\documentclass[11pt]{report}
\usepackage[left=1.2in,right=1.2in,tmargin = 25mm,bmargin = 25mm]{geometry}
\usepackage[dvipsnames]{xcolor}
\usepackage[hypertexnames = false,
            colorlinks = true,
            linkcolor = MidnightBlue,
            urlcolor  = NavyBlue,
            citecolor = RedViolet,
            anchorcolor = JungleGreen]{hyperref}
\usepackage{amsthm}
\usepackage{amssymb}
\usepackage{amsmath}
\usepackage{thmtools}
\usepackage{cleveref}
\usepackage{xeCJK} % Korean
\usepackage{tgpagella} % nice font
\usepackage[utf8x]{inputenc}
\usepackage[T1]{fontenc}
\usepackage{graphicx} % Required for inserting images
\usepackage{bm}
\usepackage{wrapfig} % for wrapfigure environment
\usepackage{pgf} 
\usepackage{tikz-cd}
\usetikzlibrary{patterns}
\usepackage{bbm} % for id etc.
\usepackage{fancyvrb} % Verbatim environment
\usepackage{bm}
\usepackage{stmaryrd} % for \llbracket and \rrbracket
\usepackage{comment}
\usepackage{ bbold }
\usepackage{enumitem}
\usepackage{bussproofs}
\usepackage{datetime}
\usepackage{perpage} % make footnote counting per page
\usepackage{setspace} % for 1.5 line spacing
\usetikzlibrary{graphs,decorations.pathmorphing,decorations.markings}

\MakePerPage{footnote}

\theoremstyle{definition}
\newtheorem{definition}{Definition}[section]
\newtheorem{remark}[definition]{Remark}

\newtheorem{example}[definition]{Example}

\newtheorem{axiom}[definition]{Axiom}

\theoremstyle{plain}
\newtheorem{proposition}[definition]{Proposition}
\newtheorem{lemma}[definition]{Lemma}
\newtheorem{theorem}[definition]{Theorem}
\newtheorem{corollary}[definition]{Corollary}

\newtheorem*{theorem*}{Theorem}

\renewcommand{\AA}{\mathcal{A}}
\newcommand{\BB}{\mathcal{B}}
\newcommand{\CC}{\mathcal{C}}
\newcommand{\DD}{\mathcal{D}}
\newcommand{\EE}{\mathcal{E}}
\newcommand{\FF}{\mathcal{F}}
\newcommand{\GG}{\mathcal{G}}
\newcommand{\HH}{\mathcal{H}}
\newcommand{\II}{\mathcal{I}}

\newcommand{\KK}{\mathcal{K}}

\newcommand{\MM}{\mathcal{M}}

\newcommand{\OO}{\mathcal{O}}
\newcommand{\PP}{\mathcal{P}}

\newcommand{\RR}{\mathcal{R}}
\renewcommand{\SS}{\mathcal{S}}

\newcommand{\N}{\mathbb{N}}

\newcommand{\st}{\,|\,}

\newcommand{\Si}{\Sigma}

\newcommand{\unit}{\mathsf{unit}}
\newcommand{\opext}{\mathrel{.}}
\newcommand{\ext}{\mathrel{\dot{\,}}}

\newcommand{\Hom}{\mathsf{Hom}}
\newcommand{\Set}{\mathsf{Set}}
\newcommand{\Psh}{\mathsf{Psh}}
\newcommand{\sSet}{\mathsf{sSet}}
\newcommand{\cSet}{\mathsf{cSet}}
\newcommand{\iso}{\cong}
\newcommand{\id}{\mathsf{id}}
\newcommand{\op}{\mathsf{op}}
\newcommand{\Cat}{\mathsf{Cat}}
\newcommand{\tcat}{\mathbb{C}\mathsf{at}}
\newcommand{\tCat}{\mathbb{C}\mathsf{AT}}
\newcommand{\Grpd}{\mathsf{Grpd}}
\newcommand{\PGrpd}{\mathsf{PGrpd}}
\newcommand{\Top}{\mathsf{Top}}

\newcommand{\cat}{\mathsf{cat}}
\newcommand{\pcat}{\mathsf{cat}_\bullet}
\newcommand{\grpd}{\mathsf{grpd}}
\newcommand{\pgrpd}{\mathsf{grpd}_\bullet}
\newcommand{\set}{\mathsf{set}}
\newcommand{\pset}{\mathsf{set}_\bullet}
\newcommand{\catp}{\mathsf{cat}^\bullet}
\newcommand{\alg}[1]{\mathsf{alg}(#1)}
\newcommand{\coalg}[1]{\mathsf{coalg}(#1)}
\newcommand{\lari}{\mathsf{lari}}

\newcommand{\two}{{\mathbb{2}}}
\newcommand{\three}{{\mathbb{3}}}

\newcommand{\opfibcart}{{\OO}}
\newcommand{\opfib}{{\overline{\opfibcart}}}
\newcommand{\smlopfibcart}{\opfibcart_{\cat}}
\newcommand{\smlopfib}{\opfib_{\cat}}
\newcommand{\fibcart}{{\FF}}
\newcommand{\fib}{{\overline{\fibcart}}}
\newcommand{\smlfib}{\fib_{\cat}}

\newcommand{\Tw}[1][]{\mathop{\mathsf{Tw}} #1}
\newcommand{\vTw}[2][]{\mathop{\mathsf{tw}_{#1}} #2}

\newcommand{\rTw}[1][]{\mathop{\mathsf{rTw}} #1}
\newcommand{\rtw}[2][]{\mathop{\mathsf{rtw}_{#1}} #2}
\newcommand{\lTw}[1][]{\mathop{\mathsf{lTw}} #1}
\newcommand{\ltw}[2][]{\mathop{\mathsf{ltw}_{#1}} #2}
\newcommand{\Op}[1][]{\mathop{\mathsf{Op}} #1}
\newcommand{\vOp}[2][]{\mathop{\mathsf{op}_{#1}} #2}
\newcommand{\Core}[1]{\mathop{\mathsf{Core}} #1}
\newcommand{\vCore}[2][]{\mathop{\mathsf{core}_{#1}} #2}

\newcommand{\src}{\mathop{\mathsf{src}}}
\newcommand{\trg}{\mathop{\mathsf{trg}}}
\newcommand{\rfl}{\mathop{\mathsf{rfl}}}
\newcommand{\inv}{\mathop{\mathsf{inv}}}
\newcommand{\inc}{\mathop{\mathsf{inc}}}

\newcommand{\U}{U}
\newcommand{\Uupp}{{\U^{\bullet}}}
\newcommand{\Ulow}{{\U_\bullet}}
\newcommand{\uupp}{{u^{\bullet}}}
\newcommand{\ulow}{{u_\bullet}}
\newcommand{\Usim}{{\U_\sim}}
\newcommand{\usim}{{u_\sim}}
\newcommand{\V}{V}
\newcommand{\Vupp}{{\V_\bullet}}
\newcommand{\Vlow}{{\V^\bullet}}
\newcommand{\vupp}{{v_\bullet}}
\newcommand{\vlow}{{v^\bullet}}
\newcommand{\Vsim}{{\V_\sim}}
\newcommand{\vsim}{{v^\sim}}
\newcommand{\Utw}{{\U_\mathsf{tw}}}
\newcommand{\Urtw}{{\U_\mathsf{rtw}}}
\newcommand{\utw}{{u_\mathsf{tw}}}
\newcommand{\urtw}{{u_\mathsf{rtw}}}
\newcommand{\uu}{{u}}
\newcommand{\vv}{{v}}
\newcommand{\pp}{{w}}

\newcommand{\Dir}{{\mathsf{Hom}}}
\newcommand{\dir}{{\mathsf{hom}}}
\newcommand{\arr}[1][]{\mathop{\mathbb{P}_{#1}}}
\newcommand{\cyl}[1][]{\mathop{\mathbb{C}_{#1}}}
\newcommand{\sym}{\mathop{\mathsf{sym}}}
\newcommand{\Path}{\mathop{\mathsf{Path}}}
\renewcommand{\path}{\mathop{\mathsf{path}}}
\newcommand{\unpath}{\mathop{\mathsf{unpath}}}
\newcommand{\Mod}{\mathop{\mathsf{Mod}}}
\renewcommand{\mod}{\mathop{\mathsf{mod}}}

\newcommand{\wCC}{\widehat\CC}
\newcommand{\wRR}{\widehat\RR}
\newcommand{\Exp}{\mathsf{Exp}}

\newcommand{\Poly}{\mathsf{Poly}}
\newcommand{\pcomp}{\triangleleft}
\newcommand{\fstProj}{\fst}
\newcommand{\sndProj}{\snd}
\renewcommand{\vert}{\mathsf{vert}}
\newcommand{\cart}{\mathsf{cart}}

\newcommand{\Unit}{\mathsf{Unit}}

\newcommand{\Ty}{\mathsf{Ty}}
\newcommand{\lam}{\mathop{\mathsf{lam}}}
\newcommand{\unlam}{\mathop{\mathsf{unlam}}}
\newcommand{\pair}{\mathsf{pair}}
\newcommand{\fst}{\mathop{\mathsf{fst}}}
\newcommand{\snd}{\mathop{\mathsf{snd}}}
\newcommand{\refl}{\mathop{\mathsf{rfl}}}
\newcommand{\Id}{\mathsf{Id}}

\newcommand{\var}{\mathsf{var}}

\renewcommand{\gg}{\mathbin{\text{\normalfont{»}}}}
\newcommand{\done}[1]{\href{#1}{$\boxbar$}}

\newcommand{\pbcorner}{
  \arrow[dr, phantom, "\scalebox{1.5}{$\lrcorner$}" , very near start]
}
\newcommand{\Sub}{\mathsf{Sub}}
\newcommand{\Sh}{\mathsf{Sh}}
\newcommand{\Topos}{\mathsf{Topos}}

\begin{document}

\begin{titlepage}
	\centering
  {\, \par}
	\vspace{2cm}
	{\Huge {\bf Internal Algebraic Type Theory}\par}
	\vspace{2cm}
	% {\huge\bfseries Pigeons love doves\par}
	% \vspace{2cm}
	{\Large Joseph Hua\par}
	\vspace{1.5cm}
  {\Large {\bf Thesis committee:} \par
    \vspace{0.4cm}
    Mathieu Anel \par
    Jeremy Avigad \par
    Steve Awodey (chair) \par
    Nicola Gambino \par }
	\vfill
	{ \large \itshape
    A dissertation submitted in fulfilment of the \par
    requirements for the degree of Doctor of Philosophy \par
    in Pure and Applied Logic.\par}
	\vfill
  { \large
    Department of Philosophy \par
    Carnegie Mellon University \par
    Pittsburgh, PA 15213 \par
    USA \par}
  \vspace{2cm}
	{\large \today \par}
% Bottom of the page
\end{titlepage}

\newpage
\vspace*{8cm}
{
  \begin{center} 
    {수빈에게}
    % To Soobin.
  \end{center}
}

\tableofcontents

\newpage

\section*{Acknowledgements}

First and foremost, I must thank Steve Awodey, who has been not only a
great mentor, but also a supportive and generous friend.
I will remember his emphasis on delivering a good talk,
and the many hours he spent helping me practice them.
Mathieu Anel and Reid Barton have both helped me tremendously,
always providing clarifying insight,
as well as forcing me to consider example after example.
I extend that gratitude towards all the committee members,
including Nicola Gambino and Jeremy Avigad,
for taking the time to read and provide feedback on this work.
I would also like to thank Emily Riehl, Bob Harper, Jonas Frey, Andrew Swan, 
Christina Bjorndahl, Jonathan Weinberger, Sina Hazrapour,
Niels van der Weide, Calum Hughes, Fernando Chu and Wojciech Nawrocki
for teaching me all sorts of things.
Finally, a thank you to my family and non-academic friends,
because I depend on them in all sorts of ways.

This material is based upon work supported by the
Air Force Office of Scientific Research under award number FA9550-21-1-0009.

\newpage
\section*{Authorship}

Much of this thesis is adapted from joint work with various coauthors.

\Cref{sec:exponentiable} is drawn from unpublished joint work with Reid Barton.
\Cref{sec:polynomials} is partly based on joint work with Yiming Xu,
which is available as a preprint \cite{hua2026}.
\Cref{sec:hottlean} is drawn from two pieces of coauthored work;
the first is from a preprint coauthored with Yi\-ming Xu \cite{hua2026},
and the second is from a long abstract submitted to the conference HoTT/UF 2026
\cite{hottuf2026} with coauthors Steve Awodey, Yiming Xu, Mario Carneiro, Sina Hazratpour,
Wojciech Nawrocki and Spencer Woolfson.
\Cref{sec:path-types} is an adaptation of joint work with Steve Awodey,
which is available as a preprint \cite{awodey2026}.
\Cref{sec:cat-as-types} is from unpublished joint work with Steve Awodey and Mathieu Anel.
% Most of the theorems about $\Cat$ are from the literature,
% and should not be credited to us.
% The axiomatization is my 

\newpage
\section*{Introduction}
This work brings us closer to applying computer-assisted, internal,
type-theoretic reasoning to a category,
such as in the category of cubical sets,
the category of groupoids, and the category of categories.
Roughly speaking, we would like to
provide efficient \emph{syntactic} proofs about \emph{semantic} objects
by writing a proof in a {domain-specific language},
and have that proof type-checked and translated
into a model of our choosing (which is also formalised in a theorem prover).
The steps we make towards this general goal are both in 
furthering the type theoretic analysis of certain categories
(\Cref{sec:exponentiable,sec:polynomials,sec:path-types,sec:cat-as-types}),
as well as implementing computer-assisted
syntax-semantic reasoning (\Cref{sec:hottlean}).

\subsection*{Background}

A \emph{domain specific language} (DSL) in our context
will mean a computer language implementing a type theory that is
tailored towards reasoning about specific models of interest.
By using a domain-specific language,
we can reason about models in a way that not only allows us to identify
when the same construction appears in different contexts,
but also to hide away details that are irrelevant to the construction being made,
elucidating the essence of a proof.

\emph{Martin-L\"of type theory} (MLTT) was developed
as a new intuitionistic foundation for mathematics \cite{martin1975,martin1998}.
It admits various extensions,
including Extensional Type Theory and Homotopy Type Theory (HoTT),
and has models in presheaves \cite{hofmann1997, awodey2024},
groupoids \cite{hofmann1995},
simplicial sets \cite{kapulkin2021} and cubical sets
\cite{bezem2014, cohen2015, awodey2026cartesian}.
One component of the HoTTLean project \cite{hua2025}
is the design of a domain-specific language implementing MLTT.

A model of MLTT necessarily admits certain categorical structure:
the presence of $\Sigma$-types in the type theory 
corresponds to a compositional structure on type families;
$\Pi$-types correspond to right adjoints of pullback functors along type families;
$\Id$-types correspond to functorial factorisation and lifting conditions.
However, as a target for interpreting the syntax of type theory,
these categorical descriptions are only stable under substitution (or pullback)
up-to-isomorphism.
This is known as the \emph{coherence problem}.

One solution to the coherence problem is to
provide and check explicit coherence equations.
This is the approach of \emph{categories with families} (CwF) \cite{clairambault2011}.
\emph{Natural models} \cite{awodey2018}
offer a simplification of CwFs
by capturing each type former as a \emph{universal} or \emph{generic}
construction on a classifier of type families,
by employing the powerful machinery of polynomial functors.

We will refer to the natural models style of semantics as \emph{algebraic}
(as in \emph{algebraic type theory} \cite{awodey2025}),
and use the word \emph{unalgebraic} in contrast.
This thesis will make the case for using both unalgebraic and algebraic semantics
(mainly for the sake of formalisation in a proof assistant):
an unalgebraic model is well-suited for writing an interpretation function
from the syntax to the model,
whereas an algebraic model is easier to construct.
Translating between the two is straightforward,
making this approach fairly practical.

In a natural model,
the classifier of all type families is necessarily \emph{external},
meaning that the object that classifies types lives in the category of presheaves
on contexts, rather than the category of contexts itself.
Unlike natural models, our approach will be \emph{internal},
meaning that we replace universal constructions on an external classifier
of all type families (in presheaves) with universal constructions
on classifiers of small type families (in the category itself).
Type theoretically,
this means that all our types are classified by a \emph{universe}.
This creates certain complications:
the category we work in is not necessarily locally Cartesian closed;
it may not even have all finite limits.
Fortunately, we demonstrate in \Cref{sec:exponentiable,sec:polynomials}
that these complications can be dealt with,
and in a way that allows us to extend our type theoretic analysis
to the category of categories $\Cat$ and the category of topological spaces $\Top$.

% For example, Lawvere algebraic theories are used to reason about
% equationally axiomatised mathematical objects such as 
% groups, rings, and modules \cite{lawvere1963}.
% Toposes admit internal reasoning in geometric logic \cite{maclane2012}.
% Simply typed $\lambda$-calculus is an internal language for
% Cartesian closed categories \cite{lambek1988}.

\subsection*{Overview of the thesis}

A significant portion of this work was motivated by
developments in HoTTLean project \cite{hua2025}.
At the very surface,
HoTTLean is a proof assistant, 
implemented as a domain-specific language (DSL) in Lean
using Lean's extensible metaprogramming features.
This DSL uses a deep embedding of MLTT in Lean,
which is then interpreted into a generic model of MLTT.
\Cref{sec:hottlean} details our contribution to this project
in formalising both the general semantics of MLTT,
as well as the Hofmann-Streicher groupoid model \cite{hofmann1995}
as a particular target for the interpretation.
This ongoing project aims to eventually fully implement
the idea of using MLTT for internal language reasoning,
with automation handling the burden of translating internal language proofs
into the model.

One outcome of HoTTLean was Steve Awodey's
cubically-inspired idea of using ``path types''
to simplify the construction of identity types in the groupoid model
(and other models as well) \cite{awodey2026}.
Instead of constructing identity types by hand,
it is enough to have path types
(which can be deduced from a closure condition on fibrations),
and a ``Hurewicz'' path-lifting condition on fibrations.
These conditions are often fairly easy to check in a model.
As a result,
identity types in the groupoid model were soon fully
formalised using path types in HoTTLean.
The path types conditions can be rephrased in a style
reminiscent of modal type theory \cite{licata2016,licata2017},
which we present in \Cref{sec:path-types}.
The main theorem of \Cref{sec:path-types} follows.
\begin{theorem*}[\labelcref{thm:id-elim}]
  Suppose $(\CC,\RR)$ is a $\pi$-preclan (\labelcref{def:pi-preclan}).
  If $\uu : \Ulow \to \U$ is an $\RR$-algebraic universe (\labelcref{def:algebraic-universe}) that
  admits $\Path$-types (\labelcref{def:mltt-alg-path-types-2})
  and a normal Hurewicz structure (\labelcref{def:hurewicz-defs}),
  then the universe also admits $\Id$-types (\labelcref{def:algebraic-id-types}).
\end{theorem*}
 
After various experiments in HoTTLean,
it became apparent that a model of MLTT (with a hierarchy of universes)
could be formulated in a manner that was internal
to the category of contexts $\CC$ itself \cite{hua2026},
rather than in the category of presheaves on contexts,
as is more typical (see \cite{clairambault2011, awodey2018}).
Roughly, this meant that we replaced the presheaf classifying all types
$\Ty : \CC^\op \to \Set$
with a hierarchy of representable presheaves $\Ty_0 \to \Ty_1 \to \Ty_2 \to \dots$,
or better yet, just with the objects (contexts)
$U_0 \to U_1 \to U_2 \to \dots$ that represent them.
This approach,
in part inspired by Uemura's notion of a representable map category
\cite{uemura2023},
proved to significantly simplify the construction of a model,
but necessitated a more general theory of polynomial functors,
since this simplification relied on an \emph{algebraic} construction
(in the sense of \cite{awodey2025}, as explained above).

Following a doctrine of working both \emph{internally} and \emph{algebraically}
has naturally led to various other results.
This includes a more general theory of exponentiability
with respect to a class of maps (a ``preclan''),
presented in \Cref{sec:exponentiable},
as well as the ensuing theory of polynomial functors,
presented in \Cref{sec:polynomials}.
Reid Barton was quick to point out a host of examples
in which the theory of exponentiability could be applied.
Insights from those examples paved the way for the rest of
the theory's development.
The main theorem, providing two equivalent definitions of
a map being \emph{$\RR$-exponentiabile} is given below.

\begin{theorem*}[\labelcref{theorem:exponentiable}]
  Let $(\CC,\RR)$ be a preclan (\labelcref{def:preclan})
  and $f : Y \to X$ a morphism in $\CC$.
  The following are equivalent.
  \begin{enumerate}
  \item
    The Beck--Chevalley condition (\labelcref{def:beck-chevalley}) holds at every pullback of
    $f$ in the preclan $(\CC,\RR)$
    (whenever such pullbacks exist).
  \item
    The pushforward in presheaves $\Pi_f : (\wCC/Y, \wRR) \to (\wCC/X, \wRR)$
    is a preclan morphism (\labelcref{def:preclan}).
  \end{enumerate}
\end{theorem*}

A new kind of dependent type theory emerges out of
the considerations of \Cref{sec:exponentiable},
which involves two different kinds of types and
which differs from MLTT in the behaviour of $\Pi$-types.
A fairly well-studied example of this is in the category of categories $\Cat$.
In short, to form a $\Pi$-type $\Pi_A B$ in $\Cat$,
$B$ must be a covariant type if $A$ is a contravariant type,
resulting in a covariant type $\Pi_A B$ (and dually).
In $\Cat$, $\Sigma$-types work in the same way as usual.
However, MLTT $\Id$-types (intensional identity types) are replaced with
$\Hom$-types that have two elimination rules,
originally due to North \cite{north2019},
relating to twisted arrow categories and various algebraic weak factorisation systems on $\Cat$.
We develop this new kind of algebraic model in \Cref{sec:cat-as-types}.
We summarise this model in the following statement.

\begin{theorem*}[\labelcref{ex:ACM-of-opfibrations}]
  There is an algebraic categorical model (\labelcref{def:ACM}) based in $\Cat$
  with a universe (\labelcref{def:ACM-universe}) classifying split opfibrations with
  $\Unit$-types, $\Sigma$-types, $\Pi$-types,
  and $\Hom$-types (\labelcref{def:ACM-universe-sigma,def:ACM-universe-pi,def:ACM-universe-hom}).
  Dually, the same holds for the universe classifying split fibrations.
\end{theorem*}

\subsection*{Structure of the thesis}
\Cref{sec:exponentiable} presents exponentiability conditions with
respect to a class of maps.
\Cref{sec:polynomials} develops polynomial functors,
using the definitions from \Cref{sec:exponentiable}.
\Cref{sec:hottlean} summarises the formalisation of the groupoid model in the HoTTLean project,
and introduces unalgebraic and algebraic models of MLTT.
The formalisation uses the theory of polynomials in $\pi$-clans,
which is an application of \Cref{sec:polynomials}.
\Cref{sec:path-types} presents an alternative construction of
algebraic identity types (from \Cref{sec:hottlean}) via path types.
Finally, \Cref{sec:cat-as-types} develops a new kind of algebraic model
in the category of categories,
providing a case study of an algebraic model based on
a $\tau$-clan (from \Cref{sec:exponentiable}).

\chapter{Exponentiable morphisms in preclans}\label{sec:exponentiable}
% This chapter presents unpublished joint work with Reid Barton \cite{barton2026}.
Many categories, such as $\Top$ and $\Cat$,
are not locally Cartesian closed,
but still have certain exponentiable objects and morphisms
\cite{niefield1982, giraud1964}.
\begin{definition}[Exponentiable map]\label{def:exp-map}
  Let $\CC$ be a category with pullbacks.
  A morphism $f : Y \to X$ in $\CC$ is \emph{exponentiable}
  when the pullback functor $f^* : \CC / X \to \CC / Y$
  has a right adjoint $\Pi_f : \CC / Y \to \CC / X$.
  The functor $\Pi_f$ is called the \emph{pushforward along $f$}.
\end{definition}
However, for analysing models of type theory,
it is important to have a more fine-grained condition of being exponentiable
relative to a class of maps $\RR$ in a category $\CC$.

For models of MLTT,
such conditions have been considered in \cite{taylor1987} and \cite{joyal2017},
in the form of $\pi$-clans (using the terminology of \cite{joyal2017}).
Uemura introduced categories with representable maps (CwRs) \cite{uemura2023}
for models of MLTT in a lex category (finite limit category).
Exponentiability conditions for $\pi$-clans are in some ways more restrictive than the
theory of exponentiable morphisms --
the example of $\Cat$ and $\Top$ are lost.
At the same time,
it introduces certain closure conditions on pushforwards --
that pushforwards preserve morphisms in the $\pi$-clan,
key to identifying a class of type families in a model.
On the other hand, the above definition of exponentiability (used in CwRs)
captures when a class of maps is exponentiable with respect to \emph{all maps},
which can recover some of examples in $\Cat$ and $\Top$,
but does not subsume the theory of $\pi$-clans
(since not all $\pi$-clans are lex)
This notion of exponentiability says nothing about the
preservation of certain morphisms upon application
of the pushforward functor.

Here, we consider morphisms that are only exponentiable with respect to
a class of maps.
Our setting not only subsumes the
definitions of exponentiability in lex categories and $\pi$-clans,
but also provides a new framework for understanding exponentiability conditions for (op)fibrations in $\Cat$,
open and closed embeddings in $\Top$,
tidy and \'etale maps of topological spaces and toposes,
and Kan fibrations in simplicial sets $\sSet$,
to name but a few examples.
This work will focus on applying this theory to
produce a study of $\Cat$ in terms of \emph{algebraic type theory},
which is the focus of \Cref{sec:cat-as-types}.
% From these considerations,
% we can glean a new kind of dependent type theory
% with models in $\Cat$, $\Top$ and so on.
% This is the focus 

\section{Preclans}

The class of maps with respect to which we consider exponentiability conditions
are required to satisfy certain basic conditions,
which have been considered in various different contexts.
Paul Taylor introduced the notion of \emph{display maps}
for the semantics of Martin-L\"of Type theory,
which could be required to satisfy various extra conditions,
each corresponding to a particular type-theoretic notion
\cite{taylor1987}.
Andr\'e Joyal called a particular combination of those conditions
a \emph{clan} \cite{joyal2017}.
For the sake of having a short name,
we introduce another abbreviation,
consistent with the ``clan'' nomenclature.

\begin{definition}\label{def:preclan}
  A \emph{preclan} $(\CC,\RR)$ consists of a category $\CC$,
  and a class of morphisms $\RR$ in $\CC$
  -- which we call $\RR$-maps -- such that 
  \begin{itemize}
    \item $\RR$ is stable under pullback,
      i.e. pullbacks of $\RR$-maps along all morphisms exist and are $\RR$-maps.
    \item All isomorphisms are in $\RR$.
    \item $\RR$ is closed under composition.
  \end{itemize}
  A \emph{preclan morphism} 
  $F : (\CC, \RR) \to (\CC', \RR')$ is a functor
  $F : \CC \to \CC'$ that
  \begin{itemize}
    \item sends $\RR$-maps to $\RR'$-maps and
    \item preserves pullbacks of $\RR$-maps,
    meaning that for any $\RR$-map $A \to X$
    and any map $Y \to X$,
    the canonical comparison map
    $F (A \times_X Y) \to F A \times_{F X} F Y$
    is an isomorphism.
  \end{itemize}
\end{definition}
Preclans, preclan morphisms,
and natural transformations naturally form a 2-category.

If we also require that $\CC$ has a terminal object
and that maps to terminal objects belong to $\RR$,
then $(\CC,\RR)$ is a \emph{clan}.
A \emph{clan morphism} is a preclan morphism between clans
that also preserves the terminal object.

The full subcategory of $\RR$-objects in a slice
-- objects whose underlying map is an $\RR$-map -- is denoted by
$\RR(X) \subseteq \CC / X$.

Like locally Cartesian closed categories,
preclans are meant to be somewhat type theoretic.
We should therefore expect that each slice
of a preclan is a preclan.
There are in fact two candidates for slicing over an object $X$,
one is $\CC / X$ and the other is $\RR(X)$.

\begin{definition}
  Let $(\CC, \RR)$ be a preclan.
  \begin{itemize}
  \item Given an object $X \in \CC$, the slice category $\CC / X$
    has a preclan structure where the chosen class of maps consists of
    morphisms in $\CC / X$ with an underlying morphism in $\RR$.
    By slight abuse of notation, we denote this preclan by $(\CC/X, \RR)$.
  \item The restriction of the above preclan structure to $\RR(X) \subseteq \CC/X$
    defines a clan $(\RR(X), \RR)$.
  \item
    The category $\wCC$ of presheaves has a preclan structure $(\wCC, \wRR)$
    \cite[Proposition 1.6.5]{joyal2017}.
    Here, $\wRR$ consists of $\RR$-representable natural transformations:
    natural transformations $A \to X$
    such that for any representable object $Y \in \CC$ and map $f : Y \to X$,
    there is an $\RR$-map $A' \to Y$ in $\CC$ representing the pullback
    \[\begin{tikzcd}
        {A'} \pbcorner & A \\
        Y & X
        \arrow[from=1-1, to=1-2]
        \arrow["{\RR\text{-map}}"', from=1-1, to=2-1]
        \arrow[from=1-2, to=2-2]
        \arrow["f"', from=2-1, to=2-2]
      \end{tikzcd}\]
  \end{itemize}
  For an object $X \in \CC$,
  there are inclusions
  \[
    \RR(X) \,\subseteq\, \CC/X \,\subseteq\, \widehat{\CC/X} \simeq \wCC/X
  \]
  These form \emph{preclan embeddings}
  $(\RR(X),\RR) \hookrightarrow \cdots \hookrightarrow (\wCC/X, \wRR)$
  in that they are fully faithful preclan morphisms that reflect $\RR$-maps.
  Furthermore, the Yoneda embedding induces an equivalence $\RR(X) \simeq \wRR(X)$,
  as a representable natural transformation with representable codomain
  also has representable domain.
\end{definition}

\begin{remark}
  For $(\RR(X), \RR)$ to be a preclan, we needed that $\RR$ is closed under composition.
  Pullbacks in $\CC / X$ are computed as pullbacks of the underlying maps in $\CC$.
  Since the inclusion $\RR(X) \to \CC / X$
  reflects pullbacks and $\RR$ is closed under composition,
  the same is also true for $\RR(X)$.
  Moreover, for $(\CC,\RR)$ satisfying only conditions
  $(1)$ and $(2)$ from the definition of preclans,
  $\RR$ is closed under composition if and only if
  $(\RR(X),\RR)$ satisfies $(1)$.
\end{remark}

\begin{definition}\label{def:partial-right-adjoint}
  A \emph{partial right adjoint} of a functor $L : \CC \to \DD$
  consists of two full subcategories
  $\partial\CC \hookrightarrow \CC$
  and $\partial\DD \hookrightarrow \DD$,
  a functor $\partial R : \partial \DD \to \partial \CC$,
  and a hom-set bijection for all
  $X \in \CC$ and $Y \in \partial\DD$
  \[ \DD(L (X), Y) \iso \CC(X,\partial R (Y))\]
  natural in both $X$ and $Y$.
  We depict such a partial right adjoint in the following manner:
  \[\begin{tikzcd}
	\CC & {\partial\CC} \\
	\DD & {\partial \DD}
	\arrow[""{name=0, anchor=center, inner sep=0}, "L"', from=1-1, to=2-1]
	\arrow[hook', from=1-2, to=1-1]
	\arrow[""{name=1, anchor=center, inner sep=0}, "{\partial R}"', from=2-2, to=1-2]
	\arrow[hook', from=2-2, to=2-1]
	\arrow["{\dashv_\partial}"{description}, draw=none, from=0, to=1]
  \end{tikzcd}\]
\end{definition}

For any $\RR$-map $f : Y \to X$ in a preclan $(\CC,\RR)$,
we have a functor $f_! : \CC / Y \to \CC / X$
between the slice categories given by composition with $f$,
with pullback as a partial right adjoint
\[\begin{tikzcd}
{\CC / Y} & {\RR(Y)} \\
{\CC / X} & {\RR(X)}
\arrow[""{name=0, anchor=center, inner sep=0}, "{f_!}"', from=1-1, to=2-1]
\arrow[hook', from=1-2, to=1-1]
\arrow[""{name=1, anchor=center, inner sep=0}, "{f^*}"', from=2-2, to=1-2]
\arrow[hook', from=2-2, to=2-1]
\arrow["{\dashv_\partial}"{description}, draw=none, from=0, to=1]
\end{tikzcd}\]

\begin{proposition}\label{prop:beck-chevalley-left}
  Let $f : Y \to X$ be an $\RR$-map in a preclan $(\CC,\RR)$.
  Then $f_! : \CC / Y \to \CC / X$ restricts to a functor
  $f_! : \RR(Y) \to \RR(X)$.
  Consider the following pullback square.
  \[
  \begin{tikzcd}
    {Y'} & Y \\
    {X'} & X
    \arrow["{s'}", from=1-1, to=1-2]
    \arrow["{f'}"', from=1-1, to=2-1]
    \arrow["\lrcorner"{anchor=center, pos=0.125}, draw=none, from=1-1, to=2-2]
    \arrow["f", from=1-2, to=2-2]
    \arrow["s"', from=2-1, to=2-2]
  \end{tikzcd}
  \]
  The associated Beck--Chevalley transformation
  $(f')_! (s')^* \to s^* f_!$ is a natural isomorphism.
  \[
  \begin{tikzcd}
    \RR(Y') \ar[d, "f'_!"'] \ar[dr, phantom, "\cong"]  &
    \RR(Y) \ar[l, "(s')^*"'] \ar[d, "f_!"] \\
    \RR(X') & \RR(X) \ar[l, "s^*"']
  \end{tikzcd}
  \]
\end{proposition}

\section{The Beck--Chevalley condition}

Variations of the following ``Beck--Chevalley condition''
are typically considered when requiring a right adjoint to pullback.
% In lex categories the ``Beck--Chevalley condition''
% holds whenever the right adjoint to pullback exists,
% making it a theorem, rather than a condition.
% In models of MLTT with $\Pi$-types,
% the ``Beck--Chevalley condition'' is automatically stable under pullback.
% We will see that for us,
% not all right adjoints satisfy the ``Beck--Chevalley condition'',
% and sometimes Beck--Chevalley holds for $f$,
% but not its pullbacks.

\begin{definition} \label{def:beck-chevalley}
  Let $f : Y \to X$ be a morphism in a preclan $(\CC,\RR)$.
  We will say that the \emph{(pushforward) Beck--Chevalley condition}
  holds at $f : Y \to X$ if for any pullback square
   \[ \begin{tikzcd}
        Y' \ar[r, "s'"] \ar[d, "f'"'] \pbcorner & Y \ar[d, "f"] \\
        X' \ar[r, "s"] & X
      \end{tikzcd} \]
    the pullback functor
    $(f')^* : \RR(X') \to \RR(Y')$
    has a right adjoint $f'_* : \RR(Y') \to \RR(X')$,
    and the associated Beck--Chevalley transformation
    $s^* f_* \to (f')_* (s')^*$
    is a natural isomorphism:
    \[
      \begin{tikzcd}
        \RR(Y') \ar[d, "f'_*"'] \ar[dr, phantom, "\cong"]  &
        \RR(Y) \ar[l, "(s')^*"'] \ar[d, "f_*"] \\
        \RR(X') & \RR(X) \ar[l, "s^*"']
      \end{tikzcd}
    \]
\end{definition}

Unlike in locally Cartesian closed categories,
the pushforward Beck--Chevalley condition does not follow
from a Beck--Chevalley isomorphism $(s')_! (f')^* \iso f^* s_!$
(like in \Cref{prop:beck-chevalley-left}) between left adjoints,
because the left adjoint to pullback
$s_! : \RR(X') \to \RR(X)$ may not even exist.
Indeed, we will see in \Cref{ex:topos-etale} that the
pushforward pushforward Beck--Chevalley condition
does not hold for all maps, 
even though we always have \Cref{prop:beck-chevalley-left}.

% We provide an equivalent formulation
% of the Beck--Chevalley condition
% (henceforth we drop the word ``pushforward'').

\begin{proposition} \label{prop:pra-defs}
  Let $f : Y \to X$ be a morphism in a preclan $(\CC,\RR)$.
  % (In particular $f^* : \RR(X) \to \RR(Y)$
  % has a right adjoint $f_* : \RR(Y) \to \RR(X)$.)
  The following conditions are equivalent
  \begin{itemize}
  \item
    The pullback functor
    $f^* : \wCC / X \to \wCC / Y$ between slices in the category of presheaves
    has a partial right adjoint (\Cref{def:partial-right-adjoint})
    $f_* : \RR(Y) \to \RR(X)$.
    \[\begin{tikzcd}
      \wCC / X & {\RR(X)} \\
      \wCC / Y & {\RR(Y)}
      \arrow[""{name=0, anchor=center, inner sep=0}, "f^*"', from=1-1, to=2-1]
      \arrow[hook', from=1-2, to=1-1]
      \arrow[""{name=1, anchor=center, inner sep=0}, "{f_*}"', dashed, from=2-2, to=1-2]
      \arrow[hook', from=2-2, to=2-1]
      \arrow["{\dashv_\partial}"{description}, draw=none, from=0, to=1]
    \end{tikzcd}\]
    
    % This means that we have
    % \[ \Hom_{\CC / Y}(f^* X', A) \iso \Hom_{\CC / X}(X', f_* A) \]
    % natural in $X' \in (\CC / X)^\op$ and $A \in \RR(Y)$.
    % In short, $f_* A$
    % is universal among all objects in the slice $\CC / X$,
    % rather than just the objects in $\RR(X)$.
  \item
    Using the Yoneda embedding $\CC \subseteq \wCC$,
    the right adjoint $\Pi_f : \wCC/Y \to \wCC/X$
    to pullback in presheaves $f^* : \wCC/X \to \wCC/Y$
    sends $\RR(Y) \subseteq \wCC/Y$
    into $\RR(X) \subseteq \wCC/X$.
    \[\begin{tikzcd}
      {\wCC / X} & {\RR(X)} \\
      {\wCC / Y} & {\RR(Y)}
      \arrow[hook', from=1-2, to=1-1]
      \arrow["{\Pi_f}", from=2-1, to=1-1]
      \arrow["{f_*}"', dashed, from=2-2, to=1-2]
      \arrow[hook', from=2-2, to=2-1]
    \end{tikzcd}\]
  \end{itemize}
\end{proposition}
\begin{proof}
  The partial right adjoint exists if and only if
  for any $A \in \RR(Y)$ there exists an object $B \in \RR(X)$
  and a natural isomorphism
  \[ \wCC/Y (f^*(-),A) \iso \wCC/X((-),B)\]
  between functors of type $(\wCC / X)^\op \to \Set$.  
  This holds if and only if
  there exists an object $B \in \RR(X)$ such that $\Pi_f A \iso B$,
  if and only if $\Pi_f$ preserves $\RR$-objects.
\end{proof}

\begin{definition} 
  We say the \emph{partial right adjoint condition}
  holds at $f : Y \to X$ if
  $f$ satisfies the equivalent conditions of
  \Cref{prop:pra-defs}.
\end{definition}

If $\CC$ is lex then the partial right adjoint condition
is equivalent to saying that the pullback
$f^* : \CC / X \to \CC / Y$
in $\CC$ has a partial right adjoint
$f_* : \RR(Y) \to \RR(X)$.
If, additionally, $f$ is exponentiable
(\Cref{def:exp-map})
then the partial right adjoint condition is also equivalent to saying that
the right adjoint $\Pi_f : \CC/Y \to \CC/X$ sends $\RR(Y)$ into $\RR(X)$.
The point is that although the functors
$f^* : \CC / X \to \CC / Y$ and $\Pi_f : \CC/Y \to \CC/X$
may not always exist,
when they do exist,
we can use them instead of their counterparts in $\wCC$.
\[
  \begin{tikzcd}
	{\RR(Y)} & {\CC / Y} & {\wCC / Y} \\
	{\RR(X)} & {\CC / X} & {\wCC / X}
	\arrow[hook, from=1-1, to=1-2]
	\arrow[""{name=0, anchor=center, inner sep=0}, "{f^*}"', shift right=3, from=1-1, to=2-1]
	\arrow[hook, from=1-2, to=1-3]
	\arrow[""{name=1, anchor=center, inner sep=0}, "{f^*}"', shift right=3, dashed, from=1-2, to=2-2]
	\arrow[""{name=2, anchor=center, inner sep=0}, "{f^*}"', shift right=3, from=1-3, to=2-3]
	\arrow[""{name=3, anchor=center, inner sep=0}, "{f_*}"', shift right=3, from=2-1, to=1-1]
	\arrow[hook, from=2-1, to=2-2]
	\arrow[""{name=4, anchor=center, inner sep=0}, "{\Pi_f}"', shift right=3, dashed, from=2-2, to=1-2]
	\arrow[hook, from=2-2, to=2-3]
	\arrow[""{name=5, anchor=center, inner sep=0}, "{\Pi_f}"', shift right=3, from=2-3, to=1-3]
	\arrow["\dashv"{anchor=center}, draw=none, from=0, to=3]
	\arrow["\dashv"{anchor=center}, draw=none, from=1, to=4]
	\arrow["\dashv"{anchor=center}, draw=none, from=2, to=5]
\end{tikzcd}\]

Next, we show that the Beck--Chevalley condition is equivalent
to the partial right adjoint condition.
This was originally observed by Thomas Streicher,
and can be found in \cite[Theorem 1.28]{streicher2012semantics}.
However,
our setting is different to Streicher's in a very subtle way.
Streicher, following Paul Taylor,
requires that $f$ itself is an $\RR$-map, whereas we do not.
We will see that the Beck--Chevalley condition in our setting
is \emph{not necessarily stable under pullback},
unlike in Taylor's setting.
The examples and \Cref{theorem:exponentiable}
will show that throwing in pullback-stability for Beck--Chevalley
actually produces a very natural notion of exponentiability.

\begin{lemma} \label{lem:pullback-global-sections}
  For $s : X' \to X$ and $A \in \RR(X)$, we have
  $\Hom_{\CC/X}(X', A) \cong \Hom_{\RR(X')}(1_{X'}, s^* A)$.
\end{lemma}
\begin{proof}
  Write $A = (A \xrightarrow{p} X)$.
  The left side consists of maps $u : X' \to A$ with $pu = s$,
  while the right side consists of sections of the pullback $s^* p : s^* A \to X'$.
  It makes sense to write $\RR(X')$ on the right hand side
  because both $1_{X'} : X' \to X'$ and $s^* p : s^* A \to X'$ are $\RR$-maps.
\end{proof}

\begin{proposition} \label{prop:beck-chevalley-defs}
  Let $f : Y \to X$ be a morphism in a preclan $(\CC,\RR)$.
  Then the Beck--Chevalley condition holds at $f$ if and only if
  the partial right adjoint condition holds at $f$.
\end{proposition}
\begin{proof}
  Since $\RR(X) \simeq \wRR(X)$,
  the Beck--Chevalley condition holds at $f : Y \to X$
  in $(\CC,\RR)$ if and only if the Beck--Chevalley condition holds at
  $f : Y \to X$ in $(\wCC,\wRR)$.
  Furthermore, since $\wCC$ is lex,
  the partial right adjoint condition holds at $f : Y \to X$
  in $(\CC,\RR)$
  if and only if it holds at $\Pi_f : \wCC / X \to \wCC / Y$
  has a partial right adjoint $f_* : \wRR(Y) \to \wRR(X)$.
  Thus, in the following, we assume that $\CC$ is lex,
  and prove that the Beck--Chevalley condition holds at $f$
  if and only if $f^* : \CC / X \to \CC / Y$
  has a partial right adjoint $f_* : \RR(Y) \to \RR(X)$.

  ($\Rightarrow$)
  We need to establish an isomorphism
  $\Hom_{\CC/X}(X', f_* A) \cong \Hom_{\CC/Y}(f^* X', A)$
  for any $(X' \xrightarrow{s} X) \in \CC/X$ and $(A \to Y) \in \RR(Y)$.
  For $Z \in \CC$, write $\Gamma : \RR(Z) \to \Set$
  for the ``global sections'' functor $\Gamma := \Hom_{\RR(Z)}(1_Z, -)$.
  Let $(Y' \xrightarrow{s'} Y) := f^* (X' \xrightarrow{s} X) \in \CC/Y$,
  with $f' : Y' \to X'$ the induced map,
  so there is a pullback square as in (1).
  Using \Cref{lem:pullback-global-sections},
  \[
    \Hom_{\CC/X}(X', f_* A)
    = \Hom_{\RR(X')}(1_{X'}, s^* f_* A)
    = \Gamma(s^* f_* A)
  \]
  and likewise
  \begin{align*}
    \Hom_{\CC/Y}(f^* X', A)
    & = \Hom_{\RR(Y')}(1_{Y'}, (s')^* A) \\
    & = \Hom_{\RR(Y')}((f')^* 1_{X'}, (s')^* A) \\
    & = \Hom_{\RR(X')}(1_{X'}, f'_* (s')^* A) \\
    & = \Gamma(f'_* (s')^* A)
  \end{align*}
  Condition (1) says that the natural map
  $s^* f_* A \to f'_* (s')^* A$ is an isomorphism,
  so we get an induced isomorphism
  \[
    \Hom_{\CC/X}(X', f_* A)
    = \Gamma(s^* f_* A)
    \cong \Gamma(f'_* (s')^* A)
    = \Hom_{\CC/Y}(f^* X', A)
  \]

  ($\Leftarrow$)
  Consider a pullback square of $f$ as shown below.
  \[
    \begin{tikzcd}
      Y' \ar[r, "s'"] \ar[d, "f'"'] \pbcorner & Y \ar[d, "f"] \\
      X' \ar[r, "s"] & X
    \end{tikzcd}
  \]
  Let $A \in \RR(Y)$.
  We must give a natural isomorphism $s^* f_* A \cong (f')_* (s')^* A \in \RR(X')$.
  We will do this by showing that $s^* f_* A$ represents
  the functor $\Hom_{\CC/Y'}((f')^*(-), (s')^*A)$ on $\CC/X'$.
  Indeed,
  \begin{align*}
    \Hom_{\CC/Y'}((f')^*(-), (s')^*A)
    & \cong \Hom_{\CC/Y}((s')_!(f')^*(-), A) \\
    & \cong \Hom_{\CC/Y}(f^* s_!(-), A) \\
    & \cong \Hom_{\CC/X}(s_!(-), f_* A) \\
    & \cong \Hom_{\CC/X'}({-}, s^* f_* A)
  \end{align*}
  In the second step, we used the two pullbacks lemma
  ($f$ has all pullbacks).
  In the remaining steps, we used the facts that
  $(s')^*$, $f_*$, $s^*$ are partial right adjoints
  to $s'_!$, $f^*$, $s_!$ respectively.
\end{proof}

By the proposition, the Beck--Chevalley condition and the partial right adjoint condition
are equivalent.
However, we will continue to use two different names
because it is often easier
to check or apply one condition over the other.

\begin{remark}
  For the Beck--Chevalley condition,
  we have not needed that $\RR$ is closed under composition,
  since we have not needed the preclan structure $(\RR(X),\RR)$.
\end{remark}

\section{Exponentiability}
\begin{theorem} \label{theorem:exponentiable}
  Let $(\CC,\RR)$ be a preclan and $f : Y \to X$ a morphism in $\CC$.
  The following are equivalent.
  \begin{enumerate}
  \item
    The Beck--Chevalley condition (\labelcref{def:beck-chevalley})
    holds at every pullback of $f$ in the preclan $(\CC,\RR)$
    (whenever such pullbacks exist).
  % \item 
  %   The Beck--Chevalley condition holds at every pullback of
  %   (the image of) $f$ in the preclan ($\wCC,\wRR$).
  \item
    The slice pushforward in presheaves $\Pi_f : (\wCC/Y, \wRR) \to (\wCC/X, \wRR)$
    is a preclan morphism.
  \end{enumerate}
  
  If $\CC$ is lex and $f$ is an exponentiable morphism
  (\Cref{def:exp-map}),
  (2) can be rephrased as saying that
  the slice pushforward $\Pi_f : (\CC/Y, \RR) \to (\CC/X, \RR)$ is a preclan morphism.
\end{theorem}

Before proving the theorem,
we provide our main definitions,
as well as some lemmas and remarks.

\begin{definition} \label{def:exponentiable}
  We say $f$ is \emph{$\RR$-exponentiable}
  if it satisfies the equivalent conditions in \Cref{theorem:exponentiable}.
  We write $\Exp_\RR$ for the class of $\RR$-exponentiable maps in $\CC$.
\end{definition}

By condition (1), $\RR$-exponentiable morphisms
are stable under pullback.
Condition (2) says that $\Pi_f : (\wCC/Y, \wRR) \to (\wCC/X, \wRR)$
not only sends $\RR(Y) \subseteq \wCC / Y$ into $\RR(X) \subseteq \wCC / X$,
as in the statement of \Cref{prop:beck-chevalley-defs},
but also preserves all $\RR$-maps in
$\RR(Y)$ and $\CC / Y$, which are subcategories of $\wCC/Y$.
Hence, assuming condition (2),
we have that the pushforward $f_* : (\RR(Y), \RR) \to (\RR(X), \RR)$
is a preclan morphism,
making it a partial right adjoint 
to pullback $f^* : (\wCC/X, \RR) \to (\wCC /Y, \RR)$
in the $2$-category of preclans.
  \[\begin{tikzcd}
    (\wCC / X, \wRR) & {(\RR(X), \RR)} \\
    (\wCC / Y, \wRR) & {(\RR(Y), \RR)}
    \arrow[""{name=0, anchor=center, inner sep=0}, "f^*"', from=1-1, to=2-1]
    \arrow[hook', from=1-2, to=1-1]
    \arrow[""{name=1, anchor=center, inner sep=0}, "{f_*}"', from=2-2, to=1-2]
    \arrow[hook', from=2-2, to=2-1]
    \arrow["{\dashv_\partial}"{description}, draw=none, from=0, to=1]
  \end{tikzcd}\]

As we explain in \Cref{sec:example},
condition~(1) is closely related to the
various notions of a map being ``proper''.
Anel and Weinberger \cite{anel2024} consider such conditions
in a different setting where the ``coefficients'' may not be given by
a full subcategory $\RR(-)$ of the slice $\CC/(-)$
(and when $\CC$ is lex, though this is minor).
In the language of comprehension categories \cite{jacobs1993},
their $\RR$ is a comprehension category over $\CC$
whereas ours is a \emph{full} one.
When both settings coincide,
their ``proper maps'' are our $\RR$-exponentiable maps,
stated via a pullback stable Beck-Chevalley condition.
The equivalence (1) $\Leftrightarrow$ (3) in \Cref{prop:beck-chevalley-defs} 
says that when both settings coincide,
their ``strict proper maps'' are the same as ``proper maps''.
% It seems that at least some of this theory works without the assumption
% of being fully faithful;
% we leave such considerations to future work.

\begin{definition}\label{def:double-preclan}
  Let $(\CC,\RR)$ be a preclan
  and $\HH$ a class of maps in $\CC$.
  We say that $\HH$ is $\RR$-exponentiable when all the
  maps in $\HH$ are $\RR$-exponentiable.

  If $(\CC,\HH)$ is also a preclan, $\HH$ is $\RR$-exponentiable
  and $\RR$ is $\HH$-exponentiable,
  then we will say that the combination
  $(\CC,\RR,\HH)$ is a $\tau$-\emph{preclan}.
\end{definition}

Naturally, if both $(\CC,\RR)$ and $(\CC,\HH)$
are clans then we would call the
combination a $\tau$-\emph{clan}.
Examples of $\tau$-preclans can be found in \Cref{ex:top-opens} and
\Cref{ex:categories-and-fibrations}.

\begin{definition}\label{def:pi-preclan}
  We call $(\CC, \RR)$ a $\pi$-\emph{preclan}
  if all $\RR$-maps are $\RR$-exponentiable.
\end{definition}

This is close to Paul Taylor's definition of a
display map category with dependent products \cite{taylor1987},
and recovers what Andr\'e Joyal calls a $\pi$-clan \cite{joyal2017}
when $(\CC,\RR)$ is a clan.

\begin{remark}[Counterexamples]\label{rmk:counterexample} 
  Exponentiable maps in a lex category
  need not be $\RR$-exponentiable
  (see \Cref{ex:categories-and-fibrations}).
  Conversely, $\RR$-exponentiable maps need not be exponentiable
  with respect to all maps, either.
  Consider the $\pi$-clan $(\CC,\RR)$, where 
  $\CC$ is the category of contexts
  (the syntactic category)
  for Martin-L\"of Type Theory with
  $\Unit$-types, $\Sigma$-types, and $\Pi$-types
  and $\RR$ is the class of type families.
  In fact, in this example, the category $\CC$ is not even lex,
  so that the class of all maps does not even produce a clan.
  
  One might wonder why the definition of an $\RR$-exponentiable morphism
  requires the Beck--Chevalley condition for all pullbacks of $f$,
  rather than just for $f$ itself.
  % Experience with examples and applications shows that
  % forcing the class of $\RR$-exponentiable maps to be pullback-stable
  % leads to a more useful notion.
  We have the following implications:
  \[
    \begin{tikzcd}
      \text{$f : Y \to X$ is $\RR$-exponentiable}
      \ar[d, Rightarrow, "\text{\,(\textit{i})}"] \\
      {\begin{array}{c}
        \text{$f$ satisfies Beck--Chevalley and $f_* : (\RR(Y), \RR) \to (\RR(X), \RR)$ is a preclan morphism}
      \end{array}}
      \ar[d, Rightarrow, "\text{\,(\textit{ii})}"] \\
      \text{$f$ satisfies Beck--Chevalley}
      \ar[d, Rightarrow, "\text{\,(\textit{iii})}"] \\
      \text{$f^* : \RR(X) \to \RR(Y)$ has a right adjoint $f_* : \RR(Y) \to \RR(X)$}
    \end{tikzcd}
  \]
  In general, none of these implications is reversible.
  We provide counterexamples for (\textit{i}), (\textit{ii}) and (\textit{iii})
  respectively in \Cref{ex:top-opens}, \Cref{ex:top-open-maps}
  and \Cref{ex:topos-etale}.
\end{remark}

We now turn to the proof of \Cref{theorem:exponentiable}.
In the proof,
we will need to analyse the functorial action of $\Pi_f : \wCC/Y \to \wCC/X$ on morphisms.
To do so, we will reintroduce \cite[Lemmas 2.3.1 and 2.3.2]{joyal2017}
in the following.
First, recall that for any functor
$F : \CC \to \DD$ and object $X \in \CC$,
there is an induced functor between slices, denoted by
\[ F / X : \CC / X \to \DD / F X\]
which acts on objects and morphisms on the slice by
applying $F$ to the underlying morphisms in $\CC$.

\begin{lemma}[{\cite[Lemma~2.3.1 + dual]{joyal2017}}] \label{lem:adjoints-of-slice-functors}
  Suppose we have an adjunction $L \dashv R : \CC \to \DD$
  and an object $X \in \CC$.
  We denote the component of the counit at $X$ as
  $\epsilon : L R X \to X$.
  Then there is an adjunction $\epsilon_! \circ L / R X \dashv R / X$,
  depicted below.

  On the other hand, suppose we have an object $Y \in \DD$
  and the component $\eta : Y \to R L Y$ of the unit at $Y$
  has pullbacks along all maps,
  so that we have a functor $\eta^* : \DD / R L Y \to \DD / Y$.
  Then there is a dual adjunction $L / Y \dashv \eta^* \circ R / L Y$.
  \[ \begin{tikzcd}
    \CC &&& {\CC / X} && {\CC / L Y} \\
    && {\CC / L R X} &&&& {\DD / R L Y} \\
    \DD &&& {\DD / R X} && {\DD / Y}
    \arrow[""{name=0, anchor=center, inner sep=0}, "R", shift left=3, from=1-1, to=3-1]
    \arrow[""{name=1, anchor=center, inner sep=0}, "{{R / X}}", from=1-4, to=3-4]
    \arrow["{R / L Y}", from=1-6, to=2-7]
    \arrow["{{\epsilon_!}}", from=2-3, to=1-4]
    \arrow["{\eta^*}", from=2-7, to=3-6]
    \arrow[""{name=2, anchor=center, inner sep=0}, "L", shift left=3, from=3-1, to=1-1]
    \arrow["{{L / R X}}", from=3-4, to=2-3]
    \arrow[""{name=3, anchor=center, inner sep=0}, "{L / Y}", from=3-6, to=1-6]
    \arrow["\dashv"{description}, draw=none, from=2-3, to=1]
    \arrow["\dashv"{description}, draw=none, from=2, to=0]
    \arrow["\dashv"{description}, draw=none, from=3, to=2-7]
  \end{tikzcd} \]
\end{lemma}
\begin{proof}
  For the first adjunction,
  let $(a : A \to R X) \in \DD / R X$ and $(b : B \to X) \in \CC / X$.
  A map $t$ in $\CC / X (\epsilon_! \circ L / R X (A), B)$
  corresponds to a diagram in $\CC$
  \[ \begin{tikzcd}
    {L A} & B \\
    {L R X} & X
    \arrow["t", dashed, from=1-1, to=1-2]
    \arrow["{L a}"', from=1-1, to=2-1]
    \arrow["b", from=1-2, to=2-2]
    \arrow["\epsilon"', from=2-1, to=2-2]
  \end{tikzcd} \]
  which transposes along the adjunction $L \dashv R$
  to a diagram in $\DD$
  \[ \begin{tikzcd}
    A & RB \\
    {R X} & {R X}
    \arrow["{\tilde{t}}", dashed, from=1-1, to=1-2]
    \arrow["a"', from=1-1, to=2-1]
    \arrow["Rb", from=1-2, to=2-2]
    \arrow[equals, from=2-1, to=2-2]
  \end{tikzcd} \]
  This is precisely a morphism
  $\tilde{t}$ in $\DD / R X (A, R / Z (B))$.
  One can verify that this bijection is natural in both $A$ and $B$.
  
  For the second adjunction,
  let $A \in \DD / Y$ and $B \in \CC / L Y$.
  A map $t$ in $\CC / L Y (L / Y (A), B)$
  corresponds to a diagram in $\CC$,
  depicted below with its transpose along $L \dashv R$ in $\DD$.
  \[ \begin{tikzcd}
    {L A} & B && A & RB \\
    {L Y} & {L Y} && Y & {R L Y}
    \arrow["t", dashed, from=1-1, to=1-2]
    \arrow["{L a}"', from=1-1, to=2-1]
    \arrow["b", from=1-2, to=2-2]
    \arrow["{\tilde{t}}", dashed, from=1-4, to=1-5]
    \arrow["a"', from=1-4, to=2-4]
    \arrow["Rb", from=1-5, to=2-5]
    \arrow[equals, from=2-1, to=2-2]
    \arrow["\eta"', from=2-4, to=2-5]
  \end{tikzcd} \]
  The transpose is a morphism in
  \[\DD / R L Y (\eta_! (A), R / L Y (B)) \cong \DD / Y (A, \eta^* \circ R / L Y)\]
  This provides the required natural bijection.
\end{proof}

It is well known that
exponentiable morphisms are stable under pullback 
(see \cite[Section 2.1]{weber2015} for example,
which uses Dubuc's adjoint triangle theorem \cite{dubuc2006}).
In the following,
we construct the pushforward functor $\Pi_{f'}$ associated
with a pullback $f'$ of an exponentiable morphism~$f$.

\begin{lemma} \label{lem:pushforward-along-pullback}
  Consider an exponentiable morphism $f : Y \to X$
  in a lex category $\CC$.
  Consider a morphism $s : X' \to X$ and its pullback along $f$.
  \[\begin{tikzcd}
    {Y'} & {X'} & {\Pi_f Y'} \\
    Y & X
    \arrow["{{{f'}}}", from=1-1, to=1-2]
    \arrow["{{{s'}}}"', from=1-1, to=2-1]
    \arrow["\lrcorner"{anchor=center, pos=0.125}, draw=none, from=1-1, to=2-2]
    \arrow["{{{\eta}}}", from=1-2, to=1-3]
    \arrow["s"', from=1-2, to=2-2]
    \arrow["{{{\Pi_f s'}}}", from=1-3, to=2-2]
    \arrow["f"', from=2-1, to=2-2]
  \end{tikzcd}\]
  where $\eta : X' \to \Pi_f Y'$ denotes the unit of the adjunction
  $f^* \dashv \Pi_f$ at $(s : X' \to X)$.
  Then the composition $\Pi_{f'} := \eta^* \circ \Pi_f / Y'$
  is right adjoint to the pullback functor
  $(f')^* : \CC / X' \to \CC / Y'$.
  \[ \begin{tikzcd}
    {\CC / Y'} & {\CC / \Pi_f Y'} & {\CC / X'}
    \arrow["{\Pi_f / Y'}", from=1-1, to=1-2]
    \arrow["{\Pi_{f'}}"', bend right, from=1-1, to=1-3]
    \arrow["{\eta^*}", from=1-2, to=1-3]
  \end{tikzcd} \]
\end{lemma}
\begin{proof}
  The left adjoint of $\eta^* \circ \Pi_f / Y'$
  is $f^* / X' : \CC / X' \to \CC / Y'$,
  by \Cref{lem:adjoints-of-slice-functors}.
  Clearly $(f')^* \cong f^* / X'$, as required. 
\end{proof}

Now suppose that we also have $g : Z \to Y$.
Then the pushforward $\Pi_f$ induces a functor
$\Pi_f / Z : \CC / Z \to \CC / \Pi_f Z$ on slices,
which we can calculate using the following lemma.

\begin{lemma}[{\cite[Lemma~2.3.2]{joyal2017}}] \label{lem:pullback-morphism-action}
  Suppose $f : Y \to X$ is an exponentiable morphism
  in a lex category $\CC$,
  and $g : Z \to Y$.
  Consider the diagram
  \[\begin{tikzcd}
    Z & {Y \times_X \Pi_f Z} & {\Pi_f Z} \\
    & Y & X
    \arrow["g"', from=1-1, to=2-2]
    \arrow["{\epsilon}"', from=1-2, to=1-1]
    \arrow["{f'}", from=1-2, to=1-3]
    \arrow[from=1-2, to=2-2]
    \arrow["\lrcorner"{anchor=center, pos=0.125}, draw=none, from=1-2, to=2-3]
    \arrow["{\Pi_f g}", from=1-3, to=2-3]
    \arrow["f"', from=2-2, to=2-3]
  \end{tikzcd}\]
  where $\epsilon : Y \times_X \Pi_f Z \to Z$
  is the counit of the adjunction $f^* \dashv \Pi_f$ at $(g : Z \to Y)$.
  Then there is a canonical natural isomorphism
  $\Pi_f/Z \iso \Pi_{f'} \circ \epsilon^*$.
  \begin{equation} \label{diag:pullback-morphism-action}
    \begin{tikzcd}
      {\CC / Z} & {\CC / \big( Y \times_X \Pi_f Z \big)} & {\CC / \Pi_f Z}
      \arrow["{{\epsilon^*}}", from=1-1, to=1-2]
      \arrow["{{\Pi_{f'}}}", from=1-2, to=1-3]
      \arrow["{{\Pi_f/Z}}"', from=1-1, to=1-3, bend right]
    \end{tikzcd}
  \end{equation}
\end{lemma}
\begin{proof}
  It suffices to construct an isomorphism between their left adjoints.
  The left adjoint of $\Pi_f / Z$ is computed as
  $\epsilon_! \circ f^* / \Pi_f Z$ by \Cref{lem:adjoints-of-slice-functors}.
  By definition,
  \[f^* / \Pi_f Z \iso (f')^* : \CC / \Pi_f Z \to \CC / \big( Y \times_X \Pi_f Z \big)\]
  By composing adjunctions $\epsilon_! \dashv \epsilon^*$ and $(f')^* \dashv \Pi_{f'}$
  we see that $\epsilon_! \circ (f')^* \dashv \Pi_{f'} \circ \epsilon^*$.
\end{proof}

It follows from \Cref{lem:pullback-morphism-action} that we have a canonical isomorphism
\begin{equation} \label{eq:type-theoretic-axiom-of-choice}
  f_* \circ A_! \iso (f_* A)_! \circ f'_* \circ \epsilon^* 
\end{equation}
This is called the \emph{distributivity law} in \cite{gambino2013}
(for $\Pi$-types and $\Sigma$-types).
It could also be called the \emph{type theoretic axiom of choice}.
In an informal type theoretic syntax, one could render it as follows.
\[
  \prod_{x : X} \sum_{y : Y(x)} A(y) \iso \sum_{f : \Pi_{X} Y} \prod_{x : X} A(f x)
\]

\begin{lemma}\label{lem:BC-iff-BC-in-presheaves}
  Let $(\CC,\RR)$ be a preclan and $f : Y \to X$ a morphism in $\CC$.
  The following are equivalent.
  \begin{itemize}
  \item
    The Beck--Chevalley condition holds at every pullback of
    $f$ in the preclan $(\CC,\RR)$
    (whenever such pullbacks exist).
  \item 
    The Beck--Chevalley condition holds at every pullback of
    (the image of) $f$ in the preclan ($\wCC,\wRR$).
  \end{itemize}
\end{lemma}
\begin{proof}
  ($\Rightarrow$)
  Suppose the partial right adjoint condition
  holds at every pullback of $f$ in $(\CC,\RR)$.
  Consider a pullback of $f$ in $\wCC$
  \[
    \begin{tikzcd}
      Y' \ar[r, "t"] \ar[d, "f'"'] \pbcorner & Y \ar[d, "f"] \\
      X' \ar[r, "s"] & X \rlap{\, .}
    \end{tikzcd}
  \]
  Note that $X'$ and $Y'$ in $\wCC$ may not be representable.
  Since $\wCC$ is locally Cartesian closed,
  it suffices to show that $\Pi_{f'} : \wCC / Y' \to \wCC / X'$
  sends $\wRR(Y')$ into $\wRR(X')$
  (this is a simplification of condition (3) in \Cref{prop:beck-chevalley-defs}).
  Let $q : A \to Y'$ be in $\wRR(Y')$
  and write $r := \Pi_{f'}(q)$.
  To check that $r$ belongs to $\wRR$,
  let $u : \widetilde X \to X'$ be a map with $\widetilde X \in \CC$;
  we must check that $u^*(r)$ is a morphism of $\CC$ that belongs to $\RR$.

  \[
    \begin{tikzcd}[row sep=small, column sep=small]
      %JH: with =small it looked like \tilde{f} is a label for another arrow
      %RB: I decided we didn't really need those "back" arrows anyways
      u^* A \ar[rd, "u^*(q)"'] & & A \ar[rd, "q"'] \\
      & \widetilde Y \ar[rr] \ar[dd, "\tilde f"'] \pbcorner & &
      Y' \ar[rr, "t"] \ar[dd, "f'"'] \pbcorner & & Y \ar[dd, "f"] \\
      u^*(\Pi_{f'} A) \ar[rd, "u^*(r)"'] & &
      \Pi_{f'} A \ar[rd, "r"'] & & \hphantom{B} \\
      & \widetilde X \ar[rr, "u"] & &
      X' \ar[rr, "s"] & & X
    \end{tikzcd}
  \]

  As shown above, write $\widetilde f : \widetilde Y \to \widetilde X$
  for the pullback $u^*(f')$.
  Since $u^* : \wCC / X' \to \wCC / \widetilde X$ preserves pushforwards,
  the morphism $u^*(r)$ is isomorphic to $\Pi_{\tilde f}(u^*(q))$.
  Now as $q \in \wRR$ and $\widetilde Y \in \CC$,
  the map $u^*(q)$ belongs to $\RR$.
  The composition $su : \widetilde X \to X$ is a morphism of $\CC$,
  so the morphism $\widetilde f : \widetilde Y \to \widetilde X$
  is a pullback of $f$ in $\CC$ (in particular, $\widetilde Y \in \CC$).
  By the hypothesis on $f$, then,
  the object $u^*(r) \cong \Pi_{\tilde f}(u^*(q))$ of $\wCC/\widetilde X$
  belongs to $\RR(\widetilde X)$.
  Thus, every pullback of $r$ to a representable $\widetilde X$
  produces a morphism in $\RR$, so $r \in \wRR$.
  
  ($\Leftarrow$)
  Now suppose the partial right adjoint condition
  holds at every pullback of $f$ in $(\wCC,\wRR)$.
  Consider a pullback
  \[
    \begin{tikzcd}
      Y' \ar[r, "t"] \ar[d, "f'"'] \pbcorner & Y \ar[d, "f"] \\
      X' \ar[r, "s"] & X
    \end{tikzcd}
  \]
  of $f$ in $\CC$.
  We must show that $\Pi_{f'} : \wCC / Y' \to \wCC / X'$
  sends $\RR(Y')$ into $\RR(X')$
  (condition (3) in \Cref{prop:beck-chevalley-defs}).
  Since the image of the given pullback
  is still a pullback in $\wCC$,
  our assumption says that $\Pi_{f'} : \wCC / Y' \to \wCC / X'$
  sends $\wRR(Y')$ into $\wRR(X')$.
  We conclude by recalling that
  $\RR(Z) \simeq \wRR(Z)$ for any $Z$ in $\CC$.
\end{proof}

\begin{proposition}\label{prop:lex-exponentiable}
  Let $(\CC,\RR)$ be a preclan such that
  $\CC$ is lex and $f : Y \to X$
  an exponentiable morphism in $\CC$.
  The following are equivalent
  \begin{itemize}
  \item
    The Beck--Chevalley condition holds at every pullback of
    $f$ in the preclan $(\CC,\RR)$.
  \item
    The slice pushforward in presheaves $\Pi_f : (\CC/Y, \RR) \to (\CC/X, \RR)$
    is a preclan morphism.
  \end{itemize}
\end{proposition}
\begin{proof}
  % We use
  % the partial right adjoint condition
  % rather than the Beck--Chevalley condition.

  ($\Rightarrow$)
  Suppose $\CC$ is lex, $f : Y \to X$ is exponentiable, and
  the partial right adjoint condition holds at every pullback of $f$.
  We must show that $\Pi_f : \CC / Y \to \CC / X$ is a preclan morphism.
  Note that $\Pi_f : \CC / Y \to \CC / X$ is a right adjoint,
  thus it preserves all pullbacks.
  It remains to show that for any $\RR$-map $h : A \to Z$ over $Y$,
  \[\begin{tikzcd}
    A & Z & Y & X
    \arrow["h", from=1-1, to=1-2]
    \arrow["g", from=1-2, to=1-3]
    \arrow["f", from=1-3, to=1-4]
  \end{tikzcd}\]
  the pushforward $\Pi_f h : \Pi_f A \to \Pi_f Z$
  is also an $\RR$-map.

  By \Cref{lem:pullback-morphism-action},
  we can analyse $\Pi_f h$ by viewing $h : A \to Z$ as an object in $\CC / Z$.
  In our case we can further view $h : A \to Z$ as an object in $\RR(Z)$
  and use that $\Pi_{f'}$ restricts to $f'_*$,
  extending \cref{diag:pullback-morphism-action}.
  \[\begin{tikzcd}
    {\RR(Z)} & {\RR(Y \times_X \Pi_f Z)} & {\RR(\Pi_f Z)} \\
    {\CC / Z} & {\CC / \big( Y \times_X \Pi_f Z \big)} & {\CC / \Pi_f Z}
    \arrow["{\epsilon^*}", from=1-1, to=1-2]
    \arrow[from=1-1, to=2-1]
    \arrow["{f'_*}", from=1-2, to=1-3]
    \arrow[from=1-2, to=2-2]
    \arrow[from=1-3, to=2-3]
    \arrow["{\epsilon^*}", from=2-1, to=2-2]
    \arrow["{\Pi_{f'}}", from=2-2, to=2-3]
    \arrow["{\Pi_f/Z}", from=2-1, to=2-3, bend right]
  \end{tikzcd}\]
  Hence $\Pi_f h$ can be viewed as an object in $\RR(\Pi_f Z)$,
  which is an $\RR$-map in $\CC$:
  \[
    \Pi_f h
    \,\iso\, \Pi_{f'} \, \epsilon^* \, h
    \,\iso\, f'_* \, \epsilon^* \, h \,.
  \]

  ($\Leftarrow$)
  Suppose $\CC$ is lex, $f : Y \to X$ is exponentiable,
  and $\Pi_f : (\CC/Y,\RR) \to (\CC/X,\RR)$ is a preclan morphism.
  Let $f' : Y' \to X'$ be a pullback of $f$
  \[
    \begin{tikzcd}
      Y' \ar[r, "s'"] \ar[d, "f'"'] \pbcorner & Y \ar[d, "f"] \\
      X' \ar[r, "s"] & X
    \end{tikzcd}
  \]
  By \Cref{lem:pushforward-along-pullback} there is
  a pullback functor
  $\Pi_{f'} = \eta^* \circ \Pi_f / Y' : \CC / Y' \to \CC / X'$.
  Since $\Pi_f$ is a preclan morphism,
  $\Pi_f / Y'$ sends $\RR(X')$ into $\RR(\Pi_f Y')$.
  Pullbacks always preserve $\RR$-maps,
  so $\eta^*$ sends $\RR(\Pi_f Y')$ into $\RR(Y')$.
  Hence $\Pi_{f'}$ sends $\RR(Y')$ into $\RR(X')$
  and so the partial right adjoint condition holds at $f'$.
\end{proof}

We can now prove \Cref{theorem:exponentiable},
by combining the previous propositions.

\begin{proof}[Proof of \Cref{theorem:exponentiable}.]
  Let $(\CC,\RR)$ be a preclan and $f : Y \to X$ a morphism in $\CC$.
  The Beck--Chevalley condition holds at every pullback of
  $f$ in the preclan $(\CC,\RR)$
  if and only if 
  it holds at every pullback of $f$ in $(\wCC,\wRR)$,
  by \Cref{lem:BC-iff-BC-in-presheaves}.
  The latter holds if and only if
  the slice pushforward in presheaves $\Pi_f : (\wCC/Y, \wRR) \to (\wCC/X, \wRR)$
  is a preclan morphism,
  by \Cref{prop:lex-exponentiable}
  and the fact that $\wCC$ is locally Cartesian closed.
\end{proof}

The usual definition of a $\pi$-(pre)clan follows.

\begin{corollary} \label{cor:pi-preclan}
  Let $(\CC, \RR)$ be a preclan.
  The following are equivalent
  \begin{enumerate}
  \item
    All $\RR$-maps are $\RR$-exponentiable,
    i.e. $(\CC, \RR)$ is a $\pi$-preclan.
  \item
    The Beck--Chevalley condition holds at all $\RR$-maps.
  \end{enumerate}
  In this case, $\Pi_f : (\wCC / X, \RR) \to (\wCC / Y, \RR)$
  is a preclan morphism for every $\RR$-map $f$,
  and so is $\Pi_f : (\CC / X, \RR) \to (\CC / Y, \RR)$
  when $\CC$ is lex and $f$ is exponentiable.
\end{corollary}
\begin{proof}
  Since $\RR$-maps are stable under pullback,
  a fact about all morphisms in $\RR$ is true if and only if
  it is true for all pullbacks of morphisms in $\RR$.
\end{proof}

\begin{proposition} \label{prop:exp-clan}
  If $(\CC,\RR)$ is a preclan
  then $(\CC,\Exp_\RR)$ is also a preclan.
\end{proposition}
\begin{proof}
  Pullbacks of $\RR$-exponentiable morphisms exist and
  are still $\RR$-exponentiable by considering condition (1)
  from the equivalent definitions of $\RR$-exponentiability.
  If $f$ is an isomorphism, then the Beck--Chevalley condition holds trivially
  at $f$ as well as its pullbacks, which are also isomorphisms.
  If $g : Z \to Y$ and $f : Y \to X$ are both $\RR$-exponentiable,
  then by condition (3)
  \[\Pi_f \circ \Pi_g \iso \Pi_{f \circ g} : (\wCC/Z, \wRR) \to (\wCC/X, \wRR)\]
  we see that $f \circ g$ is also $\RR$-exponentiable.
\end{proof}

The remainder of this section provides various lemmas
about $\RR$-exponentiability that will be useful in the coming chapters.
The first lemma \Cref{lem:stable-class-of-exponentiable-maps}
makes it easier to check that a class of maps $\HH$ is $\RR$-exponentiable,
using a more elementary condition.

\begin{lemma}\label{lem:stable-class-of-exponentiable-maps}
  Suppose $\HH$ is a class of maps that is stable under pullback,
  and for every $f : Y \to X$ in $\HH$,
  and every $\RR$-map $g : Z \to Y$, 
  there is an $\RR$-map $p : P \to X$ satisfying the partial right adjoint
  universal property of the pushforward:
  \[ \CC / Y (f^* h, g) \iso \CC / X (h,p)\]
  naturally for all $h$ in $\CC / X$.
  Then $\HH$ is $\RR$-exponentiable.
  \[\HH \subseteq \Exp_\RR\]
\end{lemma}
\begin{proof}
  Let $f : Y \to X$ be in $\HH$.
  By assumption on $\HH$, $f$ has all pullbacks.
  The elementary condition is sufficient for constructing a
  partial right adjoint $f_* : \RR(Y) \to \RR(X)$
  for the pullback functor $f^* : \CC / X \to \CC / Y$.
  \[\begin{tikzcd}
    \CC / X & {\RR(X)} \\
    \CC / Y & {\RR(Y)}
    \arrow[""{name=0, anchor=center, inner sep=0}, "f^*"', from=1-1, to=2-1]
    \arrow[hook', from=1-2, to=1-1]
    \arrow[""{name=1, anchor=center, inner sep=0}, "{f_*}"', from=2-2, to=1-2]
    \arrow[hook', from=2-2, to=2-1]
    \arrow["{\dashv_\partial}"{description}, draw=none, from=0, to=1]
  \end{tikzcd}\]
  Since $\HH$ is stable under pullback,
  the existence of the partial right adjoint also holds for all pullbacks of $f$.
\end{proof}

Like in \cite[Lemma 2.4.13]{joyal2017},
we also provide versions of 
\Cref{lem:pushforward-along-pullback}
and \Cref{lem:pullback-morphism-action}
for preclan (rather than a lex category).
These lemmas aid in various computations
involving the pushforward,
which we make use of for the theory of polynomial functors in \Cref{sec:polynomials}.

For a preclan morphism $F : (\CC, \RR) \to (\CC',\RR')$
and $X \in \CC$,
the action of $F$ on morphisms induces a functor
\[F / X : \RR(X) \to \RR'(F X)\]
Then for an $\RR$-map $f : X \to Y$
we have $(F f)_! \circ F / X \iso F / Y \circ f_!$.
\[
  \begin{tikzcd}
    {\RR(X)} & {\RR'(F X)} \\
    {\RR(Y)} & {\RR'(F Y)}
    \arrow["{F / X}", from=1-1, to=1-2]
    \arrow["{f_!}"', from=1-1, to=2-1]
    \arrow["{(F f)_!}", from=1-2, to=2-2]
    \arrow["F / Y"', from=2-1, to=2-2]
  \end{tikzcd}
\]

% In particular for 
% \begin{definition}
%   Let $c,d \in \CC$ and $F : (\RR(c), \RR) \to (\RR(d), \RR)$
%   be a preclan morphism.
%   For an object $X$ in $\RR(c)$,
%   the action of $F$ on morphisms induces a functor
%   \[F / X : \RR(X) \to \RR(F X)\]
%   Writing $X_! : \RR(X) \to \RR(c)$ for the functor composing
%   the map $X \to c$,
%   it follows that the following square commutes.
%   \[
%   \begin{tikzcd}
%     {\RR(X)} & {\RR(F X)} \\
%     {\RR(c)} & {\RR(d)}
%     \arrow["{F / X}", from=1-1, to=1-2]
%     \arrow["{X_!}"', from=1-1, to=2-1]
%     \arrow["{(F X)_!}", from=1-2, to=2-2]
%     \arrow["F"', from=2-1, to=2-2]
%   \end{tikzcd}
%   \]
% \end{definition}

\begin{lemma}
  Let $f : Y \to X$ be $\RR$-exponentiable,
  and let $A$ be an object in $\RR(X)$,
  which we pull back along $f$.
  Consider the diagram
  \[\begin{tikzcd}
    {f^* A} & {A} & {f_* f^* A} \\
    Y & X
    \arrow["{{{f'}}}", from=1-1, to=1-2]
    \arrow[from=1-1, to=2-1]
    \arrow["\lrcorner"{anchor=center, pos=0.125}, draw=none, from=1-1, to=2-2]
    \arrow["{{{\eta}}}", from=1-2, to=1-3]
    \arrow[from=1-2, to=2-2]
    \arrow["{{{f_* s'}}}", from=1-3, to=2-2]
    \arrow["f"', from=2-1, to=2-2]
  \end{tikzcd}\]
  where $\eta : A \to f_* f^* A$ denotes the unit of the adjunction
  $f^* \dashv f_*$ at $A$.
  Then there is a canonical isomorphism 
  ${f'_*} \iso \eta^* \circ f_* / f^* A$.
  \[ \begin{tikzcd}
    {\RR(f^* A)} & {\RR(f_* f^* A)} & {\RR(A)}
    \arrow["{f_* / f^* A}", from=1-1, to=1-2]
    \arrow["{{f'_*}}"', bend right, from=1-1, to=1-3]
    \arrow["{\eta^*}", from=1-2, to=1-3]
  \end{tikzcd} \]
\end{lemma}
\begin{proof}
  Embedding everything into presheaves $\RR(A) \to \wCC / A$,
  and noting that all functors commute with the embeddings,
  the result follows from \cite[Lemma 1.13]{barton2026}.
\end{proof}

\begin{lemma}\label{lem:distributivity}
  Let $f : Y \to X$ be $\RR$-exponentiable,
  and let $A$ be an object in $\RR(Y)$.
  Consider the diagram
  \[\begin{tikzcd}
    A & {f^* f_* A} & {f_* A} \\
    & Y & X
    \arrow[from=1-1, to=2-2]
    \arrow["{\epsilon}"', from=1-2, to=1-1]
    \arrow["{f'}", from=1-2, to=1-3]
    \arrow[from=1-2, to=2-2]
    \arrow["\lrcorner"{anchor=center, pos=0.125}, draw=none, from=1-2, to=2-3]
    \arrow[from=1-3, to=2-3]
    \arrow["f"', from=2-2, to=2-3]
  \end{tikzcd}\]
  where $\epsilon : f^* f_* A \to A$
  is the counit of the adjunction $f^* \dashv f_*$ at $A$.
  Then there is a canonical natural isomorphism
  $f_*/A \iso {f'}_* \circ \epsilon^*$.
  \[
    \begin{tikzcd}
      {\RR(A)} & {\RR(f^* f_* A)} & {\RR(f_* A)}
      \arrow["{{\epsilon^*}}", from=1-1, to=1-2]
      \arrow["{f'_*}", from=1-2, to=1-3]
      \arrow["{f_*/A}"', from=1-1, to=1-3, bend right]
    \end{tikzcd}
  \]
\end{lemma}
\begin{proof}
  Again,
  embedding everything into presheaves $\RR(A) \to \wCC / A$,
  the result follows from \cite[Lemma 1.14]{barton2026}.
\end{proof}

Like with \Cref{lem:pullback-morphism-action}, 
it follows from \Cref{lem:distributivity} that we have a \emph{distributivity law}
or \emph{type theoretic axiom of choice}
(cf. \cite[Proposition 2.4.15]{joyal2017}).
\begin{equation}
  f_* \circ A_! \iso (f_* A)_! \circ f'_* \circ \epsilon^* 
\end{equation}
Again, in an informal type theoretic syntax, one could render it as follows.
\[
  \prod_{x : X} \sum_{y : Y(x)} A(y) \iso \sum_{f : \Pi_{X} Y} \prod_{x : X} A(f x)
\]

\begin{lemma}\label{lem:pushforward-preserves-maps}
  Let $f : Y \to X$ be $\RR$-exponentiable.
  Suppose $(\CC,\SS)$ is another preclan such that $\SS \subseteq \RR$
  and $f$ is also $\SS$-exponentiable.
  Then the pushforward functor
  \[ f_* : (\RR(Y), \SS) \to (\RR(X), \SS)\]
  is a preclan morphism.
\end{lemma}
\begin{proof}
  The proof is similar to (1) $\Rightarrow$ (4) in the proof of \Cref{theorem:exponentiable}.
  Being a right adjoint, $f_* : \RR(Y) \to \RR(X)$ preserves all pullback squares.
  It remains to show that $f_*$ preserves $\SS$-maps.
  Consider an $\SS$-map $h : A \to B$ over $Y$.
  \[\begin{tikzcd}
    A & B & Y & X
    \arrow["h", from=1-1, to=1-2]
    \arrow[from=1-2, to=1-3]
    \arrow["f", from=1-3, to=1-4]
  \end{tikzcd}\]
  Recall the diagrams from \Cref{lem:distributivity}.
  \[\begin{tikzcd}
    B & {f^* f_* B} & {f_* B} \\
    & Y & X
    \arrow[from=1-1, to=2-2]
    \arrow["{\epsilon}"', from=1-2, to=1-1]
    \arrow["{f'}", from=1-2, to=1-3]
    \arrow[from=1-2, to=2-2]
    \arrow["\lrcorner"{anchor=center, pos=0.125}, draw=none, from=1-2, to=2-3]
    \arrow[from=1-3, to=2-3]
    \arrow["f"', from=2-2, to=2-3]
  \end{tikzcd}
  \quad \quad
  \begin{tikzcd}
    {\SS(B)} & {\SS(f^* f_* B)} & {\SS(f_* B)} \\
    {\RR(B)} & {\RR(f^* f_* B)} & {\RR(f_* B)}
    \arrow["{{{\epsilon^*}}}", from=1-1, to=1-2]
    \arrow[from=1-1, to=2-1]
    \arrow["{{f'_*}}", from=1-2, to=1-3]
    \arrow[from=1-2, to=2-2]
    \arrow[from=1-3, to=2-3]
    \arrow["{{{\epsilon^*}}}", from=2-1, to=2-2]
    \arrow["{{f_*/B}}"', bend right, from=2-1, to=2-3]
    \arrow["{{f'_*}}", from=2-2, to=2-3]
  \end{tikzcd}
  \]
  Although $A$ and $B$ in $\RR(Y)$ might not be in $\SS(Y)$,
  $h : A \to B$ is an object in $\SS(B) \subseteq \RR(B)$,
  and therefore we can compute
  \[f_* h \iso f'_* \, \epsilon^* \, h\]
  as an object in $\SS(f_* B) \subseteq \RR(f_* B)$,
  making $f_* h$ an $\SS$-map.
\end{proof}

\section{Examples} \label{sec:example}

\begin{example}
  Let $\CC$ be a lex category
  and let $\RR$ consist of all morphisms in $\CC$.
  Then $(\CC,\RR)$ forms a clan.
  Functors between slices always preserve $\RR$-maps,
  making condition (4) of \Cref{theorem:exponentiable} trivial.
  Thus, $\RR$-exponentiability reduces to
  the standard notion of exponentiability of a morphism in a lex category.
  In terms of the pullback-pushforward adjunction,
  since $\RR(X) = \CC / X$ for all $X$,
  the partial right adjoint to pullback $f^* : \CC / X \to \CC / Y$ becomes
  an actual right adjoint $f_* = \Pi_f : \CC/Y \to \CC/X$.
  % Conversely,
  % any exponentiable morphism $f : Y \to X$ has a
  % (global) pullback-pushforward adjunction,
  % and $\Pi_f : \CC / Y \to \CC / X$ will preserve pullbacks --
  % since it is a right adjoint -- and will preserve $\RR$-maps trivially.
\end{example}

\begin{example}[Topological spaces, open embeddings and closed embeddings]
  \label{ex:top-opens}
  Consider the category of topological spaces $\Top$.
  We write $\OO$ for the class of open embeddings,
  i.e. those maps homeomorphic to the inclusion of an open subset,
  and $\HH$ for the class of closed embeddings,
  i.e. those maps homeomorphic to the inclusion of a closed subset.
  These form a $\tau$-preclan $(\Top,\OO,\HH)$ (\Cref{def:double-preclan}).

  It is easy to check that embeddings and open embeddings form clans in $\Top$.
  Writing $\Sub$ for the class of embeddings,
  we can define a right adjoint to $f^* : \Sub(X) \to \Sub(Y)$ as follows.
  For $U \subseteq Y$ a subspace,
  let $(\forall f)(U) \subseteq X$ denote the subspace
  \[\{\, x \in X \mid f(y) = x \Rightarrow y \in U \,\}\]
  It is easy to see that $(\forall f)(U) \in \Top/X$
  represents the functor $\Hom_{\Top/Y}(f^*(-), U)$.
  \[ \begin{tikzcd}
    {\Top / Y} & {\Sub(Y)} & {\OO(Y)} \\
    {\Top / X} & {\Sub(X)} & {\OO(X)}
    \arrow[from=1-2, to=1-1]
    \arrow[""{name=0, anchor=center, inner sep=0}, "{\forall f}", shift left=3, from=1-2, to=2-2]
    \arrow[hook', from=1-3, to=1-2]
    \arrow["{f^*}", from=2-1, to=1-1]
    \arrow[""{name=1, anchor=center, inner sep=0}, "{f^*}", shift left=3, from=2-2, to=1-2]
    \arrow[from=2-2, to=2-1]
    \arrow["{f^*}", from=2-3, to=1-3]
    \arrow[hook', from=2-3, to=2-2]
    \arrow["\dashv"{description}, draw=none, from=1, to=0]
  \end{tikzcd} \]
  
  A map $f : Y \to X$ is $\OO$-exponentiable
  if and only if it is universally closed
  (all pullbacks of $f$ are closed).
  Indeed, we show that
  the partial right adjoint condition holds at $f : Y \to X$
  if and only if $f$ is closed ($f$-image of a closed subset is closed):
  the partial right adjoint condition holds at $f$
  if and only if 
  $\forall f : \Sub(Y) \to \Sub(X)$
  sends $\OO(Y) \subseteq \Sub(Y)$ into $\OO(X) \subseteq \Sub(X)$,
  i.e. if $U \subseteq Y$ is open then 
  $(\forall f)(U) \subseteq X$ is open.
  By passage to complements, this is the same as asking that the map $f$ is closed.
  
  By a dual argument,
  the partial right adjoint condition with respect to the closed embeddings
  $\HH$ holds at $f : Y \to X$
  if and only if $f$ is an open map ($f$-image of an open subset is open).
  Since open maps are stable under pullback (unlike closed maps; see below),
  it follows that
  $f$ is $\HH$-exponentiable if and only if
  $f$ is an open map.
  \[\OO \subseteq \{\text{open maps}\} = \Exp_\HH\]
  Also, any closed embedding is universally closed.
  \[\HH \subseteq \Exp_\OO\]
  Thus, open embeddings and closed embeddings together
  form a $\tau$-preclan $(\Top,\OO,\HH)$.

  In general, a map $f : Y \to X$ is universally closed
  if and only if it is closed and the preimage of any compact subset of $X$ is compact
  \cite[Theorem 055R]{stacks-project}.
  For maps between locally compact Hausdorff spaces,
  the assumption that $f$ is closed can be omitted.
  Depending on the source, the term ``proper'' can refer to being universally closed,
  just the condition that preimages of compact subsets are compact,
  or also include the condition that $f$ is separated.

  We see that the Beck--Chevalley condition holding at $f$
  does not imply that it holds at pullbacks of $f$,
  since there are maps $f : Y \to X$ that are closed but not universally closed.
  Consider $f : \N \to 1$.
  Though $f$ is closed, the pullback $X \times \N \to X$ may not be closed.
  An $\N$-indexed union of closed subsets $U_i$ of $X$ may not be closed,
  but the disjoint union of the subsets $U_i \times \{i\}$ in $\N \times X$ is closed.
  
  We can also see that when the Beck--Chevalley condition holds at $f : Y \to X$,
  the pushforward
  \[f_* : (\OO(Y), \OO) \to (\OO(X), \OO)\]
  is automatically a preclan morphism,
  because all morphisms between objects of $\OO(X)$ are open embeddings.
  This shows that the converse of (\textit{i}) in \Cref{rmk:counterexample}
  is not true:
  $f$ satisfies Beck–-Chevalley and $f_* : (\OO(Y), \OO) \to (\OO(X), \OO)$
  is a preclan morphism, but $f$ is not $\OO$-exponentiable.
\end{example}

\begin{example}[Topological spaces and open maps]
  \label{ex:top-open-maps}
  Another choice of preclan structure $\RR$ on $\Top$
  is the class of open maps.
  In fact, this $(\Top, \RR)$ is a clan.
  We use it to give an example of
  a map $f : Y \to X$ that satisfies the Beck--Chevalley condition
  but for which $f_* : (\RR(Y), \RR) \to (\RR(X), \RR)$ is not a preclan morphism.

  Take the unique map $f : \N \to 1$,
  where $\N$ has the discrete topology.
  Since $\N$ is locally compact, $f$ is exponentiable \cite{niefield1982}.
  Moreover, every map into $1$ (or $\N$ for that matter) is an open map,
  so $\Pi_f : \Top/\N \to \Top/1 = \Top$ certainly sends $\RR(\N)$ into $\RR(1)$.
  However, we claim that $f_* = \Pi_f$ does not preserve open maps.
  Since the exponential $(-)^\N : \Top \to \Top$
  factors as $(-)^\N = \Pi_f \circ f^*$,
  and $f^* : \Top \to \Top / \N$ preserves open maps,
  it suffices to show that
  $(-)^\N = \Pi_f \circ f^*$ does not preserve open maps.
  To see this,
  take the open point $1$ in the Sierpinski space $\mathbb{S} = \{0 < 1\}$,
  which gives us an open embedding $i : \{1\} \to \mathbb{S}$ by inclusion.
  We want to show that \[i^\N : \{1\}^\N \to \mathbb{S}^\N\]
  is not an open map,
  by showing that the constant map $c_1 : \N \to \mathbb{S}$
  at $1 \in \mathbb{S}$
  is not an open point in $\mathbb{S}^\N$.

  Recall that $\mathbb{S}^\N$ is given the Scott topology \cite{niefield1982};
  in our case we can provide a basis for the Scott topology on $\mathbb{S}^\N$
  where a basic open is of the form 
  \[\{f : \Top(\N,\mathbb{S}) \st f(F) \subseteq \{1\} \}\]
  for some finite set of naturals $F \subseteq \N$.
  Suppose that the constant function $c_1$ were an open point in $\mathbb{S}^\N$.
  Then we would have
  \[ \{c_1\} = \bigcup_{i} \{f : \Top(\N,\mathbb{S}) \st f(F_i) \subseteq \{1\} \}\]
  for some finite sets $F_i \subset \N$.
  There is some $j$ such that $c_1 (F_j) \subseteq \{1\}$.
  Consider the continuous map $g : \N \to \mathbb{S}$ that is constant at $1$
  on $F_j$, and constant at $0$ elsewhere.
  Then 
  \[g \in \bigcup_{i} \{f : \Top(\N,\mathbb{S}) \st f(F_i) \subseteq \{1\} \}\]
  is in the union of opens, but $g \ne c_1$.
\end{example}

\begin{example}[Simplicial sets and Kan fibrations]
  \label{ex:sset-kan}
  Consider the category of simplicial sets $\sSet$
  with the Kan--Quillen model structure.
  Writing $\RR$ for the Kan fibrations,
  we have that $(\sSet, \RR)$ is a $\pi$-preclan \cite{joyal2017}.
  Since $\sSet$ is locally Cartesian closed,
  it is enough to check the ``Frobenius property'':
  pushforwards of fibrations along fibrations are fibrations.

  Whilst all fibrations are $\RR$-exponentiable,
  not all $\RR$-exponentiable maps are fibrations
  (an example is given below).
  The following conditions on $f : Y \to X$ are equivalent
  to $f$ being $\RR$-exponentiable.
  \begin{enumerate}
  \item
    $f^* : \sSet/X \to \sSet/Y$ preserves weak equivalences,
    i.e., in any ``two pullbacks'' diagram
    \[
      \begin{tikzcd}
        Y'' \ar[r, "t'"] \ar[d, "f''"'] \pbcorner &
        Y' \ar[r] \ar[d, "f'"'] \pbcorner & Y \ar[d, "f"] \\
        X'' \ar[r, "t"] & X' \ar[r] & X
      \end{tikzcd}
    \]
    if $t$ is a weak equivalence, then so is $t'$.
    This is Rezk's definition of a \emph{sharp map} \cite{rezk1998}.
  \item
    Every pullback of $f$ is a homotopy pullback.
  \item
    $f^* \dashv \Pi_f$ is a Quillen adjunction.
  \item
    $\Pi_f : \sSet/Y \to \sSet/X$ preserves fibrations
    ($f$ is $\RR$-exponentiable).
  \item
    $\Pi_f : \sSet/Y \to \sSet/X$ preserves fibrant objects
    (the partial right adjoint condition holds at $f$).
  \end{enumerate}

  \begin{proof}
    (1) $\Leftrightarrow$ (2) is
    proven in \cite[Proposition 2.7]{rezk1998}.

    (1) $\Leftrightarrow$ (3) $\Leftrightarrow$ (4).
    The cofibrations of $\sSet$ are the monomorphisms,
    which are stable under pullback.
    Therefore, $f^*$ always preserves cofibrations
    and by the adjunction, $\Pi_f$ always preserves trivial fibrations.
    Furthermore,
    $f^*$ preserves weak equivalences if and only if it preserves trivial cofibrations
    if and only if $\Pi_f$ preserves fibrations.

    (4) $\Rightarrow$ (5).
    Preserving fibrations implies preserving fibrant objects.

    (5) $\Rightarrow$ (1).
    A map between cofibrant objects of any model category is a weak equivalence
    just when it induces an isomorphism on
    homotopy classes of maps into every fibrant object.
    In $\sSet$ and its slices, we can use the product $A \times \Delta^1$
    as the cylinder object to compute the (left) homotopy relation.
    Let $A \in \sSet/X$, and $Z \in \sSet/Y$.
    In addition to the adjunction
    \[
      \Hom_{/Y}(f^* A, Z) \cong \Hom_{/X}(A, \Pi_f Z),
    \]
    we also have
    \[
      \Hom_{/Y}(f^*(A) \times \Delta^1, Z) =
      \Hom_{/Y}(f^*(A \times \Delta^1), Z) \cong \Hom_{/X}(A \times \Delta^1, \Pi_f Z),
    \]
    Therefore, when $Z$ is fibrant, $\Pi_f Z$ is also fibrant by (5)
    and the adjunction isomorphism passes to homotopy classes of maps
    \[
      [f^* A, Z]_{/Y} \cong [A, \Pi_f Z]_{/X}.
    \]
    This isomorphism is natural in $A$.
    So if $g : A \to B$ is a weak equivalence in $\sSet_{/X}$,
    then 
    \[g^* : [B, \Pi_f Z]_{/X} \to [A, \Pi_f Z]_{/X}\]
    is an isomorphism for all fibrant $Z \in \sSet/X$,
    and so is \[(f^*g)^* : [f^* B, Z]_{/Y} \to [f^* A, Z]_{/Y}\]
    so $f^*g : f^* A \to f^* B$ is a weak equivalence.
    This proves (1).
  \end{proof}
  
  We provide an example of an $\RR$-exponentiable map that is not a fibration.
  One can show that a weak equivalence $f : Y \to X$ is sharp
  if and only if its pullback along
  any nondegenerate simplex $\Delta^n \to X$ is again a weak equivalence,
  if and only if the pullback (object) is weakly contractible.
  The map $f : \Lambda^2_1 \to \Delta^1$
  sending vertices $0$, $1$ to $0 \in \Delta^1$ and $2$ to $1 \in \Delta^1$
  is sharp because $\Lambda^2_1$ and the fibres $(\Lambda^2_1)_0 \cong \Delta^1$,
  $(\Lambda^2_1)_1 \cong \Delta^0$ are all weakly contractible.
  However, this $f$ is not a fibration.
\end{example}

\begin{example}[Cartesian cubical sets and (unbiased) fibrations]
  \label{ex:cSet-unbiased-fibrations}
  For the following facts, we refer to \cite{awodey2026cartesian}.
  Let $\cSet$ denote the category of Cartesian cubical sets.
  Then the class of unbiased fibrations (henceforth just ``fibrations'')
  $\RR$ provides a $\pi$-preclan structure
  $(\cSet,\RR)$.
  The cubical interval $I = y([1])$ is not a fibrant,
  i.e. $I = y([1]) \to 1$ is not an $\RR$-map,
  but $I = y([1]) \to 1$ is $\RR$-exponentiable.
  It is not fibrant because the following open box in $I$ does not have a filler.
  \[ \begin{tikzcd}
    0 & 1 \\
    0 & 0
    \arrow[from=2-1, to=1-1]
    \arrow[from=2-1, to=2-2]
    \arrow[from=2-2, to=1-2]
  \end{tikzcd} \]

  To show that $I \to 1$ is $\RR$-exponentiable,
  let us first note that exponentiating by $I$ preserves fibrations.
  Let $f : A \to B$ be a fibration.
  Consider the unique map $u : 0 \to I$
  and the following ``pullback-hom'' diagram.
  \[\begin{tikzcd}
    {A^I} && \\
    & \bullet & {A^0} \\
    & {B^I} & {B^0}
    \arrow["{u \Rightarrow f}", two heads, from=1-1, to=2-2]
    \arrow["{A^u}", bend left, from=1-1, to=2-3]
    \arrow["{f^I}"', bend right, from=1-1, to=3-2]
    \arrow[from=2-2, to=2-3]
    \arrow["\iso"', from=2-2, to=3-2]
    \arrow["\lrcorner"{anchor=center, pos=0.125}, draw=none, from=2-2, to=3-3]
    \arrow["\iso", from=2-3, to=3-3]
    \arrow["{B^u}"', from=3-2, to=3-3]
  \end{tikzcd}\]
  Since $u : 0 \to I$ is a cofibration and $f : A \to B$ is a fibration,
  $u \Rightarrow f : A^I \to \bullet$ is a fibration,
  and so is $f^I : A^I \to B^I$ by the isomorphism.

  Using this fact,
  we can construct the pushforward of a fibration $f : X \to I$
  along $I \to 1$ as follows
  \[ \begin{tikzcd}
    X & {\Pi_I f} & {X^I} \\
    I & 1 & {I^I}
    \arrow["f"', two heads, from=1-1, to=2-1]
    \arrow[from=1-2, to=1-3]
    \arrow[from=1-2, to=2-2]
    \arrow["\lrcorner"{anchor=center, pos=0.125}, draw=none, from=1-2, to=2-3]
    \arrow["{f^I}", two heads, from=1-3, to=2-3]
    \arrow[from=2-1, to=2-2]
    \arrow["{\tilde{\id}}"', from=2-2, to=2-3]
  \end{tikzcd}\]
  where $\tilde{\id} : 1 \to I^I$ is the adjoint transpose 
  of the identity on $I$.
  Then $\Pi_I f \to 1$ is a pullback of a fibration,
  so is itself a fibration.

  This example will be revisited in \Cref{sec:path-types},
  to understand ``path types'' in $\cSet$.
\end{example}

\begin{example}[Topological spaces and \'etale maps]
  \label{ex:topos-etale}
  ``Proper maps'' provide a prototypical example of $\RR$-exponentiability.
  They can be found in topological spaces, locales, 1-toposes,
  and $\infty$-toposes, where $\RR$ is the class of \'etale maps.
  These examples are very similar,
  and share a lot of terminology that may not be consistent across the literature;
  the word ``proper'' is used to mean various things
  depending on the context.
  For example, Lurie uses the word ``proper'' for \'etale-exponentiability
  in $\infty$-toposes \cite[Definition 7.3.1.4]{lurie2009},
  whereas Moerdijk and Vermeulen say ``tidy'' \cite[III Definition 1.2]{moerdijk2000proper}
  or ``stable BCC'' (Beck--Chevalley condition) for \'etale-exponentiability in 1-toposes,
  and reserve the word ``proper'' for a weaker condition.
  To keep things precise,
  in this example we will specifically consider \'etale maps
  (also called local homeomorphisms)
  of topological spaces.
  
  The category of topological spaces with the class of
  \'etale maps $(\Topos,\EE)$ forms a clan,
  such that \'etale maps over $X$ are
  equivalent to sheaves on $X$
  \[\EE(X) \simeq \Sh(X)\]
  for each topos $X$ \cite[Section II.6]{maclane2012}.
  For a continuous map $f : Y \to X$,
  the pullback $f^* : \Topos/X \to \Topos/Y$ preserves \'etale maps
  and so restricts to $f^* : \EE(X) \to \EE(Y)$,
  which corresponds to the ``inverse image'' functor $f^* : \Sh(X) \to \Sh(Y)$
  under the equivalence.
  There is also a ``direct image''
  functor $f_* : \Sh(Y) \to \Sh(X)$, right adjoint to
  the inverse image $f^*  : \Sh(X) \to \Sh(Y)$.
  This gives us a right adjoint to the pullback $f^* : \EE(X) \to \EE(Y)$
  of \'etale maps.
  \[ \begin{tikzcd}
    Y & {\Sh(Y)} & {\EE(Y)} & {\Top / Y} \\
    X & {\Sh(X)} & {\EE(X)} & {\Top / X}
    \arrow["f"', from=1-1, to=2-1]
    \arrow["\simeq"{description}, draw=none, from=1-2, to=1-3]
    \arrow[""{name=0, anchor=center, inner sep=0}, "{f_*}", shift left=2, from=1-2, to=2-2]
    \arrow[hook, from=1-3, to=1-4]
    \arrow[""{name=1, anchor=center, inner sep=0}, "{f_*}", shift left=2, dashed, from=1-3, to=2-3]
    \arrow[""{name=2, anchor=center, inner sep=0}, "{f^*}", shift left=2, from=2-2, to=1-2]
    \arrow[""{name=3, anchor=center, inner sep=0}, "{f^*}", shift left=2, from=2-3, to=1-3]
    \arrow["\simeq"{description}, draw=none, from=2-3, to=2-2]
    \arrow[hook, from=2-3, to=2-4]
    \arrow["{f^*}", from=2-4, to=1-4]
    \arrow["\dashv"{anchor=center}, draw=none, from=2, to=0]
    \arrow["\dashv"{anchor=center}, draw=none, from=3, to=1]
  \end{tikzcd}\]
  However, $f_* : \EE(Y) \to \EE(X)$ may not satisfy Beck--Chevalley;
  and even if it does, the same might not hold for pullbacks of $f$.
  Hence, the converse of $\textit(iii)$ in \Cref{rmk:counterexample} fails.
  Moerdijk and Vermeulen characterise those $\EE$-exponentiable maps
  as those that are ``tidy'' \cite[III, Corollary 4.9]{moerdijk2000proper}.
  %TODO example of a geometric morphism that doesn't satisfy BC?
\end{example}

\begin{example}[Categories and fibrations]\label{ex:categories-and-fibrations}
  The following example will be explored in greater detail in \Cref{sec:cat-as-types}.

  The category of categories contains several preclans.
  The classes of (Grothendieck) fibrations and opfibrations are both clans.
  Opfibrations are fibration-exponentiable and vice versa
  (\Cref{prop:pushforward-in-cat}),
  together forming a $\tau$-clan (\Cref{def:double-preclan}).
  However, this example is somewhat unnatural,
  since viewing (op)fibrations as clans in $\Cat$ means that
  we are considering all functors between opfibrations,
  not just opcartesian functors (\Cref{def:morphism-split-opfibration}).

  We can restrict our attention to
  \emph{groupoidal} (op)fibrations --
  those which have groupoid fibres.
  The classes of groupoidal fibrations and
  groupoidal opfibrations form a $\tau$-preclan in $\Cat$
  (as opposed to a $\tau$-clan).
  Similarly, discrete fibrations and discrete opfibrations form another $\tau$-preclan in $\Cat$.
  % All functors between groupoidal fibrations are morphisms of fibrations,
  % so this example is somewhat more natural than the previous.

  The class of (split) isofibrations in the category of groupoids forms a $\pi$-clan.
  This will be used in \Cref{sec:hottlean} to construct a model of Martin-L\"of Type Theory,
  after \cite{hofmann1995}.
  One might expect that another model of Martin-L\"of Type Theory
  could be given by using isofibrations $\II$ or groupoidal isofibrations $\GG$ in $\Cat$.
  Unfortunately, although $(\Cat,\II)$ is a clan and $(\Cat,\GG)$ is a preclan,
  neither satisfies the $\pi$-preclan condition.
  The following (standard) example of a functor that is not
  exponentiable (with respect to all morphisms) can be used to demonstrate this fact.
  Consider the embedding $f := [0,2] : \two \to \three$,
  where $\two$ is the walking arrow,
  $\three$ is the walking composition of two arrows,
  and $f$ takes the arrow to the composition.
  % It is well known that $f$ is not exponentiable:
  % $f^* : \Cat / \three \to \Cat / \two$ does not preserve colimits (this example will demonstrate this),
  % therefore it does not have a right adjoint.
  Since $\GG$ is a preclan,
  there is a pullback functor $f^* : \GG(\three) \to \GG(\two)$.
  If $f$ were $\GG$-exponentiable, $f^*$ would have a right adjoint,
  so that $f^*$ preserves colimits.
  Consider the following pushout diagram in $\GG(\three) \subseteq \II(\three) \subseteq \Cat / \three$.
  \[
    \begin{tikzcd}
      1 & \two \\
      \two & \three
      \arrow["{[0]}", from=1-1, to=1-2]
      \arrow["{[1]}"', from=1-1, to=2-1]
      \arrow["{[1,2]}", from=1-2, to=2-2]
      \arrow["{[0,1]}"', from=2-1, to=2-2]
    \end{tikzcd}
  \]
  The pullback along $f$ produces the following square, which is not a pushout in $\GG(\two)$.
  \[
    \begin{tikzcd}
      0 & 1 \\
      1 & \two
      \arrow["{!}", from=1-1, to=1-2]
      \arrow["{!}"', from=1-1, to=2-1]
      \arrow["{[1]}", from=1-2, to=2-2]
      \arrow["{[0]}"', from=2-1, to=2-2]
    \end{tikzcd}
  \]
  Finally, note that $f$ is a groupoidal isofibration,
  hence $\GG$ is not a $\pi$-preclan.
  The same example shows that $\II$ is not a $\pi$-clan and $f$ is not an exponentiable morphism.

  It is also easy to come up with examples of exponentiable maps that are not
  exponentiable with respect to opfibrations.
  All opfibrations are exponentiable \cite{giraud1964,vidmar2018},
  so it suffices to give an opfibration that is not opfibration-exponentiable.
  Consider the (discrete) opfibrations $0 \to 1$ and $b : 1 \to \{a < b\}$.
  We can calculate the pushforward explicitly as $a : 1 \to \{a < b\}$,
  which is not an opfibration.
  \[ \begin{tikzcd}
    0 & {1 = \Pi_b 0} \\
    1 & {\{a < b\}}
    \arrow["\subseteq"', from=1-1, to=2-1]
    \arrow["a", from=1-2, to=2-2]
    \arrow["b"', from=2-1, to=2-2]
  \end{tikzcd}\]
\end{example} %TODO x1

\chapter{Polynomial functors in preclans}\label{sec:polynomials}
Polynomial functors can be developed in different settings,
including locally Cartesian closed categories \cite{gambino2013},
lex categories \cite{weber2015},
and $\pi$-clans \cite{hua2026}.
In \Cref{sec:exponentiable},
we saw that the theory of preclans
with an appropriate notion of preclan-exponentiability subsumes these various cases.
This chapter reproduces the theory of polynomial functors in this more general setting.

\section{Polynomial signatures}
From now on, we fix a preclan $(\CC,\RR)$ with a terminal object $1$.

\begin{definition}[Polynomial signature]
  A (univariate) \emph{polynomial signature} in $(\CC,\RR)$ consists of
  a morphism $f : E \to B$ such that
  $B \to 1$ is an $\RR$-map
  and $f$ is $\RR$-exponentiable.
\end{definition}

For simplicity, we only work with univariate polynomials.
This can be straightforwardly adapted to the multivariate case:
a multivariate polynomial signature would consist of three maps
\[\begin{tikzcd}
	I & E & B & J
	\arrow["i"', from=1-2, to=1-1]
	\arrow["f", from=1-2, to=1-3]
	\arrow["j", from=1-3, to=1-4]
\end{tikzcd}\]
such that $j$ is an $\RR$-map and $f$ is $\RR$-exponentiable.

\begin{definition}
  The \emph{category $\Poly_{(\CC,\RR)}$ of polynomial signatures}
  in $(\CC,\RR)$ has as objects
  polynomial signatures, and morphisms between objects
  $f : E \to B$ and $f' : E' \to B'$ as pairs
  \[(\phi_1 : B \to B', \phi^\# : f' \times_{B'} \phi_1 \to E)\]
  of morphisms in $\CC$ such that
  $f \circ \phi^\# = \pi_B : f' \times_{B'} \phi_1 \to B$.
  \[
  \begin{tikzcd}
    {f' \times_{B'} \phi_1} & E & B \\
    {E'} && {B'}
    \arrow["{\phi^\#}", from=1-1, to=1-2]
    \arrow["{\pi_B}", bend left = 50, from=1-1, to=1-3]
    \arrow["{\pi_{E'}}"', from=1-1, to=2-1]
    \arrow["\lrcorner"{anchor=center, pos=0.125}, draw=none, from=1-1, to=2-3]
    \arrow["f", from=1-2, to=1-3]
    \arrow["{\phi_1}", from=1-3, to=2-3]
    \arrow["{f'}"', from=2-1, to=2-3]
  \end{tikzcd}\]
  We write $\phi = (\phi_1, \phi^\#) : f \to f'$.
  The identity morphism on $(f : E \to B)$ is given by
  computing the pullback along an identity map
  \[
    \begin{tikzcd}
      {f \times_{B} \id_B} & E & B \\
      E && B
      \arrow["{{\iso}}"{description}, draw=none, from=1-1, to=1-2]
      \arrow["{{\pi_{E}}}"', from=1-1, to=2-1]
      \arrow["f", from=1-2, to=1-3]
      \arrow["{\id_E}"{description}, from=1-2, to=2-1]
      \arrow["\lrcorner"{anchor=center, pos=0.125}, draw=none, from=1-2, to=2-3]
      \arrow["{{\id_B}}", from=1-3, to=2-3]
      \arrow["{{f}}"', from=2-1, to=2-3]
    \end{tikzcd}
  \]
  Composition of $\phi : (f_0 : E_0 \to B_0) \to (f_1 : E_1 \to B_1)$
  and $\psi : (f_1 : E_1 \to B_1) \to (f_2 : E_2 \to B_2)$
  is given as follows
  \[
  \begin{tikzcd}
    {f_2 \times_{B_2} (\psi_1 \circ \phi_1)} & {f_1 \times_{B_1} \phi_1} & {E_0} & {B_0} \\
    {f_2 \times_{B_2} \psi_1} & {E_1} && {B_1} \\
    {E_2} &&& {B_2}
    \arrow["{\phi_1^* \psi^\#}", from=1-1, to=1-2]
    \arrow[from=1-1, to=2-1]
    \arrow["\lrcorner"{anchor=center, pos=0.125}, draw=none, from=1-1, to=2-2]
    \arrow["{\phi^\#}", from=1-2, to=1-3]
    \arrow[from=1-2, to=2-2]
    \arrow["\lrcorner"{anchor=center, pos=0.125}, draw=none, from=1-2, to=2-4]
    \arrow["{f_0}", from=1-3, to=1-4]
    \arrow["{\phi_1}", from=1-4, to=2-4]
    \arrow["{\psi^\#}", from=2-1, to=2-2]
    \arrow[from=2-1, to=3-1]
    \arrow["\lrcorner"{anchor=center, pos=0.125}, draw=none, from=2-1, to=3-4]
    \arrow["{f_1}", from=2-2, to=2-4]
    \arrow["{\psi_1}", from=2-4, to=3-4]
    \arrow["{f_2}", from=3-1, to=3-4]
  \end{tikzcd}
  \]
\end{definition}

On the notational difference with \cite{gambino2013}:
when $\CC$ is LCC (locally cartesian closed) and $\RR$ is the class of all maps,
our 1-category $\Poly_{(\CC,\RR)}$ corresponds to the 1-category
of univariate polynomials $\Poly_\CC (1,1)$ in \cite{gambino2013}.
They write $\Poly_\CC$ for the bicategory of multivariate polynomials;
for objects $I$ and $J$,
$\Poly_\CC(I,J)$ consists of the 1-category of multivariate polynomials
indexed over $I$ and $J$.

\begin{remark}
  As this work is part of a formalisation effort
  \cite{hazratpour2024poly},
  we note a technical detail concerning the definition of morphisms
  between polynomial signatures.
  The object $f' \times_{B'} \phi_1$
  cannot be part of the data of a morphism,
  since this would get in the way of proving equality of morphisms,
  such as associativity of composition.
  One workaround is to require that a morphism $\phi : (f) \to (f')$
  provides {for any given pullback} $f \times_B \phi_1 \iso P$ a morphism
  $\phi^\#_P : P \to E$,
  such that for any two such pullbacks $P$ and $P'$,
  the triangle formed with the canonical isomorphism $i : P \to P'$ commutes.
  \[\phi^\#_P = \phi^\#_Q \circ i : P \to E\]
\end{remark}

\begin{definition}\label{def:vert-cart-ofs}
  We say that a morphism $\phi : (f : E \to B) \to (f' : E' \to B')$ in $\Poly_{(\CC,\RR)}$
  is \emph{vertical} when $\phi_1$ is an isomorphism;
  it is \emph{cartesian} when $\phi^\#$ is an isomorphism.
  We write $\vert$ and $\cart$ respectively for
  the classes of vertical and cartesian morphisms in $\Poly_{(\CC,\RR)}$.
\end{definition}
  A morphism
  $\phi : (f : E \to B) \to (f' : E' \to B')$
  can be factorised as a vertical morphism
  $(\id_B, \phi^\#) : (f : E \to B) \to (f \circ \phi^\# : Q \to B)$
  followed by a cartesian one
  $(\phi_1, \id_{Q}) : (f \circ \phi^\# : Q \to B) \to (f' : E' \to B')$,
  where $Q := f' \times_{B'} \phi_1$ is the pullback of $\phi_1$ along $f$.
  We note that $f \circ \phi^\#$ is $\RR$-exponentiable,
  since it is the pullback of $f'$.
  \[
  \begin{tikzcd}
    {Q} & E & B \\
    {E'} && {B'}
    \arrow["{\phi^\#}", from=1-1, to=1-2]
    \arrow[from=1-1, to=2-1]
    \arrow["f", from=1-2, to=1-3]
    \arrow["{\phi_1}", from=1-3, to=2-3]
    \arrow["{f'}"', from=2-1, to=2-3]
    \arrow["\lrcorner"{anchor=center, pos=0.125}, draw=none, from=1-1, to=2-3]
  \end{tikzcd}
  \quad = \quad
  \begin{tikzcd}
    {Q} & {Q} & E & B \\
    {Q} & {Q} && B \\
    {E'} &&& {B'}
    \arrow[equals, from=1-1, to=1-2]
    \arrow[equals, from=1-1, to=2-1]
    \arrow["\lrcorner"{anchor=center, pos=0.125}, draw=none, from=1-1, to=2-2]
    \arrow["{\phi^\#}", from=1-2, to=1-3]
    \arrow[equals, from=1-2, to=2-2]
    \arrow["\lrcorner"{anchor=center, pos=0.125}, draw=none, from=1-2, to=2-4]
    \arrow["f", from=1-3, to=1-4]
    \arrow[equals, from=1-4, to=2-4]
    \arrow[equals, from=2-1, to=2-2]
    \arrow[from=2-1, to=3-1]
    \arrow["\lrcorner"{anchor=center, pos=0.125}, draw=none, from=2-1, to=3-4]
    \arrow["{f \circ \phi^\#}"', from=2-2, to=2-4]
    \arrow["{\phi_1}", from=2-4, to=3-4]
    \arrow["{f'}"', from=3-1, to=3-4]
  \end{tikzcd}
  \]
\begin{proposition}
  The classes $(\vert,\cart)$
  form an orthogonal factorisation system on $\Poly_{(\CC,\RR)}$.
\end{proposition}
A detailed account of the factorisation system when $\CC = \Set$ and $\RR$ consists of all maps 
can be found in \cite[Section 5.5]{niu2025};
the construction here is not significantly different.

\section{Polynomial functors}\label{sec:polynomial-functor}

\begin{definition}\label{def:polynomial-functor-1}
  Let $f : E \to B$ be a polynomial signature in $(\CC,\RR)$.
  The polynomial functor $P_f : \RR(1) \to \RR(1)$
  associated to $f$ is given by the
  composition of functors $P_f = B_! \circ f_* \circ E^*$.
  \[\begin{tikzcd}
     & {\RR(1)} & {\RR(E)} & {\RR(B)} & {\RR(1)} & 
    \arrow["{E^*}", from=1-2, to=1-3]
    \arrow["{f_*}", from=1-3, to=1-4]
    \arrow["{B_!}", from=1-4, to=1-5]
  \end{tikzcd}\]
\end{definition}

Unlike in \cite{hua2026},
we do not require that all objects are $\RR$-objects,
and we do not require that all $\RR$-maps
are $\RR$-exponentiable.
This general setting permits us to apply this theory
to the examples considered in \cite{barton2026},
such as when $\CC = \Top$ and $\RR$ is the class of open embeddings,
or when $\CC = \Cat$ and $\RR$ is the class of Grothendieck opfibrations.

Of course, we could also define polynomial functors for
multivariate polynomial signatures.
\[\begin{tikzcd}
	I & E & B & J
	\arrow["i"', from=1-2, to=1-1]
	\arrow["f", from=1-2, to=1-3]
	\arrow["j", from=1-3, to=1-4]
\end{tikzcd}\]
\[\begin{tikzcd}
   & {\RR(J)} & {\RR(E)} & {\RR(B)} & {\RR(I)} & 
  \arrow["{j^*}", from=1-2, to=1-3]
  \arrow["{f_*}", from=1-3, to=1-4]
  \arrow["{i_!}", from=1-4, to=1-5]
\end{tikzcd}\]

\begin{remark}
  Our setting is fairly minimal for a well-behaved theory of polynomials.
  The definition of a polynomial functor already uses two of
  the three preclan axioms -- pullback stability for $E^*$ and
  closure under composition $B_!$.
  The theory of exponentiability
  (which uses the remaining preclan axiom)
  is required for $f_*$.

  Of course,
  one can also wonder if all the conditions
  for exponentiability are necessary.
  Indeed,
  we will see that to define polynomial composition,
  we would like $f$ to belong to a pullback stable class of maps $\HH$.
  We will also see that we need $f_*$ to be a partial right adjoint
  for the universal property of polynomial functors,
  and this fact should also hold for its pullbacks.
  By \Cref{lem:stable-class-of-exponentiable-maps}
  this means that $\HH \subseteq \Exp_\RR$.

  For an even more general approach,
  which may be useful for working with examples considered in \cite{anel2024}
  (in particular the example of (op)fibrations in $\Cat$ used in \Cref{sec:cat-as-types}),
  it appears that some of this theory could be recovered
  when $\RR$ is a comprehension category over $\CC$ that is not full
  (fullness is implicit in our assumptions).
  We leave this for future work.
\end{remark}

\begin{lemma}\label{lem:polynomial-functor-preserves-maps}
  Let $f : E \to B$ be a signature in $(\CC,\RR)$.
  Suppose $(\CC,\SS)$ is another preclan such that $\SS \subseteq \RR$
  and $f$ is also $\SS$-exponentiable.
  Then the polynomial functor
  \[ P_f : (\RR(1), \SS) \to (\RR(1), \SS)\]
  is a preclan morphism with respect to $(\RR(1),\SS)$.
\end{lemma}
\begin{proof}
  The polynomial functor is a composition of preclan morphisms.
  \[\begin{tikzcd}
    {(\RR(1),\SS)} & {(\RR(E),\SS)} & {(\RR(B),\SS)} & {(\RR(1),\SS)}
    \arrow["{{{E^*}}}", from=1-1, to=1-2]
    \arrow["{{{f_*}}}", from=1-2, to=1-3]
    \arrow["{{{B_!}}}", from=1-3, to=1-4]
  \end{tikzcd}\]
  That $E^* : (\RR(1),\SS) \to (\RR(E),\SS)$ is a preclan morphism is clear.
  We have already shown that $f_* : (\RR(E),\SS) \to (\RR(B),\SS)$
  is a preclan morphism in \Cref{lem:pushforward-preserves-maps}.
  The functor $B_! : \RR(B) \to \RR(1)$ clearly preserves $\SS$-maps;
  $B_!$ preserves pullbacks of $\SS$-maps since
  the inclusions $\RR(X) \subseteq \CC / X$ reflect pullbacks,
  and the map $B_! : \CC / B \to \CC / 1$ preserves pullbacks.
\end{proof}

\begin{proposition}\label{prop:polynomial-functor-preclan-morphism}
  Let $f : E \to B$ be a signature in $(\CC,\RR)$.
  Then the polynomial functor $P_f$ is a preclan morphism.
  \[ P_f : (\RR(1), \RR) \to (\RR(1), \RR)\]
  Moreover, it is a restriction of the usual polynomial functor in presheaves,
  which is also a preclan morphism.
  \[ \begin{tikzcd}
      {(\RR(1), \RR)} & {(\RR(1), \RR)} \\
      {(\wCC, \wRR)} & {(\wCC, {\wRR})}
      \arrow["{P_f}", from=1-1, to=1-2]
      \arrow["y"', hook, from=1-1, to=2-1]
      \arrow["y", hook, from=1-2, to=2-2]
      \arrow["{P_f}"', from=2-1, to=2-2]
    \end{tikzcd} \]
\end{proposition}
\begin{proof}
  This follows from \Cref{lem:polynomial-functor-preserves-maps}.
  % The polynomial functors are compositions of preclan morphisms.
  % \[\begin{tikzcd}
  %   {(\RR(1),\RR)} & {(\RR(E),\RR)} & {(\RR(B),\RR)} & {(\RR(1),\RR)} \\
  %   (\wCC,\wRR) & {(\wCC / E,\wRR)} & {(\wCC / B,\wRR)} & (\wCC,\wRR)
  %   \arrow["{{E^*}}", from=1-1, to=1-2]
  %   \arrow[hook, from=1-1, to=2-1]
  %   \arrow["{{f_*}}", from=1-2, to=1-3]
  %   \arrow[hook, from=1-2, to=2-2]
  %   \arrow["{{B_!}}", from=1-3, to=1-4]
  %   \arrow[hook, from=1-3, to=2-3]
  %   \arrow[hook, from=1-4, to=2-4]
  %   \arrow["{{E^*}}"', from=2-1, to=2-2]
  %   \arrow["{{\Pi_f}}"', from=2-2, to=2-3]
  %   \arrow["{{\Sigma_B}}"', from=2-3, to=2-4]
  % \end{tikzcd}\]
  % That the two $E^*$ are preclan morphisms is clear.
  % It follows from the equivalent conditions of \Cref{theorem:exponentiable}
  % that $f_*$ and $\Pi_f$ are preclan morphisms.
  % The map $B_!$ clearly preserves $\RR$-maps.
  % That $B_!$ preserves pullbacks of $\RR$-maps follows from the
  % observation that the pullback of an $\RR$-map
  % $f : A \to B$ along an arbitrary map $b : B' \to B$ in $\RR(X)$
  % is the same thing as the pullback of
  % $f : X_! A \to X_! B$ along $b : X_! B' \to X_! B$ in $\CC$.
  % It is straightforward to show that $\Sigma_B$ is a preclan morphism.
\end{proof}

The polynomial functor associated to a signature $f : E \to B$
is characterised by its universal property:
\[ P_f X = \sum_{b : B} X^{E_b}\]

More precisely, we have the following proposition.

\begin{proposition}[Universal property of polynomial functors]\label{prop:poly-up}
  Given a signature $f : E \to B$,
  there is a bijection
  \[ \CC(\Gamma, P_f X) \iso
    \coprod_{t_1 \in \CC(\Gamma,B)}\CC(t_1 \times_B f, X)\]
  that is natural in both $\Gamma \in \CC^\op$ and $X \in \RR(1)$.
  For a given $t : \Gamma \to P_f X$,
  an element of the dependent sum can be visualised as
\[\begin{tikzcd}
	X & {t_1 \times_B f} & E \\
	& \Gamma & B
	\arrow["t_2"', from=1-2, to=1-1]
	\arrow[from=1-2, to=1-3]
	\arrow[from=1-2, to=2-2]
	\arrow["\lrcorner"{anchor=center, pos=0.125}, draw=none, from=1-2, to=2-3]
	\arrow["f", from=1-3, to=2-3]
	\arrow["t_1"', from=2-2, to=2-3]
\end{tikzcd}\]
  Of particular interest is the natural bijection applied to
  the identity $\id \in \CC(P_f X, P_f X)$,
  which results in a {universal} pair
  $\fstProj = (\id)_1 : P_f X \to B$
  and $\sndProj = (\id)_2 : \fstProj \times_B f \to X$,
  through which all others factor
\[\begin{tikzcd}[column sep = large]
	& X \\
	{t_1 \times_B f} & {\fstProj \times_B f} & E \\
	\Gamma & {P_f X} & B
	\arrow["{t_2}", from=2-1, to=1-2]
  \arrow["{t \times_B f}"', from=2-1, to=2-2]
	\arrow[from=2-1, to=3-1]
	\arrow["\lrcorner"{anchor=center, pos=0.125}, draw=none, from=2-1, to=3-2]
	\arrow["{\sndProj}"', from=2-2, to=1-2]
	\arrow[from=2-2, to=2-3]
	\arrow[from=2-2, to=3-2]
	\arrow["\lrcorner"{anchor=center, pos=0.125}, draw=none, from=2-2, to=3-3]
	\arrow["f", from=2-3, to=3-3]
	\arrow["t"', from=3-1, to=3-2]
	\arrow["{t_1}"', bend right, from=3-1, to=3-3]
	\arrow["{\fstProj}"', from=3-2, to=3-3]
\end{tikzcd}\]
\end{proposition}
\begin{proof}
  See \cite[Proposition 3.6]{hua2026}.
\end{proof}

\begin{definition}\label{def:pcomp}
  If $f : E \to B$ and $f' : E' \to B'$ are both signatures,
  then we can form the \emph{polynomial composition}
  $f \pcomp f' : \sndProj \times_{B'} f' \to P_f B'$ as follows
  \[ \begin{tikzcd}
    {E'} & {\sndProj \times_{B'} f'} & \\
    {B'} & {\fstProj \times_B f} & E \\
    & {P_fB'} & B
    \arrow["{{f'}}"', from=1-1, to=2-1]
    \arrow[from=1-2, to=1-1]
    \arrow["\lrcorner"{anchor=center, pos=0.125, rotate=-90}, draw=none, from=1-2, to=2-1]
    \arrow[from=1-2, to=2-2]
    \arrow["{f \pcomp f'}"{description, pos=0.2}, shift left=10, bend left = 50, from=1-2, to=3-2]
    \arrow["{{\sndProj}}", from=2-2, to=2-1]
    \arrow[from=2-2, to=2-3]
    \arrow[from=2-2, to=3-2]
    \arrow["\lrcorner"{anchor=center, pos=0.125}, draw=none, from=2-2, to=3-3]
    \arrow["f", from=2-3, to=3-3]
    \arrow["{{\fstProj}}"', from=3-2, to=3-3]
  \end{tikzcd}\]
  The map $f \pcomp f'$ is a polynomial signature
  since $P_f B'$ is an $\RR$-object
  and $\Exp_\RR$ is stable under pullback and closed under composition
  (\Cref{prop:exp-clan}).
\end{definition}

\begin{lemma}
  The polynomial composition ${f \pcomp f'}$
  is characterised by the following natural isomorphism.
  \[ P_{f \pcomp f'} \iso P_{f} \circ P_{f'} : \RR(1) \to \RR(1) \]
\end{lemma}
\begin{proof}
  The proof of \cite[Proposition 1.12]{gambino2013}
  works mutatis mutandis in this setting,
  where slices $\CC / X$ are replaced with $\RR(X)$.
  Indeed, we have Beck--Chevalley isomorphisms for pushforward
  (\Cref{def:beck-chevalley}),
  Beck--Chevalley isomorphisms for composition
  (\Cref{prop:beck-chevalley-left}),
  the distributivity law (\Cref{lem:distributivity}),
  and pseudofunctoriality of pullback and its adjoints. 
\end{proof}

\begin{definition}\label{def:cart-trans}
  For any cartesian morphism of signatures
  $\psi : (f : E \to B) \to (f' : E' \to B')$
  we define an associated natural transformation of polynomial functors
  $P_\psi : P_f \to P_{f'}$.

  Let us rename $\psi_1 : B \to B'$ as $\delta$.
  Since $\psi$ is cartesian, we have a pullback square
  \[ \begin{tikzcd}
    E & B \\
    {E'} & {B'}
    \arrow["f", from=1-1, to=1-2]
    \arrow["\phi"', from=1-1, to=2-1]
    \arrow["\lrcorner"{anchor=center, pos=0.125}, draw=none, from=1-1, to=2-2]
    \arrow["\delta", from=1-2, to=2-2]
    \arrow["{{f'}}"', from=2-1, to=2-2]
  \end{tikzcd}\]
  We define $P_\psi : P_f \to P_{f'}$ by whiskering the following natural transformations.
  \[\begin{tikzcd}
    {\RR(1)} & {\RR(E)} & {\RR(B)} & {\RR(1)} \\
    {\RR(1)} & {\RR(E')} & {\RR(B')} & {\RR(1)}
    \arrow["{{{E^*}}}", from=1-1, to=1-2]
    \arrow[equals, from=1-1, to=2-1]
    \arrow["\iso"{marking, allow upside down}, draw=none, from=1-1, to=2-2]
    \arrow["{{{f_*}}}", from=1-2, to=1-3]
    \arrow["\iso"{marking, allow upside down}, draw=none, from=1-2, to=2-3]
    \arrow["{{{B_!}}}", from=1-3, to=1-4]
    \arrow["\Rightarrow"{marking, allow upside down}, draw=none, from=1-3, to=2-4]
    \arrow["{{{E'^*}}}"', from=2-1, to=2-2]
    \arrow["{{{\phi^*}}}", from=2-2, to=1-2]
    \arrow["{{{f'_*}}}"', from=2-2, to=2-3]
    \arrow["{{{\delta^*}}}"', from=2-3, to=1-3]
    \arrow["{{{B'_!}}}"', from=2-3, to=2-4]
    \arrow["{{\id^*}}"', equals, from=2-4, to=1-4]
  \end{tikzcd}\]
  The first isomorphism is given by pseudofunctoriality,
  the second by Beck--Chevalley for pushforward along 
  the $\RR$-exponentiable map $f$.
  The final (non-invertible)
  natural transformation is the Beck--Chevalley
  transformation for composition applied to the commutative square
  \[ \begin{tikzcd}
    B & 1 \\
    {B'} & 1
    \arrow["{{{B}}}", from=1-1, to=1-2]
    \arrow["{{{\delta}}}", from=1-1, to=2-1]
    \arrow["{{{B'}}}"', from=2-1, to=2-2]
    \arrow[equals, from=2-2, to=1-2]
  \end{tikzcd} \]
\end{definition}

Recall that a natural transformation is cartesian
when each naturality square of the transformation is
a pullback square.
\begin{proposition}
  For any cartesian morphism of signatures
  $\psi : (f : E \to B) \to (f' : E' \to B')$
  the associated natural transformation $P_\psi : P_f \to P_{f'}$
  is cartesian.
\end{proposition}
\begin{proof}
  Unfolding the whiskering diagram defining $P_\psi$,
  we see that $P_\psi$ is a composition of three transformations,
  of which two are invertible.
  Let us call the non-invertible transformation $\alpha$.
  \[ \begin{tikzcd}
    B & 1 && {\RR(B)} & {\RR(1)} & \CC \\
    {B'} & 1 && {\RR(B')} & {\RR(1)}
    \arrow["{{{{B}}}}", from=1-1, to=1-2]
    \arrow["{{{{\delta}}}}", from=1-1, to=2-1]
    \arrow["{B_!}", from=1-4, to=1-5]
    \arrow["\alpha"{description}, Rightarrow, from=1-4, to=2-5]
    \arrow["\subseteq"{description}, draw=none, from=1-5, to=1-6]
    \arrow["{{{{B'}}}}"', from=2-1, to=2-2]
    \arrow[equals, from=2-2, to=1-2]
    \arrow["{\delta^*}", from=2-4, to=1-4]
    \arrow["{B'_!}"', from=2-4, to=2-5]
    \arrow[from=2-5, to=1-5]
  \end{tikzcd}\]
  It suffices to check that the whiskering $\alpha f_* E^*$ is cartesian,
  which is true if $\alpha$ is cartesian.
  Now, the naturality squares of the natural transformation
  $\alpha$ are always pullback squares in $\CC$,
  by the definition of Beck--Chevalley for composition.
  Since $\RR(1) \subseteq \CC$ is fully faithful,
  it reflects limits,
  hence the naturality squares are also pullback squares in $\RR(1)$.
\end{proof}

\begin{definition}\label{def:vert-trans}
  For any vertical morphism of signatures
  $\psi : (f : E \to B) \to (f' : E' \to B')$
  we define an associated natural transformation of polynomial functors
  $P_\psi : P_f \to P_{f'}$.
  Since $\psi$ is vertical,
  $\psi_1 : B \to B'$ is invertible,
  and so is its pullback.
  \[
  \begin{tikzcd}
    {f' \times_{B'} \psi_1} & E & B \\
    {E'} && {B'}
    \arrow["{{\psi^\#}}", from=1-1, to=1-2]
    \arrow["\sim"', from=1-1, to=2-1]
    \arrow["\lrcorner"{anchor=center, pos=0.0125}, draw=none, from=1-1, to=2-3]
    \arrow["f", from=1-2, to=1-3]
    \arrow["{{\psi_1}}", from=1-3, to=2-3]
    \arrow["\sim"', draw=none, from=1-3, to=2-3]
    \arrow["{{f'}}"', from=2-1, to=2-3]
  \end{tikzcd}
  \]
  Let us write simplify the diagram by writing $\delta : B' \to B$ for the inverse of $\psi_1$,
  and $\phi : E' \to E$ for the composition of $\psi^\#$ with the comparison map between $E'$ the pullback.
  \[
  \begin{tikzcd}
    {E} & {B} \\
    E' & B'
    \arrow["{f}", from=1-1, to=1-2]
    \arrow["\phi", from=2-1, to=1-1]
    \arrow["\delta"', from=2-2, to=1-2]
    \arrow["\sim"', draw=none, from=1-2, to=2-2]
    \arrow["{f'}"', from=2-1, to=2-2]
  \end{tikzcd}
  \]
  We define $P_\psi: P_f \to P_{f'}$ as the following whiskering 
  \[\begin{tikzcd}
    \RR(1) & {\RR(E)} & {\RR(B)} \\
    & {\RR(E')} & {\RR(B')} & \RR(1)
    \arrow["{E^*}", from=1-1, to=1-2]
    \arrow[""{name=0, anchor=center, inner sep=0}, "{E'^*}"', from=1-1, to=2-2]
    \arrow["{f_*}", from=1-2, to=1-3]
    \arrow["{\phi^*}", from=1-2, to=2-2]
    \arrow["\Rightarrow"{marking, allow upside down}, draw=none, from=1-3, to=2-2]
    \arrow["{\delta^*}"', from=1-3, to=2-3]
    \arrow[""{name=1, anchor=center, inner sep=0}, "{B_!}", from=1-3, to=2-4]
    \arrow["{f'_*}"', from=2-2, to=2-3]
    \arrow["{B'_!}"', from=2-3, to=2-4]
    \arrow["\iso"{marking, allow upside down}, draw=none, from=0, to=1-2]    
    \arrow["{\iso}"{marking, allow upside down}, draw=none, from=2-3, to=1]
  \end{tikzcd}\]
  where the (non-invertible) natural transformation in the middle is
  Beck--Chevalley for pushforward applied to the previous commutative square.
\end{definition}

\begin{proposition}
  For any vertical morphism of signatures
  $\psi : (f : E \to B) \to (f' : E' \to B')$,
  $P_\psi$ is vertical, meaning
  \[\begin{tikzcd}
      {P_f 1} & {P_{f'} 1} \\
      B & {B'}
      \arrow["{P_\psi 1}", from=1-1, to=1-2]
      \arrow["\sim"', draw=none, from=1-1, to=1-2]
      \arrow["\sim", from=1-1, to=2-1]
      \arrow["\fstProj"', no head, from=1-1, to=2-1]
      \arrow["\sim"', from=1-2, to=2-2]
      \arrow["\fstProj", draw=none, from=1-2, to=2-2]
      \arrow["{\psi_1}"', from=2-1, to=2-2]
      \arrow["\sim", draw=none, from=2-1, to=2-2]
    \end{tikzcd}\]
\end{proposition}

\begin{definition}
Combining the $(\vert, \cart)$ orthogonal factorisation system
with \Cref{def:cart-trans} and \Cref{def:vert-trans},
we obtain a functor from the category of polynomial signatures
to the category of endofunctors on $\RR(1)$.
\[
  P : \Poly_{(\CC,\RR)} \to [\RR(1), \RR(1)]
\]
\end{definition}

\begin{remark}
  Assuming that $\CC$ is locally Cartesian closed,
  Gambino and Kock \cite{gambino2013} use an enrichment of each slice $\CC / X$ over $\CC$
  to show that when $\CC$ is locally Cartesian closed (and taking $\RR$ to be all maps),
  the functor $P$ is faithful and the full image of $P$ consists of natural tranformations
  that are ``strong'' with respect to the given enrichment.
  When $(\CC,\RR)$ is a $\pi$-clan,
  one can produce a similar enrichment of $\RR(X)$ over $\RR(1)$,
  but when $(\CC,\RR)$ is merely a clan,
  it is unclear whether a similar argument could be made to work.
\end{remark}

\chapter{HoTTLean}\label{sec:hottlean}
While it is common to use MLTT as an internal language for
certain arguments about objects in its models,
this idea has not yet been fully implemented in a proof assistant.
The HoTTLean project \cite{hua2025, nawrocki2026, hottuf2026}
(\href{https://github.com/sinhp/HoTTLean}{github.com/sinhp/HoTTLean}) 
aims to bridge this gap by
designing a domain-specific language (DSL) for Martin-L\"of type theories (MLTT)
using Lean's DSL support,
formalising the general semantics of MLTT in Lean,
and supplying a specific model
-- namely the Hofmann--Streicher groupoid model \cite{hofmann1995} --
as a test case.
Overall, this allows users to write synthetic proofs in the DSL,
and use the interpretation to produce constructions pertaining to
groupoids as defined in Mathlib \cite{mathlib2020},
the standard library for mathematics in Lean.
In short, this would make automated internal reasoning in the category of groupoids a reality.

Here, we will focus on the formalisation of the groupoid model,
with a finite hierarchy of universes,
$\Si$-types, $\Pi$-types, and $\Id$-types.

The process of formalisation involves a great deal of experimentation.
It is often difficult to ascertain what approach is best until it has been tried.
For the construction of a model of type theory,
there are many choices to be made.
For example,
just the definition of ``a model of MLTT'' can be taken to mean
a natural model \cite{awodey2018},
or a category with families (CwF) \cite{clairambault2011},
or a comprehension category \cite{jacobs1993};
the list goes on.
After various experiments,
we decided on an approach that was (1) ``internal'',
and (2) ``unalgebraic''.

% The word ``internal'' here refers to replacing constructions on an external classifier
% of all type families (in presheaves) with 
% constructions on classifiers of small type families (in the category itself).

By ``internal'', we mean that instead of working with a classifier of all types $\Ty : \CC^\op \to \Set$
(as is typically required for categories with families and natural models),
we work with a hierarchy of universes in the category of contexts
$U_0 \to U_1 \to \cdots \in \CC$
(i.e., a hierarchy of representable presheaves $\Ty_0 \to \Ty_1 \to \cdots : \CC^\op \to \Set$
classifying types of a certain size).
% such that the universes are defined polymorphically over the universe levels of Lean.
Working internally meant that the user supplying a model need not make separate constructions
for the classifier of all types and the internal universes classifying small types.
It also simplified the universe level calculus,
as presheaves involve introducing extra universe levels and size constraints therein.
Furthermore, this approach seems to be more uniform across different models of HoTT.
Usually, an ad hoc description of simplical and cubical models is given in a presheaf category,
to which we would not want to apply the theory of natural models or CwFs.
In short, we are satisfied with the internal approach.

On the other hand,
the ``unalgebraic'' construction of the groupoid model is
not entirely satisfactory.
(Recall that we refer to the natural models style of semantics as \emph{algebraic}
(as in \emph{algebraic type theory} \cite{awodey2025}),
and use the word \emph{unalgebraic} in contrast
to describe semantics with explicit coherence conditions, like a CwF.)
We make the case that an alternative ``algebraic'' approach,
presented in \Cref{sec:StrSem},
would improve upon this first attempt
(in terms of the metrics for formalisation).
We discuss this in \Cref{sec:algebraic-vs-elementary}.

In the following, we present general unalgebraic models
and the groupoid model constructed as an unalgebraic model,
as a companion to the formalisation.
We also present algebraic models and a roadmap towards formalisation
using algebraic models.
The formalisation repository can be found at
\href{https://github.com/sinhp/HoTTLean/tree/HoTTUF}{\textsf{github.com/sinhp/HoTTLean}}.
We also include links to the
formalised counterparts of definitions,
labeled using the symbol \done{https://github.com/sinhp/HoTTLean/tree/HoTTUF}.

\subsection*{Related Work}
Du \cite{du2025} formalised in Lean a simplicial model
of type theory as a contextual category with $\Pi$-types,
a single universe, and without other type formers.
Sozeau and Tabareau \cite{sozeau2014} worked towards a Rocq formalisation
of the groupoid model,
but this has not been completed, to our knowledge.
The groupoid model is formalised as an example comprehension category
in the Unimath project \cite{2024grayson},
also without type formers.

\section{Unalgebraic models}\label{sec:UnstrSem}
We fix a category $\CC$,
and call objects of $\CC$ \emph{contexts} and
morphisms of $\CC$ \emph{substitutions}. 

\begin{definition}[\done{https://github.com/sinhp/HoTTLean/blob/HoTTUF/HoTTLean/Model/Unstructured/UnstructuredUniverse.lean\#L19}]
  \label{def:elementary-universe}
  A \emph{universe}
  in $\CC$ consists of a morphism
  $\uu : \Ulow \to \U$
  and a chosen pullback $\Gamma . A$
  for each object $\Gamma$ and morphism $A : \Gamma \to \U$.
  \[\begin{tikzcd}
	{\Gamma.A} & \Ulow \\
	\Gamma & \U
	\arrow["\var_A", from=1-1, to=1-2]
	\arrow["d_A", swap, from=1-1, to=2-1]
	\arrow["\lrcorner"{anchor=center, pos=0.125}, draw=none, from=1-1, to=2-2]
	\arrow["\uu", from=1-2, to=2-2]
	\arrow["A"', from=2-1, to=2-2]
  \end{tikzcd}\]
  We will view $A : \Gamma \to \U$,
  as a type in context $\Gamma$,
  and view the chosen pullback $\Gamma . A$
  as context extension $d_A : \Gamma . A \to \Gamma$.
  We also view a morphism $a : \Gamma \to \Ulow$
  satisfying $\uu \circ a = A$ as a term of type $A$.
\end{definition}

One can view $\uu : \Ulow \to \U$ as a context extension
$d_X : (X : \U.\,x : X) \to (X : \U)$,
by taking the type $X : \U \to \U$ to be the identity
$\id_\U : \U \to \U$.

\begin{definition}[\done{https://github.com/sinhp/HoTTLean/blob/HoTTUF/HoTTLean/Model/Unstructured/UnstructuredUniverse.lean\#L98}]
  For $\sigma : \Delta \to \Gamma$,
  a type $A : \Gamma \to \U$ and a term $a : \Delta \to \Ulow$
  such that $\uu \circ a = A \circ \sigma$,
  we denote the map into the context extension
  using the notation
  $\sigma . a : \Delta \to \Gamma . A$,
  induced by the pullback
\[\begin{tikzcd}
	\Delta \\
	& {\Gamma.A} & \Ulow \\
	& \Gamma & \U
	\arrow["{\sigma.a}"{description}, dashed, from=1-1, to=2-2]
	\arrow["a", bend left, from=1-1, to=2-3]
	\arrow["\sigma"', bend right, from=1-1, to=3-2]
	\arrow[from=2-2, to=2-3]
	\arrow[from=2-2, to=3-2]
	\arrow["\lrcorner"{anchor=center, pos=0.125}, draw=none, from=2-2, to=3-3]
	\arrow["\uu", from=2-3, to=3-3]
	\arrow["A"', from=3-2, to=3-3]
\end{tikzcd}\]
\end{definition}

\begin{definition}[$\Unit$-types] \label{def:elementary-unit}
  We say that a universe $\uu : \Ulow \to \U$ admits
  \emph{unalgebraic $\Unit$-types} when we have the following.
  \begin{enumerate}
  \item
    For any context $\Gamma$, there is a type $\Unit_\Gamma : \Gamma \to \U$.
  \item
    The construction $\Unit_\Gamma$ is stable under substitution,
    meaning that for any $\sigma : \Delta \to \Gamma$ we have
    \[
      \Unit_\Gamma \circ \sigma = \Unit_\Delta
    \]
  \item
    For any context $\Gamma$, there is a unique map $\unit_\Gamma : \Gamma \to \Ulow$
    satisfying $\uu \circ \unit_\Gamma = \Unit_\Gamma$.
  \end{enumerate}
\end{definition}

\begin{definition}[$\Sigma$-types, \done{https://github.com/sinhp/HoTTLean/blob/HoTTUF/HoTTLean/Model/Unstructured/UnstructuredUniverse.lean\#L219}]
  \label{def:elementary-sigma}
  We will say $\uu : \Ulow \to \U$ admits
  \emph{unalgebraic $\Sigma$-types}
  when we have the following.
  \begin{enumerate}
  \item 
    For any context $\Gamma$, type $A : \Gamma \to \U$,
    and dependent type $B : \Gamma . A \to \U$
    there is a type $\Sigma_A B : \Gamma \to \U$.
  \item
    The construction $\Sigma_A B$ is stable under substitution,
    meaning that for any $\sigma : \Delta \to \Gamma$ we have
    \[
      \Sigma_{A \circ \sigma} (B \circ \tilde{\sigma}) =
      \Sigma_{A} B \circ \sigma
    \]
    where $\tilde{\sigma}$ is the same as in \Cref{def:elementary-pi}.
  \item
    Under the assumptions of (1),
    for any $a : \Gamma \to \Ulow$ and $b:\Gamma\to \Ulow$
    such that $\uu \circ a = A$ and $\uu \circ b = B \circ \id_\Gamma . a $,
    there is a map $\pair (a,b) : \Gamma \to \Ulow$
    satisfying $\uu \circ \pair (a,b)  = \Sigma_A B$. 
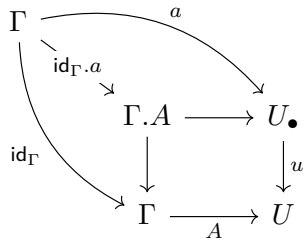
    \[
\begin{tikzcd}
	\Gamma \\
	& {\Gamma . A} & \Ulow \\
	& \Gamma & \U
	\arrow["{\id_\Gamma . a}"{description}, from=1-1, to=2-2]
	\arrow["a", bend left, from=1-1, to=2-3]
	\arrow["{\id_\Gamma}"', bend right, from=1-1, to=3-2]
	\arrow[from=2-2, to=2-3]
	\arrow[from=2-2, to=3-2]
	\arrow["\uu", from=2-3, to=3-3]
	\arrow["A"', from=3-2, to=3-3]
\end{tikzcd}
    \]
  \item
    The construction $\pair (a, b)$ is stable under substitution.
    For any $\sigma : \Delta \to \Gamma$
    \[
      \pair (a \circ \sigma,b \circ \sigma) =
       \pair(a , b) \circ \sigma
    \]
  \item
    Under the assumptions of (1),
    for any $s : \Gamma \to \Ulow$
    such that $ \uu \circ s = \Sigma_A B$,
    there are maps $\fst s : \Gamma  \to \Ulow$
    and $\snd s : \Gamma  \to \Ulow$
    satisfying $\uu \circ \fst\; s = A$
    and $\uu \circ \snd \; s = B \circ (\id_\Gamma . \fst\; s)$.
  \item ($\beta$-rules) The operations 
    $\fst$ and $\snd$ deconstruct a term formed using $\pair$.
    \[ \fst \;(\pair(a, b)) = a \quad \text{ and } 
    \quad \snd \; (\pair (a, b)) = b\]
  \item ($\eta$-rule) Conversely, the operation $\pair$
    reconstructs the term split up by $\fst$ and $\snd$.
  \[\pair (\fst \;s, \snd \;s) = s\]
  \end{enumerate}
\end{definition}

In the presence of sufficient exponentiability conditions --
for example, if there is an embient $\pi$-clan (\Cref{def:pi-preclan}) or
$\CC$ is lex and the universe is an exponentiable morphism
(see \Cref{sec:exponentiable} and \Cref{sec:polynomials}) --
then unalgebraic $\Sigma$-types on $\uu$
can be reformulated algebraically using polynomials.
We provide such a reformulation in \Cref{def:alg-sig-type-mltt}.

% \[
%     \begin{tikzcd}
%       Q & \Ulow \\
%       {P_\uu \U} & \U
%       \arrow["{\pair}", from=1-1, to=1-2]
%       \arrow["{{\uu \pcomp \uu}}"', from=1-1, to=2-1]
%       \arrow["\lrcorner"{anchor=center, pos=0.125}, draw=none, from=1-1, to=2-2]
%       \arrow["\uu", from=1-2, to=2-2]
%       \arrow["{\Sigma}"', from=2-1, to=2-2]
%     \end{tikzcd}
% \]

For reasons of generality that will become important in \Cref{sec:cat-as-types},
we will define a version of $\Pi$-types that is slightly more general than what is needed in this section,
involving two universes rather than just one.
(In fact, the formalisation is even more general,
involving three universes. See \Cref{rmk:three-universes-generality}.)
\begin{definition}[$\Pi$-types, \done{https://github.com/sinhp/HoTTLean/blob/HoTTUF/HoTTLean/Model/Unstructured/UnstructuredUniverse.lean\#L332}]
  \label{def:elementary-pi}
  Let $\uu : \Ulow \to \U$ and $\vv : \Vlow \to \V$ be two universes.
  We will say that $\uu$ admits $\vv$-indexed \emph{unalgebraic $\Pi$-types}
  when we have the following.
  \begin{enumerate}
  \item 
    For any context $\Gamma$, type $A : \Gamma \to \V$,
    and dependent type $B : \Gamma . A \to \U$
    there is a type \[\Pi_A B : \Gamma \to \U\]
  \item
    The construction $\Pi_A B$ is stable under substitution,
    meaning that for any $\sigma : \Delta \to \Gamma$ we have
    \[
      \Pi_{A \circ \sigma} (B \circ \tilde{\sigma}) =
      B \circ \sigma
    \]
    where $\tilde{\sigma} := (\sigma \circ d_{A \circ \sigma}) . \var_{A \circ \sigma}$ as in
    \[
    \begin{tikzcd}
      {\Delta . {(A \circ \sigma)}} & {\Gamma . A} & \Vlow \\
      \Delta & \Gamma & \V
      \arrow["{{\tilde{\sigma}}}", dashed, from=1-1, to=1-2]
      \arrow["{\var_{A \circ \sigma}}", bend left, from=1-1, to=1-3]
      \arrow["{d_{A \circ \sigma}}", from=1-1, to=2-1]
      \arrow["{\var_A}", from=1-2, to=1-3]
      \arrow["{d_A}", from=1-2, to=2-2]
      \arrow["\lrcorner"{anchor=center, pos=0.125}, draw=none, from=1-2, to=2-3]
      \arrow[from=1-3, to=2-3]
      \arrow["\sigma"', from=2-1, to=2-2]
      \arrow["A"', from=2-2, to=2-3]
    \end{tikzcd}
    \]
  \item
    Under the assumptions of (1),
    for any $b : \Gamma . A \to \Ulow$
    such that $\uu \circ b = B$,
    there is a map $\lam  b : \Gamma \to \Ulow$
    satisfying $\uu \circ \lam \;b  = \Pi_A B$.
  \item
    The construction $\lam  b$ is stable under substitution.
    For any $\sigma : \Delta \to \Gamma$
    \[
      \lam (b \circ \tilde{\sigma}) =
      \lam  b \circ \sigma
    \]
  \item
    Under the assumptions of (1),
    for any $f : \Gamma \to \Ulow$
    such that $\uu \circ f = \Pi_A B$,
    there is a map $\unlam  f : \Gamma . A \to \Ulow$
    satisfying $\uu \circ \unlam  f  = B$.
  \item ($\beta$- and $\eta$-rules) The operations 
    $\lam$ and $\unlam$ are inverses of one another.
    \[ \unlam (\lam  b) = b \quad \text{ and } 
    \quad \lam (\unlam  f) = f\]
  \end{enumerate}
  If $\uu$ admits unalgebraic $\uu$-indexed $\Pi$-types,
  then we will simply say that $\uu$ admits
  \emph{unalgebraic $\Pi$-types}.
\end{definition}
Again, when polynomials are available,
unalgebraic $\Pi$-types on the universe can be 
reformulated algebraically as a pullback square (\Cref{def:alg-pi-type-mltt}).

% In the presence of sufficient exponentiability conditions,
% for example if there is an ambient $\pi$-clan or
% $\CC$ is lex and the universe is an exponentiable morphism
% (see \Cref{sec:exponentiable} and \Cref{sec:polynomials}),
% then an elementary $\Pi$-types on $\uu$
% can be reformulated algebraically using polynomials.
% We provide such this reformulation in \Cref{def:alg-pi-type-mltt}.

\begin{definition}[$\Id$-types, \done{https://github.com/sinhp/HoTTLean/blob/HoTTUF/HoTTLean/Model/Unstructured/UnstructuredUniverse.lean\#L439}]
  We say that $\uu : \Ulow \to \U$
  admits \emph{unalgebraic $\Id$-types}
  when we have the following.
  \begin{enumerate}
  \item 
    For any context $\Gamma$, type $A : \Gamma \to \U$,
    and terms $a_0, a_1 : \Gamma \to \Ulow$ such that
    $\uu \circ a_i = A$,
    there is a type $\Id_A (a_0,a_1) : \Gamma \to \U$.
  \item
    The construction $\Id_A (a_0,a_1)$ is stable under substitution.
  \item
    For any context $\Gamma$, type $A : \Gamma \to \U$,
    and term $a : \Gamma \to \Ulow$ such that
    $\uu \circ a = A$,
    there is a term $\refl a : \Gamma \to \Ulow$
    satisfying
    $\uu \circ \refl a  = \Id_A (a,a)$.
  \item The construction $\refl$ is stable under substitution.
  \item
    Under the assumptions of (3)
    one can construct the context $\Gamma . (x : A) . \Id_A(a,x)$
    by taking context extension on $\Gamma . A$ for the type
    $\Id_A(a,x)$ given by (1),
    and the substitution
    \[\rho_a := \id_\Gamma . a . \refl a : \Gamma \to \Gamma . (x : A) . \Id_A(a,x)\]
    For any \emph{motive} $C : \Gamma .(x : A) . \Id_A(a,x) \to \U$
    -- that means $C$ is parametrised by a family of paths with a fixed starting point $a$ --
    and map $r : \Gamma \to \Ulow$ satisfying
    \[ \uu \circ r = C \circ \rho_a \]
    there is a term
    $j (C, r) : \Gamma . (x : A) . \Id_A(a,x) \to \Ulow$
    satisfying $\uu \circ j (C, r) = C$ and
    $j (C, r) \circ \rho_a  = r$
    \[\begin{tikzcd}
	   \Gamma & \Ulow \\
	   {\Gamma . (x : A) . \Id_A(a,x)} & \U
	   \arrow["{r}", from=1-1, to=1-2]
	   \arrow["{\rho_a}"', from=1-1, to=2-1]
	   \arrow["\uu", from=1-2, to=2-2]
	   \arrow["{j(C,r)}"{description}, dashed, from=2-1, to=1-2]
	   \arrow["C"', from=2-1, to=2-2]
    \end{tikzcd}\]
  \item The construction $j$ is stable under substitution.
  \end{enumerate}
\end{definition}

\Cref{def:algebraic-id-types} provides the algebraic version of identity type structure.

\begin{remark}\label{rmk:three-universes-generality}
% In the previous definitions,
% we have exclusively worked with a single universe $\uu$,
% for the sake of simplicity.
% However,
The definitions for $\Sigma$-types (and similarly $\Pi$-types)
% $\Sigma$-types in Lean
have $A$ in a universe $\uu_0$,
$B$ in a second universe $\uu_1$,
and $\Sigma_A B$ in a third universe $\uu_2$.
% A similar generalisation can be made for path types.
The formalisation also makes $\Id$-types \emph{large eliminating},
meaning that the motive for identity elimination $C$
need not live in the same universe as the type $A$ on which
identities are taken.
The more general definitions can be found in the formalisation.
\end{remark}

The following defines
non-cumulative universe lifts
``\`a la Coquand'' \cite{gratzer2020}.

\begin{definition}[\done{https://github.com/sinhp/HoTTLean/blob/HoTTUF/HoTTLean/Model/Unstructured/UHom.lean\#L23},
\done{https://github.com/sinhp/HoTTLean/blob/HoTTUF/HoTTLean/Model/Unstructured/UHom.lean\#L77}]
  A \emph{morphism of universes} $l : \uu_0 \to \uu_1$ 
  consists of a pair of maps $l_\U : \U_0 \to \U_1$
  and $l_\Ulow : \Ulow_0 \to \Ulow_1$ such that 
  the following square is a pullback
  \[\begin{tikzcd}
	{\Ulow_0} & {\Ulow_1} \\
	{\U_0} & {\U_1}
	\arrow["{{{l_\Ulow}}}", from=1-1, to=1-2]
	\arrow["{{{\uu_0}}}"', from=1-1, to=2-1]
	\arrow["\lrcorner"{anchor=center, pos=0.125}, draw=none, from=1-1, to=2-2]
	\arrow["{\uu_1}", from=1-2, to=2-2]
	\arrow["{{{l_\U}}}"', from=2-1, to=2-2]
\end{tikzcd}\]
  Suppose $\CC$ has a terminal object $1$.
  A \emph{universe lift} consists of
  a morphism of universes $l : \uu_0 \to \uu_1$
  and a pair of maps $\ulcorner U_0 \urcorner : 1 \to \U_1$
  and $a : \U_0 \to \Ulow_1$
  such that the following square is a pullback.
\[\begin{tikzcd}
	{\U_0} & {\Ulow_1} \\
	1 & {\U_1}
	\arrow["{{a}}", from=1-1, to=1-2]
	\arrow[from=1-1, to=2-1]
	\arrow["\lrcorner"{anchor=center, pos=0.125}, draw=none, from=1-1, to=2-2]
	\arrow["{{{\uu_1}}}", from=1-2, to=2-2]
	\arrow["{{{\ulcorner U_0 \urcorner}}}"', from=2-1, to=2-2]
\end{tikzcd}\]
\end{definition}

\subsection*{Unalgebraic path types}

Path types will be covered in more detail in \Cref{sec:path-types},
in an algebraic style.
For HoTTLean, an unalgebraic formulation of path types
was used in the HoTTLean formalisation
to construct identity types in the groupoid model.
In this section, we briefly introduce the relevant definitions
of unalgebraic path types,
and leave a more detailed discussion to \Cref{sec:path-types}.

\begin{definition}[\done{https://github.com/sinhp/HoTTLean/blob/HoTTUF/HoTTLean/Model/Unstructured/Hurewicz.lean\#L22}]
  \label{def:cylinder-structure}
  A cylinder structure on a category $\CC$
  consists of
  \begin{itemize}
    \item an endofunctor ${\cyl} : \CC \to \CC$, called the \emph{cylinder}
    \item and natural transformations
      \[\delta_0, \delta_1 : \id_\CC \to {\cyl}
      \quad \quad
      \pi : {\cyl} \to \id_\CC
      \quad \quad 
      \psi : \cyl \circ \cyl \to \cyl \circ \cyl\]
    \item satisfying the following equations
    \[
    \begin{tikzcd}
      \id & {\cyl} & \id & {{\cyl} \circ {\cyl}} & {{\cyl} \circ {\cyl}} \\
      & \id &&& {{\cyl} \circ {\cyl}}
      \arrow["{\delta_0}", from=1-1, to=1-2]
      \arrow[equals, from=1-1, to=2-2]
      \arrow["\pi"', from=1-2, to=2-2]
      \arrow["{\delta_0}"', from=1-3, to=1-2]
      \arrow[equals, from=1-3, to=2-2]
      \arrow["\psi", from=1-4, to=1-5]
      \arrow[equals, from=1-4, to=2-5]
      \arrow["\psi", from=1-5, to=2-5]
    \end{tikzcd}
    \]
    \[ \begin{tikzcd}
      {\cyl} & {\cyl} & {\cyl} & {\cyl} & {{\cyl} \circ {\cyl}} & {{\cyl} \circ {\cyl}} \\
        {{\cyl} \circ {\cyl}} & {{\cyl} \circ {\cyl}} & {{\cyl} \circ {\cyl}} & {{\cyl} \circ {\cyl}} & {\cyl} & {\cyl}
        \arrow[equals, from=1-1, to=1-2]
        \arrow["{\delta_0 {\cyl}}"', from=1-1, to=2-1]
        \arrow["{{\cyl} \delta_0}", from=1-2, to=2-2]
        \arrow[equals, from=1-3, to=1-4]
        \arrow["{\delta_1 {\cyl}}"', from=1-3, to=2-3]
        \arrow["{{\cyl} \delta_1}", from=1-4, to=2-4]
        \arrow["\psi", from=1-5, to=1-6]
        \arrow["{\pi {\cyl}}"', from=1-5, to=2-5]
        \arrow["{{\cyl} \pi}", from=1-6, to=2-6]
        \arrow["\psi", from=2-1, to=2-2]
        \arrow["\psi", from=2-3, to=2-4]
        \arrow[equals, from=2-5, to=2-6]
      \end{tikzcd}\]
  \end{itemize}
\end{definition}

For example, a cylinder structure can be defined
using a bipointed object and binary products.
Suppose $\CC$ has a terminal object and $I$ is
a bipointed object $\delta_0, \delta_1 : 1 \to I$.
Then we can define the cylinder functor by $\cyl : X \mapsto X \times I$.

\begin{definition}[$\Path$-types, \done{https://github.com/sinhp/HoTTLean/blob/HoTTUF/HoTTLean/Model/Unstructured/Hurewicz.lean\#L238}]
  Suppose $\CC$ has a terminal object and
  a cylinder structure.
  A universe $\uu : \Ulow \to \U$ admits
  \emph{unalgebraic $\Path$-types}
  when we have the following.
  \begin{enumerate}
  \item 
    For any context $\Gamma$, type $A : \Gamma \to \U$,
    and terms $a_0, a_1 : \Gamma \to \Ulow$ such that
    $\uu \circ a_i  = A$,
    there is a type $\Path_A (a_0,a_1) : \Gamma \to \U$.
  \item
    The construction $\Path_A (a_0,a_1)$ is stable under substitution.
  \item
    For any context $\Gamma$, type $A : \Gamma \to \U$,
    and ``path'' (or ``homotopy'') $p : \cyl \Gamma \to \Ulow$ such that
    $\uu \circ p = A \circ \pi$,
    there is a term $\path  p : \Gamma \to \Ulow$
    satisfying
    \[\uu \circ \path  p  = {\Path}_A (p \circ \delta_0,p \circ \delta_1)\]
  \item The construction $\path$ is stable under substitution.
  \item
    For any context $\Gamma$, type $A : \Gamma \to \U$,
    a pair of terms $a_0, a_1 : \Gamma \to \Ulow$ such that
    $\uu \circ a_0 = \uu \circ a_1 = A$,
    and a term $p : \Gamma \to \Ulow$ such that
    $\uu \circ p = \Path_A (a_0,a_1)$,
    there is a term $\unpath  p : \cyl \Gamma \to \Ulow$
    satisfying
    \[\uu \circ \unpath  p = A \circ \pi
      \quad \quad
      \unpath p \circ \delta_0  = a_0
      \quad \quad
      \unpath p \circ \delta_1  = a_1
    \]
  \item Under the assumptions of (3),
    \[ \unpath (\path p) = p \]
  \item Under the assumptions of (5),
    \[ \path (\unpath p) = p \]
  \end{enumerate}
\end{definition}

Unalgebraic $\Path$-types are reformulated algebraically in
\Cref{def:mltt-alg-path-types-2}.

% The following is closely related to opfibrancy in $\Cat$
% (\Cref{def:split-opfibration}).
\begin{definition}[\done{https://github.com/sinhp/HoTTLean/blob/HoTTUF/HoTTLean/Model/Unstructured/Hurewicz.lean\#L148}]
  \label{def:hurewicz-defs}
  Suppose $\CC$ admits a cylinder structure.
  Let $f : A \to B$ be a morphism in $\CC$.
  A \emph{Hurewicz structure $(f,l)$} consists of
  \begin{itemize} 
    \item for every object $X$ in $\CC$,
    and every pair of maps $a : X \to A$
    and $p : \cyl X \to B$ such that
    $f \circ a = p \circ \delta_0$,
    a chosen diagonal filler $l_{(a,p)} : \cyl X \to A$,
    \[\begin{tikzcd}
        X & A \\
        {\cyl X} & B
        \arrow["a", from=1-1, to=1-2]
        \arrow["\delta_0"', from=1-1, to=2-1]
        \arrow["f", from=1-2, to=2-2]
        \arrow["{l_{(a,p)}}"{description}, dashed, from=2-1, to=1-2]
        \arrow["p"', from=2-1, to=2-2]
      \end{tikzcd} \]
    \item such that for any morphism $x : X' \to X$,
    the following coherence equation between lifts
    $l_{(a \circ x,p \circ \cyl x)} : \cyl X' \to A$
    and $l_{(a,p)} : \cyl X \to A$ holds.
    \[ l_{(a,p)} \circ \cyl x = l_{(a \circ x,p \circ \cyl x)} \]
  \end{itemize}

  We say that $(f,l)$ is \emph{normal} if
  for every $b : X \to B$, we have $l_{(a,b \circ \pi)} = a \circ \pi$. 
  \[\begin{tikzcd}
      X && A \\
      {\cyl X} & X & B
      \arrow["a", from=1-1, to=1-3]
      \arrow["{\delta_0}"', from=1-1, to=2-1]
      \arrow["f", from=1-3, to=2-3]
      \arrow["{l_{(a,b \circ \pi)}}"{description}, dashed, from=2-1, to=1-3]
      \arrow["\pi"', from=2-1, to=2-2]
      \arrow["a"{description}, from=2-2, to=1-3]
      \arrow["b"', from=2-2, to=2-3]
    \end{tikzcd}
    \]  
\end{definition}
When $\cyl = (-) \times I$ is given by product with an interval object,
a Hurewicz structure on $f : A \to B$ is the usual homotopy lifting
definition from \cite{awodey2026}.

\begin{theorem}[Identity elimination, \done{https://github.com/sinhp/HoTTLean/blob/HoTTUF/HoTTLean/Model/Unstructured/Hurewicz.lean\#L743}]
  \label{thm:id-elim-2}
  If $\uu : \Ulow \to \U$ admits unalgebraic $\Path$-types
  and a normal Hurewicz structure $(\uu,l)$,
  then $\uu$ also admits unalgebraic $\Id$-types.
\end{theorem}
A proof of \Cref{thm:id-elim-2} in the style of algebraic models is provided in \Cref{thm:id-elim}.

\section{The groupoid model}

\begin{definition}[Contexts \done{https://github.com/sinhp/HoTTLean/blob/31133dd5b25226ea897f8aa5e2e43b61392459eb/HoTTLean/Groupoids/Basic.lean\#L65}]
The category $\CC$ is taken to be the category of groupoids.
In Mathlib, there are two universe-level variables
associated with the category of groupoids,
one for the size of objects and one for the size of hom-sets,
both of which we (arbitrarily) take to be $4$.
\[ \CC := \Grpd.\{4,4\} \]
\end{definition}

\begin{definition}\label{def:split-isofibration}
  A cloven isofibration consists of a pair $(F,l)$,
  where $F : \DD \to \CC$ is a functor
  and $l$ provides for each object $x \in \DD$, each object $y \in \CC$,
  and each isomorphism $f : F x \to y$ in $\CC$,
  a chosen isomorphism $l_{(x,f)} : x \to y'$
  such that $F l_{(x,f)} = f$,
  called the lift of $f$.

  A cloven isofibration $(F : \DD \to \CC,l)$
  is \emph{normal} if 
  for any $x \in \DD$ the lift of the identity
  $\id_{F x} : F x \to F x$ is the identity
  \[l_{(x,\id_{F x})} = \id_x : x \to x\]
  In particular, the endpoint of the lift satisfies $(F x)' = x$.
  We call $(F,l)$ a \emph{normal isofibration} for short.

  A normal isofibration $(F : \DD \to \CC, l)$ is \emph{split} if
  for any $x \in \DD$, $y$ and $z$ in $\CC$ and
  isomorphisms $f : F x \to y$ and $g : y \to z$,
  the lift of the composition is the composition of the lifts
  \[l_{(x,g \circ f)} = l_{(x,g)} \circ l_{(y',f)} : x \to z' \]
  We call $(F,l)$ a \emph{split isofibration} for short.
\end{definition}

\begin{definition}[Universes and context extension \done{https://github.com/sinhp/HoTTLean/blob/31133dd5b25226ea897f8aa5e2e43b61392459eb/HoTTLean/Groupoids/UnstructuredModel.lean\#L26}]
\label{sec:universe-in-groupoid-model}
We provide a universe (as in \Cref{def:elementary-universe})
for each Lean universe level $\alpha < 4$ polymorphically.
$\uu_\alpha$ will be called the \emph{universe $\alpha$-small isofibrations}
(since each of its pullbacks admits a canonical isofibration structure):
\begin{itemize}
  \item $\U$ is the core of the category of ($\alpha$-small) groupoids
  \[\U_\alpha := \Core(\Grpd.\{\alpha,\alpha\})\]
  \item $\Ulow$ is the core of the category of ($\alpha$-small) pointed groupoids,
  \[\Ulow_\alpha := \Core(\PGrpd.\{\alpha,\alpha\})\]
  where the category of $\alpha$-small pointed groupoids is the
  Grothendieck construction on the forgetful functor from groupoids to categories
  $\Grpd.\{\alpha,\alpha\} \to \Cat.\{\alpha,\alpha\}$.
  \item The map $\uu$ is the image of the forgetful functor
  from pointed groupoids to groupoids $\PGrpd \to \Grpd$.
  \[ \uu_\alpha : \Ulow_\alpha \to \U_\alpha\]
  \item Pullbacks of each universe $\uu : \Ulow \to \U$
    are computed as Grothendieck constructions
  \[ \begin{tikzcd}
	{\int (i \circ A)} & {\Core(\PGrpd)} & \PGrpd \\
	\Gamma & {\Core(\Grpd)} & \Grpd
	\arrow[dashed, from=1-1, to=1-2]
	\arrow[from=1-1, to=2-1]
	\arrow["\lrcorner"{anchor=center, pos=0.125}, draw=none, from=1-1, to=2-2]
	\arrow[from=1-2, to=1-3]
	\arrow[from=1-2, to=2-2]
	\arrow["\lrcorner"{anchor=center, pos=0.125}, draw=none, from=1-2, to=2-3]
	\arrow[from=1-3, to=2-3]
	\arrow["A"', from=2-1, to=2-2]
	\arrow["i"', from=2-2, to=2-3]
  \end{tikzcd} \]
  where $i : \Core(\Grpd) \to \Grpd$ is the usual inclusion of the core category.
  We will use context extension notation for
  the pullback $\Gamma . A := {\int (i \circ A)}$.
  (We could alternatively have taken $\Ulow := \int i$ as the Grothendieck construction.)
\end{itemize}
\end{definition}

\begin{definition}[Universe lifts \done{https://github.com/sinhp/HoTTLean/blob/31133dd5b25226ea897f8aa5e2e43b61392459eb/HoTTLean/Groupoids/UHom.lean\#L33}]
Universes in the groupoid model are related by a chain of universe lifts
$\uu.\{\alpha\} \to \uu.\{\alpha+1\}$, for each $\alpha < 3$.
First, by lifting (Lean) types in universe $\alpha$ to universe $\alpha+1$,
we have a pullback square
\[
\begin{tikzcd}
	{\PGrpd.\{\alpha,\alpha\}} & {\PGrpd.\{\alpha+1,\alpha+1\}} \\
	{\Grpd.\{\alpha,\alpha\}} & {\Grpd.\{\alpha+1,\alpha+1\}}
	\arrow[from=1-1, to=1-2]
	\arrow[from=1-1, to=2-1]
	\arrow[from=1-2, to=2-2]
	\arrow[from=2-1, to=2-2]
    \arrow["\lrcorner"{anchor=center, pos=0.125}, draw=none, from=1-1, to=2-2]
\end{tikzcd}
\]
The functor $\Core : \Cat.\{4,4\} \to \Grpd.\{4,4\}$ is right adjoint to the forgetful functor,
hence it preserves pullbacks,
providing the following pullback square as the image
of the previous
\[
\begin{tikzcd}
	{\Ulow_\alpha} & {\Ulow_{\alpha+1}} \\
	{\U_\alpha} & {\U_{\alpha+1}}
	\arrow[from=1-1, to=1-2]
	\arrow[from=1-1, to=2-1]
    \arrow["\lrcorner"{anchor=center, pos=0.125}, draw=none, from=1-1, to=2-2]
	\arrow[from=1-2, to=2-2]
	\arrow[from=2-1, to=2-2]
\end{tikzcd}
\]
Secondly, we construct a functor $U_\alpha : 1 \to \U_{\alpha+1}$,
which picks out the $(\alpha+1)$-small groupoid $\Core(\Grpd.\{\alpha,\alpha\})$.
It follows that we have the required isomorphism
\[ 1 . U_\alpha \iso \U_\alpha \]
\end{definition}

\begin{definition}[$\Sigma$-types \done{https://github.com/sinhp/HoTTLean/blob/31133dd5b25226ea897f8aa5e2e43b61392459eb/HoTTLean/Groupoids/UHom.lean\#L50}]
For the formation rule of sigma types,
we require for each pair of functors $A : \Gamma \to \U$
and $B : \Gamma . A \to \U$,
a functor $\Sigma_A B : \Gamma \to \U$.
This amounts to constructing for each pair of functors
$A' : \Gamma \to \Grpd$
and $B' : \int A' \to \Grpd$,
a functor $\Sigma_{A'} B' : \Gamma \to \Grpd$.
This can be done ``pointwise'' using Grothendieck constructions.
\[ \textstyle \Sigma_{A'} B' (x) = \int (B |_{A' x} : A' x \to \smallint A' \to \Grpd)\]
One checks that this operation is stable under substitution.
This construction of $\Sigma$-types is automatically universe-polymorphic:
if $A$ is $u$-small and $B$ is $v$-small, then $\Sigma_A B$ is $\max (u,v)$-small.

Rather than constructing unalgebraic operations that interpret the introduction
and elimination rules,
we construct an isomorphism
\[\textstyle\int (B' : \smallint A' \to \Grpd) \iso \int (\Sigma_{A'} B' : \Gamma \to \Grpd)\]
\end{definition}

\begin{definition}[$\Pi$-types \done{https://github.com/sinhp/HoTTLean/blob/31133dd5b25226ea897f8aa5e2e43b61392459eb/HoTTLean/Groupoids/UHom.lean\#L47}]
Our unalgebraic construction of $\Pi$-types
uses the construction of $\Sigma$-types:
the formation of $\Pi$-types is defined by
taking ``pointwise sections'' of the $\Sigma$-types.
As a result of constructing $\Pi$-types this way,
it is easier to not construct $\Pi$-types polymorphically, rather
if $A$ is $u$-small and $B$ is $u$-small, then $\Pi_A B$ is $u$-small.
Then, to obtain universe-polymorphic $\Pi$-types,
we simply combine the above construction with universe lifts.
\end{definition}

\begin{definition}[$\Id$-types \done{https://github.com/sinhp/HoTTLean/blob/31133dd5b25226ea897f8aa5e2e43b61392459eb/HoTTLean/Groupoids/UHom.lean\#L71}]
  \label{def:elementary-id-type}
To construct $\Id$-types in the groupoid model,
it suffices to identify a cylinder structure on the category of groupoids,
and $\Path$-types using the interval (\Cref{thm:id-elim-2}).
The cylinder endofunctor $\cyl : \CC \to \CC$ is constructed
by taking products with the walking isomorphism $I = \{\star \iso \star\}$
\[ \cyl X = X \times I\]
For the formation rule,
let $A : \Gamma \to \Grpd$ be a type and $a, b : \Gamma \to \PGrpd$ be two terms
in $A$.
We construct a functor $\Id_A (a,b) : \Gamma \to \Grpd$ such that
pointwise, it takes the discrete groupoid of morphisms
\[ \Id_A (a,b) (x) = \Hom_{A x} (a x, b x)\]
The $\path$-$\unpath$ bijection follows from
\[ \Hom_{A x} (a x, b x) = \{ f : I \to A x \st f \circ \delta_0 = a x, f \circ \delta_1  = b x\}\]

We construct a normal Hurewicz structure on $\uu : \Ulow \to \U$.
For any groupoid $\Gamma$ and functors
$p : I \times \Gamma \to \U$
and $p_0 : \Gamma \to \Ulow$ such that $\uu \circ p_0 = p \circ (0 \times \Gamma)$,
we construct a lift $l$:
\[ \begin{tikzcd}
	\Gamma & \Ulow \\
	{I \times \Gamma} & \U
	\arrow["{p_0}", from=1-1, to=1-2]
	\arrow["{0 \times \Gamma}"', from=1-1, to=2-1]
	\arrow[from=1-2, to=2-2]
	\arrow["l"{description}, dashed, from=2-1, to=1-2]
	\arrow["p"', from=2-1, to=2-2]
\end{tikzcd}\]
This is provided by the canonical split 
isofibration structure on $\uu : \Ulow \to \U$,
viewing $\Ulow \iso \int i$ as a Grothendieck construction over $\U$.
It follows that these lifts are \emph{normal},
meaning that for any $A : \Gamma \to \U$
and $a : \Gamma \to \Ulow$ such that $\uu \circ a = A$,
the lift provided for the following diagram
\[
\begin{tikzcd}
	\Gamma && \Ulow \\
	{I \times \Gamma} & \Gamma & \U
	\arrow["a", from=1-1, to=1-3]
	\arrow["{0 \times \Gamma}"', from=1-1, to=2-1]
	\arrow[from=1-3, to=2-3]
	\arrow["l"{description}, dashed, from=2-1, to=1-3]
	\arrow["{! \times \Gamma}"', from=2-1, to=2-2]
	\arrow["a"{description}, from=2-2, to=1-3]
	\arrow["A"', from=2-2, to=2-3]
\end{tikzcd}
\]
factors as $l = a \circ (! \times \Gamma)$.
\end{definition}

\section{Algebraic models}\label{sec:StrSem}

Here we provide an alternative algebraic formulation of an MLTT model,
following \cite{awodey2025}, but working in a $\pi$-clan (\Cref{def:pi-preclan}) instead.
For this section, we fix a $\pi$-clan $(\CC, \RR)$.

\begin{definition}[\done{https://github.com/sinhp/HoTTLean/blob/TYPES2026/HoTTLean/Model/Structured/StructuredUniverse.lean\#L20}]
  \label{def:algebraic-universe}
  An \emph{algebraic universe} in $(\CC,\RR)$ is a universe $\uu : \Ulow \to \U$ in
  $\CC$ which is also an $\RR$-map.
  We will just say a universe $\uu$ is $\RR$-algebraic, or algebraic when the $\pi$-clan
  is clear from the context.
\end{definition}

We can replicate the MLTT type formers described in \cite{awodey2025} in this setting,
making use of our general theory of polynomial functors.

\begin{definition}[$\Unit$-types]
  \label{def:alg-unit-type-mltt}
  Suppose $\CC$ has a terminal object $1$.
  We say that an algebraic universe $\uu : \Ulow \to \U$ admits
  \emph{algebraic $\Unit$-types} when we have maps
  $\Unit : 1 \to \Ulow$ and $\unit : 1 \to \U$, such that
  \[\begin{tikzcd}
    1 & \Ulow \\
    1 & \U
    \arrow["\unit", from=1-1, to=1-2]
    \arrow[equals, from=1-1, to=2-1]
    \arrow["\lrcorner"{anchor=center, pos=0.125}, draw=none, from=1-1, to=2-2]
    \arrow["\uu", from=1-2, to=2-2]
    \arrow["\Unit"', from=2-1, to=2-2]
  \end{tikzcd}\]
  is a pullback square.
\end{definition}

% In that case, the relevant pullback square is the following.
% \[
% \begin{tikzcd}
% 	{P_\vv \Ulow} & \Ulow \\
% 	{P_\vv \U} & \U
% 	\arrow["\lam", from=1-1, to=1-2]
% 	\arrow["{P_\vv \vv}"', from=1-1, to=2-1]
% 	\arrow["\lrcorner"{anchor=center, pos=0.125}, draw=none, from=1-1, to=2-2]
% 	\arrow["\vv", from=1-2, to=2-2]
% 	\arrow["\Pi"', from=2-1, to=2-2]
% \end{tikzcd}
% \]
% A more precise correspondence between an elementary $\Pi$-type structure
% and an algebraic one is given in \cite{hua2026}.
Recall that since $\uu$ is an $\RR$-map and $(\CC,\RR)$ is a $\pi$-clan,
we can take $\uu$ to be the signature for a polynomial functor $P_\uu : \CC \to \CC$
(\Cref{def:polynomial-functor-1}).

\begin{definition}[$\Sigma$-types, \done{https://github.com/sinhp/HoTTLean/blob/TYPES2026/HoTTLean/Model/Structured/StructuredUniverse.lean\#L651}]
    \label{def:alg-sig-type-mltt}
    
    Recall the definition of polynomial composition from \Cref{def:pcomp}.
    We consider the polynomial composition of $(\uu : \Ulow \to \U)$ with itself.
    For notational convenience, we write $Q$ for the domain of the composition.
    We say that an
    algebraic universe $\uu:\Ulow\to \U$ admits
    \emph{algebraic $\Sigma$-types} when we have maps
    $\Sigma : P_\uu \U \to \U$ and $\pair: Q \to \Ulow$,
    such that
    \[ \begin{tikzcd}
      Q & \Ulow \\
      {P_\uu\U} & \U
      \arrow["\pair", from=1-1, to=1-2]
      \arrow["{\uu \pcomp \uu}"', from=1-1, to=2-1]
      \arrow["\lrcorner"{anchor=center, pos=0.125}, draw=none, from=1-1, to=2-2]
      \arrow["\uu", from=1-2, to=2-2]
      \arrow["\Sigma"', from=2-1, to=2-2]
    \end{tikzcd} \]
    is a pullback square.
\end{definition}

\begin{definition}[$\Pi$-types, \done{https://github.com/sinhp/HoTTLean/blob/TYPES2026/HoTTLean/Model/Structured/StructuredUniverse.lean\#L359}]
    \label{def:alg-pi-type-mltt}
    We say that an algebraic universe $\uu:\Ulow\to \U$
    admits \emph{algebraic $\Pi$-types} when we have maps
    $\Pi : P_\uu \U \to \U$ and
    $\lam : P_\uu \Ulow\to \Ulow$, such that
    \[ \begin{tikzcd}
      {P_{\uu}{\Ulow}} & \Ulow \\
      {P_{\uu}{\U}} & \U
      \arrow["\lam", from=1-1, to=1-2]
      \arrow["{{P_{\uu}{\uu}}}"', from=1-1, to=2-1]
      \arrow["\lrcorner"{anchor=center, pos=0.125}, draw=none, from=1-1, to=2-2]
      \arrow["\uu", from=1-2, to=2-2]
      \arrow["\Pi"', from=2-1, to=2-2]
    \end{tikzcd} \]
    is a pullback square.
\end{definition}

% Like with elementary $\Pi$-types,
% our formalisation formulates a more general definition
% using three universes $\uu_0 : \Ulow_0 \to \U_0$,
% $\uu_1 : \Ulow_1 \to \U_1$ and $\uu_2 : \Ulow_2 \to \U_2$.
% \[\begin{tikzcd}
% 	{P_{\uu_0}{\Ulow_1}} & {\Ulow_2} \\
% 	{P_{\uu_0}{\U_1}} & {\U_2}
% 	\arrow["\lam", from=1-1, to=1-2]
% 	\arrow["{{P_{\uu_0}{\uu_1}}}"', from=1-1, to=2-1]
% 	\arrow["{\uu_2}", from=1-2, to=2-2]
% 	\arrow["\Pi"', from=2-1, to=2-2]
% \end{tikzcd}\]

% Let us begin by recalling what we mean by a universe having algebraic identity types,
% from \cite{awodey2025}.
% This could be an application of the theory of hom-types from \Cref{sec:cat-as-types},
% where the twisted structure is now degenerate,
% but we repeat the definitions for the sake of clarity.

% \Cref{def:algebraic-id-types} is reformulated in an elementary style in \Cref{def:elementary-id-type}.

\begin{definition}\label{def:algebraic-id-form-intro}
  We will say that an algebraic universe $\uu : \Ulow \to \U$
  admits identity formation and introduction when
  we have a pair of maps $i : \Ulow \to \Ulow$
  and $\Id : \Ulow \times_\U \Ulow \to \Ulow$
  forming a commutative square
  \[ \begin{tikzcd}
    \Ulow & \Ulow \\
    {\Ulow \times_\U \Ulow} & \U
    \arrow["i", from=1-1, to=1-2]
    \arrow["\Delta"', from=1-1, to=2-1]
    \arrow[from=1-2, to=2-2]
    \arrow["\Id"', from=2-1, to=2-2]
  \end{tikzcd} \]
  where the map $\Delta : \Ulow \to {\Ulow \times_\U \Ulow}$ is the diagonal. 
\end{definition}

Assuming that we have structure for formation and introduction of
identity types on a universe $\uu : \Ulow \to \U$,
we can take the pullback of $\Id$ to produce a comparison map
$\rfl : \Ulow \to I$ that factors the diagonal through the pullback.
\[\begin{tikzcd}
  \Ulow && \\
  & I & \Ulow \\
  & {\Ulow \times_\U \Ulow} & \U
  \arrow["\rfl"{description}, dashed, from=1-1, to=2-2]
  \arrow["i", bend left, from=1-1, to=2-3]
  \arrow["\Delta"', bend right, from=1-1, to=3-2]
  \arrow[from=2-2, to=2-3]
  \arrow["d_I"', from=2-2, to=3-2]
  \arrow["\lrcorner"{anchor=center, pos=0.125}, draw=none, from=2-2, to=3-3]
  \arrow[from=2-3, to=3-3]
  \arrow["\Id"', from=3-2, to=3-3]
\end{tikzcd}\]
Let us write $\pp = \uu \circ \pi_0 \circ d_I = \uu \circ \pi_1 \circ d_I : I \to \U$ for the composition 
\[ \begin{tikzcd}
	I & {\Ulow \times_\U \Ulow} & \Ulow & \U
	\arrow["{d_I}", from=1-1, to=1-2]
	\arrow["{\pi_0}", shift left, from=1-2, to=1-3]
	\arrow["{\pi_1}"', shift right, from=1-2, to=1-3]
	\arrow["\uu", from=1-3, to=1-4]
\end{tikzcd}\]
Consider the following commuting triangle.
\[ \begin{tikzcd}
	\Ulow & I \\
	& \U
	\arrow["\rfl", from=1-1, to=1-2]
	\arrow["{\uu}"', from=1-1, to=2-2]
	\arrow["\pp", from=1-2, to=2-2]
\end{tikzcd} \]
Both $(\uu : \Ulow \to \U)$ and $(\pp : I \to \U)$ are
polynomial signatures in $\Poly_{(\CC,\RR)}$,
since $(\CC,\RR)$ is a $\pi$-clan and they are both $\RR$-maps
($\pp$ is a composition of $\RR$-maps).
Using \Cref{def:vert-trans},
this triangle gives rise to a vertical natural transformation
between the polynomial functors.
\[P_{\rfl} : P_{\pp} \to P_{\uu}\]
We consider the naturality square of $P_{\rfl}$
for the morphism $\uu : \Ulow \to \U$.
\[ \begin{tikzcd}
	{ P_{\pp} \Ulow} & {P_{\uu} \Ulow} \\
	{ P_{\pp} \U} & {P_{\uu} \U}
	\arrow["{{P_{\rfl} \Ulow}}", from=1-1, to=1-2]
	\arrow["{{ P_{\pp} \uu}}"', from=1-1, to=2-1]
	\arrow["{{P_{\uu} \uu}}", from=1-2, to=2-2]
	\arrow["{{P_{\rfl} \U}}"', from=2-1, to=2-2]
\end{tikzcd} \]
Since $P_\uu$ preserves $\RR$-maps
(\Cref{prop:polynomial-functor-preclan-morphism}),
the map $P_\uu \uu : P_\uu \Ulow \to P_\uu \U$ is an $\RR$-map,
giving us a pullback $L$:
  \begin{equation}
    \label{eq:id-elimination-2}
    \begin{tikzcd}
    { P_{\pp} \Ulow} && \\
    & L & {P_{\uu} \Ulow} \\
    & { P_{\pp} \U} & {P_{\uu} \U}
    \arrow["p", dashed, from=1-1, to=2-2]
    \arrow["{P_{\rfl} \Ulow}", bend left, from=1-1, to=2-3]
    \arrow["{P_{\pp} \uu}"', bend right, from=1-1, to=3-2]
    \arrow["\phi", from=2-2, to=2-3]
    \arrow["\psi"', from=2-2, to=3-2]
    \arrow["\lrcorner"{anchor=center, pos=0.125}, draw=none, from=2-2, to=3-3]
    \arrow["{P_{\uu} \uu}", from=2-3, to=3-3]
    \arrow["{P_{\rfl} \U}"', from=3-2, to=3-3]
  \end{tikzcd}
  \end{equation}
By the universal property of polynomials
and commutativity of the square,
$\psi$ and $\phi$ are equivalent to the data of three maps
\[\psi_1 = \phi_1 : L \to \U
\quad \psi_2 : \Id_A \to \U
\quad \phi_2 : {A} \to \Ulow\]
where $A := L \opext \psi_1 \to L$ is
given by pullback of $\uu : \Ulow \to \U$ along $\psi_1 : L \to \U$,
and $\Id_A := L \times_U I$ is the pullback in the following diagram
\begin{equation}\label{eq:id-elimination-1}
  \begin{tikzcd}
	\Ulow & {{A}} & \Ulow \\
	\U & {\Id_{A}} & I \\
	& L & \U
	\arrow["\uu"', from=1-1, to=2-1]
	\arrow["{\phi_2}"', from=1-2, to=1-1]
	\arrow[from=1-2, to=1-3]
	\arrow["{\rfl}", from=1-2, to=2-2]
	\arrow["\lrcorner"{anchor=center, pos=0.125}, draw=none, from=1-2, to=2-3]
	\arrow["\rfl", from=1-3, to=2-3]
	\arrow["{\uu}", bend left = 50, from=1-3, to=3-3]
	\arrow["{\psi_2}", from=2-2, to=2-1]
	\arrow[from=2-2, to=2-3]
	\arrow[from=2-2, to=3-2]
	\arrow["\lrcorner"{anchor=center, pos=0.125}, draw=none, from=2-2, to=3-3]
	\arrow["\pp", from=2-3, to=3-3]
	\arrow["{\psi_1}"', from=3-2, to=3-3]
\end{tikzcd}
\end{equation}

\begin{lemma}
  \label{lem:psi-fibration}
  The morphism $\psi_1 : L \to \U$ is an $\RR$-map.
\end{lemma}
\begin{proof}
  The map $\psi_1 : L \to \U$ is a composition of two maps
  $\psi : L \to P_\pp \U$ and
  $\fstProj : P_\pp \U \to \U$,
  the first projection of the polynomial from \Cref{prop:poly-up}.
  The first map $\psi$ is a pullback of $P_\uu \uu$,
  which is an $\RR$-map
  since $P_\uu$ preserves $\RR$-maps
  (\Cref{prop:polynomial-functor-preclan-morphism})
  As for $\fstProj$,
  recall that evaluating a polynomial at the terminal object
  recovers the base of the signature
  $P_\pp 1 \iso \U$.
  The map $\fstProj : P_\pp \U \to \U$ is therefore isomorphic to
  $P_\pp ! : P_\pp \U \to P_\pp 1$,
  where $! : \U \to 1$.
  Finally, $P_\pp !$ is an $\RR$-map because $P_\pp$ preserves $\RR$-maps
  and $!$ is an $\RR$-map.
\end{proof}

\begin{proposition}[Identity elimination]\label{prop:id-elim-equiv}
  Suppose $\uu : \Ulow \to \U$ is an algebraic universe that admits
  identity formation and introduction.
  The following are equivalent (structures)
  % for a universe $\uu : \Ulow \to \U$ that admits
  % identity formation and introduction.
  \begin{itemize}
    \item A section $j : L \to P_{\pp} \Ulow$
    of the comparison map $p : P_{\pp} \Ulow \to L$
    from \cref{eq:id-elimination-2}.
    \[ p \circ j = \id\]
    \item A diagonal lift $j : \Id_{A} \to \Ulow$
    as indicated in the following diagram,
    extracted from \cref{eq:id-elimination-1}.
    \[\begin{tikzcd}
        {{A}} & \Ulow \\
        {\Id_A} & \U
        \arrow["{{\phi_2}}", from=1-1, to=1-2]
        \arrow["{{\rfl}}"', from=1-1, to=2-1]
        \arrow["\uu", from=1-2, to=2-2]
        \arrow["j"{description}, dashed, from=2-1, to=1-2]
        \arrow["{{\psi_2}}"', from=2-1, to=2-2]
      \end{tikzcd}\]
    \item A diagonal lift $j' : \Id_{A} \to C$
    as indicated in the following diagram.
    \[\begin{tikzcd}
        {{A}} & C & \Ulow \\
        {\Id_{A}} & {\Id_{A}} & \U
        \arrow["r", dashed, from=1-1, to=1-2]
        \arrow["{{{\phi_2}}}", bend left, from=1-1, to=1-3]
        \arrow["{{{\rfl}}}"', from=1-1, to=2-1]
        \arrow[from=1-2, to=1-3]
        \arrow[from=1-2, to=2-2]
        \arrow["\lrcorner"{anchor=center, pos=0.125}, draw=none, from=1-2, to=2-3]
        \arrow["\uu", from=1-3, to=2-3]
        \arrow["{j'}"{description}, dashed, from=2-1, to=1-2]
        \arrow[equals, from=2-1, to=2-2]
        \arrow["{{{\psi_2}}}"', from=2-2, to=2-3]
      \end{tikzcd}\]
  \end{itemize}
\end{proposition}

\begin{definition}[$\Id$-types]\label{def:algebraic-id-types}
  If an algebraic universe admits identity formation and introduction and
  any one of the equivalent conditions of \Cref{prop:id-elim-equiv} is met (elimination),
  we will say that the universe has $\Id$-type structure.
\end{definition}

\begin{proposition}
  If $\uu : \Ulow \to \U$ is an $\RR$-algebraic universe,
  and $\uu$ admits algebraic $\Unit$-types
  (respectively $\Sigma$-types, $\Pi$-types, $\Id$-types),
  it also admits unalgebraic $\Unit$-types
  (respectively $\Sigma$-types, $\Pi$-types, $\Id$-types).
\end{proposition}
\begin{proof}
  This is routine (cf. \cite{awodey2025, hua2026}).
\end{proof}
As a result, when the universe is algebraic,
we can drop the adjective ``algebraic'' and ``unalgebraic''
for type formers.

\subsection*{Constructing algebraic models}

\begin{definition}\label{def:alg-universe-classifying-map}
  Suppose $\uu : \Ulow \to \U$ is a universe.
  We denote the set of pullbacks of $\uu$ by $\RR_\U$
  and call $\RR_\U$-maps \emph{small fibrations}.
  (If $\uu$ is $\RR$-algebraic then $\RR_\U \subseteq \RR$.)
  \[
    \RR_\U := \{ A : X \to \Gamma \st \begin{tikzcd}
    X & \Ulow \\
    \Gamma & \U
    \arrow["\exists", dashed, from=1-1, to=1-2]
    \arrow["A"', from=1-1, to=2-1]
    \arrow["\lrcorner"{anchor=center, pos=0.125}, draw=none, from=1-1, to=2-2]
    \arrow["\uu", from=1-2, to=2-2]
    \arrow["\exists"', dashed, from=2-1, to=2-2]
  \end{tikzcd}\}
  \]
  Then we say that $\uu$ \emph{has chosen classifying maps}
  when every small fibration $A : X \to \Gamma$ in $\RR_\U$
  is equipped with a {chosen classifying map}
  $\ulcorner A \urcorner : \Gamma \to \U$,
  so that
  \[ \begin{tikzcd}
    X & \Ulow \\
    \Gamma & \U
    \arrow[from=1-1, to=1-2]
    \arrow["A"', from=1-1, to=2-1]
    \arrow["\lrcorner"{anchor=center, pos=0.125}, draw=none, from=1-1, to=2-2]
    \arrow["\uu", from=1-2, to=2-2]
    \arrow["{{\ulcorner A \urcorner}}"', from=2-1, to=2-2]
  \end{tikzcd} \]
  is a pullback square.
\end{definition}
When the axiom of choice is available,
this condition is trivial.
In many models, such as the groupoid model with $\uu : \Ulow \to \U$
as the universe of small isofibrations,
it is not constructively true that the universe has chosen classifying maps.
In this case, we should really work in a setting where $\RR_\U$
is not just a class of maps satisfying a property,
but a class of maps with structure.
% Corresponding to this,
% the structure of a small split isofibration is preserved under 
% pullback, composition, isomorphism, and pushforward.
For now, we will make it clear where chosen classifying maps are being used,
and have the groupoid model as a non-constructive example.
% For the universe of small split isofibrations in the category of groupoids
% (\Cref{sec:universe-in-groupoid-model}),
% this condition is constructively satisfied by taking the fibres functor
% of a split isofibration (cf. \Cref{prop:universal-opfibration-classifies}).

\begin{proposition}\label{prop:constructing-alg-models}
  Suppose $(\CC,\RR)$ is a $\pi$-clan and $\uu : \Ulow \to \U$ is an $\RR$-algebraic universe
  that has chosen classifying maps.
  If $\RR_\U$ is a $\pi$-preclan then
  $\uu$ admits $\Unit$-types, $\Sigma$-types and $\Pi$-types.
\end{proposition}
\begin{proof}
  Since $\uu$ has chosen classifying maps,
  to show that $\uu$ admits $\Unit$-types,
  we only need to note that the identity $1 \to 1$
  is a small fibration.

  To show that $\uu$ admits $\Sigma$-types,
  we can simply recall that $\uu \pcomp \uu$
  is defined as a composition of pullbacks of small fibrations,
  and is therefore also a small fibration.

  To show that $\uu$ admits $\Pi$-types,
  it suffices to show that $P_\uu \uu$ is a small fibration.
  This follows from \Cref{lem:polynomial-functor-preserves-maps} with $\SS = \RR_\U$,
  which says that $P_\uu : (\RR(1),\RR_\U) \to (\RR(1),\RR_\U)$ is a preclan morphism.
\end{proof}

Modulo concerns about constructivism,
\Cref{prop:constructing-alg-models} makes models of MLTT much easier to set up.
For example,
when working with the groupoid model,
one only needs to check that the class of small split isofibrations
(meaning the maps $F : \CC \to \DD$ on which there merely exists the
structure of a small split isofibration $(F,l)$)
are closed under isomorphisms, compositions, and pushforwards.
(A stronger claim is also true: the structure of small split isofibrations
are preserved under these operations.)
The stability of type formers under substitution automatically
follows from the general theory of algebraic models.
% To construct algebraic $\Id$-types,
% we will use algebraic $\Path$-types,
% which we present in \Cref{sec:path-types}.
For constructing algebraic $\Id$-types,
one could either consider the usual factorisation conditions,
which we do below,
or algebraic $\Path$-types, as we present in \Cref{sec:path-types}.

\begin{definition}[Anodyne maps]
  Let $\RR$ be a class of maps in a category $\CC$.
  A morphism $f : X \to Y$ is $\RR$-anodyne
  when $f$ has the left lifting property against all $\RR$-maps,
  meaning that for any $\RR$-map $d : B \to A$, and morphisms
  $b : X \to B$ and $a : Y \to A$ such that $d \circ b = a \circ f$,
  there is a diagonal lift $t : Y \to B$.
  We draw $\RR$-maps with a double head $\twoheadrightarrow$
  and anodyne maps with a tail $\rightarrowtail$. 
  \[
    \begin{tikzcd}
    X & B \\
    Y & A
    \arrow[from=1-1, to=1-2]
    \arrow[tail, from=1-1, to=2-1]
    \arrow[two heads, from=1-2, to=2-2]
    \arrow[dashed, from=2-1, to=1-2]
    \arrow[from=2-1, to=2-2]
  \end{tikzcd}
  \]
  % A pretribe $(\CC,\RR)$ is a preclan such that 
  % \begin{itemize}
  %   \item every map $f : X \to Y$ between $\RR$-objects factors as an
  %   anodyne map followed by an $\RR$-map and
  %   \item anodyne maps are stable under pullback along $\RR$-maps.
  % \end{itemize}
\end{definition}

% \begin{definition}
%   A subtribe $(\CC,\RR,\RR_\U)$ consists of 
%   \begin{itemize}
%     \item a clan $(\CC,\RR)$ such that $\RR$-anodyne maps are stable under pullback
%     \item a preclan $\RR_\U \subseteq \RR$,
%     \item 
%   \end{itemize}
%   Let us say that $\RR_U \subseteq \RR$ is a pretribe in $\RR$
%   when any 
% \end{definition}

\begin{lemma}\label{lem:anodyne-pullback}
  If $(\CC,\RR)$ is a $\pi$-clan,
  then anodyne maps are stable under pullback along $\RR$-maps.
\end{lemma}
\begin{proof}
  See \cite[Lemma 29]{awodey2018}.
\end{proof}

\begin{proposition}\label{prop:constructing-alg-id-types}
  Suppose $(\CC,\RR)$ is a $\pi$-clan and
  $\uu : \Ulow \to \U$ is an $\RR$-algebraic universe that has chosen classifying maps.
  If the diagonal on the universe
  $\Delta : \Ulow \to \Ulow \times_\U \Ulow$
  factors as an $\RR$-anodyne\footnote{
    For the sake of avoiding the axiom of choice,
    we also need to assume that anodyne maps
    come equipped with certain chosen lifts.
    Let us do so implicitly.} %TODO
  map followed by an $\RR_\U$-map
  then $\uu$ admits $\Id$-types.
\end{proposition}
\begin{proof}
  Let us denote the given factorisation of $\Delta$ as follows.
  \[ \begin{tikzcd}
      \Ulow && {\Ulow \times_\U \Ulow} \\
      & I
      \arrow["\Delta", from=1-1, to=1-3]
      \arrow["\rfl"', tail, from=1-1, to=2-2]
      \arrow["{(\src,\trg)}"', two heads, from=2-2, to=1-3]
    \end{tikzcd} \]
  Identity formation and introduction follows from
  $(\src,\trg) : I \to {\Ulow \times_\U \Ulow}$
  being in $\RR_\U$.
  \[ \begin{tikzcd}
    \Ulow & I & \Ulow \\
    & {{\Ulow \times_\U \Ulow}} & \U
    \arrow["\rfl", from=1-1, to=1-2]
    \arrow["\Delta"', from=1-1, to=2-2]
    \arrow[dashed, from=1-2, to=1-3]
    \arrow[from=1-2, to=2-2]
    \arrow[from=1-3, to=2-3]
    \arrow["{\ulcorner I \urcorner}"', dashed, from=2-2, to=2-3]
    \arrow["\lrcorner"{anchor=center, pos=0.125}, draw=none, from=1-2, to=2-3]
  \end{tikzcd}\]
  By the second (or third) condition in \Cref{prop:id-elim-equiv},
  to provide identity elimination,
  it suffices to show that the map $A \to \Id_A$
  from the ``universal elimination'' diagram
  \labelcref{eq:id-elimination-1}, reproduced below, is anodyne.
  \[
    \begin{tikzcd}
    \Ulow & {{A}} & \Ulow \\
    \U & {\Id_{A}} & I \\
    & L & \U
    \arrow["\uu"', from=1-1, to=2-1]
    \arrow["{\phi_2}"', from=1-2, to=1-1]
    \arrow[from=1-2, to=1-3]
    \arrow["{\rfl}", from=1-2, to=2-2]
    \arrow["\lrcorner"{anchor=center, pos=0.125}, draw=none, from=1-2, to=2-3]
    \arrow["\rfl", from=1-3, to=2-3]
    \arrow["{\uu}", bend left = 50, from=1-3, to=3-3]
    \arrow["{\psi_2}", from=2-2, to=2-1]
    \arrow[from=2-2, to=2-3]
    \arrow[from=2-2, to=3-2]
    \arrow["\lrcorner"{anchor=center, pos=0.125}, draw=none, from=2-2, to=3-3]
    \arrow["\pp", from=2-3, to=3-3]
    \arrow["{\psi_1}"', from=3-2, to=3-3]
  \end{tikzcd}\]
  The map $A \to \Id_A$ is (by definition) a pullback of the map
  $\rfl : \Ulow \to I$ along $\psi_1 : L \to \U$.
  By \Cref{lem:anodyne-pullback},
  any pullback of an anodyne map along an $\RR$-map is anodyne.
  As required,
  $\rfl : \Ulow \to I$ is anodyne
  and $\psi_1 : L \to \U$ is an $\RR$-map \Cref{lem:psi-fibration}.
\end{proof}

% We stated the result of \Cref{prop:constructing-alg-id-types}
% propositionally,
% but we could also make the same constructive construction:
% If additionally, the universe has chosen classifying maps,
% lifts for anodyne maps
% and the anodyne-$\RR$-map factorisations are given as structure
% (rather than mere existence),
% then the universe admits $\Id$-types.

% The equivalent condition of \Cref{prop:pi-pretribe-equiv}
% are similar to the condition of being a sub-tribe in \cite{joyal2017},
% except, using the terminology of \cite{joyal2017},
% we require that $(\CC,\RR)$ is a $\pi$-clan instead of a tribe.

\section{Future work}
\label{sec:algebraic-vs-elementary}

\subsection*{HoTT0}
We have not shown that the groupoid model of MLTT is a model of HoTT0,
a fragment of HoTT in which univalence holds only on the set-truncated types \cite{hua2025}. 
The two axioms of HoTT0 are set-truncated univalence and function extensionality.
Though a brute-force proof of the axioms should be fairly straightforward,
we would prefer to first create a general user interface
for this kind of syntax-semantics reasoning,
using the tools developed in \cite{nawrocki2026}.

\subsection*{Algebraic models}
Constructing the groupoid model as an unalgebraic model comes with several problems. 
The main problems are those of type equalities and nested dependent equalities,
which ultimately add several minutes of compile time to the files defining type formers in the model.
Another issue is the code's reusability.
The unalgebraic constructions result in long files of code that pertain only to groupoids.
The unalgebraic approach also requires explicit proofs of
stability of all operations under substitution,
though these equations are usually very easy to prove.

Nonetheless, as a target for writing an interpretation function from syntax to semantics,
the unalgebraic model is ideal.
This is because each judgment in the syntax corresponds closely to a rule
in the unalgebraic semantics.
The HoTTLean library would benefit from having
a translation between the two descriptions of models,
taking advantage of both formulations.

To construct an algebraic groupoid model instead,
we would need to formalise the following.
\begin{enumerate}
  \item The theory of algebraic models, including the theory of polynomial functors in preclans.
  \item The $\pi$-clan of isofibrations (or split isofibrations) in groupoids.
  \item The universe of split small isofibrations.
  \item Small split isofibrations form a $\pi$-preclan.
  \item The pathobject factorisation of the diagonal of an isofibration
    (analogous to \Cref{prop:src-trg-discrete-opfibration}).
  \item Either showing that isofibrations are Hurewicz
    or that the pathobject factorisation is part of a weak factorisation system
    (analogous to \Cref{thm:lari-opfib-awfs}).
  \item An algebraic-to-unalgebraic model translation.
\end{enumerate}

We have strong reasons to believe that all of the items
in the list admit quite straightforward and ``nice'' formalisations
(some of the items have already been formalised),
possibly with the exception of showing that
split isofibrations are closed under pushforwards (as part of (4)).
Since the closure conditions of (4) are only required up to isomorphism
and not equality, we expect that there will be no equalities of types 
when working with groupoids.
Instead, those dependent equalities that are inherent to a strict model of
type theory are automatically and uniformly generated in (7),
resulting in a much faster compile time.
Moreover, every item in this list would make for a good contribution to Mathlib,
and are of general interest to the formalisation community.

Finally, there is the issue of using the axiom of choice when applying
\Cref{prop:constructing-alg-id-types,prop:constructing-alg-models}
to construct the groupoid model.
The propositions are currently stated using the class of maps $\RR_\U$
that are (propositionally) pullbacks of the universe,
meaning that the axiom of choice is required to extract a classifying map.
In Mathlib, the convention is usually to apply choice whenever needed,
but a constructive alternative formulation of the result should also be possible.
We leave this for future work.

\subsection*{Algebraic categorical models}

In \Cref{sec:cat-as-types},
we will develop a blueprint
for extending the HoTTLean formalisation of the groupoid model
to the category of categories $\Cat$.
% The development of \Cref{sec:cat-as-types} is designed in part as a blueprint
% for extending the HoTTLean formalisation of the groupoid model
% to the category of categories $\Cat$.
Indeed, many parts of the groupoid model deserve to be factored through
more general theorems in $\Cat$.
Such results would include the following:
\begin{itemize}
  \item The clans of (groupoidal or discrete) (split) (op)fibrations (\Cref{prop:preclans-in-cat})
    and the clan of isofibrations in $\Cat$.
  \item Straightening and unstraightening of split (op)fibrations (\Cref{thm:opfibration-main}).
  \item The universes of split (op)fibrations (\Cref{def:universal-small-split-opfibration}).
  \item The $\tau$-clan of opfibrations and fibrations in $\Cat$ (\Cref{prop:pushforward-in-cat}).
  \item The (lari, split opfibration) algebraic weak factorisation system (\Cref{thm:lari-opfib-awfs}).
\end{itemize}
Again, these results would make for a more modular library,
and could ultimately be part of a library for automated internal reasoning in $\Cat$.

\chapter{Path types}\label{sec:path-types}
In \cite{awodey2026},
a new approach to the semantics of identity types was introduced,
in a category with finite limits with a bipointed exponentiable object $I$,
playing the role of an interval. 
From a small number of easy-to-check assumptions on a universe,
including the presence of what we will call ``path types'',
this interval can be used to produce an elimination principle for identity types.

Though not all models with intensional identity types have path types,
many of the standard models of Martin-L\"of type theory have them.
These examples include the model in groupoids \cite{hofmann1995},
the discrete model in presheaves,
Voevodsky's model in simplicial sets \cite{kapulkin2021},
and various cubical models.

In addition to simplifying the construction of identity types in models of
type theory,
path types also present identity types in a style more consistent
with that of $\Sigma$ and $\Pi$-type formers in algebraic type theory.
In algebraic type theory, a universe has $\Sigma$-types and $\Pi$-types
precisely when there are pullback squares classifying certain
universal constructions by the universe of type families
(cf. \Cref{def:alg-pi-type-mltt,def:alg-sig-type-mltt}).
\[ \begin{tikzcd}
      {P_{\uu}{\Ulow}} & \Ulow \\
      {P_{\uu}{\U}} & \U
      \arrow["\lam", from=1-1, to=1-2]
      \arrow["{{P_{\uu}{\uu}}}"', from=1-1, to=2-1]
      \arrow["\lrcorner"{anchor=center, pos=0.125}, draw=none, from=1-1, to=2-2]
      \arrow["\uu", from=1-2, to=2-2]
      \arrow["\Pi"', from=2-1, to=2-2]
    \end{tikzcd} 
    \quad \quad 
    \begin{tikzcd}
      Q & \Ulow \\
      {P_\uu\U} & \U
      \arrow["\pair", from=1-1, to=1-2]
      \arrow["{\uu \pcomp \uu}"', from=1-1, to=2-1]
      \arrow["\lrcorner"{anchor=center, pos=0.125}, draw=none, from=1-1, to=2-2]
      \arrow["\uu", from=1-2, to=2-2]
      \arrow["\Sigma"', from=2-1, to=2-2]
    \end{tikzcd}\]
A similar pullback square will be used to define path types.

A version of the argument presented here has been formalised in Lean,
as part of the HoTTLean project \cite{hua2025},
which is summarised in \Cref{sec:hottlean}.
% The assumptions for having Path types also produces
% a cubical filling structure on the universe,
% making any type family classified by the universe a cubical Kan fibration.
We present a variation of the proof in \cite{awodey2026} in a style
that is more consistent with the formalisation,
using $\pi$-preclans rather than a lex category,
and replacing the assumption of an exponentiable interval with
cylinder and path functors.
In the formalisation,
abstracting away some of the cartesian structure
from \cite{awodey2026} simplifies the formalisation in Lean.
This results in an almost syntactic presentation of the argument,
where the cylinder and path structures are controlled by a \emph{modality}
(see for example \cite{licata2016, licata2017, gratzer2020, shulman2023}).
% Unlike \cite{awodey2026},
% we do not assume the underlying category is lex.
As a result, the setup will be slightly more complicated,
but the abstractions should ultimately clarify some technical
details appearing in the argument.

\section{Cylinders and paths}

We fix a $\pi$-preclan $(\CC,\RR)$ (\Cref{def:pi-preclan}) with a terminal object
and a cylinder structure (\Cref{def:cylinder-structure}).
We can use the cylinder structure
to define an endofunctor on the slice $\CC / \Gamma$
for every object $\Gamma \in \CC$.
\[
   \cyl[\Gamma] : \CC / \Gamma \to \CC / \Gamma
\]
\[
  (d : X \to \Gamma) \mapsto (d \circ \pi : \cyl X \to \Gamma)
\]
This results in a {cylinder structure in each slice $\CC / \Gamma$}.

Let us denote the bicategory of large categories using $\tcat$.
For a category $\CC$,
let us write $S(-) : \CC \to \tcat$
for the \emph{covariant slice pseudofunctor} (in fact a strict 2-functor),
that takes an object $\Gamma \in \CC$ to the slice
$S(\Gamma) = \CC / \Gamma$ and a morphism $\sigma : \Delta \to \Gamma$
to the postcomposition functor
\[ S(\sigma) = \sigma_! : \CC / \Delta \to \CC / \Gamma\]
\[ (h : X \to \Delta) \mapsto (\sigma \circ h : X \to \Gamma)\]

\begin{proposition}
  Given a cylinder structure on a category $\CC$,
  the cylinders over each $\Gamma \in \CC$
  \[\cyl[\Gamma] : \CC / \Gamma \to \CC / \Gamma\]
  constitute a pseudotransformation
  \[ {\cyl} : S (-) \to S (-) \]
  with modifications
  \[\delta_0, \delta_1 : \id \to {\cyl}
  \quad \quad
  \pi : {\cyl} \to \id
  \quad \quad 
  \psi : \cyl \circ \cyl \to \cyl \circ \cyl\]
   satisfying the same equations as before,
   but now as modifications.
    % \[
    % \begin{tikzcd}
    %   \id & {\cyl} & \id & {{\cyl} \circ {\cyl}} & {{\cyl} \circ {\cyl}} \\
    %   & \id &&& {{\cyl} \circ {\cyl}}
    %   \arrow["{\delta_0}", from=1-1, to=1-2]
    %   \arrow[equals, from=1-1, to=2-2]
    %   \arrow["\pi"', from=1-2, to=2-2]
    %   \arrow["{\delta_0}"', from=1-3, to=1-2]
    %   \arrow[equals, from=1-3, to=2-2]
    %   \arrow["\psi", from=1-4, to=1-5]
    %   \arrow[equals, from=1-4, to=2-5]
    %   \arrow["\psi", from=1-5, to=2-5]
    % \end{tikzcd}
    % \]
    % \[ \begin{tikzcd}
    %   {\cyl} & {\cyl} & {\cyl} & {\cyl} & {{\cyl} \circ {\cyl}} & {{\cyl} \circ {\cyl}} \\
    %     {{\cyl} \circ {\cyl}} & {{\cyl} \circ {\cyl}} & {{\cyl} \circ {\cyl}} & {{\cyl} \circ {\cyl}} & {\cyl} & {\cyl}
    %     \arrow[equals, from=1-1, to=1-2]
    %     \arrow["{\delta_0 {\cyl}}"', from=1-1, to=2-1]
    %     \arrow["{{\cyl} \delta_0}", from=1-2, to=2-2]
    %     \arrow[equals, from=1-3, to=1-4]
    %     \arrow["{\delta_1 {\cyl}}"', from=1-3, to=2-3]
    %     \arrow["{{\cyl} \delta_1}", from=1-4, to=2-4]
    %     \arrow["\psi", from=1-5, to=1-6]
    %     \arrow["{\pi {\cyl}}"', from=1-5, to=2-5]
    %     \arrow["{{\cyl} \pi}", from=1-6, to=2-6]
    %     \arrow["\psi", from=2-1, to=2-2]
    %     \arrow["\psi", from=2-3, to=2-4]
    %     \arrow[equals, from=2-5, to=2-6]
    %   \end{tikzcd}\]
\end{proposition}

Pseudonaturality of ${\cyl} : S(-) \to S(-)$ says that for any
map $\sigma : \Delta \to \Gamma$ there is an isomorphism
\begin{equation}\label{eq:frobenius-identity}
\begin{tikzcd}
	{\CC / \Delta} & {\CC / \Gamma} \\
	{\CC / \Delta} & {\CC / \Gamma}
	\arrow["{\sigma_!}", from=1-1, to=1-2]
  \arrow["{\cyl[\Delta]}"', from=1-1, to=2-1]
  \arrow["{\cyl[\Gamma]}", from=1-2, to=2-2]
	\arrow["{\sigma_!}"', from=2-1, to=2-2]
\end{tikzcd} \end{equation}

\begin{example}\label{ex:exponentiable-bipointed}
Consider the assumptions of \cite{awodey2026},
where there is a bipointed exponentiable object
$\delta_0, \delta_1 : 1 \rightrightarrows I$.
We could instantiate a cylinder structure
by taking ${\cyl} X := X \times I$, 
where $I$ is the given interval,
and $X$ is an object in the category.
It follows that the cylinder over $\Gamma$ of some $X$ in the slice
$\CC / \Gamma$ is
${\cyl}_{\Gamma} X := X \times_\Gamma I_\Gamma$.
Then the pseudonaturality square looks like the topos-theoretic ``Frobenius identity'':
\[ \sigma_! (X \times_\Delta I_\Delta) \iso (\sigma_! X) \times_\Gamma I_\Gamma\]
\end{example}

% \begin{axiom}\label{ax:preclan}
%   There is a $\pi$-preclan (\Cref{def:pi-preclan}) $(\CC,\RR)$.
% \end{axiom}

Recall the definition of a partial right adjoint
from \Cref{def:partial-right-adjoint}.
\begin{definition}\label{def:path-functor}
  We say that a $\pi$-preclan $(\CC,\RR)$ with a terminal object
  and a cylinder structure (\Cref{def:cylinder-structure})
  \emph{has path functors} if, additionally,
  for each object $\Gamma$ in $\CC$,
  the cylinder functor $\cyl[\Gamma] : \CC / \Gamma \to \CC / \Gamma$ over $\Gamma$
  has a partial right adjoint $\arr[\Gamma] : \RR(\Gamma) \to \RR(\Gamma)$.
  \[\begin{tikzcd}
    {\CC / \Gamma} & {\RR(\Gamma)} \\
    {\CC / \Gamma} & {\RR(\Gamma)}
    \arrow[""{name=0, anchor=center, inner sep=0}, "{\cyl[\Gamma]}"', from=1-1, to=2-1]
    \arrow[hook', from=1-2, to=1-1]
    \arrow[""{name=1, anchor=center, inner sep=0}, "{\arr[\Gamma]}"', from=2-2, to=1-2]
    \arrow[hook', from=2-2, to=2-1]
    \arrow["{{\dashv_\partial}}"{description}, draw=none, from=0, to=1]
  \end{tikzcd}\]
\end{definition}
We will call $\arr[\Gamma] A$ the \emph{paths} in $A$ over $\Gamma$,
and $\arr[\Gamma]$ the \emph{path functor}.

Partial right adjoints have counits (but not units on all elements);
we will denote the counit at an object $A \in \RR(\Gamma)$ by
\[ \epsilon_A : \cyl[\Gamma] \arr[\Gamma]{A} \to A\]
It is the morphism in $\CC / \Gamma$
that is conjugate to the identity $\id : \arr[\Gamma]{A} \to \arr[\Gamma]{A}$.

The $\pi$-preclan $\RR$ induces a pseudofunctor
$\RR(-) : \CC^\op \to \tcat$,
where the contravariant action on morphisms is given by pullback.

\begin{proposition}
  Suppose $(\CC,\RR)$ has path functors.
  The functors $\arr[\Gamma]$ together form a pseudotransformation
  $\arr : \RR(-) \to \RR(-)$.
  This pseudotransformation is equipped with modifications 
  \[\src, \trg : \arr \to {\id}
  \quad \quad
  \pi : {\id} \to \arr
  \quad \quad 
  \sym : \arr \circ \arr \to \arr \circ \arr\]
  satisfying the following equations
    \[
      \begin{tikzcd}
        & \id &&& {{\arr} \circ {\arr}} \\
        \id & \arr & \id & {{\arr} \circ {\arr}} & {{\arr} \circ {\arr}}
        \arrow[equals, from=1-2, to=2-1]
        \arrow["\rfl"{description}, from=1-2, to=2-2]
        \arrow[equals, from=1-2, to=2-3]
        \arrow["\src", from=2-2, to=2-1]
        \arrow["\trg"', from=2-2, to=2-3]
        \arrow[equals, from=2-4, to=1-5]
        \arrow["\sym"', from=2-4, to=2-5]
        \arrow["\sym"', from=2-5, to=1-5]
      \end{tikzcd}
    \]
    \[ \begin{tikzcd}
	{{\arr} \circ {\arr}} & {{\arr} \circ {\arr}} & {{\arr} \circ {\arr}} & {{\arr} \circ {\arr}} & \arr & \arr \\
	\arr & \arr & \arr & \arr & {{\arr} \circ {\arr}} & {{\arr} \circ {\arr}}
	\arrow["\sym", from=1-1, to=1-2]
	\arrow["{{\src {\arr}}}"', from=1-1, to=2-1]
	\arrow["{{{\arr} \src}}", from=1-2, to=2-2]
	\arrow["\sym", from=1-3, to=1-4]
	\arrow["{{\trg {\arr}}}"', from=1-3, to=2-3]
	\arrow["{{{\arr} \trg}}", from=1-4, to=2-4]
	\arrow[equals, from=1-5, to=1-6]
	\arrow["{{\rfl {\arr}}}"', from=1-5, to=2-5]
	\arrow["{{{\arr} \rfl}}", from=1-6, to=2-6]
	\arrow[equals, from=2-1, to=2-2]
	\arrow[equals, from=2-3, to=2-4]
	\arrow["\sym"', from=2-5, to=2-6]
\end{tikzcd}\]
\end{proposition}
\begin{proof}
  The pseudonaturality condition follows a standard
  (partial) adjoint hom-set bijection argument,
  making use of the Frobenius identity.
  Suppose $\sigma : \Delta \to \Gamma$,
  then naturally in $X \in \CC / \Delta$ we have,
  \begin{align*}
    & \CC / \Delta (X, \sigma^* \arr[\Gamma] A) \\
     \iso \; & \CC / \Gamma (\sigma_! X, \arr[\Gamma] A) \\
     \iso \; & \CC / \Gamma (\cyl[\Gamma] \sigma_! X, A) \\
     \iso \; & \CC / \Gamma (\sigma_! \cyl[\Delta] X, A) \\
     \iso \; & \CC / \Delta (\cyl[\Delta] X, \sigma^* A) \\
     \iso \; & \CC / \Delta (X, \arr[\Delta] \sigma^* A) \\
  \end{align*}
  From the Yoneda lemma,
  it follows that $\sigma^* \arr[\Gamma] A \iso \arr[\Delta] \sigma^* A$.
  This isomorphism holds in $\RR(\Delta)$ since it is a full
  subcategory of the slice.

  To define $\src : \arr[\Gamma] A \to A$ (and similarly $\trg$)
  for some $A \in \RR(\Gamma)$,
  consider the composition 
  \[ 
  \begin{tikzcd}
    {\arr[\Gamma]A} & {\cyl[\Gamma] \arr[\Gamma]{A}} \\
    & A
    \arrow["{\delta_0 }", from=1-1, to=1-2]
    \arrow["\src"', from=1-1, to=2-2]
    \arrow["{\epsilon_A}", from=1-2, to=2-2]
  \end{tikzcd}
  \]
  Under the partial adjunction,
  the map $\pi : \cyl[\Gamma] A \to A$ has a conjugate
  $\rfl : A \to \arr[\Gamma] A$.
  We can define
  $\sym : \arr[\Gamma] \arr[\Gamma] A \to \arr[\Gamma] \arr[\Gamma] A$
  via the following hom-set bijection
  \begin{align*}
    & \CC / X (\arr \arr A, \arr \arr A) & \ni \sym \\
    \iso \; & \CC / X (\cyl \arr \arr A, \arr A) \\
    \iso \; & \CC / X (\cyl \cyl \arr \arr A, A) \\
    \iso \; & \CC / X (\cyl \cyl \arr \arr A, A) & (\text{via } \psi) \\
    \iso \; & \CC / X (\cyl \arr \arr A, \arr A) \\
    \iso \; & \CC / X (\arr \arr A, \arr \arr A) & \ni \id
  \end{align*}
  
  We leave it to the reader to check that $\src, \trg, \rfl$,
  and $\sym$ are modifications
  that satisfy the required equations.
  Here, being a modification amounts to these operations being natural in
  $A \in \RR(\Gamma)$
  and stable (up to isomorphism) under pullback along maps $\Delta \to \Gamma$.
\end{proof}

For any $\RR$-object $A$ in the slice over $\Gamma$,
we obtain a \emph{pathobject factorisation} of the diagonal
$\Delta : A \to A \times_\Gamma A$ using the path functor.
\[\begin{tikzcd}
	{A} & {\arr[\Gamma]{A}} \\
	& {{A} \times_\Gamma A}
	\arrow["\rfl", from=1-1, to=1-2]
	\arrow["{\Delta}"', from=1-1, to=2-2]
	\arrow["{(\src,\trg)}", from=1-2, to=2-2]
\end{tikzcd}\]
% We will refer to the map $(\src,\trg) : \arr[\Gamma] A \to A$
% as the pathobject of $A$.

% Let us call the morphisms in $\RR$ fibrations.

% \begin{axiom}\label{ax:path-types-universe}
%   We have an $\RR$-algebraic universe (\Cref{def:algebraic-universe})
%   $\uu : \Ulow \to \U$.
%   We call $\uu$ \emph{the universe of small fibrations}.
%   Denoting the set of pullbacks of $\uu$ by $\RR_\U$,
%   we call maps in $\RR_\U$ \emph{small fibrations}.
%   We also require that every small fibration $A : X \to \Gamma$ in $\RR_\U$
%   is equipped with a chosen \emph{classifying map}
%   $\ulcorner A \urcorner : \Gamma \to \U$,
%   so that
%   \[ \begin{tikzcd}
% 	X & \Ulow \\
% 	\Gamma & \U
% 	\arrow[from=1-1, to=1-2]
% 	\arrow["A"', from=1-1, to=2-1]
% 	\arrow["\lrcorner"{anchor=center, pos=0.125}, draw=none, from=1-1, to=2-2]
% 	\arrow["\uu", from=1-2, to=2-2]
% 	\arrow["{{\ulcorner A \urcorner}}"', from=2-1, to=2-2]
% \end{tikzcd} \]
% \end{axiom}

Let us now consider the universal pathobject factorisation
on a universe $\uu : \Ulow \to \U$ in $\RR(\U)$.
We will say that $\uu : \Ulow \to \U$ is an $\RR$-algebraic
universe if it is an $\RR$-map and $\U$ is an $\RR$-object
(cf. \Cref{def:algebraic-universe} where $\RR$ was a clan).

\begin{definition}\label{def:mltt-alg-path-types-2}
  Suppose $(\CC,\RR)$ has path functors
  and an $\RR$-algebraic universe
  $\uu : \Ulow \to \U$.
  Then we will say that $\uu$ admits algebraic $\Path$-types
  when we have maps
  $\Path : \Ulow \times_\U \Ulow \to \U$
  and $\path : \arr[\U] \Ulow \to \Ulow$
  forming a pullback square
  \[
  \begin{tikzcd}
    {{\arr[\U] \Ulow}} & \Ulow \\
    {\Ulow \times_\U \Ulow} & \U
    \arrow["\path", from=1-1, to=1-2]
    \arrow["{(\src,\trg)}"', from=1-1, to=2-1]
    \arrow[from=1-2, to=2-2]
    \arrow["\Path"', from=2-1, to=2-2]
    \arrow["\lrcorner"{anchor=center, pos=0.125}, draw=none, from=1-1, to=2-2]
  \end{tikzcd}\]
\end{definition}

For an $\RR$-algebraic universe $\uu : \Ulow \to \U$.
Let us write $\RR_\U$ for the class of pullbacks of $\uu$.
\[
  \RR_\U := \{ A : X \to \Gamma \st \begin{tikzcd}
  X & \Ulow \\
  \Gamma & \U
  \arrow["\exists", dashed, from=1-1, to=1-2]
  \arrow["A"', from=1-1, to=2-1]
  \arrow["\lrcorner"{anchor=center, pos=0.125}, draw=none, from=1-1, to=2-2]
  \arrow["\uu", from=1-2, to=2-2]
  \arrow["\exists"', dashed, from=2-1, to=2-2]
\end{tikzcd}\}
\]
Let us call the maps in $\RR_\U$ \emph{small fibrations},
and the maps in $\RR$ fibrations.

\begin{proposition}\label{prop:universal-pathobject}
  Suppose $(\CC,\RR)$ has path functors
  and an $\RR$-algebraic universe.
  The following are equivalent
  \begin{itemize}
    \item For any small fibration $A \to \Gamma$,
      the pathobject projection
      $(\src,\trg) : \arr[\Gamma] A \to A \times_\Gamma A$
      is a small fibration.
    \item The universal pathobject projection
      $(\src,\trg) : \arr[\Gamma] \Ulow \to \Ulow \times_\U \Ulow$
      is a small fibration.
  \end{itemize}
\end{proposition}
If the universe admits $\Path$-types,
then certainly the conditions of \Cref{prop:universal-pathobject} hold.
Conversely, if the universe $\uu$ has classifying maps
(\Cref{def:alg-universe-classifying-map})
then the latter condition provides $\Path$-types on $\uu$.

% \begin{definition}\label{def:AM-with-path-types}
%   Let us call the combination of Axioms
%   \labelcref{ax:cylinder,ax:path-functor,ax:path-types-universe,,ax:universal-pathobject}
%   \[(\CC,\RR,\cyl,\arr,\uu : \Ulow \to \U, \Path, \path)\]
%   an algebraic model of MLTT with (a universe that admits) $\Path$-types.
% \end{definition}

\section{Hurewicz structure}

Let us fix a $\pi$-preclan $(\CC,\RR)$ that has path functors.
Recall the definition of a Hurewicz structure on a morphism
from \Cref{def:hurewicz-defs}.

\begin{lemma}\label{lem:hurewicz-over-1}
  If $\uu : \Ulow \to \U$ is an $\RR$-algebraic universe 
  that admits a normal Hurewicz structure $(\uu, l)$,
  then any small fibration $f : A \to B$ (i.e., any pullback of $\uu$)
  inherits a normal Hurewicz structure.
\end{lemma}
\begin{proof}
  The right lifting structure of Hurewicz maps is
  preserved under pullback.
  Thus, the Hurewicz structure on the universe pulls back to 
  Hurewicz structure on all small fibrations.
\end{proof}

\begin{proposition}[Locally Hurewicz]\label{prop:hurewicz-equiv-defs}
  Let $\Gamma$ be an object of $\CC$ and
  $f : A \to B$ be a morphism in $\RR(\Gamma)$.
  (Note that $A$ and $B$ are required to be fibrant over $\Gamma$.)
  The following are equivalent (structures)
  \begin{itemize}
    \item A Hurewicz structure $(f,l)$ in the slice $\CC / \Gamma$ with respect to
      the cylinder $\cyl[\Gamma] : \CC / \Gamma \to \CC / \Gamma$.
    \item A section $l$ of the ``pullback-hom'' $h$
      given in the following diagram in $\RR(\Gamma)$.
      \[\begin{tikzcd}
      {\arr[\Gamma] A} && \\
      & P & A \\
      & {\arr[\Gamma] B} & B
      \arrow["h"{description}, dashed, from=1-1, to=2-2]
      \arrow["\src", bend left, from=1-1, to=2-3]
      \arrow["{\arr[\Gamma] f}"', bend right, from=1-1, to=3-2]
      \arrow[from=2-2, to=2-3]
      \arrow[from=2-2, to=3-2]
      \arrow["\lrcorner"{anchor=center, pos=0.125}, draw=none, from=2-2, to=3-3]
      \arrow["f", from=2-3, to=3-3]
      \arrow["\src"', from=3-2, to=3-3]
      \end{tikzcd}
      \quad \quad
      \quad \quad
      \begin{tikzcd}
        & {\arr[\Gamma] A} \\
        P & P
        \arrow["h", from=1-2, to=2-2]
        \arrow["l", from=2-1, to=1-2]
        \arrow[equals, from=2-1, to=2-2]
      \end{tikzcd} \]
  \end{itemize}

  Then a Hurewicz structure $(f,l)$ is normal
  if and only if $l \circ h \circ \rfl = \rfl$.
  \[\begin{tikzcd}
    A & {\arr[\Gamma] A} \\
    {\arr[\Gamma] A} & P
    \arrow["\rfl", from=1-1, to=1-2]
    \arrow["\rfl"', from=1-1, to=2-1]
    \arrow["h"', from=2-1, to=2-2]
    \arrow["l"', from=2-2, to=1-2]
  \end{tikzcd}\]
\end{proposition}

\begin{lemma}[Globally Hurewicz implies locally Hurewicz]\label{lem:global-hurewicz-to-local-hurewicz}
  If a morphism $f : A \to B$ in $\RR(\Delta)$ is (normal) Hurewicz
  as a map in $\CC$,
  then $f$ is also (normal) Hurewicz as a map in the slice $\CC / \Delta$.
\end{lemma}
\begin{proof}
  Let us prove a more general statement:
  if $\sigma : \Delta \to \Gamma$ is a morphism in $\CC$
  and $f : A \to B$ in $\RR(\Delta)$
  is (normal) Hurewicz over $\Gamma$,
  then $f : A \to B$ is (normal) Hurewicz over $\Delta$.

  Consider a diagram in $\CC / \Delta$:
  \[\begin{tikzcd}
      X & A \\
      {\cyl[\Gamma] X} & B
      \arrow["a", from=1-1, to=1-2]
      \arrow["\delta_0"', from=1-1, to=2-1]
      \arrow["f", from=1-2, to=2-2]
      \arrow["p"', from=2-1, to=2-2]
    \end{tikzcd} \]
  We take the image of this diagram under the composition functor
  $\sigma_! : \CC / \Delta \to \CC / \Gamma$.
  By the Frobenius identity (\cref{eq:frobenius-identity}),
  we obtain a diagonal filler in $\CC / \Gamma$.
  \[
  \begin{tikzcd}
    {\sigma_!X} & {\sigma_!X} & {\sigma_!A} \\
    {\cyl[\Gamma]\sigma_! X} & {\sigma_!\cyl[\Gamma] X} & {\sigma_!B}
    \arrow[equals, from=1-1, to=1-2]
    \arrow["{\delta_0}"', from=1-1, to=2-1]
    \arrow["{\sigma_!a}", from=1-2, to=1-3]
    \arrow["{\sigma_!\delta_0}"{pos=0.7}, from=1-2, to=2-2]
    \arrow["{\sigma_!f}", from=1-3, to=2-3]
    \arrow["l"{pos=0.2}, dashed, from=2-1, to=1-3]
    \arrow["\iso"{description}, draw=none, from=2-1, to=2-2]
    \arrow["{\sigma_!p}"', from=2-2, to=2-3]
  \end{tikzcd}
  \]
  The composition of the isomorphism 
  ${\cyl[\Gamma]\sigma_! X} \iso {\sigma_!\cyl[\Gamma] X}$
  with $l$ provides us with the required diagonal filler over $\Delta$.
\end{proof}

\begin{corollary}\label{cor:small-fibration-hurewicz}
  If $\uu : \Ulow \to \U$ is an $\RR$-algebraic universe 
  with a normal Hurewicz structure $(\uu, l)$,
  and $A$ and $B$ are objects in $\RR(\Gamma)$
  then any small fibration $f : A \to B$
  is normal Hurewicz $(f,l)$ in the slice over $\Gamma$.
\end{corollary}
\begin{proof}
  This follows from \Cref{lem:hurewicz-over-1} and \Cref{lem:global-hurewicz-to-local-hurewicz}.
% we have that the ``global'' Hurewicz condition on the universe
% implies a ``local'' Hurewicz condition on each small fibration.
\end{proof}

In the following,
we will always work in the slice over $\Gamma \in \CC$,
so that we drop the subscript in $\arr[\Gamma] A$ and $A \times_\Gamma A$
for an object $A \in \RR(\Gamma)$.
\begin{lemma}[Connection]\label{lem:connection}
  Suppose $\uu : \Ulow \to \U$ is an $\RR$-algebraic universe 
  with a normal Hurewicz structure $(\uu, l)$,
  and $\Gamma$ is an object in $\CC$.
  For any small fibration $A \to \Gamma$,
  there is a map $\chi : \arr A \to \arr \arr A$
  satisfying
  \[ \begin{tikzcd}
        A & {\arr A} \\
        {\arr A} & {\arr \arr A} \\
        & {\arr A}
        \arrow["\rfl", from=1-1, to=1-2]
        \arrow["\src_A", from=2-1, to=1-1]
        \arrow["\chi", from=2-1, to=2-2]
        \arrow[equals, from=2-1, to=3-2]
        \arrow["\src_{\arr \! A}"', from=2-2, to=1-2]
        \arrow["\trg_{\arr \! A}", from=2-2, to=3-2]
    \end{tikzcd} 
    \quad \quad
    \text{and}
    \quad \quad
    \begin{tikzcd}
	A & {\arr A} \\
        {\arr A} & {\arr \arr A}
        \arrow["\rfl_A", from=1-1, to=1-2]
        \arrow["\rfl_A"', from=1-1, to=2-1]
        \arrow["\rfl_{\arr \! A}", from=1-2, to=2-2]
        \arrow["\chi"', from=2-1, to=2-2]
    \end{tikzcd}
    \]
\end{lemma}
We will call $\chi$ a (normal) connection,
where the square involing $\rfl$ is the normality condition.
We think of $\chi$ as taking a path $p : x \rightsquigarrow y$ in $A$
to a path $\chi_p : \rfl x \rightsquigarrow p$ in $\arr A$,
or a homotopy in $A$,
such that $\chi_{\rfl x}$ is the constant path $\rfl : \rfl x \rightsquigarrow \rfl x$:
\[ \begin{tikzcd}
	x & x \\
	x & y
	\arrow[rightsquigarrow, from=1-1, to=1-2]
	\arrow[""{name=0, anchor=center, inner sep=0}, "{\rfl x}"', rightsquigarrow, from=1-1, to=2-1]
	\arrow[""{name=1, anchor=center, inner sep=0}, "p", rightsquigarrow, from=1-2, to=2-2]
	\arrow[rightsquigarrow, from=2-1, to=2-2]
	% \arrow["\chi"{description}, between={0.1}{0.9}, Rightarrow, from=0, to=1]
  \end{tikzcd} \]
We will construct $\chi$ via a ``box-filling'' argument,
schematically as follows.
First, the open box
\[ \begin{tikzcd}
	x & x \\
	x & y
	\arrow["{\rfl x}"', rightsquigarrow, from=1-1, to=2-1]
	\arrow["{\rfl x}", rightsquigarrow, from=1-1, to=1-2]
	\arrow["p"', rightsquigarrow, from=2-1, to=2-2]
  \end{tikzcd} \]
is a Hurewicz lifting problem in $(\src,\trg) : \arr A \to A \times A$.
By \Cref{cor:small-fibration-hurewicz},
the map $\arr A \to A \times A$ is Hurewicz over $\Gamma$.
Hence, we have filled box
\[ \begin{tikzcd}
	x & x \\
	x & y
	\arrow["{\rfl x}"', rightsquigarrow, from=1-1, to=2-1]
	\arrow["{\rfl x}", rightsquigarrow, from=1-1, to=1-2]
	\arrow["p"', rightsquigarrow, from=2-1, to=2-2]
	\arrow["", rightsquigarrow, from=1-2, to=2-2]
  \end{tikzcd} \]
Using $\sym : \arr \arr A \to \arr \arr A$,
we obtain the required homotopy. 
\[ \begin{tikzcd}
	x & x \\
	x & y
	\arrow["{\rfl x}", rightsquigarrow, from=1-1, to=1-2]
	\arrow["{\rfl x}"', rightsquigarrow, from=1-1, to=2-1]
	\arrow["p", rightsquigarrow, from=1-2, to=2-2]
	\arrow[rightsquigarrow, from=2-1, to=2-2]
  \end{tikzcd} \]
Notice that although we will know that the top map is $\rfl x$
(definitionally),
we will not be able to say anything about the bottom map provided by box filling.
\begin{proof}[Proof of \Cref{lem:connection}.]
  The open box can be constructed as follows.
  Since $\arr : \RR(\Gamma) \to \RR(\Gamma)$ is a partial right adoint,
  it preserves all limits,
  and so $\arr (A \times A) \iso \arr A \times \arr A$.
  We construct the open box $b$
  a map into the pullback.
  \[ \begin{tikzcd}
        {\arr A} && A \\
        & P & {\arr A} \\
        {\arr A \times \arr A} & {\arr (A \times A)} & {A \times A}
        \arrow["\src", from=1-1, to=1-3]
        \arrow["b"{description}, dashed, from=1-1, to=2-2]
        \arrow["{(\rfl \circ \src, \id)}"', from=1-1, to=3-1]
        \arrow["\rfl", from=1-3, to=2-3]
        \arrow["\src", from=2-2, to=2-3]
        \arrow["{\arr (\src,\trg)}"', from=2-2, to=3-2]
        \arrow["{(\src,\trg)}", from=2-3, to=3-3]
        \arrow["\iso"{description}, draw=none, from=3-2, to=3-1]
        \arrow["\src"', from=3-2, to=3-3]
        \arrow["\lrcorner"{anchor=center, pos=0.125}, draw=none, from=2-2, to=3-3]
    \end{tikzcd} \]
  By \Cref{cor:small-fibration-hurewicz},
  the map $\arr A \to A \times A$ is Hurewicz over $\Gamma$.
  Then we can obtain a filled box by composing $b$ with $l : P \to \arr \arr A$
  (see \Cref{prop:hurewicz-equiv-defs}),
  and then flip the box by composing with $\sym$.
%   Then we can construct a map $b : \arr A \to P$
%   \[\begin{tikzcd}
% 	{\arr A} && A \\
% 	& {\arr \arr A} & {\arr A} \\
% 	{\arr A \times \arr A} & {\arr (A \times A)} & {A \times A}
% 	\arrow["\src", from=1-1, to=1-3]
% 	\arrow["{l \circ b}"{description}, dashed, from=1-1, to=2-2]
% 	\arrow["{(\rfl \circ \src, \id)}"', from=1-1, to=3-1]
% 	\arrow["\rfl", from=1-3, to=2-3]
% 	\arrow["\src", from=2-2, to=2-3]
% 	\arrow["{\arr (\src,\trg)}"', from=2-2, to=3-2]
% 	\arrow["{(\src,\trg)}", from=2-3, to=3-3]
% 	\arrow["\iso"{description}, draw=none, from=3-2, to=3-1]
% 	\arrow["\src"', from=3-2, to=3-3]
% \end{tikzcd}\]
  \[ 
    \begin{tikzcd}
        {\arr A} & P & {\arr \arr A} & {\arr \arr A}
        \arrow["b", from=1-1, to=1-2]
        \arrow["\chi"', bend right, dashed, from=1-1, to=1-4]
        \arrow["l", from=1-2, to=1-3]
        \arrow["\sym", from=1-3, to=1-4]
    \end{tikzcd}
  \]
\end{proof}

\section{Identity elimination}

\begin{theorem}\label{thm:id-elim}
  Suppose $(\CC,\RR)$ is a $\pi$-preclan.
  If $\uu : \Ulow \to \U$ is an $\RR$-algebraic universe (\labelcref{def:algebraic-universe}) that
  admits $\Path$-types (\labelcref{def:mltt-alg-path-types-2})
  and a normal Hurewicz structure (\labelcref{def:hurewicz-defs}),
  then the universe also admits $\Id$-types (\labelcref{def:algebraic-id-types}).
\end{theorem}
\begin{proof}
  Identity formation and introduction (\Cref{def:algebraic-id-form-intro})
  are provided by the following diagram.
  \[\begin{tikzcd}
    \Ulow & {{\arr[\U] \Ulow}} & \Ulow \\
    & {\Ulow \times_\U \Ulow} & \U
    \arrow["\rfl", from=1-1, to=1-2]
    \arrow["\Delta"', from=1-1, to=2-2]
    \arrow["\path", from=1-2, to=1-3]
    \arrow[from=1-2, to=2-2]
    \arrow["\lrcorner"{anchor=center, pos=0.125}, draw=none, from=1-2, to=2-3]
    \arrow[from=1-3, to=2-3]
    \arrow["\Path"', from=2-2, to=2-3]
  \end{tikzcd}\]
  Note that $I = {\arr[\U] \Ulow}$ is provided algebraically by the path functor.
  
  Let $\Gamma \in \CC$;
  we again drop subscripts for $\arr[\Gamma]$ and always work over $\Gamma$.
  Let $A \to \Gamma$ and $d : C \to \arr A$ be
  two small fibrations,
  and $r : A \to C$ be a map satisfying $d \circ r = \rfl$.
  By \Cref{prop:id-elim-equiv},
  it suffices construct a lift $j : \arr A \to C$
  such that $d \circ j = \id$ and $j \circ \rfl = r$.
  \[ \begin{tikzcd}
	{{A}} & C \\
	{\arr{A}} & {\arr{A}}
	\arrow["r",  from=1-1, to=1-2]
	\arrow["{{{{\rfl}}}}"', from=1-1, to=2-1]
	\arrow[from=1-2, to=2-2]
	\arrow["{{j}}"{description}, dashed, from=2-1, to=1-2]
	\arrow[equals, from=2-1, to=2-2]
  \end{tikzcd} \]
  By \Cref{cor:small-fibration-hurewicz}, $d : C \to \arr A$ is Hurewicz over $\Gamma$.
  To construct $j$, let us consider the following lifting problem in $C$
  for some path $p : x \rightsquigarrow y$ in $A$.
  \[ \begin{tikzcd}
	r & {j_p} & { C} \\
	{\rfl x} & p & { \arr A}
	\arrow[dashed, from=1-1, to=1-2]
	\arrow[maps to, from=1-1, to=2-1]
	\arrow["d", from=1-3, to=2-3]
	\arrow["{\chi_p}"', squiggly, from=2-1, to=2-2]
    \end{tikzcd}\]
  More formally, the lifting problem is captured by the following dotted map.
  \[\begin{tikzcd}
        {\arr A} && A \\
        & P & C \\
        & {\arr \arr A} & {\arr A}
        \arrow["\src", from=1-1, to=1-3]
        \arrow["{(\chi,r \circ \src)}"{description}, dashed, from=1-1, to=2-2]
        \arrow["\chi"', bend right, from=1-1, to=3-2]
        \arrow["r", from=1-3, to=2-3]
        \arrow[from=2-2, to=2-3]
        \arrow["\pi_1", from=2-2, to=3-2]
        \arrow["\lrcorner"{anchor=center, pos=0.125}, draw=none, from=2-2, to=3-3]
        \arrow["d", from=2-3, to=3-3]
        \arrow["\src"', from=3-2, to=3-3]
    \end{tikzcd}\]
    By \Cref{prop:hurewicz-equiv-defs},
    $\arr C$ comes with a locally Hurewicz lifting operation $l : P \to \arr C$.
    We obtain $j$ by composing
    $(\chi,r \circ \src) : \arr A \to P$
    with $l : P \to \arr C$
    and the target map $\trg : \arr C \to C$.
    \[ \begin{tikzcd}
        {\arr A} & P & {\arr C} & C
        \arrow["{(\chi,r \circ \src)}", from=1-1, to=1-2]
        \arrow["j"', bend right, dashed, from=1-1, to=1-4]
        \arrow["l", from=1-2, to=1-3]
        \arrow["\trg", from=1-3, to=1-4]
    \end{tikzcd} \]
    The lower triangle involving $j$ commutes:
    \begin{align*}
      & d \circ j \\
      = \; & d \circ {\trg} \circ l \circ (\chi, r \circ \src) \\
      = \; & {\trg} \circ \arr (d) \circ l \circ (\chi, r \circ \src) \\
      = \; & {\trg} \circ \pi_1 \circ (\chi, r \circ \src) \\
      = \; & {\trg} \circ \chi \\
      = \; & \id \\
    \end{align*}
    The upper triangle commutes due to normality of $\chi$ and $l$.
    \begin{align*}
      & j \circ \rfl \\
      = \; & {\trg} \circ l \circ (\chi, r \circ \src) \circ \rfl \\
      = \; & {\trg} \circ l \circ (\chi\circ \rfl, r \circ \src\circ \rfl)  \\
      = \; & {\trg} \circ l \circ (\rfl \circ \rfl, r) & (\chi \text{ normal}) \\
      = \; & {\trg} \circ l \circ h \circ \rfl \circ r   \\
      = \; & {\trg} \circ \rfl \circ r & (l \text{ normal}) \\
      = \; & r  \\
    \end{align*}
\end{proof}

\section{Examples} \label{sec:examples}

\begin{example}\label{example:sets}
  Take $\RR$ to be all maps in a lex category $\CC$
  and the cylinder functor $\cyl : \CC \to \CC$
  to be the identity functor on $\CC$, 
  then the path functors $\arr[X] : \RR(X) \to \RR(X)$
  will also be identities
  and all maps are trivially normal Hurewicz.
  It follows that the pullback square appearing in \Cref{def:mltt-alg-path-types-2}
  is the same pullback square defining \emph{extensional} Identity types (cf.~\cite{awodey2025}):
  \[ \begin{tikzcd}
        \Ulow & \Ulow \\
        {\Ulow \times_\U \Ulow} & \U
        \arrow["\path", from=1-1, to=1-2]
        \arrow["{{(\src,\trg) = \Delta}}"', from=1-1, to=2-1]
        \arrow["\lrcorner"{anchor=center, pos=0.125}, draw=none, from=1-1, to=2-2]
        \arrow[from=1-2, to=2-2]
        \arrow["\Path"', from=2-1, to=2-2]
    \end{tikzcd} \]

  In particular when $\CC = \Psh(\DD)$ is a category of presheaves,
  $\RR$ is the class of all maps,
  $\uu : \Ulow \to U$ is a Hofmann-Streicher universe \cite{hofmann1997},
  such a pullback square can always be constructed (see \cite[Section {7.1}]{awodey2026}).
\end{example}

As shown in \Cref{ex:exponentiable-bipointed},
a bipointed exponentiable object
$\delta_0, \delta_1 : 1 \rightrightarrows I$
naturally gives rise to a cylinder structure on $\CC$.

\begin{example}\label{example:groupoids}
  Take $\CC$ to be the category of groupoids
  and $I = \{\star \iso \star\}$ to be the walking isomorphism.
  Then the normal Hurewicz condition is equivalent to being a normal isofibration.
  Taking $\RR$ to be the class of isofibrations,
  which is a $\pi$-clan,
  it is clear that exponentiation by $I_X$ provides a partial right adjoint
  $\RR(X) \to \RR(X)$ -- it also extends to an actual right adjoint.
  Taking $\uu : \Ulow \to \U$ to be the universal small isofibration in groupoids
  (see \Cref{sec:universe-in-groupoid-model}),
  it is easy to check the equivalent conditions of \Cref{prop:universal-pathobject}.
  In fact, pathobject projection for a small isofibration
  is a small discrete isofibration.
\end{example}

\begin{example} \label{example:cSet}
  If we take $\CC = \cSet$ to be the category of
  Cartesian cubical sets, $I = y(\Box_1)$,
  and $\RR$ to be the class of unbiased cubical fibrations
  \cite[Definition 4.6]{awodey2026cartesian},
  then $(\CC,\RR)$ is only a $\pi$-preclan.
  To see that $(\CC,\RR)$ fails to be a $\pi$-clan,
  note that the interval $I$ is not fibrant.
  However, for each $\Gamma$, taking $\arr[\Gamma] A := A^{I_\Gamma}$
  still defines an endofunctor $\arr[\Gamma] : \RR(\Gamma) \to \RR(\Gamma)$,
  where $I_\Gamma$ is the pullback of $I$ to the slice $\CC / \Gamma$.
  This is because although $I \to 1$ is not a fibration,
  it is $\RR$-exponentiable.
  We can take $\uu : \Ulow \to \U$
  to be the classifier for small unbiased fibrations
  \cite[Proposition 7.17]{awodey2026cartesian};
  checking that $\U$ itself is fibrant \cite[Proposition 9.4]{awodey2026cartesian}.
  In this case, all unbiased cubical fibrations are Hurewicz,
  and a straightforward computation shows that the universal pathobject projection
  is a small unbiased fibration.
\end{example}

\begin{example} \label{example:sSet}
  Take $\CC = \mathsf{sSet}$ to be the category of
  simplicial sets and $\RR$ to be the class of Kan fibrations.
  This again forms a $\pi$-preclan (\Cref{ex:sset-kan}).
  Exponentiation by $I = [1]$ (and its pullbacks)
  again defines an endofunctor $\arr[1] : \RR(1) \to \RR(1)$.
  Then we can take $\uu : \Ulow \to \U$
  to be the universe from \cite[Definition 2.1.9]{kapulkin2021},
  checking that $\U$ is fibrant \cite[Theorem 2.2.1]{kapulkin2021}.
  Here, it is also the case that all Kan fibrations are Hurewicz.
  A proof that the universal pathobject projection
  is a small unbiased fibration can be adapted from \cite[Proposition 2.3.3]{kapulkin2021}.
\end{example}

\section{A path types modality}
There is a modal character to path types, which we make precise now
by defining a modal theory and showing that our examples are models
of this modal theory.
This suggests a modal type theory that
has characteristics of both (plain) Homotopy Type Theory
and (various versions of) cubical type theory.
Here we will suggest what a model of such a modal type theory should satisfy,
leaving the design of a modal syntax to future work.
% So far,
% we have presented path types in a style that is reminiscent of modal type theory.
% Here, we make this precise by providing
% a modal type theory and an additional condition
% that makes our setting a model of it.

The modal theory is given as follows,
in the style of \cite{licata2016,licata2017,gratzer2020}.
Consider the 2-category $\MM$ (the \emph{mode theory}) with a single 0-cell (\emph{mode}) $\star$,
a single non-trivial 1-cell (\emph{modality}) $\arr : \star \to \star$,
and two-cells $\src, \trg : \arr \to \id$,
$\rfl : \id \to \arr$ and $\sym : \arr \circ \arr \to \arr \circ \arr$.

Let us again assume that
$(\CC,\RR)$ is a $\pi$-preclan with a terminal object,
a cylinder structure and path functors (\Cref{def:path-functor}).
It is straightforward to define a strict 2-functor
$\llbracket - \rrbracket : \MM^\mathsf{coop} \to \tcat$
(cf. \cite[Definition 5.1]{gratzer2020}),
using the cylinder structure.
The mode $\star$ is interpreted as the category $\CC$,
the modality $\arr$ is interpreted as the cylinder functor $\cyl : \CC \to \CC$
and the two-cells are interpreted as the obvious natural transformations.

\begin{definition}\label{def:modal-model}
  A modal universe with respect to the context structure
  $\llbracket - \rrbracket : \MM^\mathsf{coop} \to \tcat$
  consists of an $\RR$-algebraic universe $\uu : \Ulow \to \U$,
  maps $\Mod : \arr[1] \U \to \U$ and $\mod : \arr[1] \Ulow \to \Ulow$,
  such that
  \[\begin{tikzcd}
      {\arr[1] \Ulow} & \Ulow \\
      {\arr[1] \U} & {  \U}
      \arrow["\mod", from=1-1, to=1-2]
      \arrow["{\arr[1]\uu}"', from=1-1, to=2-1]
      \arrow["\lrcorner"{anchor=center, pos=0.125}, draw=none, from=1-1, to=2-2]
      \arrow["\uu", from=1-2, to=2-2]
      \arrow["\Mod"', from=2-1, to=2-2]
  \end{tikzcd}\]
  is a pullback square.
\end{definition}
Given an interval object $I$,
the map $\Mod : \arr[1] \U \to \U$ should be viewed as
a kind of $\Pi$-type indexed by $I$,
taking a type $A : I \to \U$ to the type $\Pi_I A : \U$.

Since we are working internally,
the above definition of a modal universe is the internal analogue of a
``modal natural model'' (\cite{gratzer2020}) with a ``dependent right adjoint''
(\cite{birkedal2020}).
By taking the image of $\uu : \Ulow \to \U$ under the Yoneda embedding,
we obtain a model of modal type theory with 
a dependent right adjoint provided by the partial right adjoints
$\arr[\Gamma] : \RR(\Gamma) \to \RR(\Gamma)$
for each $\cyl[\Gamma] : \CC / \Gamma \to \CC / \Gamma$.
One way to state a precise result is using CwDRAs
(categories with dependent right adjoints),
but the same could be done with modal natural models.

\begin{proposition}
  Suppose $(\CC,\RR)$ is a $\pi$-preclan with a terminal object,
  a cylinder structure and path functors (\Cref{def:path-functor}).
  A modal universe provides a model of the modal type theory $\MM$
  in the sense of being a CwDRA (\cite[Definition 2]{birkedal2020})
  where the presheaf of types is representable.
\end{proposition}

% Cylinders on contexts $\cyl : \CC \to \CC$
% are the functors associated with the modality
% $\arr$ and paths on types $\arr[\Gamma] : \RR(\Gamma) \to \RR(\Gamma)$
% are their so-called ``{dependent right adjoints}''.
% As usual,
% each two-cell between modalities induces a natural transformation
% between the associated functors.

If the universe $\uu : \Ulow \to \U$ has classifying maps
(\Cref{def:alg-universe-classifying-map}),
then we can also rephrase the pullback condition in \Cref{def:modal-model} as saying
that the path functor $\arr[1]$ preserves small fibrations
(pullbacks of the universe $\uu$).
\[ \arr[1] : (\RR(1),\RR_\U) \to (\RR(1),\RR_\U) \]
(If $A$ and $B$ are in $\RR(1)$ and $f : A \to B$ is a small fibration then
so is $\arr[1] f : \arr[1] A \to \arr[1] B$.)
This is useful because in practice it is often easier
to check whether the path functor preserves small fibrations,
as opposed to directly constructing the maps $\Mod$ and $\mod$.
To demonstrate this,
let us revisit the previous examples.

\begin{example}
  % The examples from before,
  % using an interval object $I$,
  % all admit modal universes,
  % and are therefore models of the model type theory
  % given by $\MM$.
  In $\Set$, taking $I = 1$ as the interval,
  exponentiation by $1$ is the identity,
  and so path functors preserve small fibrations.

  In $\Grpd$ with $I$ as the walking isomorphism,
  exponentiation by $I$ preserves small isofibrations because
  $I \to 1$ is itself a small isofibration,
  and small isofibrations form a $\pi$-preclan.
  One way to see this is by using \Cref{lem:pushforward-preserves-maps},
  with $\SS$ as small isofibrations and $\RR$ as the
  $\pi$-clan of isofibrations.

  In $\mathsf{cSet}$ with $I = y(\Box_1)$,
  although $I \to 1$ is not a fibration,
  it is $\RR$-exponentiable (see \Cref{ex:cSet-unbiased-fibrations}),
  where $\RR$ is the class of unbiased fibrations.
  It follows that exponentiating by $I$ preserves small unbiased fibrations:
  again we can use \Cref{lem:pushforward-preserves-maps},
  with $\SS$ as small unbiased fibrations and $\RR$ as the
  $\pi$-preclan of unbiased fibrations.
\end{example}

% \begin{remark}
%   We have only used the $\pi$-clan for making polynomial functors,
%   which in turn are used to construct the generic case for identity elimination.
%   As was demonstrated in \Cref{sec:exponentiability} and \Cref{sec:polynomials},
%   a $\pi$-clan is not the setting in which polynomials can be defined.
%   A similar exponentiability condition was omitted in \cite{awodey2026},
%   where coherence of identity elimination was not proven;
%   one way to fix this issue is to require that the universe is exponentiable.
%   % As noted in \Cref{a},
%   % $\pi$-clans are not the only way one can obtain polynomial functors

%   % If such a condition were added,
%   % then the axioms presented here are more general than those in \cite{awodey2026}.
% \end{remark}

\chapter{Categories as types}\label{sec:cat-as-types}
\section*{Introduction}

While Martin-L\"of Type Theory (MLTT) is a synthetic language for $\infty$-groupoids,
\emph{directed type theory} -- which is by no means a fixed construct -- 
aims to achieve the same for higher categories.

\subsection*{Size issues}

Let us assume that we have three sizes: small, large, and extra-large.
In classical foundations, this could be done by postulating enough
strongly inaccessible cardinals or
(equivalently) Grothendieck universes \cite{bourbaki1972}.
We denote the 1-category of small categories as $\cat$,
which is an object in the 1-category of large categories as $\Cat$.
We will also use $\tCat$, which differs from the others in two ways:
$\tCat$ is a strict 2-category rather than a 1-category,
and its objects are extra-large categories.

\subsection*{Overview of the literature}

There is an analogue of the Hofmann--Streicher groupoid model of MLTT \cite{hofmann1995},
where the category of categories $\Cat$ replaces the category of groupoids
as the category of contexts.
% There is a 1-categorical approach to the semantics of directed type theory,
% emulating Hofmann and Streicher's groupoid model of MLTT \cite{hofmann1995}.
North \cite{north2019} proposed a type theory for this model,
with \emph{hom-types} replacing identity types,
placing restrictions on
the introduction and elimination rules for hom-types using the groupoid core of a category.
% The model in $\Cat$ then makes use of a certain
% algebraic weak factorisation system on $\Cat$.
Altenkirch and Neumann \cite{altenkirch2024} continued this work by
formulating directed type theory in terms of a generalised algebraic theory,
providing several category-theoretic constructions in this synthetic category theory.
Chu and North \cite{rivera2026} provided an analysis of dependent 2-sided fibrations,
allowing reasoning with natural transformations,
which proved challenging in \cite{altenkirch2024}.
Also notable is their treatment of \emph{hom-types},
which no longer placed restrictions on introduction and elimination rules
using the groupoid core.

For higher-categorical models,
the model structure for complete Segal spaces in bisimplicial sets
presents the category theory of $\infty$-categories \cite{riehlverity2017}.
Riehl and Shulman \cite{riehl2017} developed a three-layered type theory,
with \emph{Simplicial Type Theory} at the highest level,
and an interpretation in complete Segal spaces.
Not all types in Simplicial Type Theory model higher categories.
Rather, $\infty$-categories are modelled by Rezk types (and Segal types).
There is a growing list of results developed in this setting,
which we do not detail here
(overviews are provided in \cite{riehl2025, gratzer2026}, for example).
Some of these results have also been formalised in the proof assistant Rzk \cite{rzk}.

A bicategorical approach to directed type theory is also natural to consider.
Licata and Harper \cite{licata2011} defined 2-Dimensional Directed Type Theory
that has a syntactic notion of a 2-cell.
The authors developed a syntax for directed $\Pi$-types
that can handle the complexity introduced by the variances (polarity) of types.
An ad-hoc model of 2-Dimensional Directed Type Theory was
then constructed in $\tCat$.
Ahrens et al. \cite{ahrens2023} defined \emph{comprehension bicategories},
capturing what it means to be a general model of 2-Dimensional Directed Type Theory,
for which the $\tCat$ is an example.

Laretto et al. \cite{laretto2026} defined a non-dependent directed type theory,
where types are modelled by categories,
special ``propositional contexts'' are modelled by profunctors,
and their associated substitutions are modelled by dinatural transformations.
Notably, this work includes an axiomatisation of ends and coends.

Our work is based upon the MLTT-style approach of \cite{north2019, altenkirch2024},
with the same focus on the model in the 1-category $\Cat$.
We will adopt the methods of \emph{algebraic type theory} \cite{awodey2025}
to analyse type formers in this setting,
using the polynomial machinery developed in \Cref{sec:polynomials}.
The aim here is not necessarily to prove something new about the model,
but rather to reformulate existing results into a more algebraic language
that can illuminate what is needed for a full axiomatisation.

\subsection*{Challenges and limitations}

An axiomatisation of the model in $\Cat$ is given in \Cref{sec:type-theoretic-axioms},
and it is by no means complete.
For example, the axioms fail to capture the fact that
the formation of $\Sigma$-types and $\Pi$-types
preserves morphisms of opfibrations.
This is discussed further in \Cref{rmk:pi-preserve-opfib-mor}.
There is no treatment of univalence here,
which remains a challenge.
There is also no consideration of directed higher inductive types.
The theory of $W$-types in $\Cat$
would be an interesting application of the theory of polynomials in $\Cat$
that we develop.

Before jumping into the model in $\Cat$,
we make a summary of the various intricacies of the model,
which will be explained in further detail during the exposition.
\begin{itemize}
  \item There are two kinds of types,
    corresponding to two kinds of variances of indexed categories,
    or two kinds of fibrations.
  \item Whilst there are the usual $\Unit$ and $\Sigma$-types for both variances,
    the $\Pi$-types in $\Cat$ involve an interaction between the two variances.
  \item $\Hom$-types give rise to twisted arrow categories,
    rather than simply arrow categories,
    and this results in restricted introduction and elimination rules
    associated with $\Hom$-types.
  \item Like in the category of groupoids,
    opfibrancy and fibrancy can be viewed as types having transport.
    Unlike in the category of groupoids,
    not all morphisms in a slice between opfibrations preserve transport,
    which is to say that not all morphisms (in a slice) between opfibrations
    are morphisms of opfibrations.
    We do not address this issue in this work.
\end{itemize}

\section{A case study in the category of categories}
\label{sec:preliminaries}
We recall some of the basic definitions in $\Cat$
that will be relevant to our model.
There are two interacting classes of maps,
namely the fibrations and opfibrations.
We will focus on opfibrations,
and only introduce notation for the dual notion of
fibrations when needed.

\section*{Opfibrations}
% \label{sec:fibrations-in-cat}

Let $\DD$ and $\CC$ be categories and
$F : \DD \to \CC$ be a functor.
We say $f : x \to y$ is a lift of its image $F f : F x \to F y$.
Recall that $f'$ is \emph{opcartesian}
when for every morphism $g : y \to z$ in $\CC$
and every lift $h' : x' \to z'$ of $g \circ f : x \to z$
there is a unique lift $g' : y' \to z'$ of $g$ such that
$h' = g' \circ f'$.
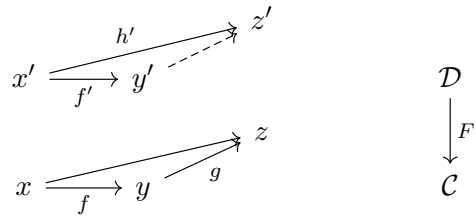
\[ \begin{tikzcd}[row sep = tiny]
  && {z'} \\
  {x'} & {y'} &&& \DD \\
  && z \\
  x & y &&& \CC
  \arrow["{h'}", from=2-1, to=1-3]
  \arrow["{f'}"', from=2-1, to=2-2]
  \arrow[dashed, from=2-2, to=1-3]
  \arrow["F", from=2-5, to=4-5]
  \arrow[from=4-1, to=3-3]
  \arrow["f"', from=4-1, to=4-2]
  \arrow["g"', from=4-2, to=3-3]
\end{tikzcd}\]
% \end{definition}

Recall also the following ``pullback-hom'' diagram in $\Cat$,
where the pullback computes a comma category $F / \CC$.
\begin{equation}\label{eq:pullback-hom}
\begin{tikzcd}
  {\DD^{\two}} \\
  & {F / \CC} & \DD \\
  & {\CC^{\two}} & \CC
  \arrow["{h_F}"{description}, dashed, from=1-1, to=2-2]
  \arrow["\src", bend left, from=1-1, to=2-3]
  \arrow["{F^{\two}}"', bend right, from=1-1, to=3-2]
  \arrow[from=2-2, to=2-3]
  \arrow[from=2-2, to=3-2]
  \arrow["\lrcorner"{anchor=center, pos=0.125}, draw=none, from=2-2, to=3-3]
  \arrow["F", from=2-3, to=3-3]
  \arrow["\src"', from=3-2, to=3-3]
\end{tikzcd}
\end{equation}

The functor $h_F$ takes an arrow $f : x \to y$ in $\DD$
and constructs a lifting problem (i.e. an element of the comma category)
by taking the domain of the arrow in $\DD$ and the image of the arrow in $\CC$.
\[ h_F (f) = (x,F f)\]

\begin{proposition}\label{prop:opfibration-equiv-defs}
  Let $F : \DD \to \CC$ be a functor.
  The following are equivalent (structures)
  \begin{enumerate}
    \item For each object $x \in \DD$, each object $y \in \CC$,
    and each morphism $f : F x \to y$ in $\CC$,
    a chosen opcartesian lift $l_{(x,f)} : x \to y'$ of $f$.
    \item A section $l : F / \CC \to \DD^\two$
    of the pullback-hom $h_F : \DD^\two \to F / \CC$,
    that factors through the full subcategory
    $\KK \subseteq \DD^\two$ consisting of opcartesian arrows.
    \[ \begin{tikzcd}
      {\KK} & {\DD^\two} \\
      {F / \CC} & {F / \CC}
      \arrow[hook, from=1-1, to=1-2]
      \arrow["{h_F}", from=1-2, to=2-2]
      \arrow[dashed, from=2-1, to=1-1]
      \arrow["l"{description}, from=2-1, to=1-2]
      \arrow[equals, from=2-1, to=2-2]
    \end{tikzcd} \]
    \item A left adjoint $l : F / \CC \to \DD^\two$
    to $h_F : \DD^\two \to F / \CC$,
    with the unit of the adjunction equal to the identity.
    (This makes $l$ a ``left-adjoint-right-inverse'' or ``lari''.)
    \[ \eta : \id_{F / \CC} = h_F \circ l \]
    % \item A section $l : F \tcomma \CC \to \Tw{\DD}$
    % of the twisted pullback-hom $t_F : \Tw{\DD} \to F \tcomma \CC$,
    % that factors through the full subcategory
    % $\TT \subseteq \Tw{\DD}$ consisting of opcartesian twisted arrows.
    % \[ \begin{tikzcd}
    %   {\TT} & {\Tw{\DD}} \\
    %   {F \tcomma \CC} & {F \tcomma \CC}
    %   \arrow[hook, from=1-1, to=1-2]
    %   \arrow["{t_F}", from=1-2, to=2-2]
    %   \arrow[dashed, from=2-1, to=1-1]
    %   \arrow["l"{description}, from=2-1, to=1-2]
    %   \arrow[equals, from=2-1, to=2-2]
    % \end{tikzcd} \]
    % \item A left-adjoint-right-inverse $l : F \tcomma \CC \to \Tw{\DD}$
    % to $t_F : \Tw{\DD} \to F \tcomma \CC$, with identity unit.
    % \[ \eta : \id_{F \tcomma \CC} = t_F \circ l \]
  \end{enumerate}
\end{proposition}
\begin{proof}
  (2) $\to$ (1) is straightforward by
  defining $l_{(x,f)} := l (x,f)$.
  The equation $\id_{F / \CC} = h_F \circ l$ ensures that
  $l$ produces correct lifts to lifting problems supplied by $F / \CC$.
  % Since $h_F \circ l = \id$,
  % it suffices to construct a counit $\epsilon : \id \to l \circ h_F$,
  % and show that $h_F \epsilon = \id_{h_F}$ and $\epsilon l = \id_l$.

  (3) $\to$ (2).
  The counit of the adjunction supplies the unique map out of
  the lift given by $l$,
  ensuring that $l$ produces opcartesian maps.
  
  (1) $\to$ (3).
  Let $l : F / \CC \to \DD^\two$ be given by taking
  $(x, f) \mapsto l_{(x,f)}$ on objects.
  For a morphism $(\psi : x \to x', \phi : y \to y') : (x,f : F x \to y) \to (x',g : F x' \to y')$,
  we define $l (\psi,\phi) : l_{(x,f)} \to l_{(x',f')}$
  by lifting the commuting square for $(\psi,\phi)$
  using the opcartesian property of $l_{(x,f)}$.
  \[\begin{tikzcd}
    Fx & y & x & {y_0} \\
    {Fx'} & {y'} & {x'} & {y_0'}
    \arrow["f", from=1-1, to=1-2]
    \arrow["{F\psi}"', from=1-1, to=2-1]
    \arrow["\phi", from=1-2, to=2-2]
    \arrow["{l_{(x,f)}}", from=1-3, to=1-4]
    \arrow["\psi"', from=1-3, to=2-3]
    \arrow["{\phi_0}", dashed, from=1-4, to=2-4]
    \arrow["{f'}"', from=2-1, to=2-2]
    \arrow["{l_{(x',f')}}"', from=2-3, to=2-4]
  \end{tikzcd}\]
  More precisely:
  \[l (\psi,\phi) := (\psi,\phi_0) : l_{(x,f)} \to l_{(x',f')}\]
\end{proof}
% Condition $2.$ can be viewed as an extension of the ``Hurewicz fibration'' condition
% used for the semantics of path types \cite{awodey2026}.
% Condition $3.$ is just a more concise rephrasing of $2.$
% Conditions $4.$ and $5.$ will replace $2.$ and $3.$ in our
% axiomatisation of directed path types in \Cref{sec:directed-types}.

\begin{definition}\label{def:split-opfibration}
  A cloven opfibration consists of a pair $(F,l)$,
  where $F : \DD \to \CC$ is a functor
  and $l$ is left-adjoint-right-inverse to
  the pullback-hom $h_F : \DD^\two \to F / \CC$
  (or one of the other equivalent structures of
  \Cref{prop:opfibration-equiv-defs}).

  A cloven opfibration $(F : \DD \to \CC,l)$
  is \emph{normal} if 
  for any $x \in \DD$ the lift of the identity
  $\id_{F x} : F x \to F x$ is the identity
  \[l_{(x,\id_{F x})} = \id_x : x \to x\]
  In particular, the endpoint of the lift satisfies $(F x)' = x$.
  We say $(F,l)$ is a \emph{normal opfibration} for short.

  A normal opfibration $(F : \DD \to \CC, l)$ is \emph{split} if
  for any $x \in \DD$, $y$ and $z$ in $\CC$ and
  morphisms $f : F x \to y$ and $g : y \to z$,
  the lift of the composition is the composition of the lifts
  \[l_{(x,g \circ f)} = l_{(x,g)} \circ l_{(y',f)} : x \to z' \]
  In particular, the endpoints $z'$ of the two lifts agree.
  We call $(F,l)$ a \emph{split opfibration} for short.
\end{definition}

We can of course phrase these conditions more diagrammatically;
we do so for the normality condition.
\begin{proposition}\label{prop:normal-opfibration-equiv-defs}
  Let $(F : \DD \to \CC, l)$ be a cloven opfibration.
  The following conditions are equivalent
  \begin{enumerate}
    \item $(F,l)$ is normal.
    \item The constant functor $\Delta : \DD \to \DD^\two$
      factors as $\Delta = l \circ L_F$, where
      $L_F := h_F \circ \Delta$.
      \[ \begin{tikzcd}
        \DD & {\DD^\two} \\
        {\DD^\two} & {F / \CC}
        \arrow["\Delta", from=1-1, to=1-2]
        \arrow["\Delta"', from=1-1, to=2-1]
        \arrow["{L_F}"{description}, dashed, from=1-1, to=2-2]
        \arrow["{{h_F}}"', from=2-1, to=2-2]
        \arrow["l"', from=2-2, to=1-2]
      \end{tikzcd} \]
    % \item The functor $\rfl : \Core{\DD} \to \Tw{\DD}$ that takes
    %   an object $d$ to the identity on $d$ and an isomorphism
    %   $f : d \to d'$ to the twisted morphism
    %   $(f^{-1}, f) : \id_d \to \id_d$
    %   factors as $\rfl = l \circ t_F \circ \rfl$.
    %   \[ \begin{tikzcd}
    %     d & {d'} && {\Core{\DD}} & {\Tw{\DD}} \\
    %     d & {d'} && {\Tw{\DD}} & {F \tcomma \CC}
    %     \arrow["{=}"', from=1-1, to=2-1]
    %     \arrow["{f^{-1}}"', from=1-2, to=1-1]
    %     \arrow["{=}", from=1-2, to=2-2]
    %     \arrow["\rfl", from=1-4, to=1-5]
    %     \arrow["\rfl"', from=1-4, to=2-4]
    %     \arrow["f"', from=2-1, to=2-2]
    %     \arrow["{t_F}"', from=2-4, to=2-5]
    %     \arrow["l"', from=2-5, to=1-5]
    %   \end{tikzcd} \]
  \end{enumerate}
\end{proposition}
\begin{proof}
  (1) $\implies$ (2).
  For an object $x \in \DD$,
  \[l \circ h_F \circ \Delta (x) = l \circ h_F (x,\id_x) = l (x, \id_{F x}) = \id_x = \Delta (x)\]  
  For a morphism $f : x \to y$ in $\DD$,
  \[l \circ h_F \circ \Delta (f) = l \circ h_F (f, f) = l (f, F f) = (f,f) = \Delta (f)\]
  \[ \begin{tikzcd}
    x & x \\
    y & y
    \arrow["{l_{(x,\id)}}", equals, from=1-1, to=1-2]
    \arrow["f"', from=1-1, to=2-1]
    \arrow["f", dashed, from=1-2, to=2-2]
    \arrow["{l_{(y,\id)}}"', equals, from=2-1, to=2-2]
  \end{tikzcd}\]
  (2) $\implies$ (1) is clear.
\end{proof}

\section*{The split opfibration factorisation system}
% \label{sec:awfs}

There are algebraic formulations of normal and split opfibrations,
which we make use of in our type theoretic analysis.
These considerations lead us to
an algebraic weak factorisation system on $\Cat$,
where the algebras are precisely split opfibrations.

Any functor $F : \DD \to \CC$ factors through the
comma category as follows.
\[ \begin{tikzcd}
  \DD && \CC \\
  & {\FF / \CC}
  \arrow["F", from=1-1, to=1-3]
  \arrow["{L_F}"', from=1-1, to=2-2]
  \arrow["{R_F}"', from=2-2, to=1-3]
\end{tikzcd} \]
The factorisation of $F$ can be given by extending the pullback-hom
diagram (\cref{eq:pullback-hom})
\[
\begin{tikzcd}
  \DD & {\DD^\two} && \\
  && {F / \CC} & \DD \\
  \CC && {\CC^\two} & \CC
  \arrow["\Delta", from=1-1, to=1-2]
  % \arrow["{L_F}"{description, pos=0.3}, bend right, dashed, from=1-1, to=2-3]
  \arrow["F"', from=1-1, to=3-1]
  \arrow["{h_F}"{description}, from=1-2, to=2-3]
  \arrow["\src", bend left, from=1-2, to=2-4]
  \arrow["\pi_\DD", from=2-3, to=2-4]
  % \arrow["{R_F}"{description}, dashed, from=2-3, to=3-1]
  \arrow["{\pi_{\CC^\two}}"', from=2-3, to=3-3]
  \arrow["\lrcorner"{anchor=center, pos=0.125}, draw=none, from=2-3, to=3-4]
  \arrow["F", from=2-4, to=3-4]
  \arrow["\trg", from=3-3, to=3-1]
  \arrow["\src"', from=3-3, to=3-4]
  \arrow["F^\two"{description}, bend right, from=1-2, to=3-3]
\end{tikzcd}
\]
so that $L_F = h_F \circ \Delta : \DD \to F / \CC$ is as defined in
\Cref{prop:normal-opfibration-equiv-defs}
and $R_F = \trg \circ \pi_{\CC^\two} : F / \CC \to \CC$ is the codomain projection.
It follows that $R_F$ and $L_F$ factor $F$.
\[R_F \circ L_F = \trg \circ F^\two \circ \Delta = \trg \circ \Delta \circ F = F\]

\begin{proposition}\label{prop:normal-opfibration-algebra}
  Let $F : \DD \to \CC$ be a functor.
  The following are equivalent (structures):
  \begin{itemize}
    \item A normal opfibration $(F,l)$.
    \item A chosen left adjoint $\alpha \dashv L_F : \DD \to F / \CC$
    such that the counit of the adjunction is the identity
    \[\alpha \circ L_F = \id_\DD\]
    (making $\alpha$ ``left-adjoint-left-inverse'' or ``lali'')
    and $F \circ \alpha = R_F$, i.e. the adjunction is over $\CC$.
  \end{itemize}
\end{proposition}
\begin{proof}
  Assuming a normal opfibration $(F,l)$,
  consider the following diagram
  \[ \begin{tikzcd}
    \DD & {\DD^\two} & {F / \CC}
    \arrow[""{name=0, anchor=center, inner sep=0}, "\Delta", shift left=2, from=1-1, to=1-2]
    \arrow["{L_F}"{description}, shift left=3, bend left, from=1-1, to=1-3]
    \arrow[""{name=1, anchor=center, inner sep=0}, "\trg", shift left=2, from=1-2, to=1-1]
    \arrow[""{name=2, anchor=center, inner sep=0}, "{h_F}", shift left=2, from=1-2, to=1-3]
    \arrow["\alpha"{description}, shift left=3, bend left, from=1-3, to=1-1]
    \arrow[""{name=3, anchor=center, inner sep=0}, "l", shift left=2, from=1-3, to=1-2]
    \arrow["\dashv"{anchor=center, rotate=90}, draw=none, from=1, to=0]
    \arrow["\dashv"{anchor=center, rotate=90}, draw=none, from=3, to=2]
  \end{tikzcd}\]
  The composition of two left adjoints provides $\alpha = \trg \circ l$.
  The counit $\eta''_d : \alpha L_F d \to d$,
  for some $d \in \DD$,
  can be computed in terms of the respective counits $\eta$ and $\eta'$ of
  $\trg \dashv \Delta$ and $l \dashv h_F$.
  \[ \eta''_d = \eta_d \circ \trg (\eta'_{\Delta d}) : \alpha L_F d \to \trg \Delta d \to d \]
  It follows from normality of $l$ that
  $\trg (\eta'_{\Delta d})$ is the identity.
  On the other hand, $\eta_d$ is always the identity by a straightforward computation.
  Hence $\alpha$ is lali.
  Since $l$ is lari, 
  \[F \circ \alpha = F \circ \trg \circ l = \trg \circ F^\two \circ l = R_F \circ h_F \circ l = R_F\]

  To sketch the converse,
  suppose $\alpha$ is left-adjoint-left-inverse to $L_F$ such that
  $F \circ \alpha = R_F$.
  We want to construct a functor $l : F / \CC \to \DD^\two$ 
  fitting into the diagram above.
  % \[\begin{tikzcd}
  %   \DD & {\DD^\two} & {F / \CC}
  %   \arrow[""{name=0, anchor=center, inner sep=0}, "\Delta", shift left=2, from=1-1, to=1-2]
  %   \arrow["{L_F}"{description}, shift left=3, bend left, from=1-1, to=1-3]
  %   \arrow[""{name=1, anchor=center, inner sep=0}, "\trg", shift left=2, from=1-2, to=1-1]
  %   \arrow[""{name=2, anchor=center, inner sep=0}, "{h_F}", shift left=2, from=1-2, to=1-3]
  %   \arrow["\alpha"{description}, shift left=3, bend left, from=1-3, to=1-1]
  %   \arrow[""{name=3, anchor=center, inner sep=0}, "l", shift left=2, dashed, from=1-3, to=1-2]
  %   \arrow["\dashv"{anchor=center, rotate=90}, draw=none, from=1, to=0]
  %   \arrow["\dashv"{anchor=center, rotate=90}, draw=none, from=3, to=2]
  % \end{tikzcd}\]
  Note that $L_F$ also has a right adjoint,
  the projection $\pi_\DD : F / \CC \to \DD$ (in fact $L_F$ is lari).
  For an object $X = (x,y,f : Fx \to y)$ in $F / \CC$,
  let us write $\epsilon_X : L_F x \to X$
  for the counit of the adjunction $L_F \dashv \pi_\DD : F / \CC \to \DD$.
  To define $l : F / \CC \to \DD^\two$ on objects,
  % we apply $\alpha$ to this counit:
  take $X = (x,y,f : Fx \to y)$ in $F / \CC$ to 
  the object (an arrow in $\DD$) $\alpha \epsilon_X$ in $\DD^\two$.
\end{proof}
In fact, there is a monad with its underlying functor
$R : \Cat^\two \to \Cat^\two$
taking $F \mapsto R_F$.
The second condition of \Cref{prop:normal-opfibration-algebra}
implies that $F$ has a
pointed endofunctor algebra structure.
(The converse is not true;
a pointed endofunctor algebra provides normal lifts,
but those lifts may not be opcartesian.
In this sense, the pointed endofunctor algebras are equivalent to
normal Hurewicz maps \Cref{def:hurewicz-defs}.)
In the following diagram,
the unit is given by the left square,
and the algebra map is given by the right square.
\[
\begin{tikzcd}
  \DD & {F / \CC} & \DD \\
  \CC & \CC & \CC
  \arrow["{L_F}", from=1-1, to=1-2]
  \arrow["F"', from=1-1, to=2-1]
  \arrow["\alpha", from=1-2, to=1-3]
  \arrow["{R_F}"', from=1-2, to=2-2]
  \arrow["F", from=1-3, to=2-3]
  \arrow[equals, from=2-1, to=2-2]
  \arrow[equals, from=2-2, to=2-3]
\end{tikzcd}
\]
Next, we consider actual monad algebras and split opfibrations.

The following theorem can be found in \cite[2.13]{garner2008}. % and \cite{bourke2016}
A version of the theorem using pseudoalgebras rather than
algebras can be found in \cite{street1974},
where cloven opfibrations replace split opfibrations.
Recall the definition of an algebraic weak factorisation system (AWFS),
or ``natural weak factorisation system'' in \cite{garner2008}.
\begin{theorem}\label{thm:lari-opfib-awfs}
  The mapping $F \mapsto L_F$ defines a comonad
  $L : \Cat^\two \to \Cat^\two$
  and $F \mapsto R_F$ defines a monad $R : \Cat^\two \to \Cat^\two$.
  This forms the basis of
  an AWFS on $\Cat$.
\end{theorem}

\begin{definition}\label{def:morphism-split-opfibration}
A morphism of split opfibrations
(also known as an opcartesian functor) between
$(F : \DD \to \CC,l)$ and $(F' : \DD' \to \CC',l')$
consists of a functor $H : \CC \to \CC'$,
a functor $G : \DD \to \DD'$,
such that $F' \circ G = H \circ F$
and $G^\two \circ l = l' \circ G / H$.
\[ 
  \begin{tikzcd}
    \DD & {\DD'} \\
    \CC & \CC'
    \arrow["G", from=1-1, to=1-2]
    \arrow["F"', from=1-1, to=2-1]
    \arrow["H"', from=2-1, to=2-2]
    \arrow["{F'}", from=1-2, to=2-2]
  \end{tikzcd}
  \quad
  \quad \begin{tikzcd}
	{\DD^\two} & {\DD'^\two} \\
	{F / \CC} & {F' / \CC'}
	\arrow["{G^\two}", from=1-1, to=1-2]
	\arrow["l", from=2-1, to=1-1]
	\arrow["{G / H}"', from=2-1, to=2-2]
	\arrow["{l'}"', from=2-2, to=1-2]
\end{tikzcd} \]
Since $\FF$-opcartesian maps in $\DD$ are
isomorphic to the lifts of their images given by $l$,
the second condition implies that
$G$ takes all opcartesian maps over $\CC$ to opcartesian maps over $\CC'$.
\end{definition}

Let us write $\opfibcart$ for the category with
split opfibrations as objects,
and morphisms of split opfibrations as morphisms.

We also write $\lari$ for the category that has lari adjunctions $L \dashv R$
as objects -- meaning that the unit is the identity --
with morphisms of lari adjunctions defined as follows.
Suppose $L : A \to B$ and $L' : A' \to B'$ are lari,
with rights adjoints $R$ and $R'$ respectively,
a morphism $(F,G) : (L \dashv R) \to (L' \dashv R')$
consists of functors $F : A \to A'$ and $G : B \to B'$
such that $L' F = G L$, $R' G = F R$,
and for any object $b$ in $B$,
the two counits $\epsilon$ and $\epsilon'$ satisfy
$G \epsilon_b = \epsilon'_{G b}$.
\[
  \begin{tikzcd}
	B & {B'} \\
	A & {A'} \\
	B & {B'}
	\arrow["G", from=1-1, to=1-2]
	\arrow["R"', from=1-1, to=2-1]
	\arrow[""{name=0, anchor=center, inner sep=0}, "{=}"', shift right = 5, bend right = 60, from=1-1, to=3-1]
	\arrow["{R'}", from=1-2, to=2-2]
	\arrow[""{name=1, anchor=center, inner sep=0}, "{=}", shift left = 5, bend left = 60, from=1-2, to=3-2]
	\arrow["F"', from=2-1, to=2-2]
	\arrow["L"', from=2-1, to=3-1]
	\arrow["{L'}", from=2-2, to=3-2]
	\arrow["G"', from=3-1, to=3-2]
	\arrow["\epsilon", Rightarrow, from=2-1, to=0, shorten >=0.5ex]
	\arrow["{\epsilon'}"', Rightarrow, from=2-2, to=1, shorten >=0.5ex]
\end{tikzcd}
\]

\begin{theorem}
  For a functor $F : \DD \to \CC$,
  the category of algebras of the monad $R : \Cat^\two \to \Cat^\two$
  (from \Cref{thm:lari-opfib-awfs})
  is equivalent to the category of split opfibrations.
  The category of coalgebras of the comonad $L : \Cat^\two \to \Cat^\two$
  is equivalent to the category of lari adjunctions.
  \[ \alg{R} \simeq \opfibcart
  \quad \quad \quad \quad
  \coalg{L} \simeq \lari\]
  \[
  (F, \alpha : R_F \to F) \mapsto (F,l)
  \quad \quad
  (F, \alpha : F \to L_F) \mapsto F \dashv A
  \]

\end{theorem}

% \begin{theorem}
%   For a functor $F : \DD \to \CC$,
%   an algebra $(F, \alpha : R_F \to F)$ of the monad $R : \Cat^\two \to \Cat^\two$
%   of \Cref{thm:lari-opfib-awfs}
%   is equivalent (as structure) to a split opfibration $(F,l)$.
%   A coalgebra $(F, \alpha : F \to L_F)$ of the comonad $L : \Cat^\two \to \Cat^\two$
%   is equivalent to an adjunction $F \dashv A$
%   in which $F$ is lari.
%   \[ \alg{R} \simeq \opfibcart
%   \quad \quad
%   \coalg{L} \simeq \lari\]
% \end{theorem}
% \begin{proof}
  % We sketch the construction of a split opfibration from a monad algebra.
  % Given a monad algebra $(F : \DD \to \CC, {\alpha} : F / \CC \to \DD)$ of $R$,
  % we obtain a left adjoint $\alpha \dashv L_F$
  % of the left-map $L_F : \DD \to \CC$
  % that is also a pointed endofunctor algebra structure on $F$.
  % By \Cref{prop:normal-opfibration-algebra},
  % this provides a normal opfibration.
  % One can check that the lift $l$ defined in
  % the proposition is also split.
% \end{proof}

\section*{Straightening and unstraightening}

\Cref{thm:opfibration-main} is the standard straightening-unstraightening result
for opfibrations.
In preparation for the theorem,
we define a category consisting of
split opfibrations over a category $\CC$.
We write $\opfib(\CC)$ for the full subcategory of the slice
$\Cat / \CC$ consisting of split opfibrations.
We write $\opfibcart(\CC)$ for
the wide subcategory of $\opfib(\CC)$
consisting of morphisms of split opfibrations in the slice.
% One candidate for this is the full subcategory of the slice,
% i.e. the objects are split opfibrations,
% with no restrictions on the morphisms.
% It turns out that a better behaved category
% is given when we restrict to morphisms
% of split opfibrations in the slice,
That is, morphisms
\[(G,H) : (F : \DD \to \CC, l) \to (F' : \DD' \to \CC, l')\]
such that $H$ is the identity functor on $\CC$
(see \cite[Definition 3.1]{streicher2018}).
% We write $\opfibcart(\CC)$ for
% the wide subcategory of $\opfib(\CC)$
% consisting of morphisms of split opfibrations in the slice.
% \[ \smlopfibcart(\CC) \to \smlopfib(\CC)\]

% \begin{example}
%   We provide an example of a morphism in $\opfib(\Set)$
%   that is not in $\opfibcart(\Set)$.
%   That is, a morphism that does not preserve lifts.
%   Since $\Set$ has all pushouts,
%   the domain map $\src : \Set^\two \to \Set$ is an opfibration,
%   we can even make explicit choices of pushouts to make it split.

% \end{example}

\begin{theorem}\label{thm:opfibration-main}
  For any category $\CC$,
  the category of functors and natural transformations
  $[\CC,\Cat]$ is equivalent to $\opfibcart(\CC)$.
  \[[\CC,\Cat] \simeq \opfibcart(\CC)\] 
\end{theorem}
\begin{proof}
  % See for example \cite{streicher2018}.
  We introduce some notation,
  and sketch the functors that make up the equivalence.
   
  Recall the \emph{Grothendieck construction} on $A$,
  denoted $\int A$:
  objects of $\int A$ consist of pairs $(x, a)$,
  where $x \in \CC$ and $a \in A(x)$;
  morphisms $(x,a) \to (y,b)$ consist of pairs $(f,\phi)$,
  where $f : x \to y$ is a morphism in $\CC$
  and $\phi : A f a \to b$ is a morphism in $A(y)$.
  There is an associated forgetful functor $u_A : \int A \to \CC$,
  which takes the first component on objects and morphisms.
  The forgetful functor $u_A : \int A \to \CC$
  has a canonical split opfibration structure.
  Overall, this makes up functor $\int : [\CC,\Cat] \to \opfibcart(\CC)$.

  Conversely, for a split opfibration $(F : \DD \to \CC, l)$ over $\CC$,
  we can construct a functor $F^* : \CC \to \Cat$ by taking
  $x$ in $\CC$ to the fibre $F^* x$.
  Then the lifting operation $l$ provides, for each $f : x \to y$ in $\CC$,
  a functor $F^* f : F^* x \to F^* y$.
\end{proof}

A version of the following lemma can be found, for example,
in \cite[Theorem 4.1]{streicher2018}.
Recall that a discrete opfibration is an opfibration
with fibres that are sets.
\begin{lemma}\label{lem:split-opfib-between-split-opfibs}
  Let $(A \to X, \alpha)$ and $(B \to X, \beta)$ be two
  split opfibrations and $F : A \to B$
  a morphism of split opfibrations over $X$.
  Then $F$ is a discrete opfibration
  if and only if 
  over each $x$ in $X$,
  the functor between fibres $F^* x : A^* x \to B^* x$
  is a discrete opfibration.

  More generally,
  suppose that $F^* : X \to \Cat^\two$ factors through the
  category of split fibrations in $\Cat$.
  \[
    \begin{tikzcd}
      & \opfibcart \\
      X & {\Cat^\two}
      \arrow[from=1-2, to=2-2]
      \arrow["{(F,\phi)}", dashed, from=2-1, to=1-2]
      \arrow["F"', from=2-1, to=2-2]
    \end{tikzcd}
  \]
  This means that for each $x$ in $X$,
  the functor between fibres $F^* x : A^* x \to B^* x$
  is equipped with the structure of a split opfibration
  $(F^* x, \phi_x)$,
  and for each $t : x \to y$ in $X$,
  transporting along $t$ forms
  a morphism of split opfibrations:
  \[(A^* t, B^* t) : (F^* x,\phi_x) \to (F^* y,\phi_y)
    \quad \quad \quad
    \begin{tikzcd}
      {A^* x} & {A^* y} \\
      {B^* x} & {B^* y}
      \arrow["{A^* t}", from=1-1, to=1-2]
      \arrow["{(F^* x,\phi_x)}"', from=1-1, to=2-1]
      \arrow["{(F^* y,\phi_y)}", from=1-2, to=2-2]
      \arrow["{B^* t}"', from=2-1, to=2-2]
    \end{tikzcd} \]
  Then we have a split opfibration structure on $F$.
\end{lemma}
\begin{proof}
  We define a lifting operation $\psi$ for $F$,
  suppose $a \in A$ and $f : F a \to b$.
  Let us write $t : x \to y$ for the image of $f : F a \to b$ in $X$.
  We can lift $t : x \to y$ using $\alpha$ and $\beta$ to get 
  $\alpha_{(a,t)} : a \to a_y$ and $\beta_{(F a, t)} : F a \to F a_y$;
  we have $F \alpha_{(a,t)} = \beta_{(F a, t)}$,
  since $F$ is a morphism of split opfibrations.
  Since $\beta_{(F a, t)}$ is opcartesian, 
  there is a unique ``vertical'' map $u : F a \to b$ --
  meaning that it is in the fibre over $y$ --
  such that $f = u \circ \beta_{(F a, t)}$.
  We lift $u$, using $\phi_y$ to obtain a morphism
  $\phi_{(a_y,u)} : a_y \to a_y'$ in the fibre of $A$ over $y$.
  \[\begin{tikzcd}
    a & {a_y} & {a_y'} \\
    \\
    Fa & {F a_y} & b \\
    \\
    x & y & y
    \arrow["{\alpha_{(a,t)}}", from=1-1, to=1-2]
    \arrow["{\psi_{(a,f)}}"', bend right, from=1-1, to=1-3]
    \arrow["{\phi_{(a_y,u)}}", from=1-2, to=1-3]
    \arrow["{\beta_{(F a, t)}}", from=3-1, to=3-2]
    \arrow["f"', bend right, from=3-1, to=3-3]
    \arrow["u", from=3-2, to=3-3]
    \arrow["t", from=5-1, to=5-2]
    \arrow[equals, from=5-2, to=5-3]
  \end{tikzcd}\]
  We define $\psi_{(a,f)}$ to be the composition
  \[\psi_{(a,f)} = u \circ \alpha_{(a,t)} : a \to a_y'\]

  It follows from this definition
  and normality of $\alpha$, $\beta$ and $\phi_y$
  that if $f$ is the identity,
  then so are $t$, $\beta_{(F a, t)}$, $u$, $\alpha_{(a,t)}$,
  and ${\phi_{(a_y,u)}}$.
  Hence $\psi$ is normal.

  To check that $\psi$ is split,
  suppose additionally $g : b \to b'$.
  Let us denote the image of $g$ in $X$ as $s : y \to z$.
  We can construct a lift $\psi_{(a'_y,g)} : a'_y \to a''_z$
  of $g$ along $F$,
  just as we constructed the lift of $f$,
  as the composition of the lift $\alpha_{(a'_y,s)} : a'_y \to a'_z$
  of $s$ followed by the lift $\phi_{(a_z,v)} : a'_z \to a''_z$ of 
  $v : F a'_z \to F a''_z$, 
  the unique map over $z$ satisfying
  $v \circ \beta_{(b,s)} = g$
  (refer to the diagram below, where the situation is similar).
  We write $\alpha_{(a_y,s)} : a_y \to a_z$,
  which is sent by $F$ to $\beta_{(F a_y,s)} : F a_y \to F a_z$ and
  $u' : F a_z \to F a'_z$ for the unique map over $z$
  satisfying $u' \circ \beta_{(F a_y,s)} = \beta_{(b,s)} \circ u$.
  The condition that
  $(A^* s, B^* s) : (F^* y,\phi_y) \to (F^* z,\phi_z)$
  is a morphism of split opfibrations can be described in
  more elementary terms as a coherence condition:
  \[\phi_{(a_z,u')} \circ \alpha_{(a_y,s)} = \alpha_{(a'_y,s)} \circ \phi_{(a_y,u)}\]
  \[\begin{tikzcd}
    &&& {a_y} & {a_z} \\
    {A^* y} & {A^* z} && {a_y'} & {a'_z} \\
    {B^* y} & {B^* z} && {F a_y} & {F a_z} \\
    y & z && b & {F a'_z}
    \arrow["{{{{\alpha_{(a_y,s)}}}}}", from=1-4, to=1-5]
    \arrow["{{{{\phi_{(a_y,u)}}}}}"', from=1-4, to=2-4]
    \arrow["{{{{\phi_{(a_z,u')}}}}}", from=1-5, to=2-5]
    \arrow["{{A^*s}}", from=2-1, to=2-2]
    \arrow["{{(F^* y,\phi_y)}}"', from=2-1, to=3-1]
    \arrow["{{(F^* z,\phi_z)}}", from=2-2, to=3-2]
    \arrow["{{{{\alpha_{(a'_y,s)}}}}}"', from=2-4, to=2-5]
    \arrow["{{B^*s}}"', from=3-1, to=3-2]
    \arrow["{{{{\beta_{(F a_y, s)}}}}}", from=3-4, to=3-5]
    \arrow["u"', from=3-4, to=4-4]
    \arrow["{{{u'}}}", from=3-5, to=4-5]
    \arrow["s", from=4-1, to=4-2]
    \arrow["{{{{\beta_{(b, s)}}}}}"', from=4-4, to=4-5]
  \end{tikzcd}\]
  Since $\phi_z$, $\alpha$ and $\beta$ are split,
  \begin{align*}
    & \psi_{(a, gf)} \\
  = \, & \phi_{(a_z,v u')} \circ \alpha_{(a,st)}\\
  = \, & \phi_{(a'_z,v)} \circ \phi_{(a_z,u')} \circ \alpha_{(a_y,s)} \circ \alpha_{(a,t)} \\
  = \, & \phi_{(a'_z,v)} \circ \alpha_{(a'_y,s)} \circ \phi_{(a_y,u)} \circ \alpha_{(a,t)} \\
  = \, & \psi_{(a'_y,g)} \circ \psi_{(a,f)}
  \end{align*}
  
  We show that $\psi_{(a,f)}$ is opcartesian over $B$.
  Suppose $h : a \to e$ is a morphism in $A$
  and $g : b \to F e$ such that $F h = g \circ f$.
  We want to show that there is a unique lift
  $g' : a'_y \to e$ of $g$ in $A$
  such that $g' \circ \psi_{(a,f)} = h$.
  Let us reuse the notation for maps $u$, $s$, $v$ and $u'$ from above,
  as well as the names for the endpoints of lifts.
  Since $\alpha_{(a,st)} : a \to a_z$ is opcartesian over $X$,
  there is a unique map $w : a_z \to e$
  over $z$ satisfying $w \circ \alpha_{(a,st)} = h$.  
  Since $\phi_{(a_z,vu')} : a_z \to a''_z$ is opcartesian over $B^* z$,
  there is a unique map $w' : a''_z \to e$
  over $F e$ satisfying $w' \circ \phi_{(a_z,vu')} = w$.
  Clearly, $F (w' \circ \psi_{(a'_y,g)}) = g$ gives us a lift.
  \[\begin{tikzcd}[column sep = large]
	&& {a_z} &&&& {F a_z} \\
	& {a_y} & {a'_z} &&& {F a_y} & {F a'_z} \\
	a & {a_y'} & {a''_z} && Fa & b & {F a''_z} \\
	&& e &&&& {F e} \\
	x & y & z && x & y & z
	\arrow["{{\phi_{(a_z,u')}}}", from=1-3, to=2-3]
	\arrow["{u'}", dashed, from=1-7, to=2-7]
	\arrow["{{\alpha_{(a_y,s)}}}", from=2-2, to=1-3]
	\arrow[from=2-2, to=2-3]
	\arrow["{{\phi_{(a_y,u)}}}"{description}, from=2-2, to=3-2]
	\arrow["{{\phi_{(a'_z,v)}}}", from=2-3, to=3-3]
	\arrow["{{\beta_{(F a_y, s)}}}", from=2-6, to=1-7]
	\arrow[from=2-6, to=2-7]
	\arrow["u"{description}, dashed, from=2-6, to=3-6]
	\arrow["v", dashed, from=2-7, to=3-7]
	\arrow["{{\alpha_{(a,t)}}}", from=3-1, to=2-2]
	\arrow["{\psi_{(a,f)}}"', from=3-1, to=3-2]
	\arrow["h"', bend right, from=3-1, to=4-3]
	\arrow["{{\alpha_{(a'_y,s)}}}"{description}, from=3-2, to=2-3]
	\arrow["{\psi_{(a'_y,g)}}"', from=3-2, to=3-3]
	\arrow["{g'}"', bend right = 10, dashed, from=3-2, to=4-3]
	\arrow["w'", dashed, from=3-3, to=4-3]
	\arrow["{{\beta_{(F a, t)}}}", from=3-5, to=2-6]
	\arrow["f"{description}, from=3-5, to=3-6]
	\arrow["{F h}"', bend right, from=3-5, to=4-7]
	\arrow["{{\beta_{(b, s)}}}"{description}, from=3-6, to=2-7]
	\arrow["g"', from=3-6, to=3-7]
	\arrow["g"', bend right = 10, from=3-6, to=4-7]
	\arrow[equals, from=3-7, to=4-7]
	\arrow["t", from=5-1, to=5-2]
	\arrow["s", from=5-2, to=5-3]
	\arrow["t", from=5-5, to=5-6]
	\arrow["s", from=5-6, to=5-7]
  \arrow["w"{description}, shift left=3, bend left = 70, dashed, from=1-3, to=4-3]
  \end{tikzcd}\]
  Let $g' : a'_y \to e$ be such that $F g' = g$.
  It suffices to show that $g' \circ \psi_{(a,f)} = h$
  if and only if $g' = w' \circ \psi_{(a'_y,g)}$.
  It is easy to check that $\phi_{(a_y, u)}$ is not just
  opcartesian with respect to $F^* y : A^* x \to B^* x$,
  but also opcartesian with respect to all of $F : A \to B$.
  Also, $\alpha_{(a,t)}$ is opcartesian over $X$.
  Thus, 
  \begin{align*}
    & g' = w' \circ \psi_{(a'_y,g)} \\
    \iff \, & g' \circ \phi_{(a_y, u)} = w' \circ \psi_{(a'_y,g)} \circ \phi_{(a_y, u)} \\
    \iff \, & g' \circ \phi_{(a_y, u)} \circ \alpha_{(a,t)} = w' \circ \psi_{(a'_y,g)} \circ \phi_{(a_y, u)} \circ \alpha_{(a,t)} \\
    \iff \, & g' \circ \psi_{(a,f)} = w' \circ \psi_{(a'_y,g)} \circ \psi_{(a,f)} \\
  \end{align*}
  Finally,
  \[ w' \circ \psi_{(a'_y,g)} \circ \psi_{(a,f)} = w' \circ \psi_{(a,gf)}
  = w' \circ \phi_{(a_z,vu')} \circ \alpha_{(a,st)} = w \circ \alpha_{(a,st)} = h \]

  Now suppose instead that each $F^* x : A^* x \to B^* x$ is a discrete opfibration.
  This is the same as saying $F^* x : A^* x \to B^* x$ is a split opfibration
  and for each $b \in B^* x$ the fibre
  $(F^* x)^* b \iso F^* b$ is a set.
  Let us first assume the coherence condition for transports holds.
  Then we have that $F$ is a split opfibration such that for any $b \in B$,
  $F^* b$ is a set.
  In other words, $F$ is a discrete opfibration.

  Finally, we verify the coherence condition for transports.
  Suppose $s : y \to z$ in $X$, $a_y \in A$ is over $y$, and $u : F a_y \to b$ in $B$.
  We generate a similar diagram as before:
  \[
  \begin{tikzcd}
    {a_y} & {a_z} & \\
    {a_y'} & {a'_z} & {a''_z} \\
    {F a_y} & {F a_z} \\
    b & {F a'_z} \\
    y & z
    \arrow["{{{{\alpha_{(a_y,s)}}}}}", from=1-1, to=1-2]
    \arrow["{{{{\phi_{(a_y,u)}}}}}"', from=1-1, to=2-1]
    \arrow["v", dashed, from=1-2, to=2-2]
    \arrow["{{{{\phi_{(a_z,u')}}}}}", from=1-2, to=2-3]
    \arrow["{{{{\alpha_{(a'_y,s)}}}}}"', from=2-1, to=2-2]
    \arrow["{{{{\beta_{(F a_y, s)}}}}}", from=3-1, to=3-2]
    \arrow["u"', from=3-1, to=4-1]
    \arrow["{{{u'}}}", from=3-2, to=4-2]
    \arrow["{{{{\beta_{(b, s)}}}}}"', from=4-1, to=4-2]
    \arrow["s", from=5-1, to=5-2]
  \end{tikzcd}
  \]
  where $u' : F a_z \to F a'_z$ is the unique map over $z$ such that
  \[u' \circ \beta_{(F a_y, s)} = \beta_{(b,s)} \circ u\]
  and $v : a_z \to a'_z$ is the unique map over $z$ such that
  \[v \circ \alpha_{(a_y,s)} = \alpha_{(a'_y,s)} \circ \phi_{(a_y,u)}\]
  It suffices to show that $v = \phi_{(a_z,u')}$.
  Since $F^* z$ is a discrete opfibration
  and $v$ and $\phi_{(a_z,u')}$ are both in the fibre over $z$,
  this amounts to showing that $F v = u'$.
  By definition of $u'$, we simply need to check the following.
  \[
    F v \circ \beta_{(F a_y, s)}
    = F (v \circ \alpha_{(a_y,s)})
    = F (\alpha_{(a'_y,s)} \circ \phi_{(a_y,u)})
    = \beta_{(b,s)} \circ u
  \]
\end{proof}

There is a similar result for lari functors between split opfibrations.
\begin{lemma} \label{lem:lari-between-split-opfibs}
  Let $(A \to X, \alpha)$ and $(B \to X, \beta)$ be two
  split opfibrations and $F : A \to B$
  a morphism of split opfibrations over $X$.
  Suppose that $F^* : X \to \Cat^\two$ factors through the
  category of lari adjunctions in $\Cat$.
  \[
    \begin{tikzcd}
      & \lari \\
      X & {\Cat^\two}
      \arrow[from=1-2, to=2-2]
      \arrow["{(F,R)}", dashed, from=2-1, to=1-2]
      \arrow["F^*"', from=2-1, to=2-2]
    \end{tikzcd}
  \]
  Then we have a lari adjunction $F \dashv R$ such that
  the right adjoint $R$ is a morphism of split opfibrations over $X$.
\end{lemma}
\begin{proof}
  We first unfold the assumption that $F^*$ factors through $\lari$.
  This means that for each $x$ in $X$,
  the functor between fibres $F^* x : A^* x \to B^* x$
  is equipped with a right adjoint $R_x : B^* x \to A^* x$,
  such that the unit is the identity.
  Also, for each $t : x \to y$ in $X$,
  % transporting along $t$ is a morphism of lari functors
  % $(A^* t, B^* t) : (F^* x \dashv R_x) \to (F^* y \dashv R_y)$.
  % Unfolding this further means that
  $F^* y \circ A^* t = B^* t \circ F^* x$, $R_y \circ B^* t = A^* t \circ R_x$,
  and for any object $b$ in $B^* x$,
  the counits satisfy
  $B^* t \epsilon_b = \epsilon'_{B^* t b}$.
  
  To define the functor $R : B \to A$,
  let $R$ take an object $b \in B$ over $x \in X$ to $R_x b$.
  For a morphism $f: b \to b'$ over $t : x \to y$,
  we take the canonical opcartesian-vertical factorisation
  \[ f = u \circ \beta_{(b,t)}\]
  for $\beta_{(b,t)} : b \to b_y$
  and $u : b_y \to b'$ over $y$.
  We take $R f := R_y u \circ \alpha_{(R_x b,t)}$,
  checking that the domain of $R_y u$
  is equal to the codomain of $\alpha_{(R_x b,t)}$,
  \[R_y b_y = R_y (B^* t (b)) = A^* t (R_x b)\]
  \[\begin{tikzcd}
    & {R_x b} && b & x \\
    {R_y b_y} & {R_y b'} & {b_y} & {b'} & y
    \arrow["{\alpha_{(R_x b,t)}}"', from=1-2, to=2-1]
    \arrow["{R f}", dashed, from=1-2, to=2-2]
    \arrow["{\beta_{(b,t)}}"', from=1-4, to=2-3]
    \arrow["f", from=1-4, to=2-4]
    \arrow["t"', from=1-5, to=2-5]
    \arrow["{R_y u}"', from=2-1, to=2-2]
    \arrow["u"', from=2-3, to=2-4]
  \end{tikzcd}\]
  It is easy to check that $R$ is functorial.
  Furthermore, it is evidently a morphism of split opfibrations over $X$,
  so that for $b \in B$ over $x \in X$ and $t : x \to y$,
  \[R \beta_{(b,t)} = \alpha{(R b,t)}\]
  and for $u : b \to b'$ vertical over $x$,
  \[R u = R_x u\]

  To make a lari adjunction $L \dashv R$,
  we check that $\id_A = R \circ F$ and construct a counit
  $\epsilon : F \circ R \to \id_B$ satisfying two equations
  $\epsilon F = \id$ and $R \epsilon = \id$.
  Firstly, for an object $a \in A$ over $x \in X$,
  \[R F a = R_x F_x a = a\]
  For morphisms, by taking opcartesian-vertical factorisations,
  it suffices to check that
  \[R F \alpha_{(a,t)} = R \beta_{(F a, t)} = \alpha_{(R F a,t)} = \alpha_{(a,t)}\]
  and for $u$ vertical over $y \in X$,
  \[R F u = R_y F_y u = u\]
  For an object $b \in B$ over $x \in X$,
  we define $\epsilon_b : F R b \to b$
  as the component of the counit of $F^* x \dashv R_x$ at $b$.
  To prove naturality of $\eta$ at a morphism
  $f : b \to b'$ over $t : x \to y$,
  factor $f = u \circ \beta_{(b,t)}$ as usual,
  where $\beta_{(b,t)} : b \to b_y$.
  Using the coherence equation, we have 
  \[B^* t \epsilon_b = \epsilon_{B^* t b} = \epsilon_{b_y}\]
  so that the following diagram commutes
  \[ \begin{tikzcd}
    {F_x R_x b} & b \\
    {F_y R_y b} & {b_y} \\
    {F_y R_y b} & {b'}
    \arrow["{\epsilon_{b}}", from=1-1, to=1-2]
    \arrow["{\beta_{(F R b,t)}}"', from=1-1, to=2-1]
    \arrow["{F R f}"', shift right=10, bend right = 40, from=1-1, to=3-1]
    \arrow["{\beta_{(b,t)}}", from=1-2, to=2-2]
    \arrow["f", shift left=5, bend left = 40, from=1-2, to=3-2]
    \arrow["{\epsilon_{b_y}}"', from=2-1, to=2-2]
    \arrow["{B^* t (\epsilon_b)}", draw=none, from=2-1, to=2-2]
    \arrow["{F_y R_y u}"', from=2-1, to=3-1]
    \arrow["u", from=2-2, to=3-2]
    \arrow["{\epsilon_{b'}}"', from=3-1, to=3-2]
  \end{tikzcd} \]
  Since each $\eta_b$ is defined as the counit of
  the adjunction $F^* x \dashv R_x$ over some $x$,
  it is vertical.
  The two remaining equations for $a \in A$ and $b \in B$ 
  $\epsilon_{F a} = \id_a$ and $R \epsilon_b = \id_{R b}$
  then follow automatically
  from each $F^* x \dashv R_x$ being lari.
\end{proof}

\section*{The universal split opfibration}
% \label{sec:universes-in-cat}

In this section, we describe the ``universal split opfibration'',
classifying \emph{small} split opfibrations in $\Cat$,
i.e. split opfibrations with small categories as fibres.
This appears, for example, in \cite{hess2007}
as the ``tautological bundle of categories''.

% We denote by $\cat$ the category of small categories,
% which is an object in $\Cat$.
We use $\smlopfibcart(\CC)$ to denote the category
of small split opfibrations over $\CC$
and morphisms of split opfibrations over $\CC$.
Let $\CC$ be a large category and
$A : \CC \to \cat$ be a functor.
Then \Cref{thm:opfibration-main}
restricts to the following equivalence.
\[[\CC,\cat] \simeq \smlopfibcart(\CC)\] 

Hess calls the following definition the
``tautological bundle of categories'' \cite{hess2007}.
Elementarily,
it is the forgetful functor from the 1-category of small pointed categories
to the 1-category of small categories.
We formulate the definition as a Grothendieck construction for convenience.

\begin{definition}\label{def:universal-small-split-opfibration}
  Consider the identity functor $\id_\cat : \cat \to \cat$.
  We write \[\pcat := \smallint \id_\cat\] for the Grothendieck construction
  on the identity functor
  and $\ulow := u_{\id_\cat} : \pcat \to \cat$ for the forgetful functor to the base.
  We call 
  % $\pcat$ the
  % \emph{(opfibrant) category of small pointed categories},
  % and
  $\ulow$ the \emph{universal small split opfibration}
  (or just the \emph{universal split opfibration}).
  This name will be justified by \Cref{prop:universal-opfibration-classifies}.
  We will also refer to $\ulow : \pcat \to \cat$ as a \emph{universe},
  as it will act like one in our model.
\end{definition}

The object $\pcat$ itself can be computed as follows:
it is the category where objects are pairs $(\CC,x)$ such that
$\CC$ is a small category and $x$ is an object in $\CC$,
and morphisms are pairs $(\FF,\phi) : (\CC,x) \to (\DD,y)$
such that $\FF : \CC \to \DD$ is a functor
and $\phi : \FF x \to y$ is a morphism in $\DD$.

For categories $\DD$ and $\CC$ and functors
$F : \DD \to \CC$ and $A : \CC \to \cat$,
there is a functor
$\int F : \int A \circ F \to \int A$
such that the following square is a pullback in $\Cat$.
\[
\begin{tikzcd}
  {\int A F} & {\int A} \\
  \DD & \CC 
  \arrow["{\int F}", from=1-1, to=1-2]
  \arrow["{u_{A F}}"', from=1-1, to=2-1]
  \arrow["{u_A}", from=1-2, to=2-2]
  \arrow["F"', from=2-1, to=2-2]
\end{tikzcd}
\]
It follows that the Grothendieck construction
is always a pullback of the universal one:
for any category $\CC$ and functor $A : \CC \to \cat$,
there is a pullback square in $\Cat$
\[
  \begin{tikzcd}
    {\int A} & {\pcat} \\
    \CC & \cat
    \arrow[from=1-1, to=1-2]
    \arrow["u_A"', from=1-1, to=2-1]
    \arrow["u_\cat", from=1-2, to=2-2]
    \arrow["A"', from=2-1, to=2-2]
  \end{tikzcd}
\]

Combining this with the equivalence
$[\CC,\cat] \simeq \smlopfibcart(\CC)$,
we have the following proposition.
\begin{proposition}\label{prop:universal-opfibration-classifies}
  Let $\CC$ and $\DD$ be large categories.
  Suppose $(F : \DD \to \CC, l)$ is a small split opfibration.
  Then $F$ is a pullback of the universal split opfibration
  $\ulow : \pcat \to \cat$ along its fibres functor
  $F^* : \CC \to \cat$.
\end{proposition}

The universal split opfibration
$\ulow : \pcat \to \cat$ is unique up to isomorphism,
but it is useful to distinguish its isomorphic ``vertical opposite''
which will be useful for our type theory.
Namely, we can consider the Grothendieck construction
on the opposite functor $\op : \cat \to \cat$.
\[ \catp := \smallint \op 
  \quad \quad \uupp := u_{\op{}} : \catp \to \cat \]
One can think of $\catp$ as $\pcat$,
except that we invert the
morphisms in the fibres.
% Since $\op{}$ is self-inverse, 
% these two universes are isomorphic.
\[\begin{tikzcd}
    \catp & \pcat \\
    \cat & \cat
    \arrow["\sim", from=1-1, to=1-2]
    \arrow["{\uupp}"', from=1-1, to=2-1]
    \arrow["\lrcorner"{anchor=center, pos=0.125}, draw=none, from=1-1, to=2-2]
    \arrow["{\ulow}", from=1-2, to=2-2]
    \arrow["\op{}"', from=2-1, to=2-2]
    \arrow["\sim", draw=none, from=2-1, to=2-2]
  \end{tikzcd} \]
More generally, viewing a split opfibration over $\CC$
as an \emph{indexed category} $A : \CC \to \cat$,
we can construct its ``vertical opposite'' by composing with the opposite functor
and taking the Grothendieck construction
$\int \op{} \circ A$,
which is to take the pullback of $\uupp$ along $A$.

Recall that a functor is a split fibration precisely when
its opposite is a split opfibration.
It follows that both
\[\vlow := {\uupp}^\op : {\catp}^\op \to {\cat}^\op
\quad \text{and} \quad
\vupp := {\ulow}^\op : {\pcat}^\op \to {\cat}^\op
\]
classify small split fibrations (they are isomorphic).
The universes $\vlow$ and $\ulow$ have opposite morphisms in their bases,
but parallel morphisms in their fibres,
because $\op$ reverses arrows in both the base and fibre.
Similarly $\vupp$ and $\uupp$ have opposite morphisms in their bases,
but parallel morphisms in their fibres.
We can also view $\vlow$ and $\vupp$ as classifying
\emph{contravariant indexed categories},
or \emph{$\cat$-valued presheaves} $A : \CC \to \cat^\op$.

We will call $\vlow : {\catp}^\op \to {\cat}^\op$
the \emph{universal split fibration},
since if we take the fiber of a category $\CC : 1 \to {\cat}^\op$
in $\vlow : {\pcat}^\op \to {\cat}^\op$,
we obtain $\CC$ itself,
whereas the fiber in $\vupp : {\pcat}^\op \to {\cat}^\op$ is $\CC^\op$.

\section*{Clans and exponentiability}

Let us write $\opfib$ for the class of split opfibrations
in $\Cat$ and $\smlopfib$ for the class of small split opfibrations.
We recall some standard facts about the two classes,
in preparation for type theoretic axiomatisation.

Both $\opfib$ and $\smlopfib$ are stable under pullback.
That $\smlopfib$ is stable under pullback follows from
\Cref{prop:universal-opfibration-classifies}.
One can check that both
$\opfib$ and $\smlopfib$ contain all isomorphisms
and are closed under composition.
Furthermore, for any category $\CC$,
the unique map $\CC \to 1$ to the terminal category
is a split opfibration,
which is small precisely when $\CC$ is.
Recalling the definition of a (pre)clan from \Cref{def:preclan},
we summarise these observations as follows.

\begin{proposition}\label{prop:preclans-in-cat}
  The category of large categories with the class of split opfibrations
  $(\Cat, \opfib)$ forms a clan.
  $\Cat$ with the class of small split opfibrations
  $(\Cat, \smlopfib)$ forms a preclan.
  
  Dually,
  $\Cat$ with the class of split fibrations
  $(\Cat, \fib)$ forms a clan
  and $\Cat$ with the class of small split fibrations
  $(\Cat, \smlfib)$ forms a preclan.
\end{proposition}

We turn our attention to the issue of exponentiability.
The failure of $\opfib$ and $\fib$ to be $\pi$-clans
distinguishes the type theory of categories from that of groupoids.
It is well known that opfibrations and fibrations are
exponentiable with respect to all maps (Conduch\'e)
\cite{giraud1964, conduche1972}.
Furthermore, pushforwards in $\Cat$ satisfy several closure conditions;
for the following proposition,
recall the definition of a $\tau$-preclan (\Cref{def:double-preclan}).

\begin{proposition}\label{prop:pushforward-in-cat}
  The opfibrations and fibrations in $\Cat$ form a $\tau$-clan
  $(\Cat,\opfib,\fib)$,
  and small opfibrations and small fibrations form a $\tau$-preclan
  $(\Cat,\smlopfib,\smlfib)$.
\end{proposition}
\begin{proof}
  By \Cref{lem:stable-class-of-exponentiable-maps},
  to show that $\fib \subseteq \Exp_\opfib$,
  it suffices to check that for any split opfibration $G : \EE \to \DD$
  and any split fibration $F : \DD \to \CC$,
  there is a split opfibration $P : \PP \to \CC$
  satisfying the universal property of the pushforward
  so that $F_* G = P$.
  Indeed, a pushforward can be constructed for
  any Conduch\'e functor \cite[Proposition 1.5.1]{vidmar2018},
  and fibrations are Conduch\'e \cite[Lemma 1.2.10]{vidmar2018}.
  In this case,
  since $G$ is a split opfibration and $F$ is a split fibration,
  one can extend the cited construction to produce a split opfibration
  $(F_* G, l)$.
\end{proof}

\begin{remark}\label{rmk:pi-preserve-opfib-mor}
  Given a fibration $f : Y \to X$,
  \Cref{prop:preclans-in-cat} and
  \Cref{prop:pushforward-in-cat} imply that there are functors
  between the full subcategories of the slice $\opfib(\star) \subseteq \Cat / (\star)$.
  \[
  \begin{tikzcd}[column sep = huge]
    {\opfib(Y)} & {\opfib(X)}
    \arrow[""{name=0, anchor=center, inner sep=0}, "{f_*}"', shift right=5, bend right, from=1-1, to=1-2]
    \arrow[""{name=1, anchor=center, inner sep=0}, "{f_!}", shift left=5, bend left, from=1-1, to=1-2]
    \arrow[""{name=2, anchor=center, inner sep=0}, "{f^*}"{description}, from=1-2, to=1-1]
    \arrow["\dashv"{anchor=center, rotate=-90}, draw=none, from=1, to=2]
    \arrow["\dashv"{anchor=center, rotate=-90}, draw=none, from=2, to=0]
  \end{tikzcd}
  \]
  In fact, it is also true that these functors restrict to the
  (non-full) subcategories of opfibrations and morphisms of opfibrations.
  \[
  \begin{tikzcd}
    {\opfibcart(Y)} & {\opfibcart(X)} \\
    {\opfib(Y)} & {\opfib(X)}
    \arrow["{f_!}", dashed, from=1-1, to=1-2]
    \arrow[from=1-1, to=2-1]
    \arrow[from=1-2, to=2-2]
    \arrow["{{f_!}}"', from=2-1, to=2-2]
  \end{tikzcd}
  \quad \quad 
  \begin{tikzcd}
    {\opfibcart(Y)} & {\opfibcart(X)} \\
    {\opfib(Y)} & {\opfib(X)}
    \arrow["{f_*}", dashed, from=1-1, to=1-2]
    \arrow[from=1-1, to=2-1]
    \arrow[from=1-2, to=2-2]
    \arrow["{{f_*}}"', from=2-1, to=2-2]
  \end{tikzcd}
  \quad \quad 
  \begin{tikzcd}
    {\opfibcart(Y)} & {\opfibcart(X)} \\
    {\opfib(Y)} & {\opfib(X)}
    \arrow["{f^*}"', dashed, from=1-2, to=1-1]
    \arrow[from=1-1, to=2-1]
    \arrow[from=1-2, to=2-2]
    \arrow["{{f^*}}", from=2-2, to=2-1]
  \end{tikzcd}
  \] 
  \[
  \begin{tikzcd}[column sep = huge]
    {\opfibcart(Y)} & {\opfibcart(X)}
    \arrow[""{name=0, anchor=center, inner sep=0}, "{f_*}"', shift right=5, bend right, from=1-1, to=1-2]
    \arrow[""{name=1, anchor=center, inner sep=0}, "{f_!}", shift left=5, bend left, from=1-1, to=1-2]
    \arrow[""{name=2, anchor=center, inner sep=0}, "{f^*}"{description}, from=1-2, to=1-1]
    \arrow["\dashv"{anchor=center, rotate=-90}, draw=none, from=1, to=2]
    \arrow["\dashv"{anchor=center, rotate=-90}, draw=none, from=2, to=0]
\end{tikzcd}\]
Unfortunately, the theory of exponentiability presented in \Cref{sec:exponentiable}
is inadequate for capturing this aspect of the model in $\Cat$.
One could naively throw the diagrams drawn above into the axiomatisation,
but this will not be enough,
as there should also be some extra coherence equations between
the 2-cells from the various related adjunctions.
A similar case is worked out in the language of comprehension categories in \cite{jacobs1993},
which could possibly be extended with the ideas from \Cref{sec:exponentiable}.
Related to this,
Gambino and Larrea construct models of Martin-L\"of Type Theory
from ``type theoretic'' AWFSs \cite{gambino2023},
which one could potentially modify to suit the case here.
Najmaei et al. present a type theory that internalises
the morphisms of algebras in an AWFS (as is the case here)
\cite{najmaei2026}.
\end{remark}

\section*{Twisted path factorisations}
% \label{sec:twisted-path-factorisation}
Motivated from the groupoid model \cite{hofmann1995},
and the algebraic semantics of identity types \cite{awodey2025}
and path types \cite{awodey2026}, 
we seek to produce a similar axiomatisation of
\emph{hom-types} in $\Cat$,
using Paige North's idea of two elimination rules \cite{north2019}
instead of the single one for identity types.
Recall from \cite{awodey2026} that in the category of groupoids $\Grpd$,
a (closed) type $A$ is interpreted as an isofibration (over $1$),
and the identity type for elements of $A$
is modelled by exponentiating $A$ by the bipointed
``walking isomorphism''.
\[ \begin{tikzcd}
	1 & I & A & {A^I} \\
	& {1+1} && {A \times A}
	\arrow["{!}"', from=1-2, to=1-1]
	\arrow[from=1-3, to=1-4]
	\arrow[from=1-3, to=2-4]
	\arrow["{{(\src, \trg)}}", from=1-4, to=2-4]
	\arrow["{!}", from=2-2, to=1-1]
	\arrow["{{(0,1)}}"', from=2-2, to=1-2]
\end{tikzcd} \]
For this to model a type,
we must ensure that the map $A^I \to A \times A$ is an isofibration,
which it is.

In $\Cat$, consider an opfibration over $1$, i.e. a category $A$.
It is natural to exponentiate by the
bipointed ``walking arrow'' $\two$, instead of $I$.
\[
\begin{tikzcd}
	1 & \two & A & {A^\two} \\
	& {1+1} && {A \times A}
	\arrow["{!}"', from=1-2, to=1-1]
	\arrow[from=1-3, to=1-4]
	\arrow[from=1-3, to=2-4]
	\arrow["{{(\src, \trg)}}", from=1-4, to=2-4]
	\arrow["{!}", from=2-2, to=1-1]
	\arrow["{{(0,1)}}"', from=2-2, to=1-2]
\end{tikzcd}
\]
Unfortunately, the map $(\src,\trg)$
is not as well-behaved as it was in the category of groupoids:
$\src : A^\two \to A$ is a fibration
whereas $\trg : A^\two \to A$ is an opfibration,
making $A^\two \to A \times A$
a two-sided fibration (or ``bifibration'' in \cite{street1974}).
% Two-sided fibrations are only classified bicategorically,
% adding a significant layer of complexity to the theory.
An approach to directed type theory using two-sided fibrations for
is explored in \cite{rivera2026}.

We consider the following alternative approach to the problem.
Recall the following definition of a twisted arrow category \cite{linton1965}.

\begin{definition}\label{def:twisted-arrow}
  For a category $A$,
  the \emph{twisted arrow category} $\Tw A$
  has morphisms of $A$ as its objects,
  and for two objects $(f : X \to Y)$ and $(g : W \to Z)$,
  a morphism $\phi : f \to g$ in $\Tw A$ consists of a pair of maps
  $\phi_0 : W \to X$ and $\phi_1 : Y \to Z$ such that
  $\phi_1 \circ f \circ \phi_0 = g$.
  \[ \begin{tikzcd}
    X & W \\
    Y & Z
    \arrow[""{name=0, anchor=center, inner sep=0}, "f"', from=1-1, to=2-1]
    \arrow["{\phi_0}"', from=1-2, to=1-1]
    \arrow[""{name=1, anchor=center, inner sep=0}, "g", from=1-2, to=2-2]
    \arrow["{\phi_1}"', from=2-1, to=2-2]
  \end{tikzcd}\]
  This is equipped with two endpoint projections
  $\src : \Tw{A} \to {A}^\op$,
  which takes an arrow to its source,
  and $\trg : \Tw{A} \to A$ which takes the target.
  There is also a functor
  $\rfl : \Core{A} \to \Tw{A}$ that takes
  an object $x$ in the groupoid core of $A$ to the identity on $x$
  and an isomorphism $f : x \to y$ to the twisted morphism
  $(f^{-1}, f) : \id_x \to \id_y$.
  \[\begin{tikzcd}
    x & {y} \\
    x & {y}
    \arrow["{=}"', from=1-1, to=2-1]
    \arrow["{f^{-1}}"', from=1-2, to=1-1]
    \arrow["{=}", from=1-2, to=2-2]
    \arrow["f"', from=2-1, to=2-2]
  \end{tikzcd}\]
\end{definition}

\begin{proposition}\label{prop:src-trg-discrete-opfibration}
  The pair of endpoint projections
  \[(\src,\trg) : \Tw{A} \to A^\op \times A\]
  is a discrete opfibration.
\end{proposition}
\begin{proof}
  Consider an object $(f : X \to Y)$ in $\Tw{A}$
  and a morphism $(\phi_0, \phi_1) : (X,W) \to (Y,Z)$ in $A^\op \times A$.
  The arrow
  $(\phi_0, \phi_1) : f \to \phi_1 \circ f \circ \phi_0$ in $\Tw{A}$
  is the unique lift of $(\phi_0, \phi_1)$.
  \[ \begin{tikzcd}
    X & W \\
    Y & Z
    \arrow[""{name=0, anchor=center, inner sep=0}, "f"', from=1-1, to=2-1]
    \arrow["{\phi_0}"', from=1-2, to=1-1]
    \arrow[""{name=1, anchor=center, inner sep=0}, "\phi_1 \circ f \circ \phi_0", from=1-2, to=2-2]
    \arrow["{\phi_1}"', from=2-1, to=2-2]
  \end{tikzcd}\]
\end{proof}

With the twisted arrow category $\Tw{A}$ replacing the arrow category $A^\two$,
we obtain an opfibration $\Tw{A} \to {A}^\op \times A$
factorising the \emph{twisted diagonal}
$(\inv, \inc) : \Core{A} \to {A}^\op \times A$,
which takes an object $x$ in the groupoid core of $A$ to
$(x,x)$ and an isomorphism $f : x \to y$
to a pair of maps $(f^{-1}, f) : (x,x) \to (y,y)$.
\[
\begin{tikzcd}
	{\Core A} & {\Tw{A}} \\
	& {{A}^\op \times A}
	\arrow["\rfl", from=1-1, to=1-2]
	\arrow["{(\inv,\inc)}"', from=1-1, to=2-2]
	\arrow["{{(\src,\trg)}}", from=1-2, to=2-2]
\end{tikzcd}
\]
We will call this the \emph{twisted pathobject factorisation}.

% In \Cref{def:twisted-arrow},
% we described the twisted arrow category on a category $A$,
% and the twisted pathobject factorisation of the map
% $(\inv,\inc) : \Core{A} \to {A}^\op \times A$.
% \[
% \begin{tikzcd}
% 	{\Core A} & {\Tw{A}} \\
% 	& {{A}^\op \times A}
% 	\arrow["\rfl", from=1-1, to=1-2]
% 	\arrow["{(\inv,\inc)}"', from=1-1, to=2-2]
% 	\arrow["{{(\src,\trg)}}", from=1-2, to=2-2]
% \end{tikzcd}
% \]

% Recall that a left adjoint $F$ in $F \dashv G : \CC \to \DD$
% is a ``left-adjoint-right-inverse'' or ``lari'' if
% the unit of the adjunction is the identity.

It will also be useful to define $\rTw{A}$,
the ``right twisted arrow category'' on $A$,
by the following pullback
\[ \begin{tikzcd}[column sep = large]
	{\rTw{A}} & {\Tw{A}} \\
	{(\Core{A}) \times A} & {{A}^\op \times A}
	\arrow[from=1-1, to=1-2]
	\arrow["{(\src,\trg)}"', from=1-1, to=2-1]
	\arrow["\lrcorner"{anchor=center, pos=0.125}, draw=none, from=1-1, to=2-2]
	\arrow["{(\src,\trg)}", from=1-2, to=2-2]
	\arrow["{\inv \times A}"', from=2-1, to=2-2]
\end{tikzcd} \]
Explicitly,
$\rTw{A}$ is the wide subcategory of the twisted arrow category $\Tw{A}$
such that morphisms $(\phi_0,\phi_1) : f \to g$ are in $\rTw{A}$
if and only if $\phi_0$ is invertible.
\[
\begin{tikzcd}
	x & w \\
	y & z
	\arrow["f"', from=1-1, to=2-1]
	\arrow["{\phi_0}"', from=1-2, to=1-1]
	\arrow["\sim", draw=none, from=1-2, to=1-1]
	\arrow["g", from=1-2, to=2-2]
	\arrow["{\phi_1}"', from=2-1, to=2-2]
\end{tikzcd}
\]
Then the twisted path factorisation can be pulled back, as follows,
to produce a \emph{right twisted pathobject factorisation}.
\begin{equation} \label{eq:global-right-twisted}
\begin{tikzcd}
	{\Core{A}} && \\
	& {\rTw{A}} & {\Tw{A}} \\
	& {(\Core{A}) \times A} & {{A}^\op \times A}
	\arrow["\rfl"{description}, dashed, from=1-1, to=2-2]
	\arrow["\rfl", bend left, from=1-1, to=2-3]
	\arrow["{\Delta := (\id, \inc)}"', bend right, from=1-1, to=3-2]
	\arrow[from=2-2, to=2-3]
	\arrow["{(\src,\trg)}"', from=2-2, to=3-2]
	\arrow["\lrcorner"{anchor=center, pos=0.125}, draw=none, from=2-2, to=3-3]
	\arrow["{(\src,\trg)}", from=2-3, to=3-3]
	\arrow["{\inv \times A}"', from=3-2, to=3-3]
\end{tikzcd}
\end{equation}
% Consider the functor
% $(\id_A, \inc) : \Core A \to \Core A \times A$ for a category $A$,
% where $\inc : \Core A \to A$ is the inclusion.
This factorisation is in fact equivalent to the canonical (lari, split opfibration)
factorisation given in \Cref{thm:lari-opfib-awfs}.
\[
\begin{tikzcd}
	& {\Delta / (\Core{A}) \times A} & \\
	{\Core{A}} && {(\Core{A}) \times A} \\
	& {\rTw{A}}
	\arrow["{R_\Delta}", from=1-2, to=2-3]
	\arrow["\simeq"{description, pos=0.3}, from=1-2, to=3-2]
	\arrow["{L_\Delta}", from=2-1, to=1-2]
	\arrow["{{\Delta}}"{description, pos=0.3}, from=2-1, to=2-3]
	\arrow["{{\text{lari}}}"{description}, from=2-1, to=3-2]
	\arrow["{{{\text{split opfibration}}}}"{description}, from=3-2, to=2-3]
\end{tikzcd}
\]

Furthermore, in this case, the second map
$\rTw{A} \to (\Core{A}) \times A$ is a discrete fibration,
as it is a pullback of one (\Cref{prop:src-trg-discrete-opfibration}).
Any lari functor is also initial (any left adjoint is initial),
so this coincides with the orthogonal
\emph{comprehensive factorisation}
of \cite{street1973}.

There is also a ``left twisted arrow category'' $\lTw{A}$,
defined similarly.
It is also an example of the (lari, split opfibration)
factorisation.
\[
\begin{tikzcd}
	{\lTw{A}} & {\Tw{A}} \\
	{{A}^\op \times \Core{A}} & {{A}^\op \times A}
	\arrow[from=1-1, to=1-2]
	\arrow["{(\src,\trg)}"', from=1-1, to=2-1]
	\arrow["\lrcorner"{anchor=center, pos=0.125}, draw=none, from=1-1, to=2-2]
	\arrow["{(\src,\trg)}", from=1-2, to=2-2]
	\arrow["{{A}^\op \times \inc}"', from=2-1, to=2-2]
\end{tikzcd}
\quad
\begin{tikzcd}
	{\Core{A}} && {{A}^\op \times \Core{A}} \\
	& {\lTw{A}}
	\arrow["{(\inv, \id)}", from=1-1, to=1-3]
	\arrow["{\text{lari}}"{description}, from=1-1, to=2-2]
	\arrow["{{\text{split opfibration}}}"{description}, from=2-2, to=1-3]
\end{tikzcd}
\]

We record these facts in the following proposition.
\begin{proposition}\label{prop:right-twisted-factorisation-equiv}
  The right twisted pathobject factorisation of the
  ``right diagonal'' \[\Core A \to (\Core{A}) \times A\] is equivalent
  to the (lari,split opfibration) factorisation (\Cref{thm:lari-opfib-awfs}),
  as well as the (initial, discrete opfibration) factorisation \cite{street1973}.
  % In particular, the left factor $\rfl : \Core A \to \rTw A$
  % is lari (hence initial),
  % and the right factor $(\src,\trg) : \rTw A \to \Core A \times A$
  % is a discrete opfibration (hence a split opfibration).
  The right adjoint of the lari functor $\rfl : \Core A \to \rTw A$
  is given by $\src : \rTw A \to \Core A$.

  Similarly, the left twisted pathobject factorisation of
  the map \[\Core A \to A^\op \times \Core A\] coincides with
  the (lari,split opfibration) factorisation,
  as well as the comprehensive factorisation.
  The right adjoint of the lari functor $\rfl : \Core A \to \lTw A$
  is given by $\trg : \lTw A \to \Core A$.
\end{proposition}

\section*{Relativised constructions} 
%  \label{sec:relativised-constructions}
Since we want to build a type-theoretic analysis of $\Cat$,
we are interested not just in constructions on objects $A$ of $\Cat$,
but on ``relativised objects'', i.e. (small) split opfibrations $A \to X$.
The analogue to this in the LCC setting would be working in a slice over $X$.
We will regard (covariant) type families over $X$ as split opfibrations over $X$.
Instead of considering the slice $\Cat / X$,
we will therefore be taking the category $\opfibcart(X)$ of split opfibrations over $X$
(for all type families),
or the category $\smlopfibcart(X)$ of small split opfibrations over $X$
(for small type families).
The lemmas in this section verify that the various constructions we
have considered so far can indeed be relativised in this way.

Taking the category of small split opfibrations and
opcartesian morphisms defines a contravariant pseudofunctor
from the 1-category of large categories
to the 2-category of extra-large categories,
\[\smlopfibcart(-) : \Cat^\op \to \tCat\]
where the action on morphisms $F : \DD \to \CC$ is given by
pullback along $F$.
It follows from \Cref{thm:opfibration-main} (or a small version thereof)
that $\smlopfibcart(-)$ is equivalent to the strict 2-functor
$[-,\cat] : \Cat^\op \to \tCat$,
using the fact that $\Cat$ is Cartesian closed\footnote{
  There are two subtle technicalities here.
  One is an issue of size.
  The category $\smlopfibcart(X)$ is extra-large,
  but is equivalent to $[X,\cat]$,
  which is just large.
  The second is an issue of dimension.
  Strictly speaking, the Cartesian closure provides a 1-functor
  $[-,\cat] : \Cat^\op \to \Cat$,
  which we compose with the inclusion $\Cat \subseteq \tCat$,
  which is a 2-functor.}.
\[\smlopfibcart(-) \simeq [-,\cat]\]
% \[
% \begin{tikzcd}
% 	{\Cat^\op} & \tCat \\
% 	& \Cat
% 	\arrow["{\smlopfibcart(-)}", from=1-1, to=1-2]
% 	\arrow["{[-,\cat]}"', from=1-1, to=2-2]
% 	\arrow["\simeq"', from=2-2, to=1-2]
% \end{tikzcd}
% \]
The twisted arrow structure can then be relativised --
constructed uniformly and locally over every object --
with $\smlopfibcart(X)$ playing the role of
a slice category $\CC / X$ in an LCC.
We outline this construction below.

Recall that there are endofunctors on $\cat$
that respectively take a category to its opposite,
a category to its core,
and a category to its twisted arrow category,
which are related by natural transformations,
variously taking source, target, and so on.
% (described in \Cref{sec:twisted-path-factorisation}).
Let us denote these functors and natural transformations as follows
(we add the overline annotation to reserve the
unannotated version for later).
\[
\overline{\vOp{}},\, \overline{\vCore{}}, \, \overline{\vTw{}}
: \cat \to \cat
\quad \quad 
\begin{tikzcd}
  & \overline{\vOp{}} \\
  \overline{\vCore{}} & \overline{\vTw{}} \\
  & {\id_{\cat}}
  \arrow["\overline\inv", from=2-1, to=1-2]
  \arrow["\overline\rfl"{description}, from=2-1, to=2-2]
  \arrow["\overline\inc"', from=2-1, to=3-2]
  \arrow["\overline\src"', from=2-2, to=1-2]
  \arrow["\overline\trg", from=2-2, to=3-2]
\end{tikzcd}
\]

Post-composition with $\overline{\vOp{}}$, $\overline{\vCore{}}$ and $\overline{\vTw{}}$
define the following (strict) 2-transformations
on $[-,\cat] : \Cat^\op \to \Cat$.
\[[-,\overline{\vOp{}}],\, [-,\overline{\vCore{}}], \, [-,\overline{\vTw{}}] : [-,\cat] \to [-,\cat]\]
% defines a (strict) 2-transformation
% \[[-,\overline{\vOp{}}] : [-,\cat] \to [-,\cat]\]
% Similarly, the core functor and twisted arrow functor
These are related by the following modifications.
\[
\begin{tikzcd}[column sep=huge]
  & {[-,\overline{\vOp{}}]} \\
  {[-,\overline{\vCore{}}]} & {[-,\overline{\vTw{}}]} \\
  & {\id}
  \arrow["{[-,\overline\inv]}", from=2-1, to=1-2]
  \arrow["{[-,\overline\rfl]}"{description}, from=2-1, to=2-2]
  \arrow["{[-,\overline\inc]}"', from=2-1, to=3-2]
  \arrow["{[-,\overline\src]}"', from=2-2, to=1-2]
  \arrow["{[-,\overline\trg]}", from=2-2, to=3-2]
\end{tikzcd}
\]

\begin{definition}\label{def:relativised-constructions}
Using the equivalence $[-,\cat] \simeq \smlopfibcart(-)$
we obtain three pseudotransformations
\[{\vOp{}}, {\vCore{}}, {\vTw{}}
: \smlopfibcart(-) \to \smlopfibcart(-) \]
which are related by the following modifications.
\[
\begin{tikzcd}
  & {\vOp{}} \\
  {\vCore{}} & {\vTw{}} \\
  & {\id_{\smlopfibcart(-)}}
  \arrow["\inv", from=2-1, to=1-2]
  \arrow["\rfl"{description}, from=2-1, to=2-2]
  \arrow["\inc"', from=2-1, to=3-2]
  \arrow["\src"', from=2-2, to=1-2]
  \arrow["\trg", from=2-2, to=3-2]
\end{tikzcd}
\]
\end{definition}
For each $X$ in $\Cat$,
the operation $\vOp[X]{} : \smlopfibcart(X) \to \smlopfibcart(X)$
applies $\vOp{}$ to each fibre of a split opfibration $A \to X$.
We therefore call $\vOp[X]{A}$ the \emph{vertical opposite} of $A$ (over $X$).
Similarly, we have the \emph{vertical core} $\vCore[X]{A}$ and \emph{vertical twisted arrows} $\vTw[X]{A}$.
\begin{definition}
  We apply the vertical core to the universe (over $\cat$) to produce
  \[ \usim := \vCore[\cat]{\ulow} : \cat_\sim \to \cat \]
\end{definition}
 
\begin{proposition}\label{prop:local-src-trg-discrete-opfibration}
  For a split opfibration $A \to X$,
  the pathobject map
  \[(\src,\trg) : \vTw[X] A \to \vOp[X]{A} \times_X A\]
  in the relativised twisted pathobject factorisation over $X$
  is a discrete opfibration,
  which is in particular a split opfibration.  
  When $A \to X$ is a small, so is $(\src,\trg)$.
\end{proposition}
\begin{proof}
  From \Cref{prop:src-trg-discrete-opfibration},
  \[(\src, \trg) : \Tw{(A^* x)} \to (A^* x)^\op \times A^* x\]
  is a discrete opfibration, for each $x$ in $X$.
  By \Cref{lem:split-opfib-between-split-opfibs},
  it follows that
  \[(\src,\trg) : \vTw[X] A \to \vOp[X]{A} \times_X A\]
  is a discrete opfibration.
  
  Suppose $A \to X$ is a small split opfibration.
  Then $\vTw[X] A \to X$ is small and
  factors as
  \[\vTw[X] A \to \vOp[X] A \times_X A \to X\]
  It follows that $\vTw[X] A \to \vOp[X] A \times_X A$ is small.
\end{proof}

What follows is a relativised version of the factorisation given in
\Cref{thm:lari-opfib-awfs}.
An explicit construction of this factorisation
is given in \cite[Lemma 4.33.14]{stacks-project}.
We provide an alternative construction,
by using straightening and unstraightening (\Cref{thm:opfibration-main})
and the global AWFS (\Cref{thm:lari-opfib-awfs}).
\begin{theorem}\label{thm:relativised-AWFS}
  Let $(A \to X,l)$ and $(B \to X,l)$
  be small split opfibrations over a category $X$.
  Any morphism of split opfibrations $F : A \to B$ over $X$
  factors in $\opfibcart(X)$ as follows.
  \[ \begin{tikzcd}
    A && B \\
    & {K_X F}
    \arrow["F", from=1-1, to=1-3]
    \arrow["{L_X F}"', from=1-1, to=2-2]
    \arrow["{R_X F}"', from=2-2, to=1-3]
  \end{tikzcd} \]
  The factorisation $(L_X F, R_X F)$
  forms the basis of an AWFS on $\opfibcart(X)$.

  Furthermore, these factorisations are stable under pullback
  along maps $G : Y \to X$.
  In particular, $K_F$ is the component of a modification
  $K : \opfibcart(-) \to \opfibcart(-)$.
  \[
  \begin{tikzcd}
    {\opfibcart(X)} & {\opfibcart(X)} \\
    {\opfibcart(Y)} & {\opfibcart(Y)}
    \arrow["K", from=1-1, to=1-2]
    \arrow["{G^*}"', from=1-1, to=2-1]
    \arrow["{G^*}", from=1-2, to=2-2]
    \arrow["K"', from=2-1, to=2-2]
  \end{tikzcd}
  \]
\end{theorem}
\begin{proof}
  To construct the monad
  $R_X : \opfibcart(X)^\two \to \opfibcart(X)^\two$
  we can use straightening and unstraightening
  (\Cref{thm:opfibration-main}),
  and the unrelativised AWFS on $\Cat$
  (\Cref{thm:lari-opfib-awfs}):
  \[ \begin{tikzcd}
    {\opfibcart(X)^\two} & {[X,\Cat]^\two} & {[X,\Cat^\two]} \\
    {\opfibcart(X)^\two} & {[X,\Cat]^\two} & {[X,\Cat^\two]}
    \arrow["\simeq"{description}, draw=none, from=1-1, to=1-2]
    \arrow["{R_X}"', dashed, from=1-1, to=2-1]
    \arrow["\simeq"{description}, draw=none, from=1-2, to=1-3]
    \arrow["{R \circ (-)}", from=1-3, to=2-3]
    \arrow["\simeq"{description}, draw=none, from=2-2, to=2-1]
    \arrow["\simeq"{description}, draw=none, from=2-3, to=2-2]
  \end{tikzcd} \]
  The unit and multiplication for $R_X$
  can be constructed by whiskering
  with the unit and multiplication for $R : \Cat^\two \to \Cat^\two$.
  The construction of $L_X$ is dual.
  % Briefly: by reindexing,
  % $F : \two \to [X,\cat]$ corresponds to a functor
  % \[ \tilde{F} : X \to \cat^\two\]
  % The comonad $L : \cat^\two \to \cat^\two$ and monad
  % $R : \cat^\two \to \cat^\two$
  % can be composed with $\tilde{F}$ to give
  % \[L \circ \tilde{F} : X \to \cat^\two \quad \text{and} \quad 
  %   R \circ \tilde{F} : X \to \cat^\two \]
  % These correspond to natural tranformations $L_F : A \to K_F$
  % and $R_F : K_F \to B$, when reindexing back to $\two \to [X,\cat]$.
\end{proof}

\begin{theorem}\label{thm:relativised-algebras}
  The coalgebras for the comonad $L_X : \opfibcart(X) \to \opfibcart(X)$
  and the algebras for the monad $R_X : \opfibcart(X) \to \opfibcart(X)$
  can be computed pointwise:
  \[
    \coalg{L_X} \simeq [X,\coalg{L}] \simeq [X,\lari]
    \quad \quad
    \alg{R_X} \simeq [X,\alg{R}] \simeq [X,\opfibcart]
  \]
  It follows from \Cref{lem:split-opfib-between-split-opfibs}
  and \Cref{lem:lari-between-split-opfibs}
  that an $R_X$-algebra structure on $F : A \to B$ over $X$
  produces a split opfibration structure on $F : A \to B$,
  and an $L_X$-algebra on $F$ produces a lari adjunction $F \dashv R$,
  where $R$ is a morphism of split opfibrations over $X$.
\end{theorem}

  % The category of coalgebras for the comonad $L_X : \opfibcart(X) \to \opfibcart(X)$
  % is equivalent to the category of functors $[X,\coalg{L}]$,
  % and the category of algebras for the monad $R_X : \opfibcart(X) \to \opfibcart(X)$
  % is equivalent to the category of $[X,\alg{R}]$.
  % $[X,\opfibcart]$
One can also construct a relativised right twisted arrow structure
\[\begin{tikzcd}
	& {\vCore{}} \\
	{\vCore{}} & {\rtw{}} \\
	& {\id_{\opfib(-)}}
	\arrow["{=}", from=2-1, to=1-2]
	\arrow["\rfl"{description}, from=2-1, to=2-2]
	\arrow["\inc"', from=2-1, to=3-2]
	\arrow["\src"', from=2-2, to=1-2]
	\arrow["\trg", from=2-2, to=3-2]
\end{tikzcd}\]

\begin{proposition}\label{prop:rtw-equiv-K}
  For a small opfibration $A \to X$,
  $\rtw[X]{A}$ and $K_\delta$ are equivalent as categories over $X$,
  where
  \[\delta = (\id,\inc) : \vCore[X]{A} \to \vCore[X]{A} \times_X A\] is the
  (right) diagonal and
  \[\begin{tikzcd}
    {\vCore[X]{A}} & {K_\delta} & {\vCore[X]{A} \times_X A}
    \arrow["{L_\delta}", from=1-1, to=1-2]
    \arrow["{R_\delta}", from=1-2, to=1-3]
  \end{tikzcd} \]
  is the functorial factorisation of $\delta$ given by the AWFS of
  \Cref{thm:relativised-AWFS}.
  Furthermore, this equivalence commutes with the two factorisations
  \[ \begin{tikzcd}
    & {\rtw[X]{A}} & \\
    {\vCore[X]{A}} && {\vCore[X]{A} \times_X A} \\
    & {K_\delta}
    \arrow["{(\src,\trg)}", from=1-2, to=2-3]
    \arrow["\simeq"{description}, from=1-2, to=3-2]
    \arrow["\rfl", from=2-1, to=1-2]
    \arrow["{L_\delta}"', from=2-1, to=3-2]
    \arrow["{R_\delta}"', from=3-2, to=2-3]
  \end{tikzcd}\]
  By \Cref{thm:relativised-algebras}, $L_\delta$ is lari.
  It follows that $\rfl : \vCore[X]{A} \to \rtw[X]{A}$
  has the left-lifting property with respect to
  split opfibrations.
\end{proposition}
\begin{proof}
  Write $\tilde{A} : X \to \cat$ for the fibres functor corresponding to
  the small opfibration $A$ in $\smlopfib(X)$.
  Then $\rtw[X]{A}$ in $\smlopfib(X)$ corresponds to the functor
  $\rtw{} \circ \tilde{A}$,
  where $\rtw{} : \cat \to \cat$ is the small version of the
  right twisted arrow functor $\rTw{} : \Cat \to \Cat$
  defined in \Cref{eq:global-right-twisted}.
  The equivalence of $K_{(\id,\inc)} \simeq \rtw{X}$ for a category $X$
  extends to an equivalence of 2-functors $K \simeq \rtw{}$
  in the 2-category $\cat \to \cat$.
  Whiskering provides an equivalence of 2-functors
  $K \circ \tilde{A} \simeq \rtw{} \circ \tilde{A}$,
  which translates to the desired equivalence $K_\delta \simeq \rtw[X]{A}$.
\end{proof}
\section{Algebraic categorical models}\label{sec:type-theoretic-axioms}

In this section,
we introduce axioms for a model of directed type theory
in the style of algebraic type theory,
that could be satisfied in models other than
the one in $\Cat$ considered so far (\Cref{ex:ACM-of-opfibrations}).
Once the axioms have all been stated,
we will check that they are indeed satisfied in $\Cat$.

The algebraic style allows us to state our axioms in a concise manner
that automatically guarantees stability under substitution of each construction,
so that when a syntax is fully worked out,
one can easily define an interpretation.
We begin with an axiom for the
external fibration structure on a category.

\section*{Opfibrations and fibrations} 
% \label{sec:types-opfibration-and-fibration}
\begin{axiom}\label{ax:opfibration-preclan}
$\CC$ is a large category with a terminal object and a
class of maps $\opfib$ forming a preclan $(\CC,\opfib)$.
We call maps in $\opfib$ \emph{opfibrations}\footnote{
  Since we are now working in an axiomatic setting,
  $\opfib$ could consist of something other than the class of split opfibrations in $\Cat$,
  and so we drop the adjective ``split''.
  Indeed $\opfib$ could be interpreted in $\Cat$
  as the class of cloven opfibrations
  or discrete opfibrations.
  It could also be interpreted in $\Grpd$
  as the class of split or cloven isofibrations.
}.
\end{axiom}

As usual, objects in the category $\CC$ are regarded as contexts,
and morphisms as substitutions.

\begin{axiom}\label{ax:global-op}
  We have an endofunctor $\Op{} : \CC \to \CC$
  satisfying $\Op \circ \Op = \id_\CC$.
\end{axiom}

The functor $\Op : \CC \to \CC$ is self-inverse,
hence preserves all limits, including the terminal object.
\[ \Op{1} \iso 1\]

\begin{definition}
  We write $\fib$ for the class of maps
  $f : X \to Y$ such that $\Op{f} : \Op{X} \to \Op{Y}$
  is an opfibration.
  Maps in $\fib$ are called \emph{fibrations}.
\end{definition}

It follows from 
\Cref{ax:opfibration-preclan} and \Cref{ax:global-op}
that $(\CC,\fib)$ is also a preclan.

\begin{proposition}\label{prop:op-preclan-morphism}
  The functor $\Op : \CC \to \CC$
  extends to two variance-reversing preclan morphisms
  \[ \Op{} : (\CC, \opfib) \to (\CC, \fib)
  \quad \quad
  \Op{} : (\CC, \fib) \to (\CC, \opfib) \]
  which are inverses of each other.
\end{proposition}

Recall the definitions of preclan-exponentiability
(\Cref{def:exponentiable})
and $\tau$-preclans (\Cref{def:double-preclan}).
The following proposition says that given half of the exponentiability conditions we require,
$\Op$ automatically generates the other half.

\begin{proposition}\label{prop:exponentiability-opposite}
  The following conditions are equivalent
  \begin{enumerate}
  \item Fibrations are opfibration-exponentiable, meaning
    \[ \fib \subseteq \Exp_\opfib \]
  \item Opfibrations are fibration-exponentiable.
    \[\opfib \subseteq \Exp_\fib\]
  \item Opfibrations $(\CC,\opfib)$ and fibrations $(\CC,\fib)$ form a $\tau$-preclan.
  \end{enumerate}
\end{proposition}
\begin{proof}
  (1) $\Rightarrow$ (2).
  Let $f : Y \to X$ be an opfibration.
  By \Cref{lem:stable-class-of-exponentiable-maps}
  it suffices to show that the pullback functor $f^* : \CC / X \to \CC / Y$
  has a partial right adjoint $f_* : \fib(Y) \to \fib(X)$.

  The self-inverse $\Op : \CC \to \CC$ induces isomorphisms between slices
  \[\Op / Y : \CC / Y \iso \CC / \Op{Y} 
  \quad \text{ and } \quad
  \Op / X : \CC / X \iso \CC / \Op{X} \]
  Since $\Op : \CC \to \CC$ preserves pullbacks,
  $f^* : \CC / X \to \CC / Y$
  is isomorphic under the above equivalences to
  $(\Op{f})^* : \CC / \Op{X} \to \CC / \Op{Y}$.
  \[
  \begin{tikzcd}
    {\CC / {Y}} & {\CC / \Op{Y}} & {\opfib(\Op{Y})} & {\fib(Y)} \\
    {\CC / {X}} & {\CC / \Op{X}} & {\opfib(\Op{X})} & {\fib(X)}
    \arrow["\iso"{description}, draw=none, from=1-1, to=1-2]
    \arrow[hook', from=1-3, to=1-2]
    \arrow["\iso"{description}, draw=none, from=1-3, to=1-4]
    \arrow[""{name=0, anchor=center, inner sep=0}, "{(\Op{f})_*}", from=1-3, to=2-3]
    \arrow["{f_*}", dashed, from=1-4, to=2-4]
    \arrow["{f^*}", from=2-1, to=1-1]
    \arrow["\iso"{description}, draw=none, from=2-1, to=2-2]
    \arrow[""{name=1, anchor=center, inner sep=0}, "{(\Op{f})^*}", from=2-2, to=1-2]
    \arrow[hook', from=2-3, to=2-2]
    \arrow["\iso"{description}, draw=none, from=2-3, to=2-4]
    \arrow["{\dashv_\partial}"{description}, draw=none, from=1, to=0]
  \end{tikzcd}
  \]
  The preclan isomorphism
  $\Op : (\CC, \fib) \to (\CC, \opfib)$
  also induces isomorphisms
  \[
  \Op{} / Y : \fib(Y) \iso \opfib(\Op{Y})
  \quad \text{ and } \quad
  \Op{} / X : \fib(X) \iso \opfib(\Op{X})
  \]
  Providing a partial right adjoint
  $f_* : \fib(Y) \to \fib(X)$ of $f^* : \CC / X \to \CC / Y$
  amounts to providing a partial right adjoint
  $(\Op{f})_*$ of $(\Op{f})^*$, which we have by (1).
  
  (2) $\Rightarrow$ (1) is similar.
  (1) $\Leftrightarrow$ (3) immediately follows from (1) $\Rightarrow$ (2).
\end{proof}

\begin{axiom}\label{ax:exponentiability}
  The equivalent conditions of \Cref{prop:exponentiability-opposite} hold.
\end{axiom}

% \begin{axiom}
% $\CC$ is a category with a terminal object with two preclans
% $(\CC,\opfib)$ and $(\CC,\fib)$
% that are exponentiable against each other (see \Cref{def:exponentiable}):
% \[\fib \subseteq \Exp_\opfib
%   \quad \text{ and } \quad
%   \opfib \subseteq \Exp_\fib\]
% \end{axiom}

\section*{Vertical opposites}

In the $\Cat$ model (\Cref{ex:ACM-of-opfibrations}), we need to be able to place restrictions
on the morphisms between opfibrations -- they should be
opcartesian functors.
Let us assume the axioms that have been introduced so far are available to us.
Consider the pseudofunctor $\opfib(-) : \CC^\op \to \tCat$,
that takes an object $X$ in $\CC$ to the category $\opfib(X)$
of opfibrant objects in the slice $\CC / X$
and a morphism $f : X \to Y$ to the pullback functor
$f^* : \opfib(Y) \to \opfib(X)$.
% Similarly, $\fib(-) : \CC^\op \to \Cat$.

\begin{axiom}\label{ax:opcartesian-map}
  For each $X \in \CC$,
  there is a wide subcategory $\opfibcart(X) \to \opfib(X)$,
  forming a pseudofunctor $\opfibcart(-) : \CC^\op \to \tCat$.
  We call the maps in $\opfibcart(X)$ \emph{opcartesian maps} over $X$.
  Furthermore, all maps between opfibrations over $1$ are opcartesian.
  \[\opfibcart(1) = \opfib(1)\]
\end{axiom}

\begin{axiom}\label{ax:local-op}
  The following structure is present in $\CC$.
  \begin{itemize}
    \item We have a pseudotransformation $\vOp{} : \opfibcart(-) \to \opfibcart(-)$ 
    \item For each object $X$ in $\CC$,
    $\vOp[X]{} : \opfibcart(X) \to \opfibcart(X)$ is self inverse.
    \[\vOp[X]{} \circ \vOp[X]{} \iso \id_{\opfibcart(X)}\]
    \item The functor $\vOp[1]{} : \opfibcart(1) \to \opfibcart(1)$
    % and $\vOp[1]{} : \fibcart(1) \to \fibcart(1)$ are both (isomorphic to)
    is (up to isomorphism) a restriction of the global opposite functor
    $\Op{} : \CC \to \CC$.
    \[ \begin{tikzcd}
        {\opfibcart(1)} & {\opfibcart(1)} \\
        \CC & \CC
        \arrow["{{{\vOp[1]{}}}}", from=1-1, to=1-2]
        \arrow[hook, from=1-1, to=2-1]
        \arrow[hook, from=1-2, to=2-2]
        \arrow["{\Op{}}"', from=2-1, to=2-2]
      \end{tikzcd} \]
    % \[\begin{tikzcd}
    %     {\opfibcart (1)} & {\opfibcart (1)} \\
    %     \CC & \CC \\
    %     {\fibcart(1)} & {\fibcart(1)}
    %     \arrow["{{\vOp[1]{}}}", from=1-1, to=1-2]
    %     \arrow[hook, from=1-1, to=2-1]
    %     \arrow[hook, from=1-2, to=2-2]
    %     \arrow["\Op{}", from=2-1, to=2-2]
    %     \arrow[hook, from=3-1, to=2-1]
    %     \arrow["{{\vOp[1]{}}}"', from=3-1, to=3-2]
    %     \arrow[hook, from=3-2, to=2-2]
    %   \end{tikzcd}\]
  \end{itemize}
\end{axiom}

For an opfibration $A \to X$, the opfibration $\vOp[X]{A} \to X$
is the \emph{vertical opposite} of $A$.
In $\Cat$, the vertical opposite of $A$ can be thought of as
the category where,
over each point $x$ in $X$,
paths in each fibre $A_x$ are reversed.
This is discussed in detail in \Cref{sec:preliminaries}.

We construct the duals to \Cref{ax:opcartesian-map} and \Cref{ax:local-op}.
\begin{definition}
  Define the pseudofunctor $\fibcart(-) : \CC^\op \to \tCat$ such that
  for each $X$ in $\CC$, $\fibcart(X)$ is the wide subcategory of $\fib(X)$ consisting of maps
  $F : A \to B$ over $X$ such that $\Op F : \Op A \to \Op B$ is
  an opcartesian in $\opfibcart(\Op X)$.
  We call maps in $\fibcart(X)$ \emph{cartesian maps} over $X$.

  Define the pseudotransformation $\vOp{} : \fibcart(-) \to \fibcart(-)$
  such that the component for each $X$ in $\CC$ is the composition
  \[ \begin{tikzcd}[column sep = large]
    {\fibcart(X)} & {\opfibcart (\Op X)} & {\opfibcart (\Op X)} & {\fibcart(X)}
    \arrow["{\Op / X}", from=1-1, to=1-2]
    \arrow["\sim"', draw=none, from=1-1, to=1-2]
    \arrow["{\vOp[\Op X]{}}", from=1-2, to=1-3]
    \arrow["{\Op / \Op X}", from=1-3, to=1-4]
    \arrow["\sim"', from=1-3, to=1-4]
  \end{tikzcd} \]
\end{definition}

\begin{proposition}
  For each object $X$ in $\CC$,
  \[\vOp[X]{} \circ \vOp[X]{} \iso \id_{\fibcart(X)} : \fibcart(X) \to \fibcart(X)\]
    % \item These operations agree up to isomorphism where their domains intersect.
    % \[
    %   \begin{tikzcd}
    %   {\opfibcart (X)} & {\opfibcart (X)} \\
    %   {\opfibcart  \cap \fib (X)} & {\opfibcart  \cap \fib(X)} \\
    %   {\fib(X)} & {\fib(X)}
    %   \arrow["{\vOp[X]{}}", from=1-1, to=1-2]
    %   \arrow[hook, from=2-1, to=1-1]
    %   \arrow[dashed, from=2-1, to=2-2]
    %   \arrow[hook, from=2-1, to=3-1]
    %   \arrow[hook, from=2-2, to=1-2]
    %   \arrow[hook, from=2-2, to=3-2]
    %   \arrow["{\vOp[X]{}}"', from=3-1, to=3-2]
    % \end{tikzcd}
    % \]
    and $\vOp[1]{} : \fibcart(1) \to \fibcart(1)$
    is (up to isomorphism) a restriction of the global opposite functor
    $\Op{} : \CC \to \CC$.
    \[ \begin{tikzcd}
        {\fibcart(1)} & {\fibcart(1)} \\
        \CC & \CC
        \arrow["{{{\vOp[1]{}}}}", from=1-1, to=1-2]
        \arrow[hook, from=1-1, to=2-1]
        \arrow[hook, from=1-2, to=2-2]
        \arrow["{\Op{}}"', from=2-1, to=2-2]
      \end{tikzcd} \]
\end{proposition}
\begin{proof}
  For the second isomorphism, consider the following diagram.
  \[ \begin{tikzcd}
    A & {\Op {A}} && {\Op {A}} & A \\
    {\fibcart(1)} & {\opfibcart (\Op{1})} & {\opfibcart (1)} & \CC & \CC \\
    {\fibcart(1)} & {\opfibcart (\Op{1})} & {\opfibcart (1)} & \CC & \CC
    \arrow[maps to, from=1-1, to=1-2]
    \arrow[maps to, from=1-2, to=1-4]
    \arrow[maps to, from=1-4, to=1-5]
    \arrow["\sim", from=2-1, to=2-2]
    \arrow["{\vOp[1]{}}"', from=2-1, to=3-1]
    \arrow["\sim", from=2-2, to=2-3]
    \arrow["{\vOp[\Op{1}]{}}", from=2-2, to=3-2]
    \arrow[hook, from=2-3, to=2-4]
    \arrow["{\vOp[1]{}}", from=2-3, to=3-3]
    \arrow["\Op{}", from=2-4, to=2-5]
    \arrow["{\Op{}}", from=2-4, to=3-4]
    \arrow["\Op{}", from=2-5, to=3-5]
    \arrow["\sim"', from=3-1, to=3-2]
    \arrow["\sim"', from=3-2, to=3-3]
    \arrow[hook, from=3-3, to=3-4]
    \arrow["\Op{}"', from=3-4, to=3-5]
  \end{tikzcd}\]
  The first square is the definition of $\vOp[] : \fibcart(-) \to \fibcart(-)$;
  the second is pseudonaturality of $\vOp[] : \opfibcart (-) \to \opfibcart (-)$
  applied to the isomorphism $1 \iso \Op 1$;
  the third square is from \Cref{ax:local-op}.
  A straightforward computation shows that the composition of the horizontal maps 
  is the inclusion $\fibcart(1) \hookrightarrow \CC$.
\end{proof}

\begin{proposition}
  Opfibrations over $1$ are the same as fibrations over $1$.
  All maps between fibrations over $1$ are cartesian.
  \[ \opfib(1) = \opfibcart(1) = \fibcart(1) = \fib(1) \]
\end{proposition}
\begin{proof}
  Let $!_X : X \to 1$ be an opfibration.
  By \Cref{ax:local-op}, viewing $X$ as an object in $\opfibcart(1)$,
  we have $\Op{X} \iso \vOp[1]{X}$ is in $\opfibcart(1)$,
  making the unique map $\Op{X} \to 1$ an opfibration.
  It follows from \Cref{prop:op-preclan-morphism} that the unique map
  $\Op{\Op{X}} \to \Op{1}$ is a fibration.
  Since $\Op{} \circ \Op{} = \id_\CC$,
  this says that $X \to 1$ is a fibration.
\end{proof}

We will say that an object $X$ in $\CC$ is \emph{fibrant} if it belongs
to $\opfibcart(1)$.

\begin{definition}\label{def:ACM}
  An \emph{op structure} (for an algebraic categorical model) consists of the structures
  \[(\CC,\opfibcart,\Op,\vOp)\]
  satisfying the axioms
  \labelcref{ax:opfibration-preclan,ax:global-op,ax:exponentiability,,ax:opcartesian-map,ax:local-op}.
\end{definition}

\section*{Universes}
From now on, we fix an op structure $(\CC,\opfibcart,\Op,\vOp)$.
% Like in \Cref{def:elementary-universe},
\begin{definition}
We say that a morphism $u : U' \to U$ in $\CC$
is a \emph{universe} if for any map $A : \Gamma \to U$
there is a chosen pullback of $u$ along $A$
(cf. \Cref{def:elementary-universe}).
\end{definition}

\begin{axiom}\label{ax:universe}
  We have a universe $\ulow : \Ulow \to \U$
  such that $\U \to 1$ is fibrant,
  and $\ulow : \Ulow \to \U$ is an opfibration.
  We call $\ulow : \Ulow \to \U$ the \emph{universe of small opfibrations}.
\end{axiom}

Since $\U$ is fibrant (hence $\U \in \opfib(1) = \fib(1)$) and $\ulow$ is an opfibration,
$\Ulow$ is also fibrant.
Any opfibration is $\fib$-exponentiable
and $\U$ is fibrant,
so $(\ulow : \Ulow \to \U)$ is also a polynomial signature in the preclan
$(\CC, \fib)$.
\[
\begin{tikzcd}
	{\opfib(1) \ni} & \Ulow & \in \fib(1) \\
	{\opfib \ni} && {\in \Exp_\fib} \\
	{\opfib(1) \ni} & \U & {\in \fib(1)}
	\arrow["\ulow", from=1-2, to=3-2]
\end{tikzcd}
\]

For a map $A : \Gamma \to \U$ in $\CC$,
we denote the chosen pullback of $\ulow : \Ulow \to \U$ along
$A$ as $d_A : \Gamma \opext A \to \Gamma$, and refer to it as the 
\emph{opfibrant context extension}, or \emph{opextension} for short.
\[ \begin{tikzcd}
    {\Gamma \opext A} & \Ulow \\
    \Gamma & \U
    \arrow[from=1-1, to=1-2]
    \arrow["{d_A}"', from=1-1, to=2-1]
    \arrow["\lrcorner"{anchor=center, pos=0.125}, draw=none, from=1-1, to=2-2]
    \arrow["\ulow", from=1-2, to=2-2]
    \arrow["A"', from=2-1, to=2-2]
  \end{tikzcd}\]

\begin{definition}
  We write $\vupp : \Vupp \to \V$ for
  the image of $\ulow : \Ulow \to \U$ under the functor
  $\Op : \CC \to \CC$.
  \[
    \V := \Op{\U} \quad
    \Vupp := \Op{\Ulow} \quad
    \vupp := \Op{\ulow}
  \]
  % \[\vlow : \Vlow \to \V \quad \quad := \quad \quad
  % \Op{\ulow} : \Op{\Ulow} \to \Op{\U}\]
\end{definition}

It follows from \Cref{ax:global-op} and \Cref{ax:local-op} that
$\V$ is fibrant and $\vupp : \Vupp \to \V$ is an fibration.

As we will see in \Cref{ex:groupoidal-opfibrations,ex:discrete-opfibrations},
we might want several different universes at the same time,
and not just stacked up in a hierarchy.
To keep things simple, however,
we will only work with a single one 
(along with its various opposites, of course).

\section*{Small opfibrations and small fibrations}
Now that we are equipped with a universe of small opfibrations
$\ulow : \Ulow \to \U$,
we can consider the class of maps classified by the universe.
Consider the subclass of maps $\opfib_\U \subseteq \opfib$
consisting of those maps that are pullbacks of $\ulow$
along some map into $\U$.
Naturally, we call the maps in $\opfib_\U$ \emph{small opfibrations}.
\[
  \opfib_\U := \{ f : E \to B \st \begin{tikzcd}
	E & \Ulow \\
	B & \U
	\arrow["\exists", dashed, from=1-1, to=1-2]
	\arrow["f"', from=1-1, to=2-1]
	\arrow["\lrcorner"{anchor=center, pos=0.125}, draw=none, from=1-1, to=2-2]
	\arrow["\ulow", from=1-2, to=2-2]
	\arrow["\exists"', dashed, from=2-1, to=2-2]
\end{tikzcd}\}
\]
Dually, we have a class of \emph{small fibrations} $\fib_\V$,
which is the class of pullbacks of $\vupp : \Vupp \to \V$.

\begin{proposition}\label{prop:op-preclan-morphism-small}
  A map $f : E \to B$ is a small fibration if and only if
  $\Op{f} : \Op{E} \to \Op{B}$ is a small opfibration.
  Hence we have two preclan morphisms
  \[ \Op{} : (\CC, \opfib_\U) \to (\CC, \fib_\V)
  \quad \quad
  \Op{} : (\CC, \fib_\V) \to (\CC, \opfib_\U) \]
\end{proposition}
\begin{proof}
  This follows from the fact that $\Op$ preserves pullbacks.
\end{proof}

% The following axiom is trivial if one assumes the axiom of choice.
% when picking a pullback square when classifying a small (op)fibration.
In the $\Cat$ model (\Cref{ex:ACM-of-opfibrations}),
we could interpret $\opfib$ as the class of split opfibrations.
Unfortunately,
since we are treating $\opfib$ as a class of maps,
i.e. maps satisfying a property,
rather than maps equipped with structure, 
we will need to use the axiom of choice
to pick classifying maps of split opfibrations.
However, a careful setup of \Cref{sec:exponentiable}
and \Cref{sec:polynomials}
that replaces the class of maps $\opfib$ with a
(not necessarily full or faithful)
functor $\opfib \to \CC^\two$ should be possible.
We leave this technical solution to future work.
For now, let us use the following axiom to keep track of when we might need
the axiom of choice.

\begin{axiom}\label{ax:classifying-map}
  Every small opfibration $A : X \to \Gamma$ in $\opfib_\U$
  is equipped with a chosen \emph{classifying map}
  $\ulcorner A \urcorner : \Gamma \to \U$,
  so that $A : X \to \Gamma$ is isomorphic to the opextension
  $\Gamma \opext \ulcorner A \urcorner \to \Gamma$ in $\opfib(\Gamma)$.
  \[ \begin{tikzcd}
      X & {\Gamma \opext \ulcorner A \urcorner} & \Ulow \\
      & \Gamma & \U
      \arrow["\sim", from=1-1, to=1-2]
      \arrow["A"', from=1-1, to=2-2]
      \arrow[from=1-2, to=1-3]
      \arrow[from=1-2, to=2-2]
      \arrow["\lrcorner"{anchor=center, pos=0.125}, draw=none, from=1-2, to=2-3]
      \arrow["\ulow", from=1-3, to=2-3]
      \arrow["{\ulcorner A \urcorner}"', from=2-2, to=2-3]
    \end{tikzcd} \]
\end{axiom}

\section*{Vertically opposite universes}

% It is convenient to have notation for the vertical opposites of the two universes
% $\ulow : \Ulow \to \U$ and $\vlow : \Vlow \to \V$.
Viewing $\ulow : \Ulow \to \U$ as an object in $\opfibcart(\U)$,
we can apply $\vOp[\U]{} : \opfibcart(\U) \to \opfibcart(\U)$ to it,
obtaining an opfibration
\[\uupp : \Uupp \to \U\]
Dually, applying $\vOp[\V]{} : \fibcart(\V) \to \fibcart(\V)$ to the object
$\vupp : \Vupp \to \V$,
we have a fibration
% Viewing $\vlow : \Vlow \to \V$ as an object in $\fibcart(\V)$,
% we can apply  to it,
% obtaining a fibration 
\[\vlow : \Vlow \to \V\]

\begin{lemma}\label{lem:vupp-eq}
  $\Vlow = \Op{\Uupp}$ and $\vlow = \Op{\uupp} : \Vlow \to \V$.
\end{lemma}
\begin{proof}
  \[ \vlow = \vOp[\V]{\vupp}
  = \vOp[\V]{\Op \ulow}
  = \Op (\vOp[\U]{\Op \Op \ulow})
  = \Op (\vOp[\U]{\ulow})
  = \Op {\uupp}\]
\end{proof}

We write $\opfibcart_\U(X)$ for the wide subcategory of $\opfib_\U(X)$
consisting only of opcartesian maps.
\[ \begin{tikzcd}
	{\opfibcart_\U (X)} & {\opfibcart (X)} \\
	{\opfib_\U (X)} & {\opfib (X)}
	\arrow[hook, from=1-1, to=1-2]
	\arrow[from=1-1, to=2-1]
	\arrow["\lrcorner"{anchor=center, pos=0.125}, draw=none, from=1-1, to=2-2]
	\arrow[from=1-2, to=2-2]
	\arrow[hook, from=2-1, to=2-2]
\end{tikzcd} \]

\begin{proposition}\label{prop:vertical-opposite-small}
  The following are equivalent
  \begin{enumerate}
    \item $\vOp[X] : \opfibcart(X) \to \opfibcart(X)$ preserves $\opfibcart_\U(X)$,
    for every $X$ in $\CC$.
    \[ \begin{tikzcd}
      {\opfibcart_\U(X)} & {\opfibcart_\U(X)} \\
      {\opfibcart(X)} & {\opfibcart(X)}
      \arrow[dashed, from=1-1, to=1-2]
      \arrow[hook, from=1-1, to=2-1]
      \arrow[hook, from=1-2, to=2-2]
      \arrow["{\vOp[X]{}}"', from=2-1, to=2-2]
    \end{tikzcd} \]
    \item $\vOp[\U] : \opfibcart(\U) \to \opfibcart(\U)$ preserves $\opfibcart_\U(\U)$.
    \item $\uupp : \Uupp \to \U$ is a small opfibration.
    \item Each $\vOp[X] : \fibcart(X) \to \fibcart(X)$ preserves $\fibcart_\V(X)$.
    \item $\vOp[\V] : \fibcart(\V) \to \fibcart(\V)$ preserves $\fibcart_\V(\V)$.
    \item $\vupp : \Vupp \to \V$ is a small fibration.
  \end{enumerate}
\end{proposition}
\begin{proof}
  (1) $\Rightarrow$ (2) $\Rightarrow$ (3) is clear.
  We prove (3) $\Rightarrow$ (1)
  Suppose $A$ is an object in $\opfibcart(X)$.
  By pseudonaturality of $\vOp[X]{}$,
  \[ \vOp[X]{A} \iso \vOp[X]{(\ulcorner A \urcorner)^* \ulow}
    \iso (\ulcorner A \urcorner)^* \vOp[\U]{\ulow}
    = (\ulcorner A \urcorner)^* \uupp \]
  Hence $\vOp[X]{A}$ is a pullback of a small opfibration,
  which is small.

  % So far we have shown that $1. \Leftrightarrow 2. \Leftrightarrow 3.$
  By duality, we also have (4) $\Leftrightarrow$ (5) $\Leftrightarrow$ (6).
  It remains to check that (3) $\Leftrightarrow$ (6).
  Indeed, this follows from 
  \Cref{prop:op-preclan-morphism-small} and \Cref{lem:vupp-eq}.
\end{proof}

\begin{axiom}\label{ax:vertical-opposite-small}
  The equivalent conditions of
  \Cref{prop:vertical-opposite-small} hold.
\end{axiom}

We can view the pullbacks of the vertical opposite of a general small
opfibration as follows.
\[ \begin{tikzcd}
	{\vOp[\Gamma]{(\Gamma \opext A)}} & {\Gamma \opext A^\op} & \Uupp & \Ulow \\
	& \Gamma & \U & \U
  \arrow["\iso"{description}, draw=none, from=1-1, to=1-2]
	\arrow[from=1-2, to=1-3]
	\arrow[from=1-2, to=2-2]
	\arrow["\lrcorner"{anchor=center, pos=0.125}, draw=none, from=1-2, to=2-3]
	\arrow[from=1-3, to=1-4]
	\arrow["\uupp"', from=1-3, to=2-3]
	\arrow["\lrcorner"{anchor=center, pos=0.125}, draw=none, from=1-3, to=2-4]
	\arrow["\ulow", from=1-4, to=2-4]
	\arrow["A", from=2-2, to=2-3]
	\arrow["{A^\op}"', bend right, from=2-2, to=2-4]
	\arrow["{\ulcorner \uupp \urcorner}", from=2-3, to=2-4]
\end{tikzcd} \]

We call $\vlow : \Vlow \to \V$ the \emph{universal small fibration}.
For a map $A : \Gamma \to \V$,
we have a chosen pullback of $\vlow : \Vlow \to \V$ along $A$
given by
\[\Gamma \ext A := \Op{(\Op{\Gamma} \opext (\ulcorner \uupp \urcorner \circ \Op{A}))}\]
\[ \begin{tikzcd}
    {\Op{\Gamma} \opext (\ulcorner \uupp \urcorner \circ \Op{A})} & \Uupp & \Ulow \\
    {\Op{\Gamma}} & \U & \U
    \arrow[from=1-1, to=1-2]
    \arrow[from=1-1, to=2-1]
    \arrow["\lrcorner"{anchor=center, pos=0.125}, draw=none, from=1-1, to=2-2]
    \arrow[from=1-2, to=1-3]
    \arrow["\uupp"', from=1-2, to=2-2]
    \arrow["\lrcorner"{anchor=center, pos=0.125}, draw=none, from=1-2, to=2-3]
    \arrow["\ulow", from=1-3, to=2-3]
    \arrow["{\Op{A}}"', from=2-1, to=2-2]
    \arrow["{\ulcorner \uupp \urcorner}"', from=2-2, to=2-3]
  \end{tikzcd}
  \quad \rightsquigarrow \quad
  \begin{tikzcd}
    {\Gamma \ext A} & \Vlow & \Vupp \\
    \Gamma & \V & \V
    \arrow[from=1-1, to=1-2]
    \arrow["{d_A}"', from=1-1, to=2-1]
    \arrow["\lrcorner"{anchor=center, pos=0.125}, draw=none, from=1-1, to=2-2]
    \arrow[from=1-2, to=1-3]
    \arrow["\vlow"', from=1-2, to=2-2]
    \arrow["\lrcorner"{anchor=center, pos=0.125}, draw=none, from=1-2, to=2-3]
    \arrow["\vupp", from=1-3, to=2-3]
    \arrow["A"', from=2-1, to=2-2]
    \arrow[from=2-2, to=2-3]
  \end{tikzcd}\]
We refer to $\Gamma \ext A$ as the \emph{fibrant context extension}.

% We can view the pullbacks of the vertical opposite of a general small
% (op)fibration as follows.
% \[ \begin{tikzcd}
% 	{\vOp[\Gamma]{(\Gamma \opext A)}} & {\Gamma \opext A^\op} & \Uupp & \Ulow \\
% 	& \Gamma & \U & \U
%   \arrow["\iso"{description}, draw=none, from=1-1, to=1-2]
% 	\arrow[from=1-2, to=1-3]
% 	\arrow[from=1-2, to=2-2]
% 	\arrow["\lrcorner"{anchor=center, pos=0.125}, draw=none, from=1-2, to=2-3]
% 	\arrow[from=1-3, to=1-4]
% 	\arrow["\uupp"', from=1-3, to=2-3]
% 	\arrow["\lrcorner"{anchor=center, pos=0.125}, draw=none, from=1-3, to=2-4]
% 	\arrow["\ulow", from=1-4, to=2-4]
% 	\arrow["A", from=2-2, to=2-3]
% 	\arrow["{A^\op}"', bend right, from=2-2, to=2-4]
% 	\arrow["{\ulcorner \uupp \urcorner}", from=2-3, to=2-4]
% \end{tikzcd} \]

% \[\begin{tikzcd}
% 	{\vOp[\Gamma]{(\Gamma \ext A)}} & {\Gamma \ext A^\op} & \Vupp & \Vlow \\
% 	& \Gamma & \V & \V
% 	\arrow["\iso"{description}, draw=none, from=1-1, to=1-2]
% 	\arrow[from=1-2, to=1-3]
% 	\arrow[from=1-2, to=2-2]
% 	\arrow["\lrcorner"{anchor=center, pos=0.125}, draw=none, from=1-2, to=2-3]
% 	\arrow[from=1-3, to=1-4]
% 	\arrow["\vupp"', from=1-3, to=2-3]
% 	\arrow["\lrcorner"{anchor=center, pos=0.125}, draw=none, from=1-3, to=2-4]
% 	\arrow["\vlow", from=1-4, to=2-4]
% 	\arrow["A", from=2-2, to=2-3]
% 	\arrow["{{A^\op}}"', bend right, from=2-2, to=2-4]
% 	\arrow["{{\ulcorner \vupp \urcorner}}", from=2-3, to=2-4]
% \end{tikzcd}
% \]

Suppose $A : X \to \Gamma$ in $\fib_\V$ is a small fibration.
Then $\Op{A} : \Op{X} \to \Op{\Gamma}$ is a small opfibration
and hence has a classifying map 
\[
\ulcorner \Op{A} \urcorner : \Op{\Gamma \to \U}
\]
We define the classifying map for $A$ by
\[ \ulcorner A \urcorner := \Op{(\ulcorner \uupp \urcorner \circ \ulcorner \Op{A} \urcorner)} : \Gamma \to \V\]
It follows that $A : X \to \Gamma$ is isomorphic to the extension
$\Gamma \ext \ulcorner A \urcorner \to \Gamma$ in $\fib(\Gamma)$.
\[ \begin{tikzcd}
    X & {\Gamma \ext \ulcorner A \urcorner} & \Vlow \\
    & \Gamma & \V
    \arrow["\sim", from=1-1, to=1-2]
    \arrow["A"', from=1-1, to=2-2]
    \arrow[from=1-2, to=1-3]
    \arrow[from=1-2, to=2-2]
    \arrow["\lrcorner"{anchor=center, pos=0.125}, draw=none, from=1-2, to=2-3]
    \arrow["\vlow", from=1-3, to=2-3]
    \arrow["{\ulcorner A \urcorner}"', from=2-2, to=2-3]
  \end{tikzcd} \]

\begin{definition}\label{def:ACM-universe}
  An algebraic categorical model (ACM)
  \[(\CC,\opfibcart,\Op,\vOp,\ulow)\]
  consists of an op structure (\Cref{def:ACM})
  and a universe $\ulow : \Ulow \to \U$ satisfying
  Axioms \labelcref{ax:universe,ax:classifying-map,ax:vertical-opposite-small}.
\end{definition}

\section*{$\Sigma$-types}
We fix an ACM $(\CC,\opfibcart,\Op,\vOp,\ulow)$.
Recall the definition of polynomial composition from \Cref{def:pcomp}.
We consider the polynomial composition of $(\ulow : \Ulow \to \U)$
in $\Poly_{(\CC,\fib)}$ with itself,
noting that $\U$ is fibrant, and $\ulow$ is $\fib$-exponentiable.
For notational convenience, we write $Q$ for the domain of the composition.
\[\begin{tikzcd}
	\Ulow & Q & \\
	\U & \fstProj \times_\U \ulow & \Ulow \\
	& {P_\ulow \U} & \U
	\arrow["\ulow"', from=1-1, to=2-1]
	\arrow[from=1-2, to=1-1]
	\arrow["\lrcorner"{anchor=center, pos=0.125, rotate=-90}, draw=none, from=1-2, to=2-1]
	\arrow[from=1-2, to=2-2]
	\arrow["{{\ulow \pcomp \ulow}}"{description, pos=0.2}, shift left=5, bend left, from=1-2, to=3-2]
	\arrow["{{{\sndProj}}}", from=2-2, to=2-1]
	\arrow[from=2-2, to=2-3]
	\arrow[from=2-2, to=3-2]
	\arrow["\lrcorner"{anchor=center, pos=0.125}, draw=none, from=2-2, to=3-3]
	\arrow["\ulow", from=2-3, to=3-3]
	\arrow["{{{\fstProj}}}"', from=3-2, to=3-3]
\end{tikzcd}\]
Similarly, we have the composition $\vlow \pcomp \vlow : Q' \to P_\vlow \V$.

\begin{theorem}\label{thm:sigma-equiv-defs}
  The following are equivalent (propositions)
  \begin{enumerate}
    \item The class of maps $\opfib_\U$ is closed under composition.
    \item $\ulow \pcomp \ulow : Q \to P_\ulow \U$ is a small opfibration.
    \item The class of maps $\fib_\V$ is closed under composition.
    \item $\vlow \pcomp \vlow : Q' \to P_\vlow \V$ is a small fibration.
  \end{enumerate}
  When these conditions hold,
  \Cref{ax:classifying-map} provides maps
  $\Sigma : P_\ulow \U \to \U$
  and $\pair : Q \to \Ulow$ forming a pullback square as follows (and dually).
  \begin{equation}\label{eq:fib-algebraic-sigma}
    \begin{tikzcd}
      Q & \Ulow \\
      {P_\ulow \U} & \U
      \arrow["{\pair}", from=1-1, to=1-2]
      \arrow["{{\ulow \pcomp \ulow}}"', from=1-1, to=2-1]
      \arrow["\lrcorner"{anchor=center, pos=0.125}, draw=none, from=1-1, to=2-2]
      \arrow["\ulow", from=1-2, to=2-2]
      \arrow["{\Sigma}"', from=2-1, to=2-2]
    \end{tikzcd}
    \quad 
    \begin{tikzcd}
      Q' & \Vlow \\
      {P_\vlow \V} & \V
      \arrow["\pair", from=1-1, to=1-2]
      \arrow["{{\vlow \pcomp \vlow}}"', from=1-1, to=2-1]
      \arrow["\lrcorner"{anchor=center, pos=0.125}, draw=none, from=1-1, to=2-2]
      \arrow["\vlow", from=1-2, to=2-2]
      \arrow["\Sigma"', from=2-1, to=2-2]
    \end{tikzcd}
  \end{equation}
\end{theorem}
\begin{proof}
  (1) $\Rightarrow$ (2) is immediate from the definition of $\ulow \pcomp \ulow$.
  (1) $\Leftrightarrow$ (3) is an immediate corollary of
  \Cref{prop:op-preclan-morphism-small}.
  (3) $\Leftrightarrow$ (4) is dual to (1) $\Leftrightarrow$ (2).
  Thus it suffices to prove (2) $\Rightarrow$ (1).
  
  If $d_C : C \to B$ and $d_B : B \to A$
  are two small opfibrations
  then \Cref{ax:classifying-map} provides classifying maps
  \[ \begin{tikzcd}
    \Ulow & C & \\
    \U & B & \Ulow \\
    & A & \U
    \arrow["\ulow"', from=1-1, to=2-1]
    \arrow["c"', from=1-2, to=1-1]
    \arrow["\lrcorner"{anchor=center, pos=0.125, rotate=-90}, draw=none, from=1-2, to=2-1]
    \arrow[from=1-2, to=2-2]
    \arrow["{\ulcorner C \urcorner}", from=2-2, to=2-1]
    \arrow["b", from=2-2, to=2-3]
    \arrow[from=2-2, to=3-2]
    \arrow["\lrcorner"{anchor=center, pos=0.125}, draw=none, from=2-2, to=3-3]
    \arrow["\ulow", from=2-3, to=3-3]
    \arrow["{\ulcorner B \urcorner}"', from=3-2, to=3-3]
  \end{tikzcd} \]
  By applying the universal property of polynomials from
  \Cref{prop:poly-up},
  we have that the composition $d_B \circ d_A : C \to A$ 
  is a pullback of $\ulow \pcomp \ulow$.
  Hence $d_B \circ d_A$ is also a small opfibration.
  \[\begin{tikzcd}
    C & Q \\
    B & {\fstProj \times_\U \ulow} \\
    A & {P_\ulow \U}
    \arrow["{(d_C, c)}", from=1-1, to=1-2]
    \arrow["{d_C}"', from=1-1, to=2-1]
    \arrow["\lrcorner"{anchor=center, pos=0.125}, draw=none, from=1-1, to=2-2]
    \arrow[from=1-2, to=2-2]
    \arrow["{(d_B, b)}", from=2-1, to=2-2]
    \arrow["{d_B}"', from=2-1, to=3-1]
    \arrow["\lrcorner"{anchor=center, pos=0.125}, draw=none, from=2-1, to=3-2]
    \arrow[from=2-2, to=3-2]
    \arrow["{({\ulcorner B \urcorner},{\ulcorner C \urcorner})}"', from=3-1, to=3-2]
  \end{tikzcd}\]
  % Thus $\ulow \pcomp \ulow : Q \to P_\ulow \U$ is classified by
  % $\ulow : \Ulow \to \U$.
  % \Cref{ax:classifying-map} provides maps $\Sigma : P_\ulow \U \to \U$
  % and $\pair : Q \to \Ulow$ forming a pullback square as follows.
\end{proof}
% \Cref{def:vert-cart-ofs} describes the $(\vert,\cart)$ factorisation system on
% $\Poly_{(\CC,\fib)}$.
Like in the theory of Martin-L\"of algebras \cite{awodey2025},
\cref{eq:fib-algebraic-sigma}
is a cartesian morphism of polynomial signatures in ${(\CC,\fib)}$.
In this sense,
the conditions of the theorem say that
the universe $\ulow : \Ulow \to \U$ has $\fib$-algebraic $\Sigma$-types
and the universe $\vlow : \Vlow \to \V$ has $\opfib$-algebraic $\Sigma$-types.

A similar analysis of $\Unit$-types can be given,
which we do not spell out.
We are led to the following axiom.

\begin{axiom}\label{ax:small-preclan}
  The class of maps $\opfib_\U$ is a preclan.
\end{axiom}
It follows that the equivalent conditions of \Cref{thm:sigma-equiv-defs}
hold and the class of maps $\fib_\V$ is a preclan.

In the $\Cat$ model (\Cref{ex:ACM-of-opfibrations}),
we will see that the $\Sigma$-type formation rule
is exactly the \emph{Grothendieck construction} operation.

\begin{definition}\label{def:ACM-universe-sigma}
  When an ACM also satisfies \Cref{ax:small-preclan},
  then we say that the universe admits $\Unit$-types and $\Sigma$-types.
\end{definition}

\section*{$\Pi$-types} 
% \label{sec:cat-types-pi-types}
We fix an ACM $(\CC,\opfibcart,\Op,\vOp,\ulow)$.
Consider the polynomial functor $P_\vlow : \opfib(1) \to \opfib(1)$
for the signature $(\vlow : \Vlow \to \V)$
in $\Poly_{(\CC,\opfib)}$.
We apply $P_\vlow$ to the morphism $\ulow : \Ulow \to \U$ in $\opfib(1)$.
\[ P_\vlow \ulow : P_\vlow \Ulow \to P_\vlow \U\]

\begin{theorem}\label{thm:pi-equiv-defs}
  The following are equivalent (propositions)
  \begin{enumerate}
    \item Small fibrations are small-opfibration-exponentiable. \[\fib_\V \subseteq \Exp_{\opfib_\U}\]
    \item $P_\vlow \ulow : P_\vlow \Ulow \to P_\vlow \U$ is a small opfibration.
    \item Small opfibrations are small-fibration-exponentiable. \[\opfib_\U \subseteq \Exp_{\fib_\V}\]
    \item $P_\ulow \vlow : P_\ulow \Vlow \to P_\ulow \V$ is a small fibration.
    \item $(\Cat,\opfib_\U, \fib_\V)$ is a $\tau$-clan.
  \end{enumerate}
  When these conditions hold,
  \Cref{ax:classifying-map} provides maps
  $\Pi : P_\vlow \U \to \U$ and $\lam : P_\vlow \Ulow \to \Ulow$
  forming the following pullback square (and dually).
\begin{equation}\label{eq:fib-algebraic-pi}
\begin{tikzcd}
	{P_\vlow \Ulow} & \Ulow \\
	{P_\vlow \U} & \U
	\arrow["\lam", from=1-1, to=1-2]
	\arrow["{P_\vlow \ulow}"', from=1-1, to=2-1]
	\arrow["\lrcorner"{anchor=center, pos=0.125}, draw=none, from=1-1, to=2-2]
	\arrow["\ulow", from=1-2, to=2-2]
	\arrow["\Pi"', from=2-1, to=2-2]
\end{tikzcd}
\quad \quad
\begin{tikzcd}
	{P_\ulow \Vlow} & \Vlow \\
	{P_\ulow \V} & \V
	\arrow["\lam", from=1-1, to=1-2]
	\arrow["{P_\ulow \vlow}"', from=1-1, to=2-1]
	\arrow["\lrcorner"{anchor=center, pos=0.125}, draw=none, from=1-1, to=2-2]
	\arrow["\vlow", from=1-2, to=2-2]
	\arrow["\Pi"', from=2-1, to=2-2]
\end{tikzcd}
\end{equation}
\end{theorem}
\begin{proof}
  (1) $\Leftrightarrow$ (3) $\Leftrightarrow$ (5) is the same as \Cref{prop:exponentiability-opposite}.
  (1) $\Leftrightarrow$ (2) and (3) $\Leftrightarrow$ (4) are dual,
  so we just show (1) $\Leftrightarrow$ (2).
  (1) $\Rightarrow$ (2) follows from viewing
  $(\vlow : \Vlow \to \V)$ as
  a signature in $\Poly_{(\CC,\opfib)}$
  and applying \Cref{lem:polynomial-functor-preserves-maps}
  with $\SS = \opfib_\U$ and $\RR = \opfib$.

  We sketch (2) $\Rightarrow$ (1).
  By \Cref{lem:stable-class-of-exponentiable-maps},
  it suffices to show that for a small opfibration $d_C : C \to B$ and
  a small fibration $d_B : B \to A$,
  there is a small opfibration $d_D : D \to A$
  satisfying the universal property of the pushforward of $d_C$ along $d_B$.
  Again, \Cref{ax:classifying-map} provides classifying maps
  \[\begin{tikzcd}
    C & \Vlow && B & \Ulow \\
    B & \V && A & \U
    \arrow[from=1-1, to=1-2]
    \arrow[from=1-1, to=2-1]
    \arrow["\lrcorner"{anchor=center, pos=0.125}, draw=none, from=1-1, to=2-2]
    \arrow["\vlow", from=1-2, to=2-2]
    \arrow[from=1-4, to=1-5]
    \arrow[from=1-4, to=2-4]
    \arrow["\lrcorner"{anchor=center, pos=0.125}, draw=none, from=1-4, to=2-5]
    \arrow["\ulow", from=1-5, to=2-5]
    \arrow["{{\ulcorner C \urcorner}}"', from=2-1, to=2-2]
    \arrow["{{\ulcorner B \urcorner}}"', from=2-4, to=2-5]
  \end{tikzcd}\]
  By applying the universal property of polynomials from \Cref{prop:poly-up},
  we have a map
  $(\ulcorner B \urcorner, \ulcorner C \urcorner) : A \to P_\vlow \U$.
  The map $d_D : D \to A$ is given by taking the pullback of
  $P_\vlow \ulow : P_\vlow \Ulow \to P_\vlow \U$ along this map.
  \[ \begin{tikzcd}[column sep = large]
      D & {P_\vlow \Ulow} \\
      A & {P_\vlow \U}
      \arrow[from=1-1, to=1-2]
      \arrow[from=1-1, to=2-1]
      \arrow["\lrcorner"{anchor=center, pos=0.125}, draw=none, from=1-1, to=2-2]
      \arrow["{P_\vlow \ulow}", from=1-2, to=2-2]
      \arrow["{(\ulcorner B \urcorner, \ulcorner C \urcorner)}"', from=2-1, to=2-2]
    \end{tikzcd} \]
\end{proof}

\begin{axiom}\label{ax:pi}
  The equivalent conditions of \Cref{thm:pi-equiv-defs} hold.
\end{axiom}

Since $P_\vlow \U$ is in $\opfib(1) = \fib(1)$,
the map $(P_\vlow \ulow : P_\vlow \Ulow \to P_\vlow \U)$ is also a
polynomial signature in $(\CC,\fib)$,
making \cref{eq:fib-algebraic-pi} a cartesian morphism of signatures.
Thus, the universe $\ulow : \Ulow \to \U$ has $\fib$-algebraic $\Pi$-types 
and the universe $\vlow : \Vlow \to \V$ has $\opfib$-algebraic $\Pi$-types.

\begin{definition}\label{def:ACM-universe-pi}
  When an ACM also satisfies \Cref{ax:pi},
  then we say that the universe admits $\Pi$-types.
\end{definition}

\section*{Hom-types}
% \label{sec:cat-types-hom-types}
We fix an ACM $(\CC,\opfibcart,\Op,\vOp,\ulow)$.
In this section,
we axiomatise a directed version of algebraic identity types
\cite{awodey2018, awodey2025},
also drawing from the analysis of algebraic path-types in
\cite{awodey2026}.
Before doing so,
we require additional external structure on $\CC$.

\begin{axiom} \label{ax:core}
  The pseudofunctor $\opfibcart(-) : \CC^\op \to \tCat$
  for opfibrations admits the following \emph{core} structure.
  \begin{itemize}
  \item We have a pseudotransformation
    \[\vCore{} : \opfibcart(-) \to \opfibcart(-)\]
  \item We have two modifications
    \[ \inv : \vCore{} \to \vOp{}
    \quad \text{and} \quad
    \inc : \vCore{} \to \id_{\opfibcart(-)} \]
  % \item The whiskerings
  % \[\inv \vCore{} : \vCore{} \vCore{} \to \vOp{} \vCore{}\]
  % and
  % \[\inc \vCore{} : \vCore{} \vCore{} \to \vCore{}\]
  % are isomorphisms.
  % We will write $\cml : \vCore{} \to \vCore{}\vCore{}$
  % for the inverse of $\inc \vCore{}$.
  \end{itemize}
\end{axiom}
For an opfibration $A \to X$,
the opfibration $\vCore[X]{A} \to X$
is the \emph{vertical core} of $A$ over $X$.
In $\Cat$, the vertical core of $A$ can be thought of as
the category where,
over each point $x$ in $X$,
the fibre is the groupoid core of $A_x$.

\begin{axiom} \label{ax:twist}
  The pseudofunctor $\opfibcart(-) : \CC^\op \to \tCat$
  for opfibrations admits the following \emph{twisted} structure.
  \begin{itemize}
  \item We have a pseudotransformation 
    \[\vTw{} : \opfibcart(-) \to \opfibcart(-)\]
  \item We have three modifications
    \[ \src : \vTw{} \to \vOp{}
      \quad \quad
      \trg : \vTw{} \to \id_{\opfibcart(-)}  \]
    \[ \rfl : \vCore{} \to \vTw{} \]
  \item These modifications satisfy two equations
    $\src \circ \rfl = \inv$ and $\trg \circ \rfl = \inc$.
    \[
      \begin{tikzcd}
        & {\vOp{}} \\
        {\vCore{}} & {\vTw{}} \\
        & {\id_{\opfibcart(-)}}
        \arrow["\inv", from=2-1, to=1-2]
        \arrow["\rfl"{description}, from=2-1, to=2-2]
        \arrow["\inc"', from=2-1, to=3-2]
        \arrow["\src"', from=2-2, to=1-2]
        \arrow["\trg", from=2-2, to=3-2]
      \end{tikzcd}
    \]
  \end{itemize}
  We call $\vTw[X]{A}$ the \emph{(vertical) twisted arrows} of $A$ over $X$.
\end{axiom}

For any object $X \in \CC$ and any opfibration $A \in \opfibcart(X)$ over it,
we have a \emph{twisted pathobject factorisation} of the (restricted) diagonal
$\Delta : \vCore[X]{A} \to \vOp[X]{A} \times_X A$.
\[\begin{tikzcd}
	{\vCore[X] A} & {\vTw{A}} \\
	& {\vOp[X]{A} \times_X A}
	\arrow["\rfl", from=1-1, to=1-2]
	\arrow["{\Delta := (\inv,\inc)}"', from=1-1, to=2-2]
	\arrow["{(\src,\trg)}", from=1-2, to=2-2]
\end{tikzcd}\]

\begin{axiom}\label{ax:src-trg-opfibration}
  For each $X \in \CC$ and $A \in \opfibcart(X)$,
  the map $(\src,\trg) : \vTw{A} \to \vOp[X]{A} \times_X A$
  is an opfibration.
\end{axiom}

Applying $\vCore[U]{}$ and $\vTw[U]{}$ to the universe of small opfibrations
$\ulow : \Ulow \to \U$, we have
\[ \begin{tikzcd}
	\Usim &&& \Utw \\
	\U &&& \U
	\arrow["{\usim \, := \, \vCore[U]{\ulow}}", from=1-1, to=2-1]
	\arrow["{\utw \, := \, \vTw[U]{\ulow}}", from=1-4, to=2-4]
\end{tikzcd} \]
% Then the modifications
% $\src : \Tw{} \to \vOp{}$ and $\trg : \Tw{} \to \id_{\opfibcart(-)}$
% By \Cref{ax:src-trg-opfibration}
% we have the following 
% universal small twisted arrow type.
We consider the universal small twisted pathobject factorisation.
\[\begin{tikzcd}
	\Usim & \Utw \\
	& {\Uupp \times_\U \Ulow}
	\arrow["\rfl", from=1-1, to=1-2]
	\arrow["\Delta"', from=1-1, to=2-2]
	\arrow["{{(\src,\trg)}}", from=1-2, to=2-2]
\end{tikzcd}\]

\begin{proposition}\label{prop:src-trg-small-opfibration}
  The following are equivalent
  \begin{itemize}
    \item For each $X \in \CC$ and small opfibration $A \in \opfibcart_\U(X)$,
      the map $(\src,\trg) : \vTw{A} \to \vOp[X]{A} \times_X A$
      is a small opfibration.
    \item $(\src,\trg) : \Utw \to \Uupp \times_\U \Ulow$
      is a small opfibration.
  \end{itemize}
  When these conditions hold,
  \Cref{ax:classifying-map} provides maps
  $\Dir : \Uupp \times_\U \Ulow \to \U$
  and $\dir : \Utw \to \Ulow$ such that
  the following square is a pullback.
  \[
    \begin{tikzcd}
      \Utw & \Ulow \\
      {\Uupp \times_\U \Ulow} & \U
      \arrow["\dir", from=1-1, to=1-2]
      \arrow["{(\src, \trg)}"', from=1-1, to=2-1]
      \arrow["\ulow", from=1-2, to=2-2]
      \arrow["\Dir"', from=2-1, to=2-2]
      \arrow["\lrcorner"{anchor=center, pos=0.125}, draw=none, from=1-1, to=2-2]
    \end{tikzcd}
  \]
\end{proposition}
\begin{proof}
  As usual, every operation in $\opfibcart_\U(X)$
  is a pullback of a universal one in $\opfibcart_\U(\U)$.
\end{proof}

This gives us a generalisation of the condition (A1) in \cite{awodey2026},
captured in the following axiom.
\begin{axiom}\label{ax:hom-formation-introduction}
  The equivalent conditions of
  \Cref{prop:src-trg-small-opfibration} hold.
\end{axiom}

From \Cref{ax:hom-formation-introduction},
we obtain formation and introduction rules for hom-types.
We now turn our attention to the hom-type elimination rule.

\begin{definition}
  We define the pseudotransformation
  $\rtw{} : \opfibcart(-) \to \opfibcart(-)$
  such that for each $X$ in $\CC$
  and object $A$ in $\opfibcart(X)$,
  $\rtw[X]{A}$ in $\opfibcart(X)$
  is the pullback in the functor category.
  Indeed, by \Cref{ax:src-trg-opfibration}
  the pullback
  $(\src,\trg) : \rtw[X]{A} \to \vCore[X]{A} \times_X A$
  exists and is an opfibration.
  \[
    \begin{tikzcd}
      {\rtw[X]{A}} & {\vTw[X]{A}} \\
      {\vCore[X]{A} \times_X A} & {\vOp[X]{A} \times_X A}
      \arrow[dashed, from=1-1, to=1-2]
      \arrow["{(\src,\trg)}"', dashed, from=1-1, to=2-1]
      \arrow["\lrcorner"{anchor=center, pos=0.125}, draw=none, from=1-1, to=2-2]
      \arrow["{(\src,\trg)}", from=1-2, to=2-2]
      \arrow["{\inv \times_X A}"', from=2-1, to=2-2]
    \end{tikzcd}
  \]
We call $\rtw[X] A$ the \emph{right-twisted arrows} of $A$ over $X$.
\end{definition}
We also obtain a \emph{right-twisted pathobject factorisation},
by the pullback property of $\rtw[X]{A}$.
\[\begin{tikzcd}
	{\vCore[X] A} & {\rtw[X]{A}} \\
	& {\vCore[X]{A} \times_X A}
	\arrow["\rfl", from=1-1, to=1-2]
	\arrow["{\Delta := (\id,\inc)}"', from=1-1, to=2-2]
	\arrow["{(\src,\trg)}", from=1-2, to=2-2]
\end{tikzcd}\]
There is, of course, also $\ltw[X] A$ the \emph{left-twisted arrows} of $A$ over $X$,
given by instead taking the pullback along the map
\[\vOp[X]{A} \times_A \inc : \vOp[X]{A} \times_X \vCore[X] A \to {\vOp[X]{A} \times_X A}\]

Applying $\rtw[\U]$ to $\ulow : \Ulow \to \U$
produces a universal small right-twisted arrow type
$\Urtw := \rtw[\U]{\Ulow}$ in $\opfibcart(\U)$
with $\urtw : \Urtw \to \U$
and an opfibration $(\src,\trg) : \Urtw \to \Usim \times_\U \Ulow$.
  \[ \begin{tikzcd}[column sep=huge]
    	\Usim & \Urtw & \Utw \\
      & {\Usim \times_\U \Ulow} & {\Uupp \times_\U \Ulow}
      \arrow["\rfl", from=1-1, to=1-2]
      \arrow["\Delta"', from=1-1, to=2-2]
      \arrow[from=1-2, to=1-3]
      \arrow["{{(\src,\trg)}}"{description}, from=1-2, to=2-2]
      \arrow["\lrcorner"{anchor=center, pos=0.125}, draw=none, from=1-2, to=2-3]
      \arrow["{{(\src,\trg)}}", from=1-3, to=2-3]
      \arrow["{{\inv \times_\U \Ulow}}"', from=2-2, to=2-3]
    \end{tikzcd} \]
Furthermore, $(\src,\trg) : \Urtw \to \Usim \times_\U \Ulow$
is a small opfibration by \Cref{ax:hom-formation-introduction}.
Consider the following commuting triangle.
\[ \begin{tikzcd}
	\Usim & \Urtw \\
	& \U
	\arrow["\rfl", from=1-1, to=1-2]
	\arrow["{\usim}"', from=1-1, to=2-2]
	\arrow["\urtw", from=1-2, to=2-2]
\end{tikzcd} \]
Both $(\usim : \Usim \to \U)$ and $(\urtw : \Urtw \to \U)$ are
polynomial signatures in $\Poly_{(\CC,\FF)}$,
since they are both opfibrations.
Using \Cref{def:vert-trans},
this triangle gives rise to a vertical natural transformation
between the polynomial functors.
\[P_{\rfl} : P_{\urtw} \to P_{\usim}\]
We consider the naturality square of $P_{\rfl}$
for the morphism $\ulow : \Ulow \to \U$.
\Cref{ax:hom-elimination}
will make this square a ``weak pullback''.
\[ \begin{tikzcd}
	{ P_{\urtw} \Ulow} & {P_{\usim} \Ulow} \\
	{ P_{\urtw} \U} & {P_{\usim} \U}
	\arrow["{{P_{\rfl} \Ulow}}", from=1-1, to=1-2]
	\arrow["{{ P_{\urtw} \ulow}}"', from=1-1, to=2-1]
	\arrow["{{P_{\usim} \ulow}}", from=1-2, to=2-2]
	\arrow["{{P_{\rfl} \U}}"', from=2-1, to=2-2]
\end{tikzcd} \]
We would like to compare $P_{\urtw}$ with an actual pullback
of the cospan in question.
Unfortunately, the other axioms do not automatically provide
such a pullback, so we postulate one.

\begin{axiom}\label{ax:pullback-L}
  There is a pullback $L$,
  with maps $\phi$ and $\psi$ as depicted in the following diagram.
  \begin{equation}
    \label{eq:hom-elimination-2}
    \begin{tikzcd}
    { P_{\urtw} \Ulow} && \\
    & L & {P_{\usim} \Ulow} \\
    & { P_{\urtw} \U} & {P_{\usim} \U}
    \arrow["p", dashed, from=1-1, to=2-2]
    \arrow["{P_{\rfl} \Ulow}", bend left, from=1-1, to=2-3]
    \arrow["{P_{\urtw} \ulow}"', bend right, from=1-1, to=3-2]
    \arrow["\phi", from=2-2, to=2-3]
    \arrow["\psi"', from=2-2, to=3-2]
    \arrow["\lrcorner"{anchor=center, pos=0.125}, draw=none, from=2-2, to=3-3]
    \arrow["{P_{\usim} \ulow}", from=2-3, to=3-3]
    \arrow["{P_{\rfl} \U}"', from=3-2, to=3-3]
  \end{tikzcd}
  \end{equation}
  We denote the comparison map $p : P_{\urtw} \Ulow \to L$.  
\end{axiom}
By the universal property of polynomials
and commutativity of the square,
$\psi$ and $\phi$ are equivalent to the data of three maps
\[\psi_1 = \phi_1 : L \to \U
\quad \psi_2 : \rtw[L]{A} \to \U
\quad \phi_2 : \vCore[L]{A} \to \Ulow\]
where $A := L \opext \psi_1 \to L$ is the opfibration over $L$
given by pullback of $\ulow : \Ulow \to \U$ along $\psi_1 : L \to \U$.
The latter two maps fit into a commutative square
\begin{equation}\label{eq:hom-elimination-1}
  \begin{tikzcd}
	\Ulow & {\vCore[L]{A}} & \Usim \\
	\U & {\rtw[L]{A}} & \Urtw \\
	& L & \U
	\arrow["\ulow"', from=1-1, to=2-1]
	\arrow["{\phi_2}"', from=1-2, to=1-1]
	\arrow[from=1-2, to=1-3]
	\arrow["{\rfl}", from=1-2, to=2-2]
	\arrow["\lrcorner"{anchor=center, pos=0.125}, draw=none, from=1-2, to=2-3]
	\arrow["\rfl", from=1-3, to=2-3]
	\arrow["{\usim}", bend left = 50, from=1-3, to=3-3]
	\arrow["{\psi_2}", from=2-2, to=2-1]
	\arrow[from=2-2, to=2-3]
	\arrow[from=2-2, to=3-2]
	\arrow["\lrcorner"{anchor=center, pos=0.125}, draw=none, from=2-2, to=3-3]
	\arrow["\urtw", from=2-3, to=3-3]
	\arrow["{\psi_1}"', from=3-2, to=3-3]
\end{tikzcd}
\end{equation}

\begin{proposition}\label{prop:hom-elimination-equiv}
  The following are equivalent (structures)
  \begin{itemize}
    \item A section $j : L \to P_{\urtw} \Ulow$
    of the comparison map $p : P_{\urtw} \Ulow \to L$,
    from \cref{eq:hom-elimination-2}.
    \[ p \circ j = \id\]
    This is akin to \cite[eq. (12)]{awodey2025}
    used in the elimination principle for algebraic identity types.
    \item A diagonal lift $j : \rtw[L]{A} \to \Ulow$
    as indicated in the following diagram,
    extracted from \cref{eq:hom-elimination-1}.
    \[\begin{tikzcd}
        {\vCore[L]{A}} & \Ulow \\
        {\rtw[L]{A}} & \U
        \arrow["{{\phi_2}}", from=1-1, to=1-2]
        \arrow["{{\rfl}}"', from=1-1, to=2-1]
        \arrow["\ulow", from=1-2, to=2-2]
        \arrow["j"{description}, dashed, from=2-1, to=1-2]
        \arrow["{{\psi_2}}"', from=2-1, to=2-2]
      \end{tikzcd}\]
    \item A diagonal lift $j' : \rtw[L]{A} \to C$
    as indicated in the following diagram.
    \[\begin{tikzcd}
        {\vCore[L]{A}} & C & \Ulow \\
        {\rtw[L]{A}} & {\rtw[L]{A}} & \U
        \arrow["r", dashed, from=1-1, to=1-2]
        \arrow["{{{\phi_2}}}", bend left, from=1-1, to=1-3]
        \arrow["{{{\rfl}}}"', from=1-1, to=2-1]
        \arrow[from=1-2, to=1-3]
        \arrow[from=1-2, to=2-2]
        \arrow["\lrcorner"{anchor=center, pos=0.125}, draw=none, from=1-2, to=2-3]
        \arrow["\ulow", from=1-3, to=2-3]
        \arrow["{j'}"{description}, dashed, from=2-1, to=1-2]
        \arrow[equals, from=2-1, to=2-2]
        \arrow["{{{\psi_2}}}"', from=2-2, to=2-3]
      \end{tikzcd}\]
  \end{itemize}
\end{proposition}
\begin{proof}
  (1) $\Leftrightarrow$ (2) follows from the universal property of polynomial functors.
  (2) $\Leftrightarrow$ (3) is clear.
\end{proof}

The equivalent conditions of \Cref{prop:hom-elimination-equiv}
can be presented using an informal syntax as follows.

\begin{prooftree}
% \AxiomC{$\Gamma \vdash A : U$}
\AxiomC{$\Gamma \opext x : \vCore A \opext y : A \opext h : \Hom_A (\inv x, y) \vdash C(x,y,h) : U$}
\noLine
\UnaryInfC{$\Gamma \opext x : \vCore A \vdash r : C(x,\inc x, \rfl x)$}
\UnaryInfC{$\Gamma \opext x : \vCore A \opext y : A \opext h : \Hom_A (\inv x, y) \vdash j_{(C,r)}(x,y,h) : C$}
\noLine
\UnaryInfC{$\Gamma \opext x : \vCore A \vdash j_{(C,r)}(x,\inc x, \rfl x) \equiv r : C(x,\inc x, \rfl x)$}
\end{prooftree}

This is essentially the syntax from \cite{north2019},
slightly simplified due to the presence of $\Pi$-types.
Rather than providing a formal syntax,
we will reformulate \Cref{prop:hom-elimination-equiv}
in ``unalgebraic'' terms in \Cref{def:elementary-hom-type},
as was done in \Cref{sec:hottlean}.

\begin{axiom}\label{ax:hom-elimination}
  We have one of the equivalent structures from \Cref{prop:hom-elimination-equiv}.
  Also, we have the analogous structure for the left-twisted arrows $\ltw$,
  which we do not spell out.
\end{axiom}

% Like in \cite{hua2026},
% \Cref{ax:hom-elimination} could be unfolded into an ``elementary'' condition instead:
% for all objects $\Gamma$, all maps $A : \Gamma \to \U$,
% all maps $C : \Gamma \opext x : \vCore A \opext y : A \opext \Hom_A (\inv x, y) \to \U$
% and maps $r : \Gamma \opext x : \vCore A \to \Ulow$
% such that $\ulow \circ r = C \circ RRRRR$,
% there is a map $j : \Gamma \opext x : \vCore A \opext y : A \opext \Hom_A (\inv x, y) \to \Ulow$
% satisfying a typing rule $\ulow \circ j = C$,
% a computation rule $j \circ RRRRR = r$,
% as well as a substitution stability condition.

\begin{definition}\label{def:ACM-universe-hom}
  When an ACM also satisfies Axioms
  \labelcref{ax:core,,ax:twist,,ax:src-trg-opfibration,,ax:hom-formation-introduction,ax:pullback-L,ax:hom-elimination},
  then we say that the universe admits $\Hom$-types.
\end{definition}

\section{Unalgebraic categorical models}

% Rather than presenting a syntax to interpret into an ACM,
% which comes with its own set of difficulties,
% we provide a more syntax reformulation of an ACM,
% called an elementary categorical model (UCM) (following \cite{hua2026}).
% The idea is that when constructing a model,
% one would construct it as an ACM,
% whereas one would target an UCM when interpreting syntax into a model.
Rather than presenting a syntax to interpret into an ACM,
we provide an ``unalgebraic'' reformulation of the definitions in algebraic categorical models
that is more syntactic in style.
Like in \Cref{sec:hottlean},
we will introduce ``unalgebraic categorical models''.
The idea is that one would typically construct models as ACMs,
as ACMs are closer to how they appear ``in the wild''.
At the same time, defining an interpretation of the syntax into the model
is straightforward using an UCM.
An appropriate translation from one to the other is therefore useful.
This approach also avoids getting into the intricacies of designing a well-behaved syntax
for dependent type theory,
which is a challenge in and of itself.

\begin{definition}
  An unalgebraic categorical model (UCM) 
  \[(\CC,\U,\V,\ulow,\uupp,\usim,\vlow,\vupp,\vsim,\inc,\inv)\]
  consists of the following.
  \begin{itemize}
  \item A category $\CC$ with a terminal object.
  \item Three universes $\ulow : \Ulow \to \U, \uupp:\Uupp \to \U$ and $\usim : \Usim \to \U$ in $\CC$,
    in the sense of \Cref{def:elementary-universe}.
  \item Three universes $\vlow : \Vlow \to \V, \vupp:\Vupp \to \V$ and $\vsim : \Vsim \to \V$ in $\CC$.
  \item Morphisms $\inv : \usim \to \uupp$ and $\inc : \usim \to \ulow$ over $\U$, and
    $\inv : \vsim \to \vupp$ and $\inc : \vsim \to \vlow$ over $\V$.
    \[
    \begin{tikzcd}
      \Uupp & \Usim & \Ulow \\
      & \U
      \arrow["\uupp"', from=1-1, to=2-2]
      \arrow["\inv"', from=1-2, to=1-1]
      \arrow["\inc", from=1-2, to=1-3]
      \arrow["\usim"{description}, from=1-2, to=2-2]
      \arrow["\ulow", from=1-3, to=2-2]
    \end{tikzcd}
    \quad \quad 
    \begin{tikzcd}
      \Vupp & \Vsim & \Vlow \\
      & \V
      \arrow["\vupp"', from=1-1, to=2-2]
      \arrow["\inv"', from=1-2, to=1-1]
      \arrow["\inc", from=1-2, to=1-3]
      \arrow["\vsim"{description}, from=1-2, to=2-2]
      \arrow["\vlow", from=1-3, to=2-2]
    \end{tikzcd}
    \]
  \end{itemize}
\end{definition}
Like in \Cref{sec:UnstrSem},
an object of $\CC$ will be called a \emph{context} and a morphism will be called a substitution.
A map $A : \Gamma \to \U$ will be called a \emph{covariant type},
and a map $A : \Gamma \to \V$ will be called a \emph{contravariant type}.
Each universe admits its own kind of context extension.
For a covariant type $A : \Gamma \to \U$,
let us introduce the suggestive notation for context extensions
$\Gamma \opext (x : A)$ and $\Gamma \opext (x : \vOp A)$, and $\Gamma \opext (x : \vCore A)$
respectively for $\ulow, \uupp$ and $\usim$.
Similarly, we denote the context extensions for $\vlow, \vupp, \vsim$
respectively as $\Gamma \ext A$, $\Gamma \ext \vOp A$, and $\Gamma \ext \vCore A$
(or with explicit variables).

\begin{definition}[Unalgebraic $\Hom$-types] \label{def:elementary-hom-type}
  Let $(\CC,\U,\V,\dots)$ be an UCM.
  Then an \emph{unalgebraic $\Hom$-type structure}
  on $\U$ consists of the following.
  \begin{enumerate}
  \item
    For any context $\Gamma$, covariant type $A : \Gamma \to \U$,
    maps $a_0 : \Gamma \to \Uupp$ and
    $a_1 : \Gamma \to \Ulow$ such that
    $\uupp \circ a_0 = A$ and $\ulow \circ a_1 = A$,
    there is a type $\Hom_A (a_0,a_1) : \Gamma \to \U$.
  \item
    The construction $\Hom_A (a_0,a_1)$ is stable under substitution,
    meaning that for any $\sigma : \Delta \to \Gamma$ we have
    \[ \Hom_{A \circ \sigma} (\sigma \gg a_0, \sigma \gg a_1)
    = \sigma \gg \Hom_A (a_0,a_1) \]
  \item
    For any context $\Gamma$, covariant type $A : \Gamma \to \U$,
    and map $a : \Gamma \to \Usim$ such that
    $\usim \circ a = A$,
    there is a map $\rfl a : \Gamma \to \Ulow$
    satisfying
    $\ulow \circ \rfl a  = \Hom_A (\inv \circ \, a,\inc \circ \, a)$.
  \item The construction $\rfl$ is stable under substitution,
    meaning that for any $\sigma : \Delta \to \Gamma$ we have
    \[ \rfl (\sigma \circ a) = \sigma \circ \rfl a \]
  \item
    Given a covariant type $A : \Gamma \to \U$,
    one can construct the context 
    \[\Gamma \opext \vCore A \opext A \opext \Hom_A := \Gamma \opext (x : \vCore A) \opext (y : A) \opext \Hom_A(\inv \circ \, x,y)\]
    and the substitution
    \[\rho : \Gamma \opext (x : \vCore A) \to \Gamma \opext (x : \vCore A) \opext (y : A) \opext \Hom_A(\inv \circ \, x,y)\]
    \[\rho := \id_{(\Gamma \opext \vCore A)} \opext (\inc \circ \, x) \opext \rfl x\]
    For any \emph{motive} \[C : \Gamma \opext \vCore A \opext A \opext \Hom_A \to \U\]
    and map $r : \Gamma \opext \vCore A \to \Ulow$ satisfying
    \[ \ulow \circ r = C \circ \rho \]
    there is a map, called the \emph{right eliminator},
    $j (C, r) : \Gamma \opext \vCore A \opext A \opext \Hom_A \to \Ulow$
    satisfying $\ulow \circ j(C,r) = C$ and
    $j (C, r) \circ \rho  = r$
    \[\begin{tikzcd}
	   \Gamma \opext \vCore A & \Ulow \\
	   {\Gamma \opext \vCore A \opext A \opext \Hom_A} & \U
	   \arrow["{r}", from=1-1, to=1-2]
	   \arrow["{\rho}"', from=1-1, to=2-1]
	   \arrow["\ulow", from=1-2, to=2-2]
	   \arrow["{j(C,r)}"{description}, dashed, from=2-1, to=1-2]
	   \arrow["C"', from=2-1, to=2-2]
    \end{tikzcd}\]
  \item The construction $j$ is stable under substitution.
  \item
    Dually, given a covariant type $A : \Gamma \to \U$,
    one can construct the context 
    \[\Gamma \opext \vOp A \opext \vCore A \opext \Hom_A :=
    \Gamma \opext (x : \vOp A) \opext (y : \vCore A) \opext \Hom_A(x,\inc \circ \, y)\]
    and the substitution
    \[\rho' : \Gamma \opext (x : \vCore A) \to \Gamma \opext \vOp A \opext \vCore A \opext \Hom_A\]
    \[\rho' := \id_{\Gamma} \opext (\inv \circ \, x) \opext x \opext \rfl x\]
    For any \emph{motive} \[C : \Gamma \opext \vOp A \opext \vCore A \opext \Hom_A \to \U\]
    and map $r : \Gamma \opext \vCore A \to \Ulow$ satisfying
    \[ \ulow \circ r = C \circ \rho' \]
    there is a map, called the \emph{left eliminator},
    $j' (C, r) : \Gamma \opext \vOp A \opext \vCore A \opext \Hom_A \to \Ulow$
    satisfying $\ulow \circ j'(C,r) = C$ and
    $j' (C, r) \circ \rho'  = r$
    \[\begin{tikzcd}
	   \Gamma \opext \vCore A & \Ulow \\
	   {\Gamma \opext \vOp A \opext \vCore A \opext \Hom_A} & \U
	   \arrow["{r}", from=1-1, to=1-2]
	   \arrow["{\rho'}"', from=1-1, to=2-1]
	   \arrow["\ulow", from=1-2, to=2-2]
	   \arrow["{j'(C,r)}"{description}, dashed, from=2-1, to=1-2]
	   \arrow["C"', from=2-1, to=2-2]
    \end{tikzcd}\]
  \item The construction $j'$ is stable under substitution.
  \end{enumerate}
\end{definition}

\begin{theorem}
  Given an algebraic categorical model with a universe that has
  $\Unit$-types, $\Sigma$-types, $\Pi$-types, and $\Hom$-types,
  there is an unalgebraic categorical model $(\CC,\U,\V,\dots)$
  with a $\vlow$-indexed {unalgebraic $\Pi$-type structure} on $\ulow$ (\Cref{def:elementary-pi}),
  a $\ulow$-indexed {unalgebraic $\Pi$-type structure} on $\vlow$,
  an {unalgebraic $\Sigma$-type structure} on $\ulow$ (\Cref{def:elementary-sigma}),
  an {unalgebraic $\Sigma$-type structure} on $\vlow$,
  an {unalgebraic $\Hom$-type structure} on $\U$,
  and an {unalgebraic $\Hom$-type structure} on $\V$.
\end{theorem}
\begin{proof}
  The proof of this is routine;
  similar proofs can be found in \cite{awodey2018,awodey2025,hua2026}.
\end{proof}

\section{Examples}

Models of Martin-L\"of type theory are degenerate categorical algebraic models.
\begin{example}[Models of Martin-L\"of type theory]
  Suppose $\CC$ is an algebraic model of Martin-L\"of type theory
  on a $\pi$-clan $(\CC,\RR)$.
  Let us take $\opfib := \RR$,
  and $\Op : \CC \to \CC$ to be the identity,
  so that we have $\RR = \opfib = \fib$.
  The $\tau$-clan condition in \Cref{ax:exponentiability}
  follows from $\RR$ being a $\pi$-preclan.
  We also set $\opfibcart := \RR$ so that all maps in the slice are
  morphisms of opfibrations.
  We then take $\vOp : \RR(-) \to \RR(-)$ to be the identity map.
  This results in an op structure, which we will call $\MM_\RR$.
  
  It follows that an ACM extending this op structure (\Cref{def:ACM-universe}) is the same thing as
  an ``$\RR$-algebraic universe'' in an algebraic model of MLTT (\Cref{def:algebraic-universe}).
  Let us assume now that we have such a universe $\uu : \Ulow \to \U$ in $\MM_\RR$.
  Then $\uu : \Ulow \to \U$ admits $\Unit$-types and $\Sigma$-types as an ACM (\Cref{def:ACM-universe-sigma})
  if and only if it does so in the MLTT model (\Cref{def:alg-unit-type-mltt,def:alg-sig-type-mltt}),
  and $\uu : \Ulow \to \U$ admits $\Pi$-types as an ACM (\Cref{def:ACM-universe-pi})
  if and only if it does so in the MLTT model (\Cref{def:alg-pi-type-mltt}).

  For $\Hom$-types,
  let us suppose that the universe in the MLTT model admits $\Path$-types
  (in the sense of \Cref{def:mltt-alg-path-types-2}).
  Take $\vCore : \RR(-) \to \RR(-)$ to be identity
  and the twisted arrows map $\vTw : \RR(-) \to \RR(-)$
  to be the path functor $\arr : \RR(-) \to \RR(-)$ (\Cref{def:path-functor}).
  Assuming that the universe admits a normal Hurewicz structure (\Cref{def:hurewicz-defs}),
  then the universe admits $\Id$-types (\Cref{thm:id-elim}),
  which coincide with $\Hom$-types.
\end{example}

There are several different algebraic categorical models in $\Cat$ with 
$\Unit$-types, $\Sigma$-types, $\Pi$-types, and $\Hom$-types.
One is with covariant types as split opfibrations.
In the following,
we will present this model in a way that uses the axiom of choice,
namely to satisfy \Cref{ax:classifying-map},
the axiom for having classifying maps.
It should be possible to reformulate the general theory of algebraic models
in such a way that avoids choice.
However, we will leave these constructive considerations for future work.

\begin{example}[$\Cat$ and the universe of opfibrations]\label{ex:ACM-of-opfibrations}
  \Cref{ax:opfibration-preclan}.
  Disregarding size issues that can be resolved by postulating
  more Grothendieck universes,
  we can take $\CC$ to be the 1-category of large categories $\Cat$.
  We take $\opfib$ to be the class of split opfibrations (\Cref{def:split-opfibration}).
  These together form a preclan (\Cref{prop:preclans-in-cat}).

  \Cref{ax:global-op} ($\Op$).
  Of course, taking opposite categories is an endofunctor on $\Cat$.
  The class of maps $\fib$ defined above is exactly the class of split fibrations.

  \Cref{ax:exponentiability}.
  We showed that $\Cat$ with opfibrations and fibrations forms a $\tau$-clan in \Cref{prop:pushforward-in-cat}.

  \Cref{ax:opcartesian-map} (morphisms of opfibrations).
  For each $X$, $\opfibcart(X)$ has morphisms as morphisms of
  split opfibrations over $\CC$ \Cref{def:morphism-split-opfibration}.
  For any category $\CC$, the map $\CC \to 1$ is a split opfibration,
  and the lift of the identity in $1$ is unique -
  since split opfibrations are also {normal}.
  It follows that any functor is also a
  split opcartesian map over $1$.

  \Cref{ax:local-op} ($\vOp$).
  Adjusting \Cref{thm:opfibration-main} for size,
  we have for any category $X$,
  \[\opfibcart(X) \simeq [X,\Cat]\]
  Postcomposing with the endofunctor $\Op : \Cat \to \Cat$
  yields a functor
  \[ \vOp[X] : \opfibcart(X) \to \opfibcart(X) \]
  as described in \Cref{def:relativised-constructions},
  which was for \emph{small} opfibrations.

  {\bf So far, we have constructed an op structure on $\Cat$.
  Let us call this op structure $\MM_\Cat$.}

  \Cref{ax:universe} (universe of small opfibrations).
  The universal small split opfibration in $\Cat$
  \[\ulow : \pcat \to \cat\]
  is an opfibration.
  This is defined in \Cref{def:universal-small-split-opfibration}.
  All objects are fibrant.

  \Cref{ax:classifying-map} (classifying maps).
  This follows from the axiom of choice.
  % Every small split opfibration $A \to X$ in $\Cat$ has a classifying map
  % $A^* : X \to \cat$, which takes a point in $X$ to its fibre, which is a category.
  % The mapping
  % \[A \mapsto A^* : \smlopfibcart(X) \to [X,\cat]\]
  % is one of the functors making up the equivalence
  % described in \Cref{thm:opfibration-main}.
  % \[[X,\cat] \simeq \smlopfibcart(X)\]

  \Cref{ax:vertical-opposite-small}.
  The first condition of \cref{prop:vertical-opposite-small} follows from
  the fact that straightening and unstraightening preserve size.
  \[
  \begin{tikzcd}
    {[A,\cat]} & {\smlopfibcart(A)} \\
    {[A,\Cat]} & {\opfibcart(A)}
    \arrow["\simeq"{description}, draw=none, from=1-1, to=1-2]
    \arrow[hook, from=1-1, to=2-1]
    \arrow[hook, from=1-2, to=2-2]
    \arrow["\simeq"{description}, draw=none, from=2-1, to=2-2]
  \end{tikzcd} 
  \]
  {\bf Thus $\MM_\Cat$ and $\ulow : \pcat \to \cat$ together form an ACM}.
  
  \Cref{ax:small-preclan} ($\Unit$-types and $\Sigma$-types).
  $\Cat$ with the class of small split opfibrations
  $(\Cat,\smlopfib)$
  is a preclan (\Cref{prop:preclans-in-cat}).
  In this case, given a small category $A$ and a functor $B : A \to \cat$,
  the point $(A,B)$ in $P_{\ulow} \U$ is taken by
  $\Sigma : P_{\ulow} \U \to \U$
  to the \emph{Grothendieck construction} $\int_A B$ in $\U$.

  \Cref{ax:pi} ($\Pi$-types).
  We showed that $\fib_V = \smlfib \subseteq \Exp_{\opfib_\U} = \Exp_{\smlopfib}$
  in \Cref{prop:pushforward-in-cat}.  
 
  \Cref{ax:core} ($\vCore$) and \Cref{ax:twist} ($\vTw$).
  The relativised twisted arrow structure for small split opfibrations
  is described in \Cref{def:relativised-constructions}.
  The same construction works for large split opfibrations.

  \Cref{ax:src-trg-opfibration} and \Cref{ax:hom-formation-introduction}
  ($\Hom$-type formation and introduction).
  We verified in \Cref{prop:local-src-trg-discrete-opfibration}
  that when $A \to X$ is a split opfibration,
  $\vTw[X] A \to \vOp[X] A \times_X A$ is a discrete opfibration,
  and therefore a split opfibration,
  and that this preserves size.

  \Cref{ax:pullback-L}.
  $\Cat$ has all 1-limits, giving us the required pullback $L$.

  \Cref{ax:hom-elimination} ($\Hom$ elimination).
  It is easiest to check the second condition in
  \Cref{prop:hom-elimination-equiv},
  namely that we have a lift $j : {\vCore[L]{A}} \to \Ulow$
  for the square
  \[\begin{tikzcd}
      {\vCore[L]{A}} & \Ulow \\
      {\rtw[L]{A}} & \U
      \arrow["{{\phi_2}}", from=1-1, to=1-2]
      \arrow["{{\rfl}}"', from=1-1, to=2-1]
      \arrow["\ulow", from=1-2, to=2-2]
      \arrow["{{\psi_2}}"', from=2-1, to=2-2]
    \end{tikzcd}\]
  Recall from \Cref{thm:lari-opfib-awfs} that
  we have an AWFS where the
  right maps are split opfibrations, and the left maps are lari.
  Of course, $\ulow : \Ulow \to \U$ is a split opfibration.
  By \Cref{prop:rtw-equiv-K},
  ${\rfl} : {\vCore[L]{A}} \to {\rtw[L]{A}}$ is lari.
  Hence, the AWFS provides a lift $j : {\vCore[L]{A}} \to \Ulow$, as required.

  {\bf In summary, the op structure $\MM_\Cat$ has a universe --
  the universe of split opfibrations -- with
  $\Unit$-types, $\Sigma$-types, $\Pi$-types, and $\Hom$-types.}
\end{example}

Next, we continue working with the op structure $\MM_\Cat$,
but we switch the universe of opfibrations
for the universe of \emph{groupoidal opfibrations}.
% the third has covariant types as discrete opfibrations

\begin{example}[$\Cat$ and the universe of groupoidal opfibrations]\label{ex:groupoidal-opfibrations}
  A split groupoidal opfibration is a
  split opfibration with fibres that are groupoids.
  Taking the pullback of $\ulow : \pcat \to \cat$ along the inclusion
  of the category of small groupoids
  $\grpd \to \cat$ produces the universes
  of \emph{small groupoidal opfibrations}.
  \[\quad \quad \pgrpd \to \grpd\]
  Although (large) groupoidal opfibrations also form preclans,
  these preclans are not suitable candidates for $\opfib$,
  as $\grpd$ is not itself a groupoid (so it is not in $\opfibcart(1)$).
  Hence, we still use the (large non-groupoidal)
  opfibrations as the ambient preclan $\opfib$.

  {\bf The universe of small split groupoidal opfibrations in the op structure $\MM_\Cat$ admits
  $\Unit$-types, $\Sigma$-types, $\Pi$-types, and $\Hom$-types.}
  For $\Hom$-types, the $\vOp[]{}$ and $\vCore[]{}$ operations can be taken to be the identity.
  \Cref{ax:hom-formation-introduction} holds for both of these universes;
  as noted above $\vTw[X] A \to \vOp[X] A \times_X A$
  is always a discrete opfibration.
  \Cref{ax:hom-elimination} for the universes of groupoidal opfibrations
  is a corollary of the case for general opfibrations.
  Moreover, as we will see in \Cref{prop:lari-orthogonal-groupoidal-opfibration},
  the diagonal filler $j$ in \Cref{prop:hom-elimination-equiv}
  is unique up to unique isomorphism.
\end{example}

\begin{example}[$\Cat$ and the universe of discrete opfibrations]\label{ex:discrete-opfibrations}
  Recall that a discrete opfibration is a
  split opfibration with fibres that are sets.
  Taking the pullback of $\ulow : \pcat \to \cat$ along the inclusion
  of the categories of small sets
  $\set \to \cat$ produces the universe
  of \emph{small discrete opfibrations}.
  \[\pset \to \set\]
  One can also check that the universe of small discrete opfibrations (in the op structure $\MM_\Cat$) admits
  $\Unit$-types, $\Sigma$-types, $\Pi$-types, and $\Hom$-types.
\end{example}

  % It is also straightforward to check that
  % $\Cat$ with the class of small groupoidal opfibrations
  % $(\Cat,\opfib_{\grpd})$
  % and the class of small discrete opfibrations
  % $(\Cat,\opfib_{\grpd})$
  % are preclans.
  % It is also straightforward to check that
  % discrete opfibrations are preserved under pushforward along
  % discrete fibrations
  % and groupoidal opfibrations are preserved under pushforward along
  % groupoidal fibrations.
  % \[ \fib_{\set} \subseteq \Exp_{\opfib_{\set}}
  %   \quad \quad \text{and} \quad \quad
  %   \fib_{\grpd} \subseteq \Exp_{\opfib_{\grpd}}\] 
\section{The Yoneda lemma}
We conclude our analysis of the model in $\Cat$ by
providing some further insight into identity elimination.
We saw in \Cref{prop:right-twisted-factorisation-equiv}
that the right twisted arrow factorisation 
\[\Core{A} \to \rTw{A} \to \Core{A} \times A\]
is the comprehensive factorisation of the restricted diagonal
$\Core{A} \to \Core{A} \times A$
into an initial functor followed by a discrete opfibration.
We also saw that it was the algebraic factorisation
of the restricted diagonal into a lari functor
followed by a split opfibration.
It turns out to also be determined by the $\Grpd$-enriched
comprehensive factorisation system on $\Cat$,
which lives ``in between'' the two.
\[ \begin{tikzcd}
	{\text{lari}} & \pitchfork & {\text{split opfibration}} \\
	{1\text{-initial}} & {\bot_{g}} & {\text{groupoidal opfibration}} \\
	{\text{initial}} & \bot & {\text{discrete opfibration}}
	\arrow[from=1-1, to=2-1]
	\arrow[from=2-1, to=3-1]
	\arrow[from=2-3, to=1-3]
	\arrow[from=3-3, to=2-3]
\end{tikzcd} \]
Here we use $\pitchfork$ to indicate that lifts are not unique,
$\bot_g$ when lifts are unique up to unique isomorphism
(unique up to a contractible groupoid),
and $\bot$ when lifts are unique.
The best reference we can find for this result is a comment by Kelly in
\cite[4.7]{kelly1982}.
However, we will only need the following weaker result,
which avoids an explicit description of the 1-initial functors.

\begin{proposition}\label{prop:lari-orthogonal-groupoidal-opfibration}
  Lari functors left-lift against groupoidal opfibrations,
  and those lifts are unique up to unique isomorphism.
\end{proposition}
\begin{proof}
  Consider a diagram of the form
  \[ \begin{tikzcd}
	\AA & \CC \\
	\BB & \DD
	\arrow["G", from=1-1, to=1-2]
	\arrow["L"', from=1-1, to=2-1]
	\arrow["F", from=1-2, to=2-2]
	\arrow["H"', from=2-1, to=2-2]
    \end{tikzcd}\]
  where $(F,l)$ is a groupoidal opfibration and
  $L \vdash R$ is lari.
  We construct a diagonal filler $\phi : \BB \to \CC$.
  Let $b$ be an object in $\BB$.
  Taking the counit of the adjunction $L \vdash R$ at $b$,
  we obtain a lifting problem in $\CC$ over $\DD$.
  \[\begin{tikzcd}
	GRb & {\phi(b)} \\
	HLRb & Hb
	\arrow["l", dashed, from=1-1, to=1-2]
	\arrow["{H \epsilon_b}"', from=2-1, to=2-2]
    \end{tikzcd}\]
  Lifting this using opfibrancy provides an object $\phi(b)$ over $H b$.
  The functorial action of $\phi$ is then determined by
  opcartesianness of the lift $l$.
  The triangle below commutes by definition of $\phi(b)$.
  \[ \begin{tikzcd}
	& \CC \\
	\BB & \DD
	\arrow["F", from=1-2, to=2-2]
	\arrow["\phi", from=2-1, to=1-2]
	\arrow["H"', from=2-1, to=2-2]
    \end{tikzcd}\]  
  To show that the upper triangle commutes,
  \[\begin{tikzcd}
	\AA & \CC \\
	\BB
	\arrow["G", from=1-1, to=1-2]
	\arrow["L"', from=1-1, to=2-1]
	\arrow["\phi"', from=2-1, to=1-2]
    \end{tikzcd}\]
  consider an object $a$ in $\AA$.
  The counit $\epsilon_{La} : L R L a \to L a$
  is the inverse of the unit,
  which is the identity.
  \[\begin{tikzcd}
	{L a} & {L R L a}
	\arrow["{=}", from=1-1, to=1-2]
	\end{tikzcd}\]
  Hence $\epsilon_{La}$ is also the identity.
  Since $(F,l)$ is normal, the lift is also the identity
  \[ \phi L a = GRLa = G a\]
  Hence $\phi$ is a lift for the diagram.
  Any other lift $\phi : \BB \to \CC$ applied to the counit of
  an object $b$ provides a map $\psi(\epsilon_b) : \psi L R b = G R b \to \psi b$
  which also lives over $H \epsilon_b$.
  By opcartesianness of the lift $l$,
  there is a unique morphism $u : \phi b \to \psi b$
  in the fibre over $H b$, which is a groupoid.
  Hence $\phi \iso \psi$ and that isomorphism is unique.
\end{proof}

Now consider a category $A$ with an object $x$.
Taking the pullback of the right twisted arrow category
along a point $x$ in the core,
we obtain the coslice under $x$.
\[
\begin{tikzcd}
	1 & {\Core{A}} \\
	{x /A} & {\rTw{A}} \\
	A & {\Core{A} \times A} \\
	1 & {\Core{A}}
	\arrow["x", from=1-1, to=1-2]
	\arrow["{\rfl x}"', from=1-1, to=2-1]
	\arrow["\rfl", from=1-2, to=2-2]
	\arrow[from=2-1, to=2-2]
	\arrow["\trg"', from=2-1, to=3-1]
	\arrow["\lrcorner"{anchor=center, pos=0.125}, draw=none, from=2-1, to=3-2]
	\arrow["{(\src,\trg)}", from=2-2, to=3-2]
	\arrow[from=3-1, to=3-2]
	\arrow[from=3-1, to=4-1]
	\arrow["\lrcorner"{anchor=center, pos=0.125}, draw=none, from=3-1, to=4-2]
	\arrow["{\pi_1}", from=3-2, to=4-2]
	\arrow["x"', from=4-1, to=4-2]
\end{tikzcd}
\]
Note that since $(\src, \trg) : \rTw A \to (\src, \trg)$
is a discrete opfibration,
so is $\trg : x / A \to A$.
Now suppose we have a groupoidal opfibration $d : C \to A$
with a point $r : 1 \to C$ making the following square commute.
\[ \begin{tikzcd}
	1 & C \\
	{x / A} & A
	\arrow["r", from=1-1, to=1-2]
	\arrow["x"{description}, from=1-1, to=2-2]
	\arrow["{\rfl x}"', from=1-1, to=2-1]
	\arrow["d", from=1-2, to=2-2]
	\arrow["\trg"', from=2-1, to=2-2]
\end{tikzcd} \]
By \Cref{prop:lari-orthogonal-groupoidal-opfibration},
since $\rfl x$ is lari and $d : C \to A$
is a groupoidal opfibration,
there is a diagonal filler $\phi : x / A \to C$.
In fact, since $x / A$ and $C$
are groupoidal fibrations over $A$
(and all functors between groupoidal fibrations
in the slice are morphisms of fibrations)
there is an equivalence of groupoids
\[ \opfibcart(A)(x / A, C) \simeq \{r : 1 \to C \st d r = x\} \]
where both sides have the obvious natural transformations as morphisms.
After straightening (\Cref{thm:opfibration-main}),
the more familiar Yoneda lemma emerges:

\begin{proposition}
  For a category $A$, an object $x \in A$,
  and a functor $C : A \to \Grpd$,
  there is an equivalence of categories
  \[ [A,\Grpd](y (x), C^*) \simeq C^*(x)\]
  where $y(x) : A \to \Grpd$ is given pointwise by the discrete groupoid
  $y(x)(a) := \Hom (x,a)$.
\end{proposition}
\begin{proof}
  Applying \Cref{thm:opfibration-main},
  and noting that $y(x) : A \to \Grpd$ corresponds to the coslice $x / A \in \opfibcart(A)$,
  we have 
  \[ [A,\Grpd](y (x), C^*) \simeq
  \opfibcart(A)(x / A, C) \simeq \{r : 1 \to C \st d r = x\}
  = C^*(x) \]
\end{proof}

\begin{comment}
  DONE
  - British spelling
  - Grammar
  - check Bibliography
  - check [?] references and citations
  - line breaks
  - overfull boxes/warnings in general

  TODO

  - dedication
  
  - Paul Taylor Partial products, what's going on there?
  - [cat] Pushforward opfibration by split opfibration
  - [exponentiable] Geometric morphism not BC
  - [pathtypes] twisted cylinder as a modality

  Give up
  - [cat] Gambino+Larrea prop 4.6 + lari stable under pullback -> pushforwards along fibrations preserve opfibrations
  - [exponentiable] Condensed example

  Unaddressed comments from Steve
  - Maybe the op actually generates everything, so it’s like adding inverses or negative values to an algebra? 
\end{comment}

\bibliography{main.bib}
\bibliographystyle{alpha}
\end{document}